\documentclass[10pt]{amsart}
\usepackage{amsmath}
\usepackage{amssymb}
\usepackage{amsthm}
\usepackage{tikz}
\usetikzlibrary{decorations.markings}
\usepackage{graphicx}
\usepackage{hyperref}
\usepackage{fullpage}
\usepackage{float}
\usepackage{caption}
\usepackage{xcolor}
\usepackage{color}
\usepackage[color,all]{xy}
\usepackage{comment}
\usepackage{boxhandler}

\newtheorem{thm}{Theorem}[section]
\newtheorem{prop}[thm]{Proposition}
\newtheorem{cor}[thm]{Corollary}
\newtheorem{lem}[thm]{Lemma}
\newtheorem{constr}[thm]{Construction}

\theoremstyle{definition}
\newtheorem{defn}[thm]{Definition}
\newtheorem{rmk}[thm]{Remark}
\newtheorem{exmp}[thm]{Example}

\newcommand{\bC}{\mathbb{C}}
\newcommand{\bD}{\mathbb{D}}
\newcommand{\bG}{\mathbb{G}}
\newcommand{\D}{\mathbb{D}}
\newcommand{\bH}{\mathbb{H}}
\newcommand{\bR}{\mathbb{R}}

\newcommand{\bZ}{\mathbb{Z}}
\newcommand{\bP}{\mathbb{P}}
\newcommand{\ww}{\mathfrak{w}}
\newcommand{\wv}{\mathfrak{v}}
\renewcommand{\tt}{\mathfrak{t}}
\newcommand{\rank}{\mathrm{rank}}
\newcommand{\Gr}{\mathrm{Gr}}
\newcommand{\GL}{\mathrm{GL}}

\newcommand{\cA}{\mathcal{A}}
\newcommand{\cX}{\mathcal{X}}
\newcommand{\cP}{\mathcal{P}}
\newcommand{\cI}{\mathcal{I}}
\newcommand{\Br}{\mathrm{Br}}

\newcommand{\dM}{\mathfrak{M}^{\vee}}

\newcommand{\cL}{\mathcal{L}}
\newcommand{\cF}{\mathcal{F}}
\newcommand{\Fl}{\mathrm{Fl}}
\newcommand{\bW}{\mathbb{W}}
\newcommand{\Span}{\mathrm{Span}}
\newcommand{\bN}{\mathbb{N}}
\newcommand{\uf}{\mathrm{uf}}
\newcommand{\fr}{\mathrm{fr}}
\newcommand{\tw}{\mathrm{tw}}
\newcommand{\DT}{\mathrm{DT}}
\newcommand{\Conf}{\mathrm{Conf}}
\newcommand{\dconf}{\mathrm{Conf}^d}
\newcommand{\id}{\mathrm{Id}}
\newcommand{\doublearrow}[1]{\overset{#1}{\implies}}
\newcommand{\tNec}{\overrightarrow{\cI}}
\newcommand{\sNec}{\overleftarrow{\cI}}
\newcommand{\tI}[1]{\overrightarrow{I_{#1}}}
\newcommand{\sI}[1]{\overleftarrow{I_{#1}}}
\newcommand{\Pio}{\Pi^{\circ}}
\def\bdmn{\mathrm{Bd}(m,n)}

\newcommand{\bGshift}{\mathbb{G}^{\downarrow}}
\def\tcF{\tilde{\cF}}

\def\modsp{\mathfrak{M}}
\def\frmodsp{\mathfrak{M}}

\DeclareMathOperator{\tog}{tog}

\newcommand{\veq}{\mathrel{\rotatebox{90}{$=$}}}

\newcommand{\dd}{\partial}
\newcommand{\sse}{\subset}
\newcommand{\lr}{\longrightarrow}
\newcommand{\la}{\lambda}
\newcommand{\La}{\Lambda}

\newcommand{\stacknumber}[2]{\ \genfrac{}{}{0pt}{1}{#1}{#2}\ }

\newcommand{\MSB}[1]{{\color{blue}[MSB: #1]}}
\newcommand{\ITL}[1]{{\color{teal}[ITL: #1]}}
\newcommand{\DW}[1]{{\color{red}[DW: #1]}}

\newtheorem{mainthm}{Theorem}

\def\empver[#1]#2(#3)#4(#5,#6)%
{
\coordinate (#3) at (#5,#6);
\draw [#1] (#3) circle [radius=0.2];
}

\def\solver[#1]#2(#3)#4(#5,#6)%
{
\coordinate (#3) at (#5,#6);
\draw [#1,fill=#1] (#3) circle [radius=0.2];
}

\def\edge[#1]#2(#3)#4(#5)%
{
\draw[#1, shorten <=0.2cm, shorten >=0.2cm] (#3) -- (#5);
}

\newcommand{\inprod}[2]{\left\langle #1, #2\right\rangle}

\colorlet{lightblue}{blue!30!white}
\colorlet{lightred}{red!30!white}
\colorlet{lightteal}{teal!30!white}

\title[Weaves for reduced plabic graphs]{Demazure weaves for reduced plabic graphs\\\vspace{0.5cm} {\bf\footnotesize{-- with a proof that Muller-Speyer twist is Donaldson-Thomas --}}}

\author{Roger Casals}
	\address{University of California Davis, Dept. of Mathematics, CA, USA}
	\email{casals@math.ucdavis.edu}

\author{Ian Le}
\address{Australian National University, Mathematical Sciences Institute, Australia}
\email{ian.le@anu.edu.au}

\author{Melissa Sherman-Bennett}
\address{Massachusetts Institute of Technology, Dept. of Mathematics, MA, USA}
\email{msherben@mit.edu}

\author{Daping Weng}
	\address{University of California Davis, Dept. of Mathematics, CA, USA}
	\email{dweng@ucdavis.edu}

\begin{document}

\maketitle

\begin{abstract}
First, this article develops the theory of weaves and their cluster structures for the affine cones of positroid varieties. In particular, we explain how to construct a weave from a reduced plabic graph, show it is Demazure, compare their associated cluster structures, and prove that the conjugate surface of the graph is Hamiltonian isotopic to the Lagrangian filling associated to the weave. The T-duality map for plabic graphs has a surprising key role in the construction of these weaves. Second, we use the above established bridge between weaves and reduced plabic graphs to show that the Muller-Speyer twist map on positroid varieties is the Donaldson-Thomas transformation. This latter statement implies that the Muller-Speyer twist is a quasi-cluster automorphism. An additional corollary of our results is that target labeled seeds and the source labeled seeds are related by a quasi-cluster transformation.
\end{abstract}

\setcounter{tocdepth}{1}
\tableofcontents

\section{Introduction}

The object of this article will be to develop the connection between weaves and reduced plabic graphs and show that the Muller-Speyer twist map on positroid strata is the Donaldson-Thomas transformation. First, we explain how to use $T$-duality to associate a weave to each reduced plabic graph, prove it is Demazure, and show that the cluster seeds on the corresponding positroid strata coincide. That is, the cluster structure associated to weaves via the quiver of relative Lusztig cycles and their microlocal merodromies is the same as that constructed from the reduced plabic graph quiver and its Pl\"ucker coordinates. Second, we use this relation between weaves and reduced plabic graphs to show that the Muller-Speyer twist map coincides with the cluster-theoretic Donaldson-Thomas transformation on positroid strata. In particular, we establish that the Muller-Speyer twist is a quasi-cluster automorphism. An added corollary of this result is that the target labelled seed and the source labelled seed of a reduced plabic graph are related by a quasi-cluster automorphism. For these results, we use that weaves allow for a geometric understanding of the Donaldson-Thomas transformation as induced by symmetries of Legendrian links in space.
\subsection{Scientific Context}\label{ssec:scientific_context} Positroid strata are algebraic subvarieties of Grassmannians, first appearing in the study of total positivity, see \cite{Lusztig98},\cite{Rietsch06}, \cite[Section 3]{Pos} and \cite[Section 5]{KLS}. From a combinatorial standpoint, they can be labeled via reduced plabic graphs \cite[Section 11]{Pos}. Reduced plabic graphs themselves can be used to identify coordinate rings of positroid strata with cluster algebras; see \cite{GL,SSBW}, \cite[Chapter 7]{FWZ} and references therein. This generalizes the recipe in \cite{Scott06} for the case of top-dimensional positroid strata.\\

\noindent Recently, weaves \cite{CZ} have also been shown to provide a versatile alternative for the construction of cluster algebras for certain varieties, including positroid strata; see \cite{CGGLSS,CW} and references therein. A core contribution of this article is to establish the relation between weaves and reduced plabic graphs, in particular comparing the cluster structures they respectively induce on positroid strata. A surprising piece of this relation is the appearance of the T-duality map on plabic graphs from \cite{Galashin21_Critical,PSBW}, as presented in Section \ref{sec:iterative_Tmap}, which plays a key role in the construction of the weave associated to a reduced plabic graph. The precise bridge between weaves and reduced plabic graphs is summarized in Theorem \ref{thm:mainA} below.\\

\begin{rmk} Weaves can describe seeds for positroid strata which have non-Pl\"ucker cluster variables, in contrast to reduced plabic graphs. To wit, given a weave corresponding to a plabic graph and the associated seed $\Sigma$, \emph{any} mutation of $\Sigma$ can again be described with a weave. Nevertheless, only mutations of $\Sigma$ at square faces of the plabic graph can again be described with plabic graphs. Figure \ref{fig:ConjugateSurface10} gives an instance of such a seed, described by a weave but not by a plabic graph.\hfill$\Box$
\end{rmk}

\begin{center}
	\begin{figure}
		\centering
		\includegraphics[scale=1]{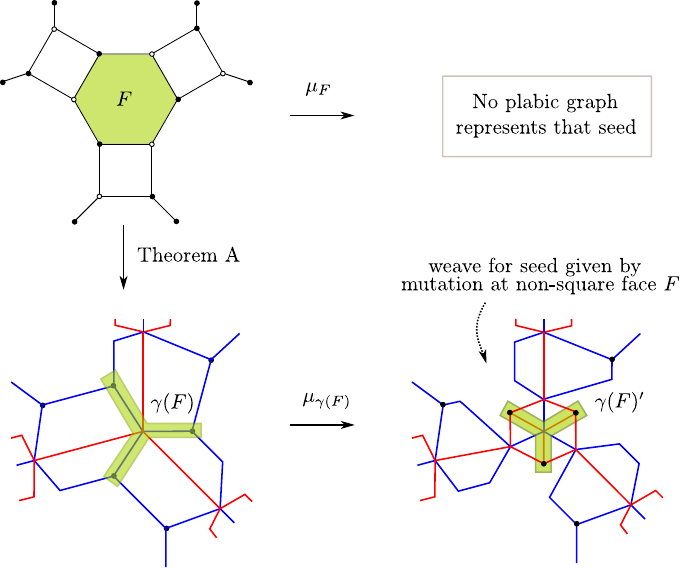}
		\caption{(Upper left) A reduced plabic graph $\bG$ for $\Gr(3,6)$. Label the boundary vertices $1, 2, \dots, 6$ clockwise, with the upper right boundary vertex labeled 1. The cluster variable for the hexagonal face $F$ is $\Delta_{135}$. (Lower left) The weave $\ww(\bG)$ corresponding to this plabic graph using Theorem \ref{thm:mainA}. The highlighted Lusztig cycle $\gamma(F)$ corresponds to the face $F$ via Theorem \ref{thm:mainA}.(ii). The seed $\Sigma(\ww(\bG))$ is equal to $\Sigma_T(\bG)$, the target seed of the plabic graph. (Lower right) The weave $\mu_{\gamma(F)}(\ww(\bG))$ obtained from mutating the weave $\ww(\bG)$ at $\gamma(F)$, which corresponds to a mutation of $\Sigma_T(\bG)$ \emph{not} represented by a plabic graph. The new cycle $\gamma(F)'$, corresponding to the vertex where we have mutated, is highlighted. The cluster variable associated to $\gamma(F)'$ is $\Delta_{136}\Delta_{245} - \Delta_{126} \Delta_{345}$, which is not a Pl\"ucker coordinate. 
  }
		\label{fig:ConjugateSurface10}
	\end{figure}
\end{center}

Second, once this bridge in Theorem \ref{thm:mainA} has been established, we use it to relate to two automorphisms that have appeared in the study of positroid strata. Indeed, the cluster algebra associated to a positroid stratum has two remarkable automorphisms:

\begin{itemize}
    \item[(i)] The Muller-Speyer twist map, introduced in \cite{MullerSpeyertwist} building on work of \cite{MarshScott16} for the top-dimensional case. This is a biregular automorphism of the positroid variety and thus an automorphism of its coordinate ring of functions. Prior to the present manuscript, it was {\it not} known to be a quasi-cluster automorphism. Theorem \ref{thm:mainB} below will show that it is a quasi-cluster automorphism.\\

    \item[(ii)] The Donaldson-Thomas transformation, defined using the cluster structure. In our context, this is an automorphism of the coordinate ring of the positroid variety. Combinatorially, it is equivalent to the existence of a reddening sequence for any quiver associated to a positroid. The existence of such reddening sequences for a quiver of positroids is proven in \cite{FordSerhiyenko18}. By construction, Donaldson-Thomas transformations are quasi-cluster automorphisms.\\
\end{itemize}

\noindent One of our main results, Theorem \ref{thm:mainB}, will show that these two automorphisms of the coordinate ring of (the affine cone of) a positroid stratum coincide. A significant consequence is that the Muller-Speyer twist map must be a quasi-cluster automorphism, since it equals the Donaldson-Thomas transformation.\\

In the context of weaves, \cite[Section 5]{CW} provides a neat geometric interpretation for the Donaldson-Thomas transformation: it is obtained by a series of Reidemeister moves on a link composed with a reflection. This particular weave description of the Donaldson-Thomas transformation allows for a comparison between the Muller-Speyer twist map and the Donaldson-Thomas transformation. Theorem \ref{thm:mainA} and the theory developed around it are independently valuable on their own. Then Theorem \ref{thm:mainB} and Corollary \ref{cor:quasiequivalence} provide a new application of the combinatorics of weaves to cluster algebras. Note that, even if inspired by symplectic topology, our proof of Theorem \ref{thm:mainB} is entirely within the realm of algebraic combinatorics and cluster algebras: it does not use any symplectic topology and it is developed directly using weaves and flags understood as combinatorial and Lie-theoretic objects.\\

\subsection{Main Results}\label{ssec:main_results} Let $\bG$ be a reduced plabic graph. In \cite[Section 8.2]{PSBW} an operation is defined: it inputs a plabic graph $\bG$ and it outputs a plabic graph $\bGshift$, the \emph{$T$-shift} of $\bG$, which is of smaller rank. Iterating this operation yields a sequence of plabic graphs $(\bG^{\downarrow})^{\downarrow},((\bG^{\downarrow})^{\downarrow})^{\downarrow},\ldots$, which eventually hits a trivial plabic graph. This is a finite sequence of plabic graphs, which we call the $T$-shifted reduced plabic graphs of $\bG$.\\

Following \cite[Section 2]{CZ}, a weave $\ww$ on a disk is combinatorially described as an ordered collection of (often many) trivalent graphs satisfying certain conditions on their overlaps. The intersection of $\ww$ with the boundary of the disk gives a positive braid word $\dd\ww$. Demazure weaves were introduced in \cite[Section 4]{CGSS20} and their Lusztig cycles were defined in \cite[Section 4]{CGGLSS}. The articles \cite{CGGLSS,CW} explain how to use weaves and their cycles to show that the coordinate rings of certain affine varieties associated to $\dd\ww$ are cluster algebras, with each weave giving a toric chart and each Lusztig cycle providing a cluster variable in that chart. These affine varieties are closely related to the following moduli space; if $\beta=\dd\ww= s_{i_1} s_{i_2} \cdots s_{i_{l(\beta)}}$ is a positive braid word on $m$-strands, then we define
$$\modsp(\beta):=\left\{(\cF_0,\cF_1,\dots, \cF_{l(\beta)}=\cF_0)\in (\Fl_m)^{l(\beta)} :  \cF_{j-1}\xrightarrow{s_{i_j}}\cF_j \mbox{ for all }1\leq j\leq l(\beta) \right\}/\GL_m(\bC),$$
where $\Fl_m$ is the variety of complete linear flags in $\bC^m$, $\GL_m(\bC)$ acts diagonally, and we have written $\cF\xrightarrow{s_{k}}\mathcal{G}$ if the flags $\cF,\mathcal{G}$ differ at precisely the $k$-dimensional subspace. 
A subset of points $\mathfrak{t}\sse\beta$ of the braid $\beta$, understood as a link diagram, leads to a similarly defined moduli $\modsp(\beta;\mathfrak{t})$: this latter moduli is an affine variety if $\mathfrak{t}$ contains at least one point per component of $\beta$; see \cite{CW,CGGLSS}. Different choices of decorations $\mathfrak{t}$ lead to different (if related) smooth affine varieties $\modsp(\beta;\mathfrak{t})$, such as the braid varieties introduced in \cite{CGSS20,CGGLSS}. In this manuscript, we explain how the affine cone of any positroid variety arises as $\modsp(\beta;\mathfrak{t})$ for a particular $\beta$ and choice of decoration $\mathfrak{t}$. The results from \cite{CW,CGGLSS} then endow such $\modsp(\beta;\mathfrak{t})$, and thus (affine cones of) positroid strata, with cluster structures. In precise terms, the results from \cite{CW,CGGLSS} show that the ring of regular functions $\bC[\modsp(\beta;\mathfrak{t})]$ is equal to a cluster algebra, with a specific initial seed constructed from the braid word $\beta$. Initial seeds for these cluster structures are given by Demazure weaves and the cluster variables are microlocal merodromies along their Lusztig cycles. The general algebro-combinatorial description of the cluster variables is given in \cite{CGGLSS} and the symplectic geometric meaning of the cluster variables is explained in \cite{CW}. See also \cite{CG23} for results towards explicitly describing all cluster seeds by these symplectic geometric methods.\\

Independently, reduced plabic graphs can also be used to endow the (affine cone of) positroid strata with cluster structures, see \cite[Chapter 7]{FWZ}. Initial seeds for these cluster structures are given by plabic graphs and cluster variables are (certain subsets of) Pl\"ucker coordinates, indexed by the faces of the plabic graph; see \cite{GL,MullerSpeyertwist}. The first main result of this article is the construction of a weave for any given reduced plabic graph such that all the combinatorics of the weave (including its quiver) match with the combinatorics of the given plabic graph, and the cluster structure constructed using the weave coincides with that constructed using reduced plabic graphs. The precise statement reads as follows:

\begin{mainthm}\label{thm:mainA}
Let $\bG$ be a reduced plabic graph. Then the union $\ww(\bG)$ of all the $T$-shifted reduced plabic graphs of $\bG$ is a weave such that

\begin{itemize}
    \item[(i)] The Lagrangian projection of the Legendrian surface associated to $\ww(\bG)$ is Hamiltonian isotopic to an exact Lagrangian representative of the conjugate surface of $\bG$.\\

    \item[(ii)] $\ww(\bG)$ is a Demazure weave and its Lusztig cycles are in natural bijection with the faces of $\bG$. In particular, the quiver of cycles associated to $\ww(\bG)$ coincides with the quiver of $\bG$.\\

    \item[(iii)] There exists a set of marked points $\mathfrak{t}$ in the positive braid $\beta=\dd\ww(\bG)$ such that the moduli $\mathfrak{M}(\beta;\mathfrak{t})$ is isomorphic to the affine cone of the positroid stratum associated to $\bG$.\\

    \item[(iv)] The cluster seed on $\mathfrak{M}(\beta;\mathfrak{t})$ associated to the weave $\ww(\bG)$ coincides with the cluster seed associated to the reduced plabic graph $\bG$ under the isomorphism of (iii). That is, the microlocal merodromies along the Lusztig relative cycles pull-back to the Pl\"ucker coordinates given by $\bG$, including frozens.\\
\end{itemize}

\noindent Furthermore, a variation on the moduli $\mathfrak{M}(\beta;\mathfrak{t})$, which is a cluster $\mathcal{A}$-scheme, leads to a Poisson $\mathcal{X}$-scheme $\dM(\beta;\mathfrak{t})$ such that the pair $(\mathfrak{M}(\beta;\mathfrak{t}),\dM(\beta;\mathfrak{t}))$ forms a cluster ensemble.
\hfill$\Box$
\end{mainthm}

In addition to cluster and combinatorial statements, Theorem \ref{thm:mainA} contains statements that are symplectic geometric in nature, starting with Theorem \ref{thm:mainA}.(i) and the use of microlocal merodromies in Theorem \ref{thm:mainA}.(iv). That said, we build on our previous works \cite{CL22,CW,CZ} so as to translate any symplectic geometric aspects into accessible combinatorics. For instance, Theorem \ref{thm:mainA}.(i) is proven using the table of diagrammatic moves in Figure \ref{fig:TableMovesIntro} below, which was established in \cite[Theorem 3.1]{CL22}. In this manner, it is our hope that readers not yet acquainted with symplectic geometry can still follow all the arguments for Theorem \ref{thm:mainA} if they are willing to rely on the geometric results from \cite{CL22,CW,CZ}.\\

\noindent Theorem \ref{thm:mainA}.(i) is proved in Theorem~\ref{thm:weave_and_conjugate} using the diagrammatic moves for hybrid surfaces established in \cite{CL22}, as explained above. Theorem \ref{thm:mainA}.(ii) is proved in Theorem~\ref{thm:weave construction}, Theorem \ref{thm:Demazure weave} and Corollary~\ref{cor:Lusztig cycles}. Note that most weaves are not Demazure, most $Y$-cycles are not Lusztig cycles, nor all Lusztig cycles are typically $Y$-trees, so this result asserts that $\ww(\bG)$ is particularly well-behaved. 
Theorem \ref{thm:mainA}.(iii) is proved in Theorem~\ref{thm: frame sheaf moduli = positroid} and deals with the subtle aspect of marked points: the set of marked points $\mathfrak{t}$ featured there is {\it not} the same as the set used in \cite{CGSS20,CGGLSS,CW}, e.g. for braid varieties. (Indeed, those choices would lead to the wrong frozen part for the quivers.) Finally, Theorem \ref{thm:mainA}.(iv), proved in Proposition~\ref{prop:merodromy-is-plucker} requires finding appropriate relative cycles Poincar\'e dual to the Lusztig cycles, then computing the microlocal merodromies, in line with \cite[Section 4]{CW}, and finally showing that they indeed coincide with the corresponding Pl\"ucker coordinates. The statement about the dual space $\dM(\beta;\mathfrak{t})$ is proven in Theorem~\ref{thm:cluster-ensemble}.\\

The correspondence in Theorem \ref{thm:mainA} can be used to study the Muller-Speyer twist as follows. For cluster structures associated to a Demazure weave $\ww$, the square of the Donaldson-Thomas transformation has a neat geometric interpretation in terms of the positive braid $\dd\ww$: it is realized by a full cyclic rotation of the braid word. The Donaldson-Thomas transformation itself is realized by a cyclic rotation moving crossings to the top composed with a reflection. We first proved this in \cite[Section 5]{CW} for grid plabic graphs and in \cite[Section 8]{CGGLSS} we established the general case.\\

\noindent The application of Theorem \ref{thm:mainA} that we present exploits this particular description, using items (ii),(iii) and (iv) in Theorem \ref{thm:mainA}. Namely, we start with a reduced plabic graph $\bG$, construct the weave $\ww(\bG)$ and compute how its associated seed in $\mathfrak{M}(\dd\ww(\bG);\mathfrak{t})$ changes as we perform the geometric transformations above, i.e.~ when implementing cyclic rotations and reflections to $\dd\ww(\bG)$. The explicit nature of the coordinates in $\mathfrak{M}(\dd\ww(\bG);\mathfrak{t})$ then allows us to write the transformation in terms of Pl\"ucker coordinates and finally show that such map coincides with the Muller-Speyer twist map. The statement we prove reads as follows:

\begin{mainthm}\label{thm:mainB}
The Donaldson-Thomas transformation on the affine cone of a positroid stratum coincides with the Muller-Speyer twist map. In particular, the twist map is a quasi-cluster transformation of the target cluster structure.\hfill$\Box$
\end{mainthm}

\noindent We prove Theorem~\ref{thm:mainB} in Theorem~\ref{thm:DT-with-froz}, once we have established Theorem \ref{thm:mainA} and the necessary results on Donaldson-Thomas transformations in the context of flag moduli spaces.\\

Finally, given a reduced plabic graph $\bG$, there are two cluster structures that can be naturally associated to $\bG$. The first cluster structure is obtained by consider the so-called target labeling. The second cluster structure is obtained using the source labeling. We refer to \cite{MullerSpeyertwist,FSB} or Section \ref{sec:prelim} below for further details. These two different choices of labeling lead to different cluster structures, i.e.~they are not related by a cluster transformation. That said, following \cite[Remark 4.7]{MullerSpeyertwist}, it was conjectured in \cite[Conjecture 1.1]{FSB} that these two cluster structure associated to a reduced plabic graph are related by a quasi-cluster transformation. Our results prove this conjecture:

\begin{cor}\label{cor:quasiequivalence}
Let $\bG$ be a reduced plabic graph, $\Sigma_T(\bG)$ the cluster seed associated to the target labeling, and $\Sigma_S(\bG)$ the cluster seed associated to the source labeling. Then $\Sigma_T(\bG)$ and $\Sigma_S(\bG)$ are related by a quasi-cluster transformation.\hfill$\Box$
\end{cor}

The appearance of quasi-cluster transformation is natural from the viewpoint of weaves, as explained in \cite{CW}. In particular, cyclic rotation of the braid $\dd\ww$ leads to quasi-cluster transformations, not necessarily cluster. Since the Donaldson-Thomas transformation is described via cyclic rotations (and a reflection), we can prove Corollary \ref{cor:quasiequivalence} even though it involves only quasi-cluster, instead of cluster, transformations.\\

\noindent {\bf Acknowledgements}. We thank Chris Fraser for many helpful conversations. R.~Casals is supported by the NSF CAREER DMS-1942363, a Sloan Research Fellowship of the Alfred P. Sloan Foundation and a UC Davis College of L\&S Dean's Fellowship. M. Sherman-Bennett is supported by the National Science Foundation under Award No. DMS-2103282. Any opinions, findings, and conclusions or recommendations expressed in this material are
those of the author(s) and do not necessarily reflect the views of the National Science
Foundation.\hfill$\Box$\\

\noindent {\bf A related reference}. In the final stage of preparation of this manuscript, M.~Pressland announced in \cite{Pressland23} an alternative proof that the Muller-Speyer twist is quasi-cluster. Their proof differs significantly from ours. In particular, their argument depends critically on categorification, as explained in ibid., whereas ours occurs directly on the positroid variety and its cluster algebra. We think him for kind and helpful discussions.\hfill$\Box$\\ 

\noindent{\bf Notation}. Given $n\in\mathbb{N}$, we denote by $w_{0,n}$ the longest permutation in the symmetric group $S_n$. If $n$ is understood by context, then $w_{0,n}$ is denoted by $w_{0}$. For $n,m\in\mathbb{N}$, we use the notation $[n]:=\{1, \dots, n\}$ and $\binom{[n]}{m}:=\{I \subset [n]: |I|=m\}$.\hfill$\Box$

\section{Preliminaries}\label{sec:prelim}

This section presents the necessary ingredients and results on positroids, in Subsection \ref{ssec:preliminaries on positroids}, on reduced plabic graphs, in Subsection \ref{ssec:prelim_plabicgraphs}, and on weaves, in Subsection \ref{subsec: prelim on weaves}.

\subsection{Preliminaries on Positroids} \label{ssec:preliminaries on positroids} Positroids were first introduced by Postnikov in \cite{Pos} in a stratification of the totally non-negative Grassmannian. Knutson, Lam, and Speyer generalized Postnikov's real strata to the complex setting in \cite{KLS}, and showed they are projected Richardson varieties, which had been studied extensively by G.~Lusztig and K.~Rietsch \cite{Lusztig98,Rietsch06}. Cluster structures on positroid strata have been studied extensively in \cite{Lec, SSBW, GL,FSB}. We review the definition and some basic properties of positroid strata in this subsection.

Let $m\leq n$ be positive integers. We denote by $\Gr_{m,n}$ the Grassmannian of $m$-dimensional subspaces of $\bC^n$ and by $\Gr_{m,n}(\bR)$ the subset consisting of real subspaces. We represent a subspace $V \in \Gr_{m,n}$ by any full rank $m \times n$  matrix $A$ whose rowspan is $V$. For $A$ a full-rank $m \times n$ matrix, we often abuse notation and write $A \in \Gr_{m,n}$, identifying $A$ with its rowspan. 

We embed $\Gr_{m,n}$ into $\bP^{\binom{n}{m}-1}$ via the Pl\"ucker embedding. For $A \in \Gr_{m,n}$ and $I=\{i_1 < i_2 < \cdots<i_m\} \subset [n]$, the \emph{Pl\"ucker coordinate} $\Delta_I(A)$ is the maximal minor of $A$ located in column set $I$. Choosing a different representative matrix $A'$ for the same subspace rescales all Pl\"ucker coordinates by the same factor.

The \emph{totally nonnegative Grassmannian} \cite[Section 3]{Pos} is
\[\Gr_{m,n}^{\geq 0}:= \left\{A \in \Gr_{m,n}(\bR): \Delta_I(A) \geq 0 \text{ for all } I \in \binom{[n]}{m}\right\}.\]

\begin{defn}
    For $\cP\subset \binom{[n]}{m}$, we define the (possibly empty) stratum 
    \[S_{\cP}:=\{V \in \Gr_{m,n}^{\geq 0}: \Delta_I \neq 0 \text{ if and only if } I \in \cP\}.\]
    If $S_{\cP}$ is nonempty, it is a \emph{positroid cell} and $\cP$ is a \emph{positroid}.\hfill$\Box$
\end{defn}

Positroids are in bijection with many combinatorial objects, among them Grassmann necklaces and bounded affine permutations. We review bounded affine permutations first.

\begin{defn} A \emph{bounded affine permutation} (of \emph{type} $(m,n)$) is a bijection $f:\bZ\rightarrow \bZ$ such that:
\begin{itemize}
    \item[-] $i\leq f(i)\leq i+n$
    \item[-] $f(i+n)=f(i)+n$ for all $i \in \bZ$
    \item[-] $\frac{1}{n}\sum_{i=1}^n(f(i)-i)=m$.
\end{itemize}
Note that $f$ is uniquely determined by its action on $[n]$. We let $\bdmn$ denote the set of bounded affine permutations of type $(m,n)$. For $f \in \bdmn$, we define the \emph{associated permutation} $\pi_f \in S_n$ by reducing $f(1), \dots, f(n)$ modulo $n$; that is, $\pi_f$ is uniquely determined by the condition
\[\pi_f(i) \equiv f(i) \mod{n}.\]
\hfill$\Box$
\end{defn}

\begin{rmk}\label{rmk:pi_f not enough} Note that $f \in \bdmn$ is not uniquely determined by its associated permutation $\pi_f$. If $\pi_f(i) \neq i$, then $\pi_f(i)$ determines $f(i)$; if $\pi_f(i)=i$, it does not. So $f \in \bdmn$ is determined by $\pi_f$ together with a list of which fixed points of $\pi_f$ are also fixed points of $f$.\hfill$\Box$
\end{rmk}

\begin{defn}\label{def:perm of positroid}
    Let $\cP$ be a positroid and let $A \in S_\cP \subset \Gr_{m, n}^{\geq 0}$. Denote the columns of $A$ by $v_1, \dots, v_n$ and extend $A$ periodically to an $m \times \bZ$ matrix by setting $v_{i+n}=v_i$. For $i \in \bZ$, define $f(i)$ to be the minimal $r \geq i$ for which
    \[v_i \in \Span(v_{i+1}, v_{i+2}, \dots, v_r).\]
    The permutation $f$ is a bounded affine permutation of type $(m,n)$ which we call the \emph{bounded affine permutation of $\cP$}.\hfill$\Box$
\end{defn}
\noindent Note that $f(i)=i$ if and only if $v_i$ is the zero vector and $f(i)=i+n$ if and only if $v_i$ is not in the span of the other columns of $A$. In both cases, $i$ will be a fixed point of $\pi_f$.\\

We now review Grassmann necklaces and use them to define the positroid strata $\Pio_\cP$. For $i\in [n]$, we define the linear ordering $<_i$ on $[n]$ by
\[
i<_i i+1<_i\cdots <_i n<_i 1<_i \cdots <_i i-1.
\]
Note that $<_1$ coincides with the usual ordering of $[n]$. The linear ordering $<_i$ induces a partial order on $\binom{[n]}{m}$: given two $m$-element subsets $J=\{j_1<_ij_2<_i\cdots <_i j_m\}$ and $K=\{k_1<_ik_2<_i\cdots <_i k_m\}$, we say $J\leq_i K$ if $j_s\leq_i k_s$ for all $1\leq s\leq m$.

\begin{defn}
    Let $\cP$ be a positroid. Let $\tI{a}$ denote the unique $<_a$-minimal element of $\cP$ and $\sI{a-1}$ the unique $<_{a}$-maximal element of $\cP$.
    The \emph{target (Grassmann) necklace} of $\cP$ is the tuple 
    $\tNec=(\tI{1}, \tI{2},\dots, \tI{n})$.
    
    \noindent The \emph{source (Grassmann) necklace}\footnote{The target Grassmann necklace is also referred to as just the ``Grassmann necklace" in the literature. Similarly, the source Grassmann necklace is also called ``reverse Grassmann necklace" in the literature.} of $\cP$ is the tuple
    $\sNec=(\sI{1},\sI{2},\dots, \sI{n})$.\hfill$\Box$
\end{defn}

\noindent Using these necklaces, we define positroid strata in $\Gr_{m,n}$, a.k.a.~ {\it open positroid varieties} in \cite{KLS}.

\begin{defn}
    Let $\cP$ be a positroid, with target necklace $\tNec=(\tI{1}, \dots, \tI{n})$ and source necklace $\sNec=(\sI{1}, \dots, \sI{n})$. The \emph{positroid stratum} of $\cP$ is
    \begin{align*}
        \underline{\Pi}^\circ_{\cP}&:= \{V \in \Gr_{m,n}: \Delta_I(V)=0 \text{ for }I \notin \cP \text{ and } \Delta_{\tI{1}}(V) \Delta_{\tI{2}}(V) \cdots \Delta_{\tI{n}}(V) \neq 0\}\\
        &=\{V \in \Gr_{m,n}: \Delta_I(V)=0 \text{ for }I \notin \cP \text{ and } \Delta_{\sI{1}}(V) \Delta_{\sI{2}}(V) \cdots \Delta_{\sI{n}}(V) \neq 0\}.
    \end{align*}
    \hfill$\Box$
\end{defn}

\noindent Positroid strata do stratify the Grassmannian and are smooth, irreducible affine varieties \cite{KLS}. Please note that these differ from \emph{matroid strata}, and are sometimes called \emph{open positroid varieties} in the literature.

\begin{rmk}
    For the remainder of the paper, we will deal exclusively with the affine cone over $\underline{\Pi}^\circ_{\cP} \subset \bP^{\binom{n}{m}-1}$, which we denote by $\Pio_{\cP}$. Abusing notation, we also call $\Pio_{\cP}$ a positroid stratum.\hfill$\Box$ 
\end{rmk}

So far, we have given a map from positroids to bounded affine permutations and to target and source necklaces. We now discuss the map from target and source necklaces to bounded affine permutations. The target necklace of a positroid satisfies the following exchange property \cite[Definition 16.1, Lemma 16.3]{Pos}:
\begin{itemize}
    \item if $a\in \tI{a}$, then $\tI{a+1}=(\tI{a}\setminus \{a\})\cup \{b\}$ for some $b$ ($b$ may or may not equal $a$);
    \item if $a\notin \tI{a}$, then $\tI{a+1}=\tI{a}$.
\end{itemize}
Furthermore, all tuples $\tNec=(\tI{1}, 
\dots, \tI{n})$ with $\tI{a} \in \binom{[n]}{m}$ satisfying this exchange property are target necklaces for some positroid.

\noindent A target necklace $\tNec$ in $\binom{[n]}{m}$ corresponds bijectively to a bounded affine permutation $f \in \bdmn$, i.e.~ to a permutation $\pi_f$ and a list of fixed points of $f$, by Remark~\ref{rmk:pi_f not enough}. The correspondence is:
\begin{itemize}
    \item If $a \in \tI{a}$, then $\pi_f(a)=b$ and $f(a) \neq a$.
    \item If $a \notin \tI{a}$, then $\pi_f(a)=a$ and $f(a)=a$.
\end{itemize}

\noindent If $\tNec$ is the target necklace for positroid $\cP$, then the construction above gives the bounded affine permutation $f$ for $\cP$ \cite[Proposition 16.4]{Pos}. There is an analogous story with source necklaces $\sNec$, see \cite[Section 2.1]{MullerSpeyertwist}. In particular, the relation $\tI{a+1}=\pi_f(\sI{a})$ for all $1\leq a\leq n$, indices mod $n$, holds between the source and target necklaces of a bounded affine permutation $f$. Finally, \cite[Theorem 6]{Oh} shows that one can recover the positroid from either necklace.

In summary, any of the following four items determines the remaining three.
\[
\begin{tikzpicture}[scale=0.6]
    \node (s) at (0,0) [] {bounded affine permutation $f$};
    \node (e) at (6,3) [] {Source necklace $\sNec$};
    \node (w) at (-6,3) [] {Target necklace $\tNec$};
    \node (n) at (0,6) [] {Positroid $\cP$};
    \draw [<->] (e) -- node [above] {$\pi_f$} (w);
    \draw [<->] (s) -- node [below right] {exchange property} (e);
    \draw [<->] (s) -- node [below left] {exchange property} (w);
    \draw [<->] (n) -- node [above right] {maximal w.r.t. $<_{a+1}$} (e);
    \draw [<->] (n) -- node [above left] {minimal w.r.t. $<_a$} (w);
\end{tikzpicture}
\]

\noindent The remaining part of this subsection describes the construction of a positive braid from a given positroid. 

\begin{defn}\label{defn:positroid-braid-word}
    Let $\cP$ be a positroid of type $(m,n)$ with bounded affine permutation $f$ and target Grassmann necklace $\tNec=(\tI{1}, \dots, \tI{n})$. Denote by $b_i:=|\{a \in \tI{i}: a <_i \pi_f(i-1)\}|$. We define the positive braid word $\beta_\cP:=\beta_1 \beta_2 \cdot\ldots\cdot \beta_n$, where $\beta_i$ is the empty word if $i-1$ is a fixed point of $f$ and otherwise
    \[\beta_i :=s_{m-1}\cdot s_{m-2}\cdot \ldots\cdot s_{m-b_i}.\]
    We call $\beta_\cP$ a \emph{positroid braid word}.
    It represents a positive braid $[\beta_\cP]$ in $\Br_m$.\hfill$\Box$
\end{defn}

\begin{rmk}\label{rmk:drawing-braid}
 One way to draw the wiring diagram of $\beta_{\cP}$ is as follows: write the elements of $\tI{i}$ in a column so they increase with respect to $<_i$, reading top-to-bottom. Arrange these columns from left to right in the order $\tI{n}, \tI{1}, \dots, \tI{n-1}, \tI{n}$. Then draw a line from $a \in \tI{i-1}$ to $a \in \tI{i}$ for $a \neq i-1$, and, if $i-1$ appears in $\tI{i-1}$, draw a line connecting $i-1 \in \tI{i-1}$ with $\pi_f(i-1) \in \tI{i}$. See Figure ~\ref{fig:positroid positive braid} (left) for an example.\hfill$\Box$
\end{rmk}

We will also use a periodically repeating version of the wiring diagram of $\beta_{\cP}$. 

\begin{defn}\label{def:grid-pattern} The \emph{periodic grid pattern} of $\beta_\cP$ is defined as follows. Arrange $\bZ^2$ in the plane using matrix conventions. For each $i \in \bZ$, draw a vertical segment from $y=i$ to $y=f(i)$ at horizontal position $x=i+0.5$. Call this segment the \emph{$(i, f(i))$-chord}. Then draw a horizontal segment at $y=f(i)$ connecting the bottom of the $(i, f(i))$-chord to the top of the $(f(i), f^2(i))$-chord. Note that between $x=i-1 + an$ and $x=i + an$ we see the wiring diagram of $\beta_{i}$, so the periodic grid pattern is the wiring diagram of $\dots \beta_n \beta_1 \beta_2 \dots \beta_n \beta_1 \dots$.\hfill$\Box$
\end{defn}

\noindent Note that the vertical slice $x=i$ through the periodic grid pattern intersects horizontal segments at heights $\tI{i} \pmod n$. Also, the periodic grid pattern is a periodic covering of the cyclic braid $\beta_\cP$. 

\begin{rmk}

    One may make an analogous definition to Definition~\ref{defn:positroid-braid-word} using the source Grassmann necklace rather than the target. For $\cP$ a positroid, let $d_i:=|\{a \in \sI{i-1}: \pi_f^{-1}(i) <_{i} a \}|$ and set $\delta_i$ to be the empty word if $f(i)=i$ and otherwise
    \[\delta_i=s_{d_i}\cdot s_{d_i-1}\cdot\ldots\cdot s_1.\]
    Then we define $\delta_\cP:=\delta_1 \delta_2\cdot\ldots\cdot\delta_n$ to be the \emph{source positroid braid word}. One may also adapt the recipes in Remark~\ref{rmk:drawing-braid} to $\delta_\cP$. In particular, for the first way of drawing the wiring diagram, write the elements of $\sI{i}$ so they increase with respect to $<_{i+1}$, reading top to bottom. Draw a line from $\pi_f^{-1}(i) \in \sI{i-1}$ with $i \in \sI{i}$, and lines connecting all other $a \in \sI{i-1}$ with $a \in \sI{i}$.There is also a periodic grid pattern for the cyclic braid $\delta_\cP$, analogous to the periodic grid pattern for $\beta_\cP$, and it coincides with the pattern described in terms of rank matrices in \cite[Section 3]{STWZ}.\hfill$\Box$
\end{rmk}

\noindent In Subsection \ref{subsec: toggles and RIII moves} we prove that the $(-1)$-closure of the positive braid $[\delta_\cP]$, cf.~\cite[Section 2.2]{CasalsNg}, is Legendrian isotopic to the Legendrian link $\Lambda_\cP$, which is defined in terms of $[\beta_\cP]$. Also, despite the fact that we only use the positive braid word $\beta_\cP$ in the next few sections, the positive braid word $\delta_\cP$ will come in handy in proving Theorem \ref{thm: frame sheaf moduli = positroid}.

\begin{exmp}\label{exmp:positroid braid} Consider the positroid $\cP$ corresponding the bounded affine permutation
\[
f = \begin{pmatrix}\cdots & 1 & 2 & 3 & 4 & 5 & 6 & 7 & \cdots \\
                     \cdots & 3 & 9 & 8 & 7 & 5 & 11 & 13 &\cdots \end{pmatrix}
\]
Its target Grassmann necklace $\overrightarrow{\cI}$ is drawn in Figure \ref{fig:positroid positive braid}, together with the wiring diagram of the positive braid word $\beta_\cP$ it produces. The periodic grid braid pattern is drawn on the right. \hfill$\Box$
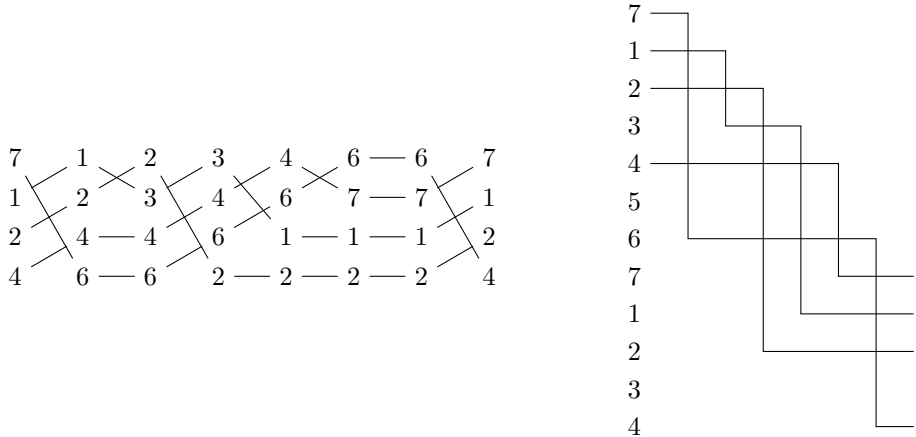
\begin{figure}[H]
    \centering
    \begin{tikzpicture}[scale=0.75,baseline=0]
    \node (01) at (-1.2,2.1) [] {$7$};
    \node (02) at (-1.2,1.4) [] {$1$};
    \node (03) at (-1.2,0.7) [] {$2$};
    \node (04) at (-1.2,0) [] {$4$};
    \node (11) at (0,2.1) [] {$1$};
    \node (12) at (0,1.4) [] {$2$};
    \node (13) at (0,0.7) [] {$4$};
    \node (14) at (0,0) [] {$6$};
    \node (21) at (1.2,2.1) [] {$2$};
    \node (22) at (1.2,1.4) [] {$3$};
    \node (23) at (1.2,0.7) [] {$4$};
    \node (24) at (1.2,0) [] {$6$};
    \node (31) at (2.4,2.1) [] {$3$};
    \node (32) at (2.4,1.4) [] {$4$};
    \node (33) at (2.4,0.7) [] {$6$};
    \node (34) at (2.4,0) [] {$2$};
    \node (41) at (3.6,2.1) [] {$4$};
    \node (42) at (3.6,1.4) [] {$6$};
    \node (43) at (3.6,0.7) [] {$1$};
    \node (44) at (3.6,0) [] {$2$};
    \node (51) at (4.8,2.1) [] {$6$};
    \node (52) at (4.8,1.4) [] {$7$};
    \node (53) at (4.8,0.7) [] {$1$};
    \node (54) at (4.8,0) [] {$2$};
    \node (61) at (6,2.1) [] {$6$};
    \node (62) at (6,1.4) [] {$7$};
    \node (63) at (6,0.7) [] {$1$};
    \node (64) at (6,0) [] {$2$};
    \node (71) at (7.2,2.1) [] {$7$};
    \node (72) at (7.2,1.4) [] {$1$};
    \node (73) at (7.2,0.7) [] {$2$};
    \node (74) at (7.2,0) [] {$4$};
    \draw (11) -- (22);
    \draw (12) -- (21);
    \draw (13) -- (23);
    \draw (14) -- (24);
    \draw (31) -- (22);
    \draw (34) -- (21);
    \draw (32) -- (23);
    \draw (33) -- (24);
    \draw (31) -- (43);
    \draw (34) -- (44);
    \draw (32) -- (41);
    \draw (33) -- (42);
    \draw (53) -- (43);
    \draw (54) -- (44);
    \draw (52) -- (41);
    \draw (51) -- (42);
    \foreach \i in {1,...,4}
    {
        \draw (5\i) -- (6\i);
    }
    \draw (61) -- (74);
    \draw (62) -- (71);
    \draw (63) -- (72);
    \draw (64) -- (73);
    \draw (13) -- (04);
    \draw (14) -- (01);
    \draw (11) -- (02);
    \draw (12) -- (03);
    \end{tikzpicture} \quad \quad \quad \quad 
    \begin{tikzpicture}[scale=0.5,baseline=0]
        \draw (0,7) node [left] {$7$} -- (1,7) -- (1,1) -- (6,1) -- (6,-4) -- (7,-4);
        \draw (0,6) node [left] {$1$} -- (2,6) -- (2,4) -- (4,4) -- (4,-1) -- (7,-1);
        \draw (0,5) node [left] {$2$} -- (3,5) -- (3,-2) -- (7,-2);
        \draw (0,3) node [left] {$4$} -- (5,3) -- (5,0) -- (7,0);
        \node at (0,4) [left] {$3$};
        \node at (0,2) [left] {$5$};
        \node at (0,1) [left] {$6$};
        \node at (0,0) [left] {$7$};
        \node at (0,-1) [left] {$1$};
        \node at (0,-2) [left] {$2$};
        \node at (0,-3) [left] {$3$};
        \node at (0,-4) [left] {$4$};
        \end{tikzpicture}
    \caption{Example of a positive braid word $\beta_\cP$ and its periodic grid pattern.}
    \label{fig:positroid positive braid}
\end{figure}

\end{exmp}

\subsection{Preliminaries on Reduced Plabic Graphs}\label{ssec:prelim_plabicgraphs}

Plabic graphs were introduced by Postnikov \cite{Pos} and are often used to describe cluster structures associated with positroids. We assume some familiarity with plabic graphs, cf.~ibid.

\begin{defn} Let $\bD$ be the unit disk with $n$ marked points on its boundary, labeled $1,2,\dots, n$ clockwise. A \emph{plabic graph} $\bG$ on $\bD$ is a planar graph embedded in $\bD$ such that each boundary marked point is a vertex (called a \emph{boundary vertex}), each boundary vertex is adjacent to a unique internal vertex (i.e.~vertex in the interior of $\bD$) and each internal vertex of $\bG$ is colored \emph{solid} or \emph{empty}.\footnote{We use solid and empty instead of the traditional black and white to avoid confusion with colors on weaves.}\\

\noindent By definition, the \emph{faces} of $\bG$ are the connected components of $\bD\setminus \bG$. A face containing a boundary vertex is said to be a \emph{boundary face}. If a boundary vertex is adjacent to a degree 1 internal vertex $v$, then $v$ is said to be a \emph{lollipop}.\hfill$\Box$
\end{defn}

\begin{rmk} For the remainder of the paper, we assume that all plabic graphs are \emph{reduced}, cf.~ \cite[Section 7.4]{FWZ}). We henceforth omit this adjective.\hfill$\Box$
\end{rmk}

\noindent We now describe how to associate a bounded affine permutation to a plabic graph. Recall that a bounded affine permutation $f \in \bdmn$ is determined by $\pi_f$ and a list of fixed points of $f$, cf.~Remark \ref{rmk:pi_f not enough}.

\begin{defn}\label{def:zig-zag-strand}
    A \emph{zig-zag strand} on $\bG$ is an oriented curve on $\bD$ that begins at a boundary vertex, obeys the ``rules of the road" in Figure~\ref{fig:rule of the road}, and ends at a boundary vertex. We denote the strand starting at boundary vertex $i$ by $\zeta_i$. We define the bounded affine permutation $f$ of $\bG$ as follows: if $\zeta_i$ ends at $j$, we set $\pi_f(i):=j$ and declare $i \in [n]$ to be a fixed point of $f$ if and only if there is a solid lollipop at $i$.\\
    
    \noindent If the bounded affine permutation $f$ of $\bG$ is of type $(m,n)$, then we say that $\bG$ has \emph{rank} $m$. By definition, a plabic graph $\bG$ is associated to a positroid $\cP$ if its bounded affine permutation $f$ corresponds to $\cP$. \hfill$\Box$
\end{defn}

\begin{figure}[H]
    \centering
    \begin{tikzpicture}
    \empver[black] (0) (0,0);
    \foreach \i in {0,...,4}{
        \draw[shorten <=0.2cm] (0) -- (90+72*\i:1.5);
        \draw[red,decoration={markings,mark=at position 0.5 with {\arrow{>}}},postaction={decorate}] (85+72*\i:1.5) to [out=-90+72*\i,in=144+72*\i] (54+72*\i:0.6) to [out=-36+72*\i,in=198+72*\i] (23+72*\i:1.5);
    }
    \end{tikzpicture}\hspace{2cm}
    \begin{tikzpicture}
    \solver[black] (0) (0,0);
    \foreach \i in {0,...,4}{
        \draw[shorten <=0.2cm] (0) -- (90+72*\i:1.5);
        \draw[red,decoration={markings,mark=at position 0.5 with {\arrow{<}}},postaction={decorate}] (95+72*\i:1.5) to [out=-90+72*\i,in=144+72*\i] (54+72*\i:0.6) to [out=-36+72*\i,in=198+72*\i] (13+72*\i:1.5);
    }
    \end{tikzpicture}\hspace{2cm}
    \begin{tikzpicture}
        \draw [very thick] (-1,0) -- (1,0);
        \draw (0,0) node [below] {$i$} -- (0,0.8);
        \empver[black] (0) (0,1);
        \draw [red,decoration={markings,mark=at position 0.5 with {\arrow{>}}},postaction={decorate}] (-0.3,0) -- (-0.3,1) arc (180:0:0.3) -- (0.3,0);
    \end{tikzpicture}\hspace{2cm}
    \begin{tikzpicture}
        \draw [very thick] (-1,0) -- (1,0);
        \draw (0,0) node [below] {$i$} -- (0,0.8);
        \solver[black] (0) (0,1);
        \draw [red,decoration={markings,mark=at position 0.5 with {\arrow{>}}},postaction={decorate}] (-0.3,0) to [in=-90] (0.3,1) arc (0:180:0.3) to [out=-90] (0.3,0);
    \end{tikzpicture}
    \caption{Rules of the road: a zig-zag strand always makes the sharpest possible at left turn at an empty vertex and makes the sharpest possible right turn at a solid vertex.
    }
    \label{fig:rule of the road}
\end{figure}
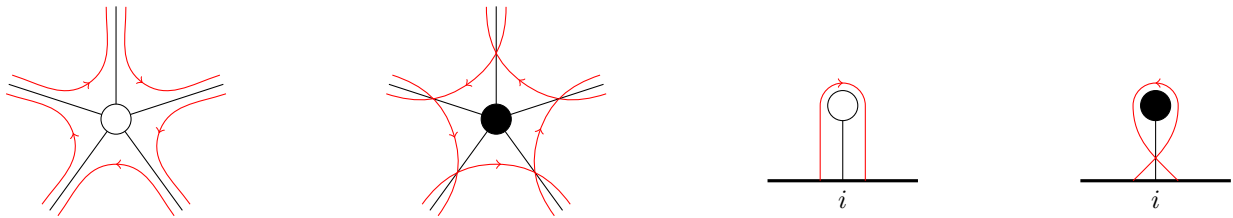

Let $\bG$ be a reduced plabic graph. We associate a subset of $\{1,2,\dots, n\}$ with each face of $\bG$ in two ways. Let us denote the \emph{source} and \emph{target} of $\zeta$ by $s(\zeta)$ and $t(\zeta)$, respectively. For $F$ a face of $\bG$, we define
\[
\overrightarrow{I_F}:=\{t(\zeta)\mid \text{$F$ is to the left of $\zeta$}\},
\]
\[
\overleftarrow{I_F}:=\{s(\zeta)\mid \text{$F$ is to the left of $\zeta$}\}.
\]
We call $\overrightarrow{I_F}$ the \emph{target label} of $F$ and  $\overleftarrow{I_F}$ the \emph{source label} of $F$. In particular, the target (resp. source) labels of the boundary faces (counted with multiplicity when lollipops are present) form the target (resp. source) Grassmann necklace corresponding to $f$.

\begin{exmp}\label{exmp:plabic graph example 1} The plabic graph $\bG$ in Figure \ref{fig:plabic graph example 1} is associated with the positroid $\cP$ in Example \ref{exmp:positroid braid}. \hfill$\Box$
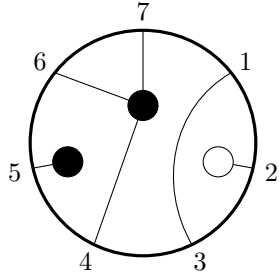
\begin{figure}[H]
    \centering
    \begin{tikzpicture}
    \draw [very thick] (0,0) circle [radius=1.5];
    \foreach \i in {1,...,7}
    {
    \coordinate (\i) at (90-\i*360/7:1.5);
    \node at (90-\i*360/7:1.75) [] {$\i$};
    }
    \empver[black] (8) (1,-0.25);
    \solver[black] (9) (-1,-0.25);
    \solver[black] (10) (0,0.5);
    \draw[shorten <=0.2cm] (8) -- (2);
    \draw[shorten <=0.2cm] (9) -- (5);
    \draw[shorten <=0.2cm] (10) -- (7);
    \draw[shorten <=0.2cm] (10) -- (6);
    \draw[shorten <=0.2cm] (10) -- (4);
    \draw (1) to [out=-150,in=120] (3);
    \end{tikzpicture}
    \caption{A plabic graph associated with the positroid in Example \ref{exmp:positroid braid}.}
    \label{fig:plabic graph example 1}
\end{figure}
\end{exmp}

Next, let us discuss some moves that can be applied to reduced plabic graphs.

\begin{itemize}
    \item \textbf{Contraction/expansion}. An edge connecting two internal vertices of the same color can be contracted. The expansion move is the reverse operation.
    \item \textbf{Bivalent vertex insertion/deletion}. We can add and remove degree-2 internal vertices.
\end{itemize}

\begin{defn}\label{defn:equivalence of reduced plabic graphs} Two plabic graphs are \emph{equivalent} if they are related by a sequence of contraction/expansion and bivalent vertex insertion/deletion moves\footnote{This is a slight departure from the literature, which usually treats square moves as a plabic graph equivalence. However, we choose to distinguish square moves since they correspond to mutation at the level of seeds.}. Note that equivalent plabic graphs have the same permutation and the same face labels (for both source and target labelings).\hfill$\Box$
\end{defn}

\begin{rmk}\label{rmk:trivalent solid vertices} Using contraction/expansion and bivalent vertex deletion, we can turn all vertices that are not lollipops, solid and empty, into trivalent vertices. The fact that all non-lollipop vertices can be made trivalent is used in our construction of weaves from reduced plabic graphs.\hfill$\Box$
\end{rmk}

\noindent In addition to the two types of equivalence moves, there is one more move that we do {\bf not} consider to be a plabic graph equivalence. It is the following move:

\begin{itemize}
    \item \textbf{Square Move}. The square move, also known as a \emph{mutation} on reduced plabic graphs, changes a reduced plabic graph locally as follows.
    \begin{figure}[H]
        \centering
        \begin{tikzpicture}[baseline=0]
        \empver[black] (1) (-0.4,-0.4);
        \solver[black] (2) (-0.4,0.4);
        \empver[black] (3) (0.4,0.4);
        \solver[black] (4) (0.4,-0.4);
        \edge[black] (1) (2);
        \edge[black] (2) (3);
        \edge[black] (3) (4);
        \edge[black] (1) (4);
        \draw[shorten <=0.2cm] (1) -- (-1,-1);
        \draw[shorten <=0.2cm] (2) -- (-1,1);
        \draw[shorten <=0.2cm] (3) -- (1,1);
        \draw[shorten <=0.2cm] (4) -- (1,-1);
        \end{tikzpicture} \hspace{1cm} $\longleftrightarrow$ \hspace{1cm}
        \begin{tikzpicture}[baseline=0]
        \solver[black] (1) (-0.4,-0.4);
        \empver[black] (2) (-0.4,0.4);
        \solver[black] (3) (0.4,0.4);
        \empver[black] (4) (0.4,-0.4);
        \edge[black] (1) (2);
        \edge[black] (2) (3);
        \edge[black] (3) (4);
        \edge[black] (1) (4);
        \draw[shorten <=0.2cm] (1) -- (-1,-1);
        \draw[shorten <=0.2cm] (2) -- (-1,1);
        \draw[shorten <=0.2cm] (3) -- (1,1);
        \draw[shorten <=0.2cm] (4) -- (1,-1);
        \end{tikzpicture}
        \caption{Square move.}
        \label{fig:square move}
    \end{figure}
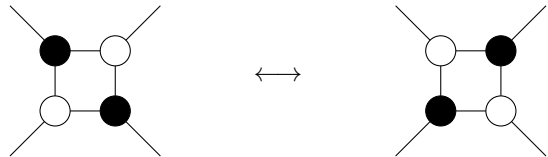
\end{itemize}
Note that under a square move, the bounded affine permutation does not change, but the source and target labels of the center face do change.


\begin{thm}[{\cite[Theorem 12.7]{Pos}}]\label{thm: square moves} Two reduced plabic graphs have the same bounded affine permutation if and only if they are related by a sequence of equivalences and mutations.
\end{thm}

A plabic graph $\bG$ has an associated quiver $Q_\bG$, e.g.~as constructed in \cite[Chapter 7.1]{FWZ}. In particular, the vertex set of $Q_\bG$ are the faces of $\bG$, and we freeze all quiver vertices corresponding to the boundary faces. The arrows of $Q_\bG$ can be obtained by drawing a counterclockwise oriented cycle (of arrows) around each solid vertex and then deleting a maximal collection of 2-cycles. See Figure \ref{fig: quiver example} for an example of a $\bG$ and its quiver $Q_\bG$. By definition, the \emph{target seed} $\Sigma_T(\bG)$ is the seed with cluster $\{\Delta_{\tI{}(F)}: F \text{ a face of }\bG\}$ and quiver $Q_\bG$. The \emph{source seed} $\Sigma_S(\bG)$ is the seed with cluster $\{\Delta_{\sI{}(F)}: F \text{ a face of }\bG\}$ and quiver $Q_\bG$. 

\begin{figure}[H]
    \centering
    \begin{tikzpicture}[baseline=0]
        \empver[black] (1) (-0.4,-0.4);
        \solver[black] (2) (-0.4,0.4);
        \empver[black] (3) (0.4,0.4);
        \solver[black] (4) (0.4,-0.4);
        \edge[black] (1) (2);
        \edge[black] (2) (3);
        \edge[black] (3) (4);
        \edge[black] (1) (4);
        \foreach \i in {1,...,4}
        {
        \edge[black] (\i) (-45-90*\i:1.4);
        }
        \draw [very thick] (0,0) circle [radius=1.2];
    \end{tikzpicture} \hspace{2cm}
    \begin{tikzpicture} [baseline=0]
    \node (0) at (0,0) [] {$\bullet$};
    \node (1) at (-1.2,0) [] {$\square$};
    \node (2) at (0,-1.2) [] {$\square$};
    \node (3) at (1.2,0) [] {$\square$};
    \node (4) at (0,1.2) [] {$\square$};
    \draw [->] (0) -- (4);
    \draw [->] (4) -- (1);
    \draw [->] (1) -- (0);
    \draw [->] (0) -- (2);
    \draw [->] (2) -- (3);
    \draw [->] (3) -- (0);
    \end{tikzpicture}
    \caption{Example of a reduced plabic graph (left) and its quiver (right).}
    \label{fig: quiver example}
\end{figure}
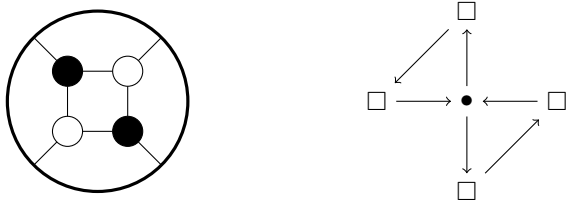

The results of \cite{GL} show that for any reduced plabic graph $\bG$ associated to $\cP$, $\bC[\Pio_\cP]$ is equal to the cluster algebra with initial seed $\Sigma_S(\bG)$. This in turn implies $\bC[\Pio_\cP]$ is equal to the cluster algebra with initial seed $\Sigma_T(\bG)$ (see e.g. \cite[Remark 2.16]{FSB}). We emphasize however that $\Sigma_S(\bG)$ and $\Sigma_T(\bG)$ are \emph{not} related by mutation, and so the cluster algebras $\cA(\Sigma_S(\bG))$ and $ \cA(\Sigma_T(\bG))$ are different (e.g.~ they have different cluster variables). In this article, we use the \emph{target cluster structure} on $\Pio_\cP$, i.e.~we view its coordinate ring as the cluster algebra $\cA(\Sigma_T(\bG))$.

\begin{defn} An orientation of the edges of $\bG$ is called a \emph{perfect orientation} if each solid vertex has a unique outgoing edge and each empty vertex has a unique incoming edge. An orientation is \emph{acyclic} if there are no oriented cycles. A boundary marked point is a \emph{source} of a perfect orientation if the adjacent edge is oriented away from the boundary point, and is a \emph{sink} otherwise.\hfill$\Box$
\end{defn}

\begin{exmp} Figure \ref{fig:perfect orientations} shows three different perfect orientations on the same reduced plabic graph. Note that the first two are acyclic and the last one is not acyclic. The source sets of the these perfect orientations are $\{1,2\}$, $\{1,3\}$, and $\{2,4\}$, respectively.\hfill$\Box$
\begin{figure}[H]
    \centering
    \begin{tikzpicture}[baseline=0]
        \foreach \i in {1,...,4}
        {
        \coordinate (\i) at (135-90*\i:1.2);
        \node at (135-90*\i:1.4) [] {$\i$};
        }
        \coordinate (5) at (0.4,0.4);
        \coordinate (6) at (0.4,-0.4);
        \coordinate (7) at (-0.4,-0.4);
        \coordinate (8) at (-0.4,0.4);
        \draw [decoration={markings,mark=at position 0.5 with {\arrow{>}}},postaction={decorate}] (1) -- (5);
        \draw [decoration={markings,mark=at position 0.5 with {\arrow{>}}},postaction={decorate}] (5) -- (6);
        \draw [decoration={markings,mark=at position 0.5 with {\arrow{>}}},postaction={decorate}] (5) -- (8);
        \draw [decoration={markings,mark=at position 0.5 with {\arrow{>}}},postaction={decorate}] (2) -- (6);
        \draw [decoration={markings,mark=at position 0.5 with {\arrow{>}}},postaction={decorate}] (6) -- (7);
        \draw [decoration={markings,mark=at position 0.5 with {\arrow{>}}},postaction={decorate}] (7) -- (8);
        \draw [decoration={markings,mark=at position 0.5 with {\arrow{>}}},postaction={decorate}] (7) -- (3);
        \draw [decoration={markings,mark=at position 0.5 with {\arrow{>}}},postaction={decorate}] (8) -- (4);
        \draw [fill=white] (5) circle [radius=0.2];
        \draw [fill=black] (6) circle [radius=0.2];
        \draw [fill=white] (7) circle [radius=0.2];
        \draw [fill=black] (8) circle [radius=0.2];
        \draw [very thick] (0,0) circle [radius=1.2];
    \end{tikzpicture}\hspace{2cm}
    \begin{tikzpicture}[baseline=0]
        \foreach \i in {1,...,4}
        {
        \coordinate (\i) at (135-90*\i:1.2);
        \node at (135-90*\i:1.4) [] {$\i$};
        }
        \coordinate (5) at (0.4,0.4);
        \coordinate (6) at (0.4,-0.4);
        \coordinate (7) at (-0.4,-0.4);
        \coordinate (8) at (-0.4,0.4);
        \draw [decoration={markings,mark=at position 0.5 with {\arrow{>}}},postaction={decorate}] (1) -- (5);
        \draw [decoration={markings,mark=at position 0.5 with {\arrow{>}}},postaction={decorate}] (5) -- (6);
        \draw [decoration={markings,mark=at position 0.5 with {\arrow{>}}},postaction={decorate}] (5) -- (8);
        \draw [decoration={markings,mark=at position 0.5 with {\arrow{<}}},postaction={decorate}] (2) -- (6);
        \draw [decoration={markings,mark=at position 0.5 with {\arrow{<}}},postaction={decorate}] (6) -- (7);
        \draw [decoration={markings,mark=at position 0.5 with {\arrow{>}}},postaction={decorate}] (7) -- (8);
        \draw [decoration={markings,mark=at position 0.5 with {\arrow{<}}},postaction={decorate}] (7) -- (3);
        \draw [decoration={markings,mark=at position 0.5 with {\arrow{>}}},postaction={decorate}] (8) -- (4);
        \draw [fill=white] (5) circle [radius=0.2];
        \draw [fill=black] (6) circle [radius=0.2];
        \draw [fill=white] (7) circle [radius=0.2];
        \draw [fill=black] (8) circle [radius=0.2];
        \draw [very thick] (0,0) circle [radius=1.2];
    \end{tikzpicture}\hspace{2cm}
    \begin{tikzpicture}[baseline=0]
        \foreach \i in {1,...,4}
        {
        \coordinate (\i) at (135-90*\i:1.2);
        \node at (135-90*\i:1.4) [] {$\i$};
        }
        \coordinate (5) at (0.4,0.4);
        \coordinate (6) at (0.4,-0.4);
        \coordinate (7) at (-0.4,-0.4);
        \coordinate (8) at (-0.4,0.4);
        \draw [decoration={markings,mark=at position 0.5 with {\arrow{<}}},postaction={decorate}] (1) -- (5);
        \draw [decoration={markings,mark=at position 0.5 with {\arrow{<}}},postaction={decorate}] (5) -- (6);
        \draw [decoration={markings,mark=at position 0.5 with {\arrow{>}}},postaction={decorate}] (5) -- (8);
        \draw [decoration={markings,mark=at position 0.5 with {\arrow{>}}},postaction={decorate}] (2) -- (6);
        \draw [decoration={markings,mark=at position 0.5 with {\arrow{<}}},postaction={decorate}] (6) -- (7);
        \draw [decoration={markings,mark=at position 0.5 with {\arrow{<}}},postaction={decorate}] (7) -- (8);
        \draw [decoration={markings,mark=at position 0.5 with {\arrow{>}}},postaction={decorate}] (7) -- (3);
        \draw [decoration={markings,mark=at position 0.5 with {\arrow{<}}},postaction={decorate}] (8) -- (4);
        \draw [fill=white] (5) circle [radius=0.2];
        \draw [fill=black] (6) circle [radius=0.2];
        \draw [fill=white] (7) circle [radius=0.2];
        \draw [fill=black] (8) circle [radius=0.2];
        \draw [very thick] (0,0) circle [radius=1.2];
    \end{tikzpicture}
    \caption{Different perfect orientations on the same plabic graph.}
    \label{fig:perfect orientations}
\end{figure}
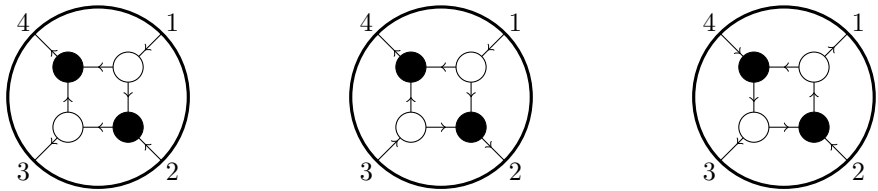
\end{exmp}

\begin{defn}[{\cite[Proposition 5.13]{MullerSpeyertwist}}]\label{defn: perfect orientation with index a} Let $\bG$ be a reduced plabic graph. For any element $\overrightarrow{I_i}$ in the target Grassmann necklace, there exists a unique acyclic perfect orientation on $\bG$ with source set $\overrightarrow{I_i}$, defined as follows\footnote{The construction of \cite[Proposition 5.13]{MullerSpeyertwist} is phrased in terms of matchings: what we write here is their construction translated to perfect orientations.}: if strands $\zeta_a, \zeta_b$ touch edge $e$ with $a <_i b$, then $e$ is oriented in the direction of $\zeta_a$. We denote this perfect orientation by $O_i$.\hfill$\Box$
\end{defn}

When $\bG$ is equipped with an acyclic perfect orientation, the boundary of any face $F$ consists of several oriented edges. We say a vertex $v$ on $\partial F$ is a \emph{source vertex} (resp. \emph{sink vertex}) if the two edges of $\partial F$ incident to $v$ are both out-going (resp. in-coming). The following property of the acyclic perfect orientation with source set $\overrightarrow{I_i}$ will be useful to us.

\begin{prop}\label{prop: unique sink} Let $\bG$ be a reduced plabic graph. In the perfect orientation $O_i$, there is at most one source vertex and at most one sink vertex along $\partial F$ for any face $F$. In particular, a non-boundary face $F$ has exactly one source vertex and one sink vertex along $\partial F$ (as $\partial F$ cannot be an oriented cycle).
\end{prop}
\begin{proof} By construction, each solid vertex has only one out-going edge and each empty vertex has only one in-coming edge. Therefore, the edges along $\partial f$ linking consecutive vertices of the same type cannot change orientation. Thus, by using the contraction-expansion move, we may assume without loss of generality that $\partial f$ is bipartite and all vertices along $\partial f$ are trivalent.\\

\noindent Now consider the solid vertices on $\partial f$ and those nearby zig-zag strands that travel in the counterclockwise direction along $\partial f$, see Figure \ref{fig: zig-zag strands near a face}. Since $\bG$ is reduced, the indices of these zig-zags respect the cyclic order on $\{1,2,\dots, n\}$, i.e. for the collection of zig-zags labeled as in Figure \ref{fig: zig-zag strands near a face} we must have $i_1<i_2<\cdots<i_k<i_1$ with respect to the cyclic order.
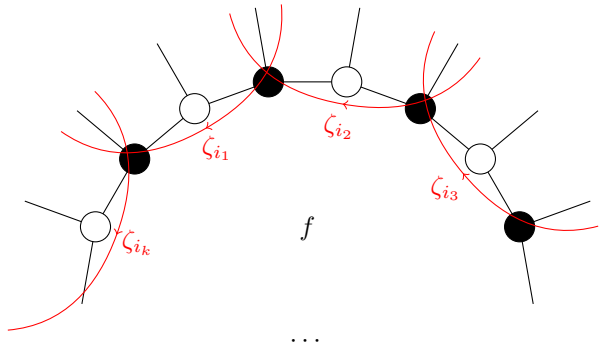
\begin{figure}[H]
    \centering
    \begin{tikzpicture}
        \foreach \i in {0,...,8}
        {
            \draw (20*\i:3) -- (20*\i+20:3);
        }
        \foreach \i in {1,...,8}
        {
            \draw (20*\i:3) -- (20*\i:4);
        }
        \foreach \i in {0,...,3}
        {
            \draw [fill=black] (20+40*\i:3) [] circle [radius=0.2];
            \draw [fill=white] (40+40*\i:3) [] circle [radius=0.2];
        }
        \foreach \i in {0,...,3}
        {
            \draw [red, decoration={markings,mark=at position 0.5 with {\arrow{>}}},postaction={decorate}] (15+40*\i:4) to [out=195+40*\i, in=310+40*\i] (40+40*\i:2.7) to [out=130+40*\i,in=245+40*\i] (65+40*\i:4); 
        }
        \node [red] at (40:2.4) [] {$\zeta_{i_3}$};
        \node [red] at (80:2.4) [] {$\zeta_{i_2}$};
        \node [red] at (120:2.4) [] {$\zeta_{i_1}$};
        \node [red] at (160:2.4) [] {$\zeta_{i_k}$};
        \node at (0,-0.5) [] {$\cdots$};
        \node at (0,1) [] {$f$};
    \end{tikzpicture}
    \caption{Nearby zig-zags traveling along $\partial f$ counterclockwise.}
    \label{fig: zig-zag strands near a face}
\end{figure}
Since $<_i$ breaks the cyclic order into a linear order, there is at most one solid vertex $v$ on $\partial f$ where $i_{t+1}<_ii_t$. Then $\zeta_{i_{t+1}}$ must be the zig-zag strand with the smallest index among the three zig-zag strands near $v$. Therefore in this case $v$ is a sink along $\partial f$. Now, any other solid vertex on $\partial f$ which is not $v$ cannot be a sink vertex for $\partial f$ because $i_s<_ii_{s+1}$, and any empty vertex cannot be a sink vertex for $\partial f$ either due to the condition that any empty vertex has only one in-coming edge attached. Therefore, $\partial f$ has at most one sink vertex as required. The argument for $\partial f$ having at most one source vertex is analogous, using empty vertices and clockwise zig-zag strands instead.
\end{proof}

\subsection{Preliminaries on weaves} \label{subsec: prelim on weaves}
Weaves were introduced in \cite{CZ}. They were originally devised as combinatorial tools to describe Legendrian surfaces in $1$-jet spaces. For the purpose of the present manuscript, they are a diagrammatic planar calculus that is well-suited to study cluster algebras on certain spaces of configurations of flags. In this subsection, we briefly review some basics about weaves. For simplicity, we only consider weaves on a disk $\bD$ in this article.

\begin{defn} Let $k\in\bN$, a \emph{$k$-weave} is a graph embedded in the 2-disk with edges labeled by elements in $[k-1]$ and such that each vertex is of one of the three types listed in Figure \ref{fig: weave vertices}.\\

\noindent Elements of $[k-1]$ are referred to as the \emph{colors} or \emph{levels} of the edges of a $k$-weave. We often refer to $k$-weaves simply as {\emph weaves} if $k$ is implicit by context or arbitrary.\hfill$\Box$
\begin{figure}[H]
    \centering
    \begin{tikzpicture}[baseline=0]
        \foreach \i in {0,1,2}
        {
        \draw[blue] (90-120*\i:1.2)  -- (0,0);
        }
        \node at (0,-1.8) [] {$\begin{array}{c} \text{A monochromatic} \\ \text{trivalent vertex}\end{array}$};
    \end{tikzpicture}\hspace{1cm}
    \begin{tikzpicture}[baseline=0]
        \foreach \i in {0,1,2} 
        {
        \draw [blue] (-120*\i:1) -- (0,0);
        \draw [red] (60-120*\i:1) -- (0,0);
        \node[blue] at (-10-120*\i:1) [] {$i$};
        \node [red] at (50-120*\i:1) [] {$j$};
        }
        \node at (0,-1.8) [] {$\begin{array}{c}
             \text{A hexavalent vertex} \\
             \text{with colors $|i-j|=1$} 
        \end{array}$};
    \end{tikzpicture}\hspace{1cm}
    \begin{tikzpicture}[baseline=0]
        \foreach \i in {0,1} 
        {
        \draw [blue] (45-180*\i:1) -- (0,0);
        \draw [teal] (-45-180*\i:1) -- (0,0);
        \node [blue] at (35-180*\i:1) [] {$i$};
        \node [teal] at (-55-180*\i:1) [] {$j$};
        }
        \node at (0,-1.8) [] {$\begin{array}{c}
        \text{A tetravalent vertex} \\
        \text{with colors $|i-j|>1$}
        \end{array}$};
    \end{tikzpicture}
    \caption{Allowable vertices in a weave. Note that conditions are imposed for the labels. For a 3-valent vertex, all labels must coincide. For a 6-valent vertex the labels must alternate between $i$ and $i+1$, $i\in[k-1]$. For a 4-valent vertex the labels must alternate between $i$ and $j$, with $i,j\in[k-1]$ and $|i-j|>1$.}
    \label{fig: weave vertices}
\end{figure}
\end{defn}

In the study of positroid varieties in the Grassmannian $\Gr(m,n)$, we shall consider $k=m$. That is, the weaves that we construct for positroids in $\Gr(m,n)$ are all $m$-weaves.

\begin{exmp}
A 2-weave is a trivalent graph. Note that the cluster algebra for the top-dimensional positroid in $\Gr(2,n)$ can be studied via triangulations of the $(n+2)$-gon, cf.~\cite[Chapter 5.3]{FWZ}. These triangulations are naturally dual to certain 2-weaves.\hfill$\Box$
\end{exmp}

\noindent To distinguish edges in a weave from edges in a plabic graph, we refer to edges in a weave as \emph{weave lines}. By definition, weave lines incident to the boundary $\partial \bD$ of the 2-disk $\bD$ are said to be \emph{external weave lines}.

\begin{defn}\label{defn: Y-cycles} Let $\ww$ be a weave and let $E_\ww$ be the set of weave lines in $\ww$. A \emph{$Y$-cycle} is a map $\gamma:E_\ww\rightarrow \mathbb{Z}_{\geq 0}$ satisfying the following three conditions:\\
\begin{itemize}
    \item[-] Among the three weave lines $a,b,c$ incident to a trivalent weave vertex, the minimum of $\gamma(a)$, $\gamma(b)$, and $\gamma(c)$ is achieved at least twice.\\
    
    \item[-] Among the four weave lines $a,b,c,d$ incident to a tetravalent weave vertex (in a cyclic order), $\gamma(a)=\gamma(c)$ and $\gamma(b)=\gamma(d)$.\\
    
    \item[-] Among the six weave lines $a,b,c,d,e,f$ incident to a hexavalent weave vertex (in a cyclic order), $\gamma(a)-\gamma(d)=\gamma(e)-\gamma(b)=\gamma(c)-\gamma(f)$. The minimum of $\gamma(a)$, $\gamma(c)$, and $\gamma(e)$ is achieved at least twice, and the minimum of $\gamma(b)$, $\gamma(d)$, and $\gamma(f)$ is achieved at least twice.\\
\end{itemize}
The \emph{support} of a $Y$-cycle is the subgraph of $\ww$ spanned by the weave lines $e$ with $\gamma(e)>0$. We denote the set of $Y$-cycles on $\ww$ by $Y_\ww$.\hfill$\Box$
\end{defn}

\begin{defn}\label{defn: Y tree} A $Y$-cycle is said to be a \emph{$Y$-tree} if it satisfies the following three conditions:
\begin{itemize}
    \item[-] $\gamma(e)\in\{0,1\}$ for all weave lines $e$;\\
    
    \item[-] $\min\{\gamma(a),\gamma(b),\gamma(c)\}=0$ for the weave lines $a,b,c$ near each trivalent weave vertex;\\
    
    \item[-] its support is homeomorphic to a connected trivalent tree.
\end{itemize}
A $Y$-tree is \emph{mutable} if $\gamma(e)=0$ for all external weave lines $e$. A $Y$-tree that is not mutable is said to be  \emph{frozen}. A mutable $Y$-tree is a \emph{short $I$-cycle} if $\gamma(e)=0$ for all but one single weave line.\hfill$\Box$
\end{defn}

In a $Y$-tree, trivalent vertices of the tree only occur at hexavalent weave vertices, 2-valent vertices occur at either hexavalent or tetravalent weave vertices, while all leaves of the tree end at trivalent weave vertices. The sugar-free hull cycles introduced in \cite{CW} are examples of $Y$-trees, and the Lusztig cycles introduced in \cite{CGGLSS} are examples of $Y$-cycles.\\

We need an intersection pairing for $Y$-cycles to construct a quiver from a collection of $Y$-cycles.

\begin{defn}\label{defn:local intersection pairing} Let $V_\ww$ denote the set of weave vertices in a weave $\ww$. We define an \emph{intersection pairing} $\inprod{\cdot}{\cdot}:Y_\ww\times Y_\ww\rightarrow \bZ$ between any two $Y$-cycles by
\[
\inprod{\gamma}{\gamma'}:=\sum_{v\in V_\ww}\inprod{\gamma}{\gamma'}_v,
\]
where
\[
\inprod{\gamma}{\gamma'}_v:=\left\{\begin{array}{ll}\vspace{1cm} \det \begin{pmatrix}
    1 & 1 & 1 \\
    \gamma(a) & \gamma(b) & \gamma(c)\\
    \gamma'(a) & \gamma'(b) & \gamma'(c)
\end{pmatrix} & \text{if \quad  $\begin{tikzpicture}[baseline=0]
    \draw[blue] (0,0) node [black, below] {$v$} --  node [right] {$a$} (90:1);
    \draw[blue] (0,0) -- node [above left] {$b$} (-150:1) ;
    \draw [blue] (0,0) -- node [above right] {$c$} (-30:1) ;
\end{tikzpicture}$} \\
\vspace{1cm}
\frac{1}{2}\left(\det\begin{pmatrix} 1 & 1 & 1\\
\gamma(a) & \gamma(c) & \gamma(e) \\
\gamma'(a) & \gamma'(c) & \gamma'(e)
    \end{pmatrix} +\det\begin{pmatrix} 1 & 1 & 1\\
    \gamma(b) & \gamma(d) & \gamma(f) \\ 
    \gamma'(b) & \gamma'(d) & \gamma'(f)
    \end{pmatrix}\right) & \text{if \quad $\begin{tikzpicture}[baseline=0]
        \draw[blue] (0,0) node [black, below] {$v$} --  node [above right] {$a$} (0:1.2);
        \draw [red] (0,0) -- node [right] {$b$} (60:1.2);
    \draw[blue] (0,0) -- node [left] {$c$} (120:1.2) ;
    \draw [red] (0,0) -- node [above left] {$d$} (180:1.2);
    \draw [blue] (0,0) -- node [left] {$e$} (-120:1.2) ;
    \draw [red] (0,0) --  node [right] {$f$} (-60:1.2);
    \end{tikzpicture}$} \\
    0 & \text{otherwise.}
\end{array}\right.
\]
The intersection pairing is skew-symmetric by construction.\hfill$\Box$
\end{defn}

There are certain moves between weaves that are considered equivalences. Figure~\ref{fig:weave equivalence} lists the allowed equivalence moves between weaves.\footnote{These six moves are not entirely independent: Move III can be deduced from Moves I and II. In practice, it is nevertheless useful to emphasize Move III.} By definition, two weaves are said to be \emph{equivalent} if they are related by a sequence of equivalence moves.
\begin{figure}[H]
    \centering
    \includegraphics[scale=0.9]{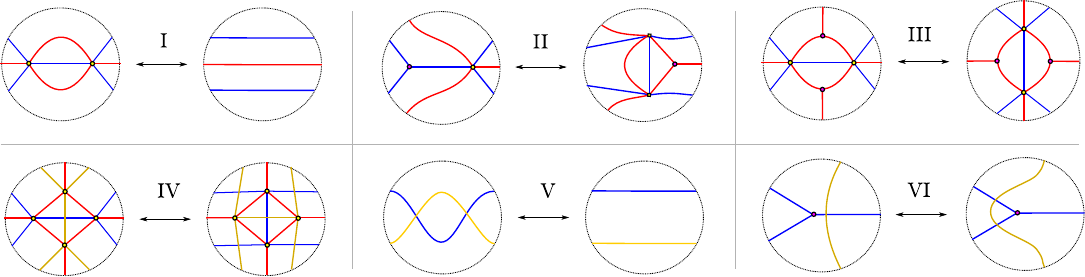}
    \caption{List of weave equivalences.}
    \label{fig:weave equivalence}
\end{figure}

\noindent For a short $I$-cycle $\gamma$ on $\ww$, one can perform a move known as \emph{weave mutation} to produce a new weave $\ww'$. By definition, weave mutation is the move depicted in Figure \ref{fig: weave mutation}. Note that two weaves that differ by a weave mutation are \emph{not} equivalent to each other.

\begin{figure}[H]
    \centering
    \begin{tikzpicture}[baseline=0]
        \draw [blue] (-0.5,0) -- node [below] {$\gamma$} (0.5,0);
        \foreach \i in {0,1}
        {
        \draw [blue] (45-90*\i:1) -- (0.5,0);
        \draw [blue] (135+90*\i:1) -- (-0.5,0);
        }
    \end{tikzpicture} \hspace{1cm} $\longleftrightarrow$ \hspace{1cm}
    \begin{tikzpicture}[baseline=0]
        \draw [blue] (0, -0.5) -- node [right] {$\gamma'$} (0, 0.5);
        \foreach \i in {0,1}
        {
        \draw [blue] (135-90*\i:1) -- (0,0.5);
        \draw [blue] (-45-90*\i:1) -- (0,-0.5);
        }
    \end{tikzpicture}
    \caption{A weave mutation.}
    \label{fig: weave mutation}
\end{figure}
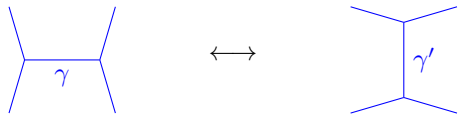

\noindent Weave mutation is defined at a short $I$-cycle. That said, the following lemma can be used to mutate at any (mutable) $Y$-tree, as introduced in Definition \ref{defn: Y tree}.

\begin{lem}[{\cite[Prop.~3.5]{CW}}]\label{lem: turning Y-trees into short I cycles} For any weave $\ww$ and a mutable $Y$-tree $\gamma$, there exists a weave equivalence $\ww'\sim \ww$ that turns $\gamma$ into a short $I$-cycle on $\ww'$.\hfill$\Box$
\end{lem}

\noindent Indeed, given a mutable $Y$-tree in a weave $\ww$, we can use Lemma \ref{lem: turning Y-trees into short I cycles} to perform weave equivalences on $\ww$ until the $Y$-tree becomes a short $I$-cycle. Then we can implement the mutation in Figure \ref{fig: weave mutation} so as to achieve the desired mutation at the $Y$-tree.

\begin{rmk} There is contact and symplectic geometry behind weaves. In a nutshell, each weave $\ww$ is describing the singular locus of a front projection of some Legendrian surface $\Lambda_\ww$ embedded in the standard contact $\mathbb{R}^5$. A $Y$-cycle is topologically a $1$-cycle on $\Lambda_\ww$ relative to the boundary. A weave equivalence induces a contact isotopy between these Legendrian surfaces. A weave mutation encodes a Lagrangian disk surgery on the exact Lagrangian projections of the corresponding Legendrian surfaces. See \cite[Section 3]{CZ} for more details.\hfill$\Box$
\end{rmk}

\begin{rmk} In Section \ref{sec:iterative_Tmap} we describe how to produce a weave from a plabic graph such that each bounded face of the plabic graph corresponds to a $Y$-tree on the weave. By combining Lemma \ref{lem: turning Y-trees into short I cycles} and weave mutations, we are able to diagramatically describe mutations at all mutable vertices of the initial quiver using weaves. In this aspect, weaves are more versatile than plabic graphs, which can only describe mutations that correspond to square moves on bounded square faces. In fact, there are known cases in which infinitely many seeds can be described with weaves, cf.~\cite[Section 4]{CG22} or \cite[Section 7.3]{CZ}.\hfill$\Box$
\end{rmk}

In \cite[Section 4.3]{CGSS20} and \cite[Section 4.1]{CGGLSS}, we considered a special family of weaves called Demazure weaves. We used them in \cite[Section 5]{CGGLSS} to construct and study cluster structures on braid varieties. We will prove in Section \ref{sec: Demazure} that weaves from plabic graphs are equivalent to Demazure weaves. We review some basics aspects of Demazure weaves here.

\begin{defn}\label{defn: demazure weave} A \emph{Demazure weave} is a weave drawn on the rectangle $R=[0,1]\times[0,1]$ such that:
\begin{enumerate}
    \item All external weave lines are incident to either the top or the bottom boundary of $R$.
    \item No point along any weave line admits a horizontal tangent line.
    \item Each trivalent weave vertex is incident to two weave lines above it and one below it.
    \item Each tetravalent weave vertex is incident to two weave lines above it and two below it.
    \item Each hexavalent weave vertex is incident to three weave lines above it and three below it.
\end{enumerate}

It is useful to picture the weave lines in a Demazure weave as being oriented from top to bottom. For a Demazure $k$-weave $\ww$, let $\beta_\text{top}$ (resp. $\beta_\text{bottom}$) be the positive braid word formed by the external weave lines along the top boundary $[0,1]\times\{1\}$, resp. bottom boundary $[0,1]\times\{0\}$, of $R$. A Demazure weave $\ww$ is \emph{complete} if $\beta_\text{bottom}$ is a reduced word for the longest permutation $w_0$ in $S_k$. In that case, the Demazure product of $[\beta_\text{top}]$ is $w_0$.\hfill$\Box$
\end{defn}

Demazure weaves have an important property: there is a particularly well-behaved collection of weave cycles associated to them, which we named {\it Lusztig} cycles in \cite[Section 4.4]{CGGLSS}. In many cases, such as the inductive Demazure weaves from \cite[Section 4.3]{CGGLSS}, these cycles are all $Y$-trees. These Lusztig cycles are defined as follows:

\begin{defn}\label{defn:Lusztig cycles} Let $\ww$ be a Demazure weave and let $v$ be a trivalent weave vertex in $\ww$. The \emph{Lusztig cycle} associated with $v$ is constructed as we scan $\ww$ from top to bottom as follows.
\begin{enumerate}
    \item Any weave line $e$ that begins above $v$ is assigned with $\gamma(e)=0$.
    \item The unique weave line that begins at $v$ is assigned with $\gamma(e)=1$.
    \item For any trivalent vertex $w$ below $v$, the assignments satisfy $\gamma(c)=\min\{\gamma(a),\gamma(b)\}$, where $c$ is the unique weave line that begins at $w$.
    \item For any tetravalent vertex $w$, the assignments satisfy $\gamma(a)=\gamma(c)$ and $\gamma(b)=\gamma(d)$, where $a,b,c,d$ are weave lines incident to $w$ in a cyclic order, with $a,b$ above $w$ and $c,d$ below $w$.
    \item For any hexavalent vertex $w$, the assignments satisfy $\gamma(d)=\gamma(a)+\gamma(b)-\min\{\gamma(a),\gamma(c)\}$, $\gamma(e)=\min\{\gamma(a),\gamma(c)\}$, and $\gamma(f)=\gamma(b)+\gamma(c)-\min\{\gamma(a),\gamma(c)\}$, where $a,b,c,d,e,f$ are weave lines incident to $w$ in a cyclic order, with $a,b,c$ above $w$ and $d,e,f$ below $w$.
\end{enumerate}
\hfill$\Box$
\end{defn}

\section{Weaves from Reduced Plabic Graphs}\label{sec:iterative_Tmap} 

In this section, we describe how to construct a weave from a reduced plabic graph, in Prop.~\ref{prop:weaves from plabic graphs}, and prove that equivalent reduced plabic graphs produce equivalent weaves, in Theorem~\ref{thm:weave construction}. The key ingredient is the T-shift, a recursive procedure for constructing a new reduced plabic graph from a previous one, cf.~Definition~\ref{def:Tmap}. If $\bG$ has no solid lollipops, this procedure is exactly ``T-duality" as defined in \cite[Definition 8.7]{PSBW} to study the $m=2$ amplituhedron.\\

\noindent In previous sections, we used $[n]$ as the ground set for a positroid, the columns of matrices, and the boundary vertices of plabic graphs. We may instead use any subset $I \subset \bZ_{>0}$ as the ground set, ordered in the standard way, and define positroids, plabic graphs, Grassmann necklaces and other concepts with ground set $I$.

\subsection{T-shifts of plabic graphs}\label{ssec:Tshift} We build weaves from plabic graphs using the following operation.

\begin{defn}\label{def:Tmap}
    Let $\bG$ be a reduced plabic graph whose solid vertices are either lollipops or trivalent. The \emph{T-shift} of $\bG$ is the plabic graph obtained from the following steps:
    \begin{enumerate}
        \item Delete solid lollipops and its incident marked points on $\partial \bD$.
        \item For each marked point $i$ not incident to a solid lollipop, add a new marked point $i'$ to $\partial\bD$ slightly counterclockwise of $i$.
        \item Place a new empty vertex $v_s'$ on top of each trivalent solid vertex $s$ and place a new solid vertex $v_F'$ in each face $F$ of $\bG$. 
        \item For each solid vertex $s$ of $\bG$ on the boundary of a face $F$, add an edge between $v_s'$ and $v_F'$. Also add an edge between $v_F'$ and any new boundary marked points $i'$ in $F$.
        \item Delete degree 2 solid vertices and replace each degree $d>2$ solid vertex with a trivalent tree of solid vertices with $d$ leaves. Call the collection of these trivalent trees $\tau$.
    \end{enumerate}
    We denote the T-shift of $\bG$ by $\bGshift(\tau)$. Different choices of $\tau$ produce graphs which differ by equivalences (involving only solid vertices), so we frequently abuse notation and write $\bGshift$ instead of $\bGshift(\tau)$.\hfill$\Box$
\end{defn}

\begin{rmk}\label{rmk:inverse-T-shift}
    One may define the operation ``inverse T-shift" by switching ``empty" and ``solid" in Definition~\ref{def:Tmap}. If $\bG$ has no solid lollipops, then applying inverse T-shift to $\bGshift$ gives a plabic graph equivalent to $\bG$ using moves only involving empty vertices, cf.~\cite[Remark 8.9]{PSBW}.\hfill$\Box$
\end{rmk}

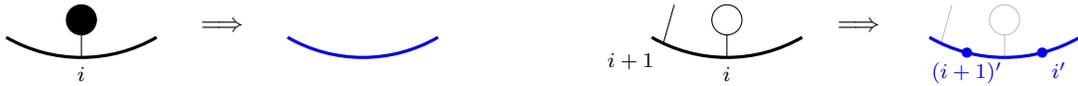
\begin{figure}[H]
    \centering
    \begin{tikzpicture}[baseline=10]
        \draw[very thick] (0,0) arc (-90:-60:2);
        \draw [very thick] (0,0) arc (-90:-120:2);
        \solver[black](0) (0,0.5);
        \draw (0,0.5) -- (0,0) node [below] {\footnotesize{$i$}};
    \end{tikzpicture} \quad $\implies$ \quad \begin{tikzpicture}[baseline=10]
        \draw[very thick,blue] (0,0) arc (-90:-60:2);
        \draw [very thick,blue] (0,0) arc (-90:-120:2);
    \end{tikzpicture}\hspace{2cm}
    \begin{tikzpicture}[baseline=10]
        \draw[very thick] (0,0) arc (-90:-60:2);
        \draw [very thick] (0,0) arc (-90:-120:2);
        \draw (0,0) arc (-90:-115:2) node [below left] {\footnotesize{$i+1$}} -- (-0.7,0.7);
        \draw (0,0) node [below] {\footnotesize{$i$}}--(0,0.3);
        \empver[black](0) (0,0.5);
    \end{tikzpicture}\quad $\implies $ \quad 
    \begin{tikzpicture}[baseline=10]
        \draw [lightgray] (0,0) arc (-90:-115:2)  -- (-0.7,0.7);
        \draw [lightgray] (0,0) --(0,0.3);
        \empver[lightgray](0) (0,0.5);
        \draw[blue, very thick] (0,0) arc (-90:-60:2);
        \draw [blue, very thick] (0,0) arc (-90:-120:2);
        \node [blue] at (0.5,0.06) [] {$\bullet$};
        \node [blue] at (0.5,0.06) [below right] {\footnotesize{$i'$}};
       \node [blue] at (-0.5,0.06) [] {$\bullet$};
        \node [blue] at (-0.5,0.06) [below] {\footnotesize{$(i+1)'$}};
    \end{tikzpicture}
    \caption{Steps (1) and (2) in the T-shift procedure.}
\end{figure}

\begin{figure}[H]
    \centering
    \begin{tikzpicture}[baseline=0]
        \foreach \i in {0,...,4}
        {
            \draw [shorten <=0.2cm, shorten >=0.2cm] (90+\i*72:1) -- (18+\i*72:1);
            \draw [shorten <=0.2cm] (90+\i*72:1) -- (90+\i*72:1.5);
        }
        \foreach \i in {0,2,3}
        {
            \draw [fill=black] (90+\i*72:1) circle [radius=0.2];
        }
        \foreach \i in {1,4}
        {
            \draw  (90+\i*72:1) circle [radius=0.2];
        }
        \end{tikzpicture}
    \quad $\implies$ \quad \begin{tikzpicture}[baseline=0]
        \foreach \i in {0,...,4}
        {
            \draw [lightgray, shorten <=0.2cm, shorten >=0.2cm] (90+\i*72:1) -- (18+\i*72:1);
            \draw [lightgray, shorten <=0.2cm] (90+\i*72:1) -- (90+\i*72:1.5);
        }
        \foreach \i in {0,2,3}
        {
            \draw [lightgray, fill=lightgray] (90+\i*72:1) circle [radius=0.2];
            \draw [blue] (90+\i*72:1) circle [radius=0.2];
        }
        \foreach \i in {1,4}
        {
            \draw [lightgray] (90+\i*72:1) circle [radius=0.2];
        }
        \solver[blue] (0) (0,0);
        \end{tikzpicture}\quad $\implies$ \quad \begin{tikzpicture}[baseline=0]
        \foreach \i in {0,...,4}
        {
            \draw [lightgray, shorten <=0.2cm, shorten >=0.2cm] (90+\i*72:1) -- (18+\i*72:1);
            \draw [lightgray, shorten <=0.2cm] (90+\i*72:1) -- (90+\i*72:1.5);
        }
        \foreach \i in {0,2,3}
        {
            \draw [lightgray, fill=lightgray] (90+\i*72:1) circle [radius=0.2];
            \draw [blue] (90+\i*72:1) circle [radius=0.2];
            \draw [blue, shorten <=0.2cm] (90+\i*72:1) -- (75+\i*72:1.5);
            \draw [blue, shorten <=0.2cm] (90+\i*72:1) -- (105+\i*72:1.5);
            \draw [blue, shorten <=0.2cm] (90+\i*72:1) -- (0,0);
        }
        \foreach \i in {1,4}
        {
            \draw [lightgray] (90+\i*72:1) circle [radius=0.2];
        }
        \solver[blue] (0) (0,0);
        \end{tikzpicture}
    \caption{Steps (3) and (4) in the T-shift procedure.}
\end{figure}
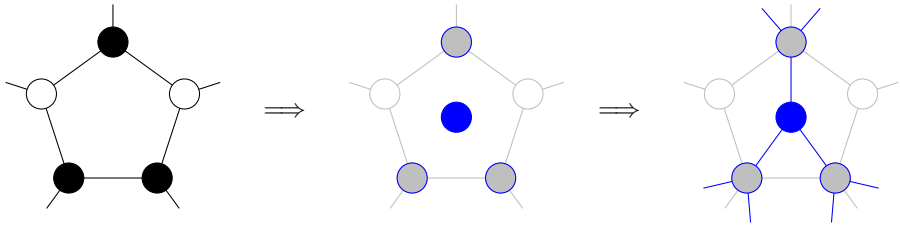

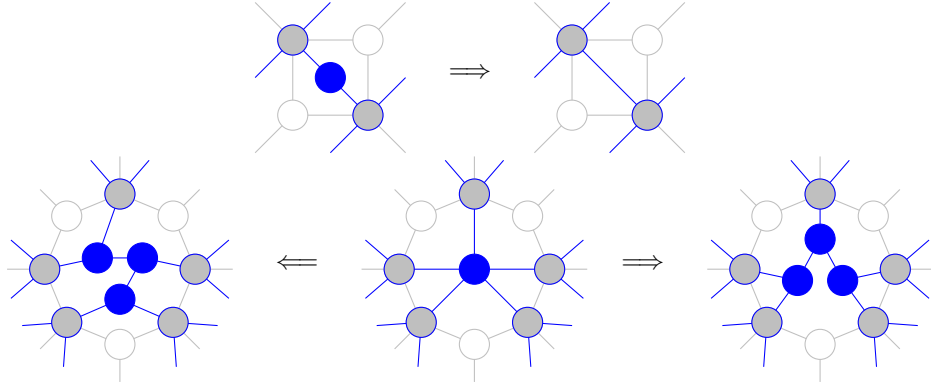
\begin{figure}[H]
    \centering
    \begin{tikzpicture}[baseline=0]
        \empver[lightgray] (1) (-0.5,-0.5);
        \solver[lightgray] (2) (-0.5,0.5);
        \empver[lightgray] (3) (0.5,0.5);
        \solver[lightgray] (4) (0.5,-0.5);
        \edge[lightgray] (1) (2);
        \edge[lightgray] (2) (3);
        \edge[lightgray] (3) (4);
        \edge[lightgray] (1) (4);
        \draw[lightgray, shorten <=0.2cm] (1) -- (-1,-1);
        \draw[lightgray, shorten <=0.2cm] (2) -- (-1,1);
        \draw[lightgray, shorten <=0.2cm] (3) -- (1,1);
        \draw[lightgray, shorten <=0.2cm] (4) -- (1,-1);
        \empver[blue] (5) (-0.5,0.5);
        \empver[blue] (6) (0.5,-0.5);
        \solver[blue] (0) (0,0);
        \edge[blue] (5) (6);
        \draw[blue, shorten <=0.2cm] (5) -- (0,1);
        \draw[blue, shorten <=0.2cm] (5) -- (-1,0);
        \draw[blue, shorten <=0.2cm] (6) -- (0,-1);
        \draw[blue, shorten <=0.2cm] (6) -- (1,0);
        \end{tikzpicture} \quad $\implies $ \quad 
        \begin{tikzpicture}[baseline=0]
        \empver[lightgray] (1) (-0.5,-0.5);
        \solver[lightgray] (2) (-0.5,0.5);
        \empver[lightgray] (3) (0.5,0.5);
        \solver[lightgray] (4) (0.5,-0.5);
        \edge[lightgray] (1) (2);
        \edge[lightgray] (2) (3);
        \edge[lightgray] (3) (4);
        \edge[lightgray] (1) (4);
        \draw[lightgray, shorten <=0.2cm] (1) -- (-1,-1);
        \draw[lightgray, shorten <=0.2cm] (2) -- (-1,1);
        \draw[lightgray, shorten <=0.2cm] (3) -- (1,1);
        \draw[lightgray, shorten <=0.2cm] (4) -- (1,-1);
        \empver[blue] (5) (-0.5,0.5);
        \empver[blue] (6) (0.5,-0.5);
        \edge[blue] (5) (6);
        \draw[blue, shorten <=0.2cm] (5) -- (0,1);
        \draw[blue, shorten <=0.2cm] (5) -- (-1,0);
        \draw[blue, shorten <=0.2cm] (6) -- (0,-1);
        \draw[blue, shorten <=0.2cm] (6) -- (1,0);
        \end{tikzpicture}\hspace{2cm} \\
    \begin{tikzpicture}[baseline=0]
        \foreach \i in {0,...,7}
        {
            \draw [lightgray, shorten <=0.2cm, shorten >=0.2cm] (90+\i*45:1) -- (45+\i*45:1);
            \draw [lightgray, shorten <=0.2cm] (90+\i*45:1) -- (90+\i*45:1.5);
        }
        \foreach \i in {0,2,3,5,6}
        {
            \draw [lightgray, fill=lightgray] (90+\i*45:1) circle [radius=0.2];
            \draw [blue] (90+\i*45:1) circle [radius=0.2];
            \draw [blue, shorten <=0.2cm] (90+\i*45:1) -- (75+\i*45:1.5);
            \draw [blue, shorten <=0.2cm] (90+\i*45:1) -- (105+\i*45:1.5);
        }
        \foreach \i in {1,4,7}
        {
            \draw [lightgray] (90+\i*45:1) circle [radius=0.2];
        }
        \draw [blue, shorten <=0.2cm, shorten >=0.2cm] (180:1) -- (-0.3,0.15) -- (90:1);
        \draw [blue, shorten <=0.2cm, shorten >=0.2cm] (0:1) -- (0.3,0.15);
        \draw [blue, shorten <=0.2cm, shorten >=0.2cm] (-45:1) -- (0,-0.4) -- (-135:1);
        \solver[blue] (0) (-0.3,0.15);
        \solver[blue] (1) (0.3,0.15);
        \solver[blue] (2) (0,-0.4);
        \edge[blue] (0) (1);
        \edge[blue] (1) (2);
        \end{tikzpicture} \quad $\impliedby$ \quad     
     \begin{tikzpicture}[baseline=0]
        \foreach \i in {0,...,7}
        {
            \draw [lightgray, shorten <=0.2cm, shorten >=0.2cm] (90+\i*45:1) -- (45+\i*45:1);
            \draw [lightgray, shorten <=0.2cm] (90+\i*45:1) -- (90+\i*45:1.5);
        }
        \foreach \i in {0,2,3,5,6}
        {
            \draw [lightgray, fill=lightgray] (90+\i*45:1) circle [radius=0.2];
            \draw [blue] (90+\i*45:1) circle [radius=0.2];
            \draw [blue, shorten <=0.2cm] (90+\i*45:1) -- (75+\i*45:1.5);
            \draw [blue, shorten <=0.2cm] (90+\i*45:1) -- (105+\i*45:1.5);
        }
        \foreach \i in {1,4,7}
        {
            \draw [lightgray] (90+\i*45:1) circle [radius=0.2];
        }
        \draw [blue, shorten <=0.2cm, shorten >=0.2cm] (180:1) -- (0,0) -- (-135:1);
        \draw [blue, shorten <=0.2cm, shorten >=0.2cm] (0:1) -- (0,0) -- (-45:1);
        \draw [blue, shorten <=0.2cm, shorten >=0.2cm] (90:1) -- (0,0);
        \solver[blue] (0) (0,0);
        \end{tikzpicture}\quad $\implies$ \quad
    \begin{tikzpicture}[baseline=0]
        \foreach \i in {0,...,7}
        {
            \draw [lightgray, shorten <=0.2cm, shorten >=0.2cm] (90+\i*45:1) -- (45+\i*45:1);
            \draw [lightgray, shorten <=0.2cm] (90+\i*45:1) -- (90+\i*45:1.5);
        }
        \foreach \i in {0,2,3,5,6}
        {
            \draw [lightgray, fill=lightgray] (90+\i*45:1) circle [radius=0.2];
            \draw [blue] (90+\i*45:1) circle [radius=0.2];
            \draw [blue, shorten <=0.2cm] (90+\i*45:1) -- (75+\i*45:1.5);
            \draw [blue, shorten <=0.2cm] (90+\i*45:1) -- (105+\i*45:1.5);
        }
        \foreach \i in {1,4,7}
        {
            \draw [lightgray] (90+\i*45:1) circle [radius=0.2];
        }
        \draw [blue, shorten <=0.2cm, shorten >=0.2cm] (180:1) -- (-0.3,-0.15) -- (-135:1);
        \draw [blue, shorten <=0.2cm, shorten >=0.2cm] (0:1) -- (0.3,-0.15) -- (-45:1);
        \draw [blue, shorten <=0.2cm, shorten >=0.2cm] (90:1) -- (0,0.4);
        \solver[blue] (0) (-0.3,-0.15);
        \solver[blue] (1) (0.3,-0.15);
        \solver[blue] (2) (0,0.4);
        \edge[blue] (0) (2);
        \edge[blue] (1) (2);
        \end{tikzpicture}
    \caption{Two instances of performing Step (5) in the T-shift procedure. In the second row, we show two possible resolutions of the $5$-valent solid vertex using trivalent trees. There are other resolutions as well.
    }
    \label{fig:my_label}
\end{figure}

\noindent Note that a boundary face of $\bG$ which contains only empty vertices will produce a solid lollipop in $\bGshift$. We call such faces \emph{exceptional boundary faces}, and they will require special care in some of our arguments.

    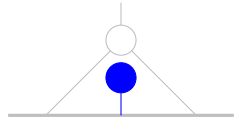
\begin{figure}[H]
        \centering
        \begin{tikzpicture}[baseline=0]
        \draw [very thick, lightgray] (-1.5,0) -- (1.5,0);
        \empver[lightgray] (0) (0,1);
        \draw [lightgray, shorten <=0.2cm] (0) -- (-1,0);
        \draw [lightgray, shorten <=0.2cm] (0) -- (1,0);
        \draw [lightgray, shorten <=0.2cm] (0) -- (0,1.5);
        \draw [blue, fill=blue] (0,0.5) circle [radius=0.2];
        \draw [blue] (0,0) -- (0,0.3);
        \end{tikzpicture}
        \caption{An exceptional boundary face where T-shift produces a solid lollipop.}
        \label{fig:exceptional boundary face}
    \end{figure}

\noindent Here are properties of T-shifts, allowing us to iterate. See Definition~\ref{def:zig-zag-strand} for rank of a plabic graph.


 \begin{prop} \label{prop:single T-shift}
     Let $\bG$ be a reduced plabic graph whose solid vertices are lollipops or trivalent. Then $\bGshift$ is again reduced, with solid vertices that are lollipops or trivalent, and has rank 1 less than the rank of $\bG$. 
     Also, $\bGshift$ and the boundary vertex labels of $\bG$ determine $\bG$, up to equivalences involving empty vertices. 
 \end{prop}

 \begin{proof} 
     Let $\bG'$ be the graph obtained from $\bG$ by deleting solid lollipops and the adjacent boundary vertices, and relabeling the $r$ remaining boundary vertices with $[r]$. Note that $\bG'$ and $\bG$ have the same rank. Applying the T-shift to $\bG'$ results in a relabeling of $\bGshift$, cf.~\cite[Definition 8.7]{PSBW}, where the T-shift is called T-duality. As proved in \cite[Proposition 8.8]{PSBW}, the zig-zag strands in $\bGshift$ are pushed to the left, see Figure \ref{fig:corresponence between zig-zags}. As a result, the bounded affine permutation $f^{\downarrow}$ of $\bGshift$ relates to $f$ of $\bG'$ by $f^{\downarrow}(i)=f(i-1)$, cf.~Figure \ref{fig:T-shift from zig-zags}. This implies that $\bGshift$ is of rank 1 less than $\bG'$. 
     
     For the second statement, applying inverse T-shift to the T-shift of $\bG'$ gives $\bG'$ up to empty-empty edge contraction/expansion and bivalent empty vertex insertion/removal, cf.~ Remark~\ref{rmk:inverse-T-shift}. This implies that $\bGshift$ determines $\bG$ up to equivalences involving empty vertices and up to solid lollipops. Solid lollipops of $\bG$ exactly correspond to boundary vertex labels of $\bG$ which are not boundary vertex labels of $\bGshift$. So altogether, $\bGshift$ and the boundary vertex labels of $\bG$ determine $\bG$ up to empty equivalence.
 \end{proof}
 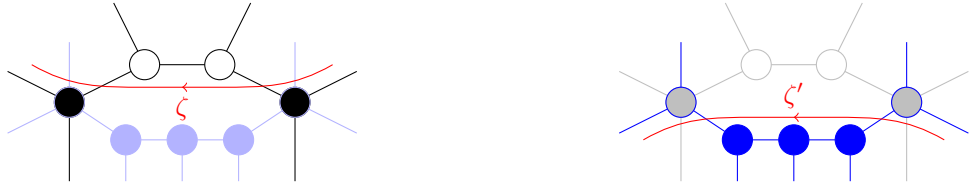
\begin{figure}[H]
    \centering
    \begin{tikzpicture}
    \solver[black] (0) (0,0.5);
    \empver[black] (1) (1,1);
    \empver[black] (2) (2,1);
    \solver[black] (3) (3,0.5);
    \edge[black] (0) (1);
    \edge[black] (1) (2);
    \edge[black] (2) (3);
    \edge[black] (0) (0,-0.75);
    \edge[black] (3) (3,-0.75);
    \edge[black] (0) (-1,1);
    \edge[black] (3) (4,1);
    \edge[black] (1) (0.5,2);
    \edge[black] (2) (2.5,2);
    \empver[lightblue] (4) (0,0.5);
    \empver[lightblue] (5) (3,0.5);
    \solver[lightblue] (6) (0.75,0);
    \solver[lightblue] (7) (1.5,0);
    \solver[lightblue] (8) (2.25,0);
    \edge[lightblue] (4) (0,1.5);
    \edge[lightblue] (4) (-1,0);
    \edge[lightblue] (5) (3,1.5);
    \edge[lightblue] (5) (4,0);
    \edge[lightblue] (4) (6);
    \edge[lightblue] (6) (7);
    \edge[lightblue] (8) (7);
    \edge[lightblue] (5) (8);
    \edge[lightblue] (6) (0.75,-0.75);
    \edge[lightblue] (7) (1.5,-0.75);
    \edge[lightblue] (8) (2.25,-0.75);
    \draw[red,decoration={markings,mark=at position 0.5 with {\arrow{>}}},postaction={decorate}] (3.5,1) to [out=-150,in=0] (2,0.7) -- node [below] {$\zeta$} (1,0.7) to [out=180,in=-30] (-0.5,1);
    \end{tikzpicture}\hspace{3cm}
    \begin{tikzpicture}
    \solver[lightgray] (0) (0,0.5);
    \empver[lightgray] (1) (1,1);
    \empver[lightgray] (2) (2,1);
    \solver[lightgray] (3) (3,0.5);
    \edge[lightgray] (0) (1);
    \edge[lightgray] (1) (2);
    \edge[lightgray] (2) (3);
    \edge[lightgray] (0) (0,-0.75);
    \edge[lightgray] (3) (3,-0.75);
    \edge[lightgray] (0) (-1,1);
    \edge[lightgray] (3) (4,1);
    \edge[lightgray] (1) (0.5,2);
    \edge[lightgray] (2) (2.5,2);
    \empver[blue] (4) (0,0.5);
    \empver[blue] (5) (3,0.5);
    \solver[blue] (6) (0.75,0);
    \solver[blue] (7) (1.5,0);
    \solver[blue] (8) (2.25,0);
    \edge[blue] (4) (0,1.5);
    \edge[blue] (4) (-1,0);
    \edge[blue] (5) (3,1.5);
    \edge[blue] (5) (4,0);
    \edge[blue] (4) (6);
    \edge[blue] (6) (7);
    \edge[blue] (8) (7);
    \edge[blue] (5) (8);
    \edge[blue] (6) (0.75,-0.75);
    \edge[blue] (7) (1.5,-0.75);
    \edge[blue] (8) (2.25,-0.75);
    \draw[red,decoration={markings,mark=at position 0.5 with {\arrow{>}}},postaction={decorate}] (3.5,0) to [out=150,in=0] (2,0.3) -- node [above]{ $\zeta'$} (1,0.3) to [out=180,in=30] (-0.5,0);
    \end{tikzpicture}
    \caption{Bijection between zig-zag strands in $\bG$ (left) and those in $\bGshift$ (right).}
    \label{fig:corresponence between zig-zags}
\end{figure}

\begin{figure}[H]
    \centering
    \begin{tikzpicture}
    \foreach \i in {-30,-60,-120,-150}
    {
    \draw (\i:1.5) -- (\i:1);
    }
    \foreach \i in {-45,-135}
    {
    \draw [lightblue] (\i:1.5) -- (\i:1);
    }
    \draw[red,decoration={markings,mark=at position 0.5 with {\arrow{>}}},postaction={decorate}] (-35:1.5) to [out=150,in=0] (0,-0.25) node [above] {$\zeta$} to [out=180,in=30] (-145:1.5);
    \draw [very thick] (0,0) circle [radius=1.5];
    \end{tikzpicture} \hspace{2cm}
    \begin{tikzpicture}
    \foreach \i in {-30,-60,-120,-150}
    {
    \draw [lightgray] (\i:1.5) -- (\i:1);
    }
    \foreach \i in {-45,-135}
    {
    \draw [blue] (\i:1.5) -- (\i:1);
    }
    \draw[red,decoration={markings,mark=at position 0.5 with {\arrow{>}}},postaction={decorate}] (-50:1.5) to [out=135,in=0] (0,-0.45) node [above] {$\zeta'$} to [out=180,in=45] (-130:1.5);
    \draw [very thick, blue] (0,0) circle [radius=1.5];
    \end{tikzpicture}
    \caption{End points of corresponding zig-zag strands in $\bG$ (left) and $\bGshift$ (right).}
    \label{fig:T-shift from zig-zags}
\end{figure}
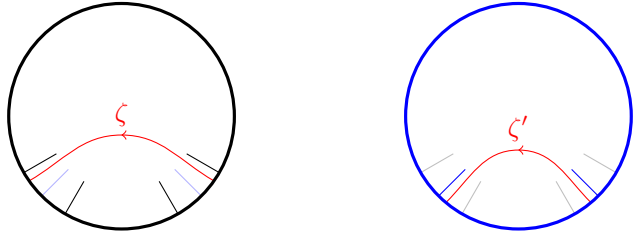

\noindent By Definition~\ref{defn:positroid-braid-word}, there is a positive braid word $\beta_\cP$ associated with a positroid $\cP$. 
The next proposition shows how T-shift changes this positive braid word, which we use later in Theorem~\ref{thm:weave construction}. Let $\beta=s_{i_1}s_{i_2}\cdots s_{i_l}$ be a positive braid word for an $m$-strand positive braid. We define $\beta^{\downarrow}$ to be the positive braid word obtained from $\beta$ by removing all occurrences of $m-1$. Then $\beta^{\downarrow}$ defines an $(m-1)$-strand positive braid.

\begin{rmk} Note that the operation $\beta \rightarrow \beta^{\downarrow}$ is an operation on positive braid words, but it is not a well-defined operation on positive braids.\hfill$\Box$
\end{rmk}

\begin{prop}\label{prop:T-shift of positive braids} Let $\bG$ be a plabic graph associated to a positroid $\cP$ and $\cP^\downarrow$ be the positroid associated with $\bGshift$. Then $(\beta_{\cP})^\downarrow$ is equal to $\beta_{\cP^{\downarrow}}$, the braid word for $\cP^\downarrow$.
\end{prop}
\begin{proof} Let us assume for simplicity that $\bG$ has no solid lollipops. The necklace $I_i$ consists of those elements $j$ such that $f^{-1}(j) >_i j$. We have that $f^{\downarrow}(i)=f(i-1)$, or in other words, $(f^{\downarrow})^{-1}(j) = f^{-1}(j) + 1$. Thus $I^{\downarrow}_i$ is $I_{i-1} \setminus \{i-1\}$. In other words, to obtain the wiring diagram for $\cP^\downarrow$, we just remove the top row from the wiring diagram for $\cP$ as in Figure \ref{fig:slicing the top row}. It is a check that this remains true in the presence of solid lollipops.
\end{proof}

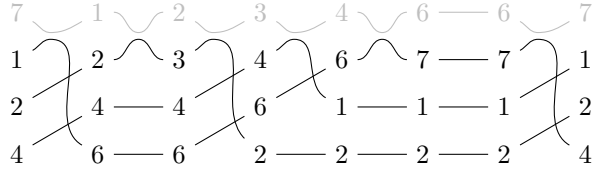
\begin{figure}[H]
    \centering
    \begin{tikzpicture}[scale=0.9,baseline=0]
    \node [lightgray] (01) at (-1.2,2.1) [] {$7$};
    \node (02) at (-1.2,1.4) [] {$1$};
    \node (03) at (-1.2,0.7) [] {$2$};
    \node (04) at (-1.2,0) [] {$4$};
    \node [lightgray](11) at (0,2.1) [] {$1$};
    \node (12) at (0,1.4) [] {$2$};
    \node (13) at (0,0.7) [] {$4$};
    \node (14) at (0,0) [] {$6$};
    \node [lightgray](21) at (1.2,2.1) [] {$2$};
    \node (22) at (1.2,1.4) [] {$3$};
    \node (23) at (1.2,0.7) [] {$4$};
    \node (24) at (1.2,0) [] {$6$};
    \node [lightgray](31) at (2.4,2.1) [] {$3$};
    \node (32) at (2.4,1.4) [] {$4$};
    \node (33) at (2.4,0.7) [] {$6$};
    \node (34) at (2.4,0) [] {$2$};
    \node [lightgray](41) at (3.6,2.1) [] {$4$};
    \node (42) at (3.6,1.4) [] {$6$};
    \node (43) at (3.6,0.7) [] {$1$};
    \node (44) at (3.6,0) [] {$2$};
    \node [lightgray](51) at (4.8,2.1) [] {$6$};
    \node (52) at (4.8,1.4) [] {$7$};
    \node (53) at (4.8,0.7) [] {$1$};
    \node (54) at (4.8,0) [] {$2$};
    \node [lightgray](61) at (6,2.1) [] {$6$};
    \node (62) at (6,1.4) [] {$7$};
    \node (63) at (6,0.7) [] {$1$};
    \node (64) at (6,0) [] {$2$};
    \node [lightgray](71) at (7.2,2.1) [] {$7$};
    \node (72) at (7.2,1.4) [] {$1$};
    \node (73) at (7.2,0.7) [] {$2$};
    \node (74) at (7.2,0) [] {$4$};
    \draw [lightgray] (11) to [out=0,in=180] (0.6,1.8) to [out=0,in=180] (21);
    \draw (12) to [out=0,in=180] (0.6,1.7) to [out=0,in=180] (22);
    \draw (13) -- (23);
    \draw (14) -- (24);
    \draw [lightgray] (21) to [out=-30,in=180] (1.7,1.8) to [out=0,in=-150] (31);
    \draw (22) to [out=30,in=180] (1.7,1.7) to [out=0] (34);
    \draw (32) -- (23);
    \draw (33) -- (24);
    \draw [lightgray] (31) to [out=-30,in=180] (2.9,1.8) to [out=0,in=-150] (41);
    \draw (34) -- (44);
    \draw (32) to [out=30,in=180] (2.9,1.7) to [out=0,in=150] (43);
    \draw (33) -- (42);
    \draw (53) -- (43);
    \draw (54) -- (44);
    \draw [lightgray] (41) to [out=0,in=180] (4.2,1.8) to [out=0,in=180] (51);
    \draw (42) to [out=0,in=180] (4.2,1.7) to [out=0,in=180] (52);
    \foreach \i in {2,...,4}
    {
        \draw (5\i) -- (6\i);
    }
    \draw [lightgray] (51) -- (61);
    \draw [lightgray] (61) to [out=-30,in=180] (6.5,1.8) to [out=0,in=-150] (71);
    \draw (62) to [out=30,in=180] (6.5,1.7) to [out=0,in=150] (74);
    \draw (63) -- (72);
    \draw (64) -- (73);
    \draw (13) -- (04);
    \draw [lightgray] (01) to [out=-30,in=180] (-0.7,1.8) to [out=0,in=-150] (11);
    \draw (02) to [out=30,in=180] (-0.7,1.7) to [out=0,in=150](14);
    \draw (12) -- (03);
    \end{tikzpicture}
    \caption{The effect of T-shift on the cyclic positive braid in Example \ref{exmp:positroid braid}. This is the first T-shift done in Example \ref{exmp: positroid weave}.}
    \label{fig:slicing the top row}
\end{figure}

\noindent As an aside, we note that we may recover the face labels of $\bG$ from the face labels of $\bGshift$ as follows. 


 \begin{prop}\label{prop:face label after a single T-shift} Let $\bG$ be a reduced plabic graph whose solid vertices are lollipops or trivalent. Let $\bGshift$ be the T-shift of $\bG$. Then for any face $F$ of $\bG$,
 \[
    \overrightarrow{I_F}= \left\{\begin{array}{c}\text{indices of solid lollipops}  \\ \text{of $\bGshift$ contained in $F$}\end{array}\right\} \cup \bigcup_{\substack{\text{faces $F'$ of $\bGshift$}\\ F'\cap F\neq \emptyset}}\overrightarrow{I_{F'}}  .
 \]
 \end{prop}
\noindent Note that by construction, a face $F$ in $\bG$ may contain a solid lollipop of $\bGshift$ only if $F$ is a boundary face, in which case it contains at most one solid lollipop.




\subsection{Positroid weaves}\label{ssec:positroid_weaves} We now present a central construction of this manuscript, which associates a weave to every reduced plabic graph. The construction relies on the fact that T-shift may be iterated and it reduces rank by 1, as proven in Proposition~\ref{prop:single T-shift} above.

\begin{prop}[Weave from plabic graph] \label{prop:weaves from plabic graphs} Let $\bG$ be a reduced plabic graph of rank $m$ whose solid vertices are lollipops or trivalent. Set $\bG_m:=\bG$ and recursively define $\bG_{k} := \bGshift_{k+1} (\tau_{k+1})$ for $k\in[m-1]$ for some choice of trivalent trees $\tau_{k+1}$. Consider the union $\bW:=\bG_{m-1}\cup\bG_{m-2}\cup\dots\cup \bG_1$, with $\bG_{k}$ drawn in color $k$.\\

\noindent Let $\ww(\bG, \mathbf{\tau})$ be the graph obtained from $\bW$ by identifying the solid vertex $s$ in $\bG_{k+1}$ with the empty vertex $v_s'$ in $\bG_{k}$, erasing the empty and solid dots at the vertices, and deleting all lollipops. Then $\ww(\bG, \mathbf{\tau})$ is a weave. 
\end{prop}

\begin{proof}
    By construction, the empty vertices of $\bG_{m-1}$ are trivalent in both $\bG_{m-1}$ and in $\bW$. Thus they are monochromatic trivalent vertices in $\ww(\bG, \mathbf{\tau})$. By construction, all remaining vertices in $\ww(\bG, \mathbf{\tau})$ correspond to a solid vertex of $\bW$ which has an empty vertex on top of it, i.e.~identified with it. Indeed, this holds by construction for all solid (non-lollipop) vertices which are not in $\bG_1$. Proposition~\ref{prop:single T-shift} implies that $\bG_1$ is a rank 1 reduced plabic graph and thus there are no trivalent solid vertices in $\bG_1$. This implies that all remaining vertices in $\ww(\bG, \mathbf{\tau})$ are hexavalent and edges are colored as in Figure~\ref{fig: weave vertices}.
\end{proof}

\begin{defn}[Positroid weaves]\label{def:positroid_weaves} A weaves $\ww(\bG, \mathbf{\tau})$ obtained via Proposition~\ref{prop:weaves from plabic graphs} from a reduced plabic graph $\bG$, and any choice of $\tau$, is said to be a \emph{positroid weaves}. By convention, color $(m-1)$ is blue, color $(m-2)$ is red, and color $(m-3)$ is green, where $m$ is the rank of $\bG$. (Thus blue is the top color.)\hfill$\Box$
\end{defn}

\begin{exmp}\label{exmp: positroid weave} Figure \ref{fig:positroid weave example 1} depicts the iterative process and a weave from the plabic graph $\bG$ in Example \ref{exmp:plabic graph example 1}. Since $\rank(\bG)=4$, we start with blue (color 3), and then red (color 2), and then green (color 1). 
Note that the positive braid word at the boundary of the weave, read clockwise starting from the $0$ degree angle on the unit circle, is $(3,3,2,1,3,2,3,3,2,1,3,2,1)$, which is a cyclic rotation of the positive braid word from Example \ref{exmp:positroid braid}. This is not a coincidence and is proven in general in Theorem \ref{thm:weave construction}.\hfill$\Box$

\begin{figure}[H]
    \centering
    \begin{tikzpicture}
    \draw [very thick, lightgray] (0,0) circle [radius=1.5];
    \foreach \i in {1,...,7}
    {
    \coordinate (\i) at (90-\i*360/7:1.5);
    \node [lightgray] at (90-\i*360/7:1.75) [] {$\i$};
    }
    \empver[lightgray, lightgray] (8) (1,-0.25);
    \solver[lightgray] (9) (-1,-0.25);
    \solver[lightgray] (10) (0,0.5);
    \draw[lightgray, shorten <=0.2cm] (8) -- (2);
    \draw[lightgray, shorten <=0.2cm] (9) -- (5);
    \draw[lightgray, shorten <=0.2cm] (10) -- (7);
    \draw[lightgray, shorten <=0.2cm] (10) -- (6);
    \draw[lightgray, shorten <=0.2cm] (10) -- (4);
    \draw [lightgray] (1) to [out=-150,in=120] (3);
    \empver[blue] (11) (0,0.5);
    \solver[blue] (12) (0.15,0);
    \draw[blue, shorten <=0.2cm] (11) -- (-190:1.5);
    \draw[blue, shorten <=0.2cm] (11) -- (120:1.5);
    \edge[blue] (11) -- (12);
    \draw [blue](12) to  (-90:1.5);
    \draw [blue](12) to  (60:1.5);
    \draw [blue] (15:1.5) --(15:1) to [out=-165,in=135] (-45:1) -- (-45:1.5);
    \node [blue] at (15:1.75) [] {$2$};
    \node [blue] at (-45:1.75) [] {$3$};
    \node [blue] at (-90:1.75) [] {$4$};
    \node [blue] at (-190:1.75) [] {$6$};
    \node [blue] at (120:1.75) [] {$7$};
    \node [blue] at (60:1.75) [] {$1$};
    \end{tikzpicture}\hspace{1cm}
    \begin{tikzpicture}
    \draw [very thick, lightblue] (0,0) circle [radius=1.5];
    \empver[lightblue] (11) (-0.3,0.5);
    \solver[lightblue] (12) (0,0);
    \draw[lightblue, shorten <=0.2cm] (11) -- (-135:1.5);
    \draw[lightblue, shorten <=0.2cm] (11) -- (120:1.5);
    \edge[lightblue] (11) -- (12);
    \draw [lightblue](12) to  (-90:1.5);
    \draw [lightblue](12) to  (60:1.5);
    \draw [lightblue] (15:1.5) to [out=-165,in=135] (-45:1.5);
    \node [blue] at (15:1.75) [] {$2$};
    \node [blue] at (-45:1.75) [] {$3$};
    \node [blue] at (-90:1.75) [] {$4$};
    \node [blue] at (-135:1.75) [] {$6$};
    \node [blue] at (120:1.75) [] {$7$};
    \node [blue] at (60:1.75) [] {$1$};
    \empver[red] (1) at (0,0);
    \solver[red] (2) at (0.6,0);
    \solver[red] (13) at (-1,0);
    \solver[red] (14) at (1.2,-0.3);
    \draw [red] (1.2,-0.3) -- (-15:1.5);
    \draw [red] (180:1.5) -- (13);
    \edge[red] (1) (2);
    \draw [red, shorten <=0.2cm] (1) -- (-115:1.5);
    \draw [red, shorten <=0.2cm] (1) -- (90:1.5);
    \draw [red] (2) -- (-70:1.5);
    \draw [red] (2) -- (40:1.5);
    \node [red] at (-70:1.75) [] {$4$};
    \node [red] at (-115:1.75) [] {$6$};
    \node [red] at (180:1.75) [] {$7$};
    \node [red] at (90:1.75) [] {$1$};
    \node [red] at (40:1.75) [] {$2$};
    \node [red] at (-15:1.75) [] {$3$};
    \end{tikzpicture}\hspace{1cm}
    \begin{tikzpicture}
    \draw [very thick, lightred] (0,0) circle [radius=1.5];
    \empver[lightred] (1) at (0,0);
    \solver[lightred] (2) at (0.6,0);
    \solver[lightred] (14) at (1.1,0);
    \edge[lightred] (1) (2);
    \draw [lightred, shorten <=0.2cm] (1) -- (-115:1.5);
    \draw [lightred, shorten <=0.2cm] (1) -- (90:1.5);
    \draw [lightred] (2) -- (-70:1.5);
    \draw [lightred] (2) -- (40:1.5);
    \draw [lightred] (14) -- (0:1.5);
    \node [red] at (-70:1.75) [] {$4$};
    \node [red] at (-115:1.75) [] {$6$};
    \node [red] at (90:1.75) [] {$1$};
    \node [red] at (40:1.75) [] {$2$};
    \node [red] at (0:1.75) [] {$3$};
    \node [teal] at (150:1.75) [] {$1$};
    \solver[teal] (13) at (-0.86,0.5);
    \draw [teal] (150:1.5) -- (13);
    \empver[teal] (0) (0.6,0);
    \solver[lightred] (3) (-1,0);
    \draw [lightred] (3) -- (-1.5,0);
    \draw [teal, shorten <=0.2cm] (0) -- (-30:1.5);
    \draw [teal, shorten <=0.2cm] (0) -- (-90:1.5);
    \draw [teal, shorten <=0.2cm] (0) -- (60:1.5);
    \node [teal] at (-90:1.75) [] {$6$};
    \node [red] at (-1.75,0) [] {$7$};
    \node [teal] at (60:1.75) [] {$2$};
    \node [teal] at (-30:1.75) [] {$4$};
    \end{tikzpicture}\hspace{1cm}
    \begin{tikzpicture}
    \draw [very thick] (0,0) circle [radius=1.5];
    \coordinate (11) at (-0.3,0.5);
    \coordinate (12) at (0,0);
    \draw[blue] (11) -- (-135:1.5);
    \draw[blue] (11) -- (120:1.5);
    \draw [blue] (11) -- (12);
    \draw [blue] (12) to  (-90:1.5);
    \draw [blue](12) to  (60:1.5);
    \draw [blue] (0:1.5) to [out=180,in=135] (-45:1.5);
    \coordinate (1) at (0,0);
    \coordinate (2) at (0.6,0);
    \draw [red] (1) -- (2);
    \draw [red] (1) -- (-115:1.5);
    \draw [red] (1) -- (90:1.5);
    \draw [red] (2) -- (-60:1.5);
    \draw [red] (2) -- (40:1.5);
    \draw [teal] (2) -- (15:1.5);
    \draw [teal] (2) -- (-75:1.5);
    \draw [teal] (2) -- (70:1.5);
    \end{tikzpicture}
    \caption{The iterative process in Proposition \ref{prop:weaves from plabic graphs}: the first row depicts the third steps in this case. The weave, obtained by overlapping blue, red and green, is on the second row.} 
    \label{fig:positroid weave example 1}
\end{figure}
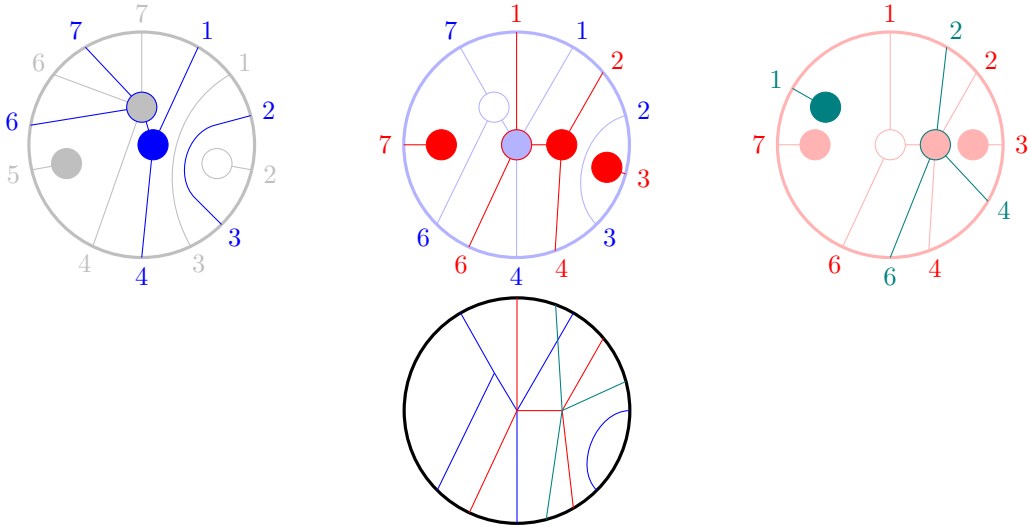
\end{exmp}

We now prove that the weave $\ww(\bG)$ encodes the quiver of $\bG$, the positroid braid word $\beta_\cP$, and that equivalences of plabic graphs give equivalences of positroid weaves. This proves part of Theorem~\ref{thm:mainA}(ii).

\begin{thm}\label{thm:weave construction} Let $\bG$ be a reduced plabic graph associated for a positroid $\cP$ whose solid vertices are lollipops or trivalent. Let $\ww:= \ww(\bG, \mathbf{\tau})$ be a weave constructed from $\bG$ via iterative T-shifts. Then the following holds:\\ 
\begin{enumerate}
    \item Let $F$ be a face of $\bG$ which is not an exceptional boundary face.\footnote{Recall that Figure \ref{fig:exceptional boundary face} depicts an instance of an exceptional boundary face.} Then:\\
    
    \begin{itemize}
        \item[-] The blue top level weave lines of $\ww$ contained inside of the face $F$ form a $Y$-tree in $\ww$.\\

       \item[-] The intersection pairing between these $Y$-trees produces the quiver $Q_\bG$, with isolated frozen vertices (corresponding to the exceptional boundary faces) removed.\\
    \end{itemize}
    
    \item The boundary positive braid word $\dd\ww$ of the weave $\ww$ is the braid word $\beta_\cP$, up to a cyclic shift.\\
    \item For any reduced plabic graph $\bH$ equivalent to $\bG$ with all non-lollipop solid vertices trivalent, and any choice of trivalent trees $\tau'$, the positroids weaves $\ww(\bH, \tau')$ and $\ww$ are weave equivalent.\\
    
    \noindent In particular, any choice of $\tau$ for the T-shift steps of $\bG$ gives rise to an equivalent weave.\\
\end{enumerate}
\end{thm}
\begin{proof} Let us start with (1). In the construction of $\ww$, the only trivalent weave vertices are on the top weave level. These trivalent vertices are end points of a $3$-regular tree in $\bGshift$ that spans a face of $\bG$, and the internal vertices of such a $3$-regular tree are all hexavalent weave vertices after the next stage. Thus, each such $3$-regular tree is a Y-tree in $\ww$. Moreover, these $3$-regular trees only intersect each other at the solid vertices of $\bG$, which are trivalent weave vertices in $\ww$. Computing their local intersection pairing at such a vertex, using Definition \ref{defn:local intersection pairing}, gives that the three Y-trees incident to a trivalent weave vertex form an oriented $3$-cycle quiver. This coincides with the local rule used to construct the quiver $Q_\bG$ (cf.~Figure \ref{fig: quiver example}). Hence, the intersection pairing between these Y-trees agrees with the arrows in the quiver $Q_\bG$.\\

\noindent For (2) and (3) we induct on $m=\rank(\bG)$. For the base case $m=1$, there are no solid vertices in $\bG$ and hence the weave $\ww$ is empty and the statement holds. Inductively, let us suppose both statements are true for reduced plabic graphs of rank smaller than $m$. For (2), we observe that in the construction of $\bGshift$ from $\bG$, there is an external edge of $\bGshift$ appearing in between the boundary marked points $i$ and $i+1$ of $\bG$ if and only if there is a crossing at level $m-1$ (the top level) between $I_i$ and $I_{i+1}$. Thus, the external edges of $\bGshift$ record all occurrences of $m-1$ in the positive braid word $\beta_\cP$. It also follows from Proposition \ref{prop:T-shift of positive braids} that the remaining part is equal to $\beta_{\cP^\downarrow}$, which is the boundary of the positroid weave $\ww^\downarrow$ constructed from $\bGshift$. By construction, $\ww$ is precisely the union $\bGshift\cup \ww^\downarrow$, and hence the boundary of $\ww$ is equal to $\beta_\cP$, which is equal to the stacking of the crossings on level $m-1$ on top of $\beta_{\cP^\downarrow}$, with a slight shift to the left, as required.\\

\noindent For (3), recall from Definition \ref{defn:equivalence of reduced plabic graphs} that two reduced plabic graphs are equivalent if they are related by a sequence of contraction-expansion moves and bivalent-vertex moves. Since we have assumed that all non-lollipop solid vertices in $\bG,\bH$ are trivalent, we may choose a sequence which does not use solid bivalent-vertex moves. Also, observe that moves involving empty vertices of $\bG$ do not affect $\bGshift$. Thus, it suffices to just study the contraction-expansion move on solid vertices of $\bG$, as drawn in Figure \ref{fig:elementary move}.

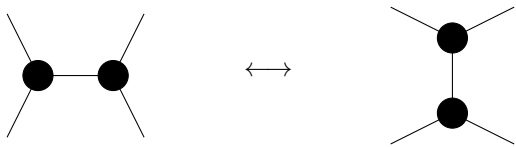
\begin{figure}[H]
    \centering
    \begin{tikzpicture}[baseline=0]
    \solver[black] (0) (-0.5,0);
    \solver[black] (1) (0.5,0);
    \edge[black] (0) (1);
    \edge[black] (0) (-1,1);
    \edge[black] (0) (-1,-1);
    \edge[black] (1) (1,1);
    \edge[black] (1) (1,-1);
    \end{tikzpicture} \hspace{1cm} $\longleftrightarrow$\hspace{1cm}
    \begin{tikzpicture}[baseline=0]
    \solver[black] (0) (0,-0.5);
    \solver[black] (1) (0,0.5);
    \edge[black] (0) (1);
    \edge[black] (0) (1,-1);
    \edge[black] (0) (-1,-1);
    \edge[black] (1) (1,1);
    \edge[black] (1) (-1,1);
    \end{tikzpicture}
    \caption{An elementary contraction-expansion move.}
    \label{fig:elementary move}
\end{figure}
Let us suppose we have two plabic graphs that differ by this move. We have some choices as to how to construct the next level, given by the blue plabic graph. These choices differ by expansion-contraction moves on the blue plabic graph and by induction, the resulting weaves will be equivalent. Hence, we may choose the blue plabic graphs so that they have the local behavior as in Figure \ref{fig:flop induced by elementary move}. By construction, the next levels in the respective weaves must then look like the red weave lines in Figure \ref{fig:flop induced by elementary move}. Therefore, the two weaves associated to two graphs differing by one elementary contraction-expansion move locally differ as the two weaves in Figure \ref{fig:flop induced by elementary move}. This difference in the weaves is the flop move on weaves: the weave equivalence in Figure \ref{fig:weave equivalence}.(III). Thus, equivalent reduced plabic graphs induce equivalent weaves.
\end{proof}

\begin{figure}[H]
    \centering
    \begin{tikzpicture}[baseline=0]
    \solver[lightgray] (0) (-0.5,0);
    \solver[lightgray] (1) (0.5,0);
    \edge[lightgray] (0) (1);
    \edge[lightgray] (0) (-1,1);
    \edge[lightgray] (0) (-1,-1);
    \edge[lightgray] (1) (1,1);
    \edge[lightgray] (1) (1,-1);
    \empver[blue] (2) (-0.5,0);
    \empver[blue] (3) (0.5,0);
    \solver[lightblue] (4) (0,0.5);
    \solver[lightblue] (5) (0,-0.5);
    \edge[blue] (2) (4);
    \edge[blue] (2) (5);
    \edge[blue] (3) (4);
    \edge[blue] (3) (5);
    \edge[blue] (2) (-1.5,0);
    \edge[blue] (3) (1.5,0);
    \edge[blue] (4) (0,1.5);
    \edge[blue] (5) (0,-1.5);
    \empver[red] (6) (0,0.5);
    \empver[red] (7) (0,-0.5);
    \edge[red] (6) (7);
    \edge[red] (6) (-0.75,1.5);
    \edge[red] (6) (0.75,1.5);
    \edge[red] (7) (-0.75,-1.5);
    \edge[red] (7) (0.75,-1.5);
    \end{tikzpicture} \hspace{1cm} $\longleftrightarrow$\hspace{1cm}
    \begin{tikzpicture}[baseline=0]
    \solver[lightgray] (0) (0,-0.5);
    \solver[lightgray] (1) (0,0.5);
    \edge[lightgray] (0) (1);
    \edge[lightgray] (0) (1,-1);
    \edge[lightgray] (0) (-1,-1);
    \edge[lightgray] (1) (1,1);
    \edge[lightgray] (1) (-1,1);
    \empver[blue] (2) (0,-0.5);
    \empver[blue] (3) (0,0.5);
    \solver[lightblue] (4) (0.5,0);
    \solver[lightblue] (5) (-0.5,0);
    \edge[blue] (2) (4);
    \edge[blue] (2) (5);
    \edge[blue] (3) (4);
    \edge[blue] (3) (5);
    \edge[blue] (2) (0,-1.5);
    \edge[blue] (3) (0,1.5);
    \edge[blue] (4) (1.5,0);
    \edge[blue] (5) (-1.5,0);
    \empver[red] (6) (0.5,0);
    \empver[red] (7) (-0.5,0);
    \edge[red] (6) (7);
    \edge[red] (6) (1.5,-0.75);
    \edge[red] (6) (1.5,0.75);
    \edge[red] (7) (-1.5,-0.75);
    \edge[red] (7) (-1.5,0.75);
    \end{tikzpicture}
    \caption{A weave equivalence induced by the contraction-expansion plabic graph move.}
    \label{fig:flop induced by elementary move}
\end{figure}
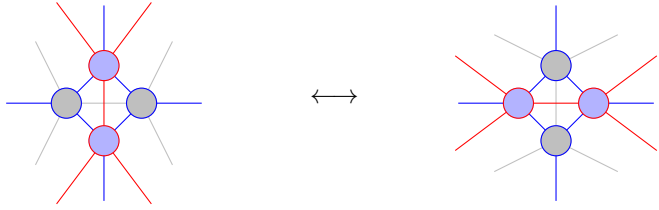

\begin{exmp}
Figure \ref{fig:ConjugateSurface4}.(i) depicts a plabic graph and its associated positroid weave, overlapped. Figure \ref{fig:ConjugateSurface4}.(ii) depicts an instance of Theorem \ref{thm:weave construction}.(1) in this case: we have drawn two of the Y-cycles in the positroid weave, each corresponding to one of the two faces $F_1,F_2$ in the plabic graph $\bG$.\hfill$\Box$
\end{exmp}

\begin{center}
	\begin{figure}[H]
		\centering
		\includegraphics[scale=1.2]{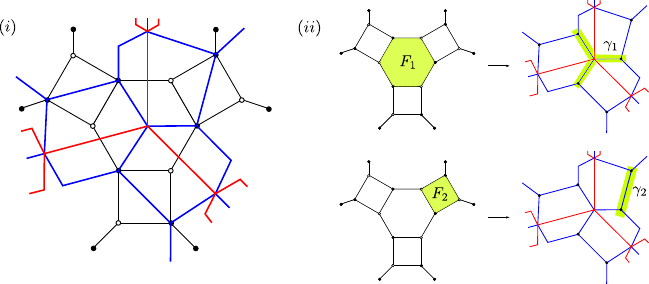}
		\caption{(Left) An example of a positroid weave $\ww$ for the top positroid stratum in $\Gr_{3,6}$. (Right) Faces in $\bG$ in green and their corresponding $Y$-trees in $\ww$, also highlighted in green. This illustrates the correspondence between faces and Y-cycles in Theorem \ref{thm:weave construction}.(1).}
		\label{fig:ConjugateSurface4}
	\end{figure}
\end{center}

\begin{rmk} By construction, square faces in $\bG$ correspond to short $I$-cycles on $\ww$. Thus, a square move on a reduced plabic graph corresponds to a mutation along a short $I$-cycle on its positroid weave. That said, note that a mutation at \emph{any} $Y$-tree on a positroid weave can be combinatorially described using weaves. Below is an example of a mutation taking place along the $Y$-tree $\gamma_1$ at the center in Figure \ref{fig:ConjugateSurface4}. Note that the resulting mutated seed cannot be described using plabic graphs.\hfill$\Box$
\end{rmk}

\begin{center}
	\begin{figure}[H]
		\centering
		\includegraphics[scale=1.2]{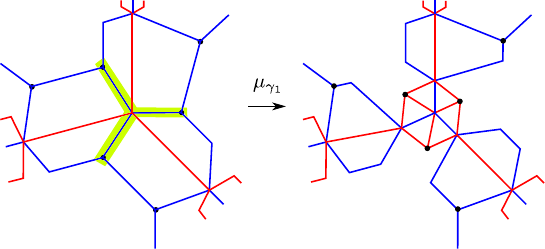}
		\caption{Mutation along the $Y$-tree $\gamma_1$ corresponding to the hexagon face of the reduced plabic graph in Figure \ref{fig:ConjugateSurface4}. The $Y$-tree $\gamma_1$ being mutated is highlighted on the left. The weave resulting from the mutation at $\gamma_1$ is drawn on the right.}
		\label{fig:ConjugateSurface5}
	\end{figure}
\end{center}

\begin{rmk}\label{rmk:combinatorial description of Pluckers} Theorem~\ref{thm:weave construction} shows that the quiver $Q_\bG$ can be recovered from the weave $\ww(\bG, \tau)$. Repeatedly applying Proposition \ref{prop:face label after a single T-shift} gives a way to recover the target face labels of $\bG$ from $\ww(\bG, \tau)$ together with the list of boundary indices for each iterated T-shift of $\bG$. In other words, the data of $\Sigma_T(\bG)$ can be recovered combinatorially from $\ww(\bG, \tau)$. See Figure~\ref{fig:ConjugateSurface8} for an example.\hfill$\Box$
\end{rmk}

\begin{center}
	\begin{figure}[H]
		\centering
		\includegraphics[width=\textwidth]{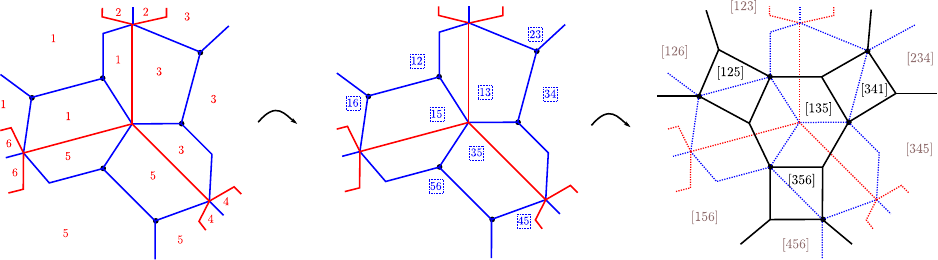}
		\caption{How to recover the target face labels of $\bG$ (right) from the weave $\bG$ (left) by repeatedly applying Proposition~\ref{prop:face label after a single T-shift}. Since all layers of the weave have 6 boundary legs, the list of boundary labels for each layer is not necessary.}
		\label{fig:ConjugateSurface8}
	\end{figure}
\end{center}

\section{Positroid Weaves are Complete Demazure Weaves} \label{sec: Demazure}

In this section we prove that positroid weaves are complete Demazure weaves, up to weave equivalences. This is established in Theorem~\ref{thm:Demazure weave}. For reference, positroid weaves have been defined in Definition \ref{def:positroid_weaves} above, Demazure weaves in Definition \ref{defn: demazure weave} above, and weave equivalences are depicted in Figure \ref{fig:weave equivalence}.\\

\subsection{Braid words and orientations under T-shifts}\label{ssec:prelim1_demazure} By Theorem \ref{thm:weave construction}.(2), the boundary positive braid word of a positroid weave is cyclically equivalent to $\beta_\cP$. Therefore, a first necessary condition for positroid weaves to be complete Demazure weaves is that after cyclically rotating and applying braid moves to $\beta_\cP$, we can obtain a braid word with $w_0$ as a consecutive subword. Let us first prove this.

\begin{prop}\label{prop: containing w_0} Let $\cP$ be a positroid. Then the braid word $\beta_\cP$ is cyclically equivalent to a positive braid word that contains a reduced word for $w_0$ as a consecutive subword.
\end{prop}
\begin{proof} Recall that $\beta_\cP$ is constructed from the target Grassmann necklace $\overrightarrow{\cI}=(\overrightarrow{I_1},\overrightarrow{I_2}, \dots, \overrightarrow{I_n})$ by writing down entries of each entry $\overrightarrow{I_i}$ in a column and then drawing a wiring diagram between consecutive columns. 
Let us fix an entry $\overrightarrow{I_i}=(i<_i i_2<_i i_3<_i \cdots <_i i_m)$ of $\overrightarrow{\cI}$, where $i$ is not a solid lollipop for $\cP$. By construction, each of the elements $i_2, i_3,\dots, i_m$ gradually moves up as we scan the columns to the right (cyclically) of $\overrightarrow{I_i}$. Follow the upward trajectory of each $i_k$ (with $2\leq k\leq m$) and mark all the crossings the trajectory passes through before $i_k$ reaches the top row. The crossings for which the trajectory of $i_k$ passes through are necessarily $s_{m-k+1}s_{m-k+2}\cdots s_{m-1}$ and thus the subword consisting of all the marked crossings spells out a reduced word for $w_0$. Let us now to use braid equivalences to move these letters together into a cyclically consecutive subword, as follows.

By the construction of $\beta_\cP$, we can already change $\beta_\cP$ into a positive braid word where the $k-1$ marked crossings of $i_k$ form a consecutive subword before using any braid moves (Reidemeister III moves). Call this consecutive subword $\tau_k:=s_{m-k+1}s_{m-k+2}\cdots s_{m-1}$. Thus, without loss of generality, we may assume that, up to a cyclic shift, $\beta_\cP$ is of the form
\[
\cdots\tau_2 \cdots \tau_3 \cdots \underbrace{\cdots}_\text{more $\tau_k$'s} \cdots \tau_m \cdots ,
\]
Moreover, from the construction of $\beta_\cP$, the crossings between $\tau_{k-1}$ and $\tau_k$ are among the $s_l$'s with $l>m-k+1$. Now, using a single braid move we have $s_l\tau_k \sim \tau_k s_{l+1}$ for $l>m-k+1$. Thus, by using braid moves, we can move the unmarked crossings between the $\tau_k$'s all the way to the right and get a consecutive subword $\tau_2\tau_3\cdots \tau_m$, which is a reduced word of $w_0$ as required. 
\end{proof}

\noindent To show that a positroid weave $\ww(\bG)$ is equivalent to a complete Demazure weave, we must draw $\ww$ so that no weave lines have horizontal tangents. For this, we use acyclic perfect orientations, cf.~Definition \ref{defn: perfect orientation with index a}.

\begin{defn}\label{defn: induced perfect orientation} Let $\bG$ be a reduced plabic graph and $O_i$ the unique acyclic perfect orientation on $\bG$ with source set $\tI{i}=(i<_ij<_i\cdots)$. By definition, $O^\downarrow_i$ is the unique acyclic perfect orientation $O_j$ on $\bGshift$.\hfill$\Box$
\end{defn}

We first determine the relationship between $O_i$ and $O^\downarrow_i$.

\begin{prop}\label{prop: compatible perfect orientations} Let $\bG$ be a reduced plabic graph equipped with the acyclic perfect orientation $O_i=O_{i_1}$ and suppose  this perfect orientation has source set $\{i_1, i_2, \dots, i_m\}$. Consider $\bGshift$ equipped with the acyclic perfect orientation $O^\downarrow_{i}$, which will have source set $\{i_2, \dots, i_m\}$.\\

\noindent Then, for any shared vertex between $\bG$ and $\bGshift$, the outgoing (resp.~incoming) edges in $\bG$ must be opposite to the incoming (resp.~outgoing) edges in $\bGshift$.
\end{prop}
\begin{proof} Let $v$ be a solid vertex in $\bG$ with incident edges oriented as in Figure \ref{fig: compatible perfect orientations} (left). By construction, we must have $i_1<_ii_2<_ii_3$ for the three zig-zag strands in Figure \ref{fig: compatible perfect orientations} (left). After constructing $\bGshift$ from $\bG$, the counterparts of these three zig-zag strands are drawn in Figure \ref{fig: compatible perfect orientations} (right). Therefore, the acyclic perfect orientation $O_i^\downarrow$ on $\bGshift$ must assign orientations to the new edges adjacent to $v$ in $\bGshift$, which is now an empty vertex, as in see Figure \ref{fig: compatible perfect orientations} (right), and the statement follows.
\end{proof}

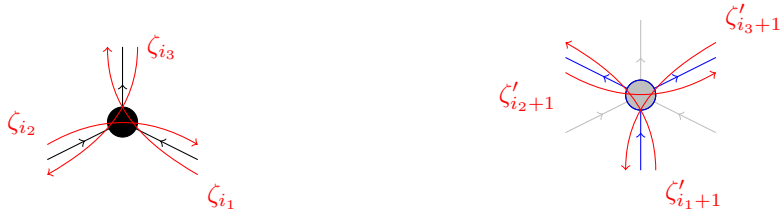
\begin{figure}[H]
    \centering
    \begin{tikzpicture}
    \draw [decoration={markings,mark=at position 0.5 with {\arrow{>}}},postaction={decorate}] (-1,0) -- (0,0.5);
    \draw [decoration={markings,mark=at position 0.5 with {\arrow{>}}},postaction={decorate}] (1,0) -- (0,0.5);
    \draw [decoration={markings,mark=at position 0.5 with {\arrow{>}}},postaction={decorate}] (0,0.5) -- (0,1.5);
    \draw [fill=black] (0,0.5) circle [radius=0.2];
    \draw [red, ->] (1,-0.2) node [below right] {$\zeta_{i_1}$}to [out=150,in=-90] (-0.2,1.5);
    \draw [red, ->] (0.2,1.5) node [right] {$\zeta_{i_3}$} to [out=-90,in=30] (-1,-0.2);
    \draw [red,->] (-1,0.2) node [above left] {$\zeta_{i_2}$} to [out=30,in=150] (1,0.2);
    \end{tikzpicture} \hspace{3cm}
    \begin{tikzpicture}
    \draw [lightgray,decoration={markings,mark=at position 0.5 with {\arrow{>}}},postaction={decorate}] (-1,0) -- (0,0.5);
    \draw [lightgray,decoration={markings,mark=at position 0.5 with {\arrow{>}}},postaction={decorate}] (1,0) -- (0,0.5);
    \draw [lightgray,decoration={markings,mark=at position 0.5 with {\arrow{>}}},postaction={decorate}] (0,0.5) -- (0,1.5);
    \draw [blue, decoration={markings,mark=at position 0.5 with {\arrow{>}}},postaction={decorate}] (0,-0.5) -- (0,0.5);
    \draw [blue, decoration={markings,mark=at position 0.5 with {\arrow{>}}},postaction={decorate}] (0,0.5) -- (1,1);
    \draw [blue, decoration={markings,mark=at position 0.5 with {\arrow{>}}},postaction={decorate}] (0,0.5) -- (-1,1);
    \draw [fill=lightgray] (0,0.5) circle [radius=0.2];
    \draw [blue] (0,0.5) circle [radius=0.2];
    \draw [red,->] (-1,0.8) node [below left] {$\zeta'_{i_2+1}$} to [out=-30,in=-150] (1,0.8);
    \draw [red,->] (1,1.2) node [above right] {$\zeta'_{i_3+1}$} to [out=-150,in=90] (-0.2,-0.5);
    \draw [red,->] (0.2,-0.5) node [below right] {$\zeta'_{i_1+1}$} to [out=90,in=-30] (-1,1.2);
    \end{tikzpicture}
    \caption{Corresponding zig-zag strands in $\bG$ (left) and $\bGshift$ (right).}
    \label{fig: compatible perfect orientations}
\end{figure}

\subsection{An ingredient from graph theory}\label{ssec:prelim2_demazure}  By definition, a directed planar graph $G$  that can be isotoped so that all edges are  oriented upward  is said to be an \emph{upward planar drawing}. We use the following result from  graph theory, cf.~\cite [Section 3.2]{garg1995upward}:

\begin{thm}[\cite{garg1995upward}] \label{thm:upward} Let $G$ be a directed planar graph. Suppose that $G$ can be embedded in the plane and the angles between any two adjacent incoming edges, or any two adjacent outgoing edges  can be labelled {\it small} or {\it large}  in a manner such that:
\begin{enumerate}
\item At every vertex, the incoming edges (and therefore the outgoing edges) are cyclically consecutive;
\item Every source or sink has exactly one large angle;
\item Every interior face has two more small angles than large angles; and
\item The exterior face has two more large angles than small angles.
\end{enumerate}
Then $G$ can be isotoped to an upward planar drawing such that all small angles measure less than $\pi$ and all large angles measure more than $\pi$.
\end{thm}

\noindent If $G$ has an upward planar drawing, then $G$ must admit such a small-large edge labelling, cf.~\cite [Lemma 2]{garg1995upward}. Theorem \ref{thm:upward} states that this also sufficient.

\subsection{Proof that positroid weaves are complete Demazure} \label{ssec: Demazure_proof} Let us prove the first part of Theorem~\ref{thm:mainA}(ii). We use the ingredients in Subsections \ref{ssec:prelim1_demazure} and \ref{ssec:prelim2_demazure}. 

\begin{thm}\label{thm:Demazure weave} Let $\bG$ be a reduced plabic graph and $\ww(\bG)$ its positroid weave. Then $\ww(\bG)$ is equivalent to a complete Demazure weave.
\end{thm}
\begin{proof} Let us fix $i$. Proposition \ref{prop: compatible perfect orientations} gives orientations $O_i$ and $O^\downarrow_i$ on the edges of $\bG$ and $\bGshift$, respectively. Iteratively  applying Proposition \ref{prop: compatible perfect orientations} yields orientations on all the T-shifts. Therefore,  by the constructive definition of $\ww(\bG)$, we obtain an orientation on all the edges of $\ww(\bG)$. Note that the weave $\ww(\bG)$ naturally sits within the same disk in which the graph $\bG$ is drawn. The proof has the following two steps:

\begin{enumerate}

\item In the first step of the proof, we show that the positroid weave and the disk it sits inside can be isotoped so that all the oriented edges point upwards. This happens at a cost: the disk becomes a non-convex subset of the plane, and the sources along the boundary, which are a subword of $\beta_\cP$ spelling out a reduced word for $w_0$, are not consecutive. See Figure \ref{fig: demazure weave example}(left) for an illustration of the type of modification  that the disk must undergo.\\

\item The second step is to use braid equivalences along the boundary to move the sources into a consecutive subword, as in the proof of Proposition \ref{prop: containing w_0}. The key step is to show that this can be done in a way such that all edges remain oriented upward. Once we have performed these braid equivalences, the source (and therefore also the target) boundary vertices  become consecutive  and so the disk can be isotoped to be convex.\footnote{Once it is convex, the disk can further be isotoped to a rectangle, as technically required for a Demazure weave.}

\end{enumerate}

For step (1), we use Theorem~\ref{thm:upward}. To  apply it, we need  the following facts. First, Proposition \ref{prop: unique sink}  implies that each interior face has exactly one source and one sink. Second, each interior vertex has either a unique incoming edge or a unique outgoing edge. Let us label the angles between two adjacent incoming edges at an interior vertex as {\it small}, and do the same with angles between two adjacent outgoing edges at interior vertices. At the external vertices, on the boundary of the disk,  we label the unique angle as {\it large}.

For a reduced plabic graph $\bG$, the first two  hypotheses of Theorem~\ref{thm:upward}  hold by construction. For the third condition in the hypothesis, all the interior faces have exactly two small angles and no large angles. For the remaining condition  in the hypothesis, Proposition \ref{prop: unique sink} implies that the external face has two more large angles than small angles. Thus all the conditions are satisfied, Theorem~\ref{thm:upward} can be applied to $\bG$, and  therefore $\bG$ has an upward planar drawing,  as desired.\\

Let us now show that the positroid weave can be drawn with its edges oriented upward, working one layer of the weave at a time. We start with $\bG$ drawn so that all edges are oriented upward  and consider $\bGshift$. 
By Proposition \ref{prop: compatible perfect orientations}, near the solid vertices of $\bG$ where the edges of $\bGshift$ start, there is no obstruction to orienting these edges upwards: we therefore start by drawing these as half-edges. Within each face of $\bG$, all but one of these half edges is incoming and one half-edge is outgoing.  This, within  each face, the edges of $\bGshift$ form a tree, so we may arrange all the edges of this tree to be oriented upwards. Hence, once $\bG$ is oriented upward, we can simultaneously orient all the edges of $\bGshift$ upward.

Iteratively, we continue this process with each T-shift of $\bG$ and orient all the edges in these layers upwards. Finally, we remove all the lollipops as well as the graph $\bG$, as it is not a layer in our weave. The end result is an embedding of the positroid weave in the plane so that all directed edges are oriented upwards. Moreover, all hexavalent weave vertices arise as in Figure \ref{fig: compatible perfect orientations}, all tetravalent vertices arise from non-consecutive T-shifts, and all trivalent weave vertices are originally empty vertices of $\bGshift$. Therefore, we conclude that conditions (2)-(5) of Definition \ref{defn: demazure weave} are  satisfied. The only remaining condition to be satisfied is Definition \ref{defn: demazure weave}.(1): all external weave lines must be incident to either the top or the bottom boundary. If the source set $\overrightarrow{I_i}$ consists of cyclically consecutive indices, then the bottom of the weave consists of boundary sources that give a reduced word for $w_0$. Then the positroid weave satisfies condition (1) as well and is therefore a complete Demazure weave.\\

For the general case we must proceed with step (2). Indeed, if there are sinks $j$ with $j <_i i_m$, then we need to use braid moves to move these strands higher in the ordering $<_i$, so we can extend them to the top. Let us call these sinks \emph{blocked}. First, we claim that none of the blocked sinks can be rightward pointing, i.e.~of the forms depicted in Figure \ref{fig: forbidden sinks}. This is because in either case, the indices must satisfy $j<_ik$ from the cyclic ordering of $i, j, k$ along the boundary, while the perfect orientation agrees with the zig-zag strand $\zeta_k$ and is opposite to $\zeta_j$, implying $k <_i j$, a contradiction. Note however, leftward pointing sinks can occur  as illustrated in Figure \ref{fig: allowed sinks}): both the cyclic ordering and the perfect orientation give $k <_i j$.
\begin{figure}[H]
    \centering
    \begin{tikzpicture}[scale=0.7]
    \draw [decoration={markings,mark=at position 0.5 with {\arrow{>}}},postaction={decorate}] (-135:1) -- (0,0);
    \draw [decoration={markings,mark=at position 0.5 with {\arrow{>}}},postaction={decorate}] (-45:1) -- (0,0);
    \draw [->] (0,0) -- (0,0.5);
    \draw [dashed] (0,0.5) -- (0,1.25) node [above] {$j$};
    \draw [red, dashed] (0.2,0.5) -- (0.2,1.25);
    \draw [fill=black] (0,0) circle [radius=0.2];
    \draw [dashed] (-135:1) to [out=-135, in=-90] (-1.5,-0.5) -- (-1.5,1.5) to [out=90,in=90] (2.5,1.5) -- (2.5,-1.5);
    \draw [dashed] (-45:1) to [out=-45,in=90] (1,-1.5) node [below] {$k$};
    \draw [red,->] (0.2,0.5) node [right] {$\zeta_j$} to [out=-90,in=45] (-125:1);
    \draw [red,->] (-55:1) node [below] {$\zeta_k$} to [out=135,in=-90] (-0.2,0.8);
    \draw [red,dashed] (-55:1) to [out=-55,in=90] (0.85,-1.5);
    \draw [very thick] (3.5,-1.5) -- (2,-1.5) -- (2,1.25) to [out=90,in=90] (-0.5,1.25)-- (0.5,1.25) to [out=0,in=90] (1.5,0.5) -- (1.5,-1.5) -- (-2,-1.5);
    \node at (3,-1.5) [] {$\bullet$};
    \node at (3,-1.5) [below] {$i$};
    \end{tikzpicture}\hspace{3cm}
    \begin{tikzpicture}[scale=0.7]
    \draw [decoration={markings,mark=at position 0.5 with {\arrow{>}}},postaction={decorate}] (0,-0.8) -- (0,0);
    \draw [dashed] (0,-1.5) -- (0,-0.8);
    \draw [decoration={markings,mark=at position 0.5 with {\arrow{>}}},postaction={decorate}] (0,0) -- (45:1);
    \draw [dashed] (45:1) -- (1.2,1) node [above right] {$j$};
    \draw [decoration={markings,mark=at position 0.5 with {\arrow{>}}},postaction={decorate}] (0,0) -- (135:1);
    \draw [fill=white] (0,0) circle [radius=0.2];
    \draw [dashed] (135:1) to [out=135,in=-90] (-1,1) -- (-1,1.5) to [out=90,in=90] (2.5,1.5) -- (2.5,-1.5);
    \draw [red,->]  (35:1)node [below] {$\zeta_j$} to [out=-135,in=90] (0.2,-0.8);
    \draw [red,dashed] (35:1) to [out=35,in=-150] (1.3,0.9);
    \draw [red,->] (-0.2,-0.8) node [left] {$\zeta_k$} to [out=90,in=-45] (145:1);
    \draw [red,dashed] (-0.2,-0.8) -- (-0.2,-1.5);
    \draw [very thick] (3.5,-1.5) -- (2,-1.5) -- (2,1.25) to [out=90,in=90] (-0.5,1.25) -- (0.5,1.25) to[out=0,in=90] (1.5,0.5) -- (1.5,-1.5) -- (-2,-1.5);
    \node at (0,-1.5) [below] {$k$};
    \node at (3,-1.5) [] {$\bullet$};
    \node at (3,-1.5) [below] {$i$};
    \end{tikzpicture}
    \caption{Forbidden rightward pointing blocked sinks of perfect orientations. The thick solid curve indicates the boundary of the disk after the isotopy.}
    \label{fig: forbidden sinks}
\end{figure}
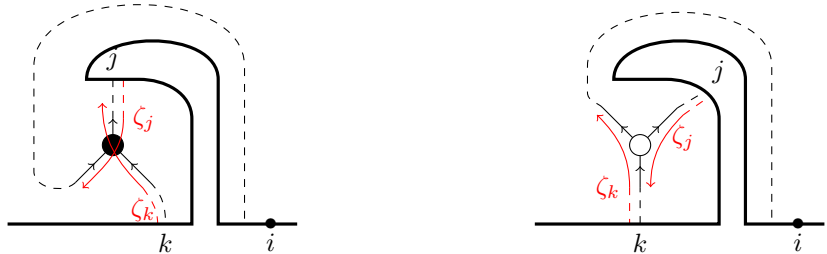
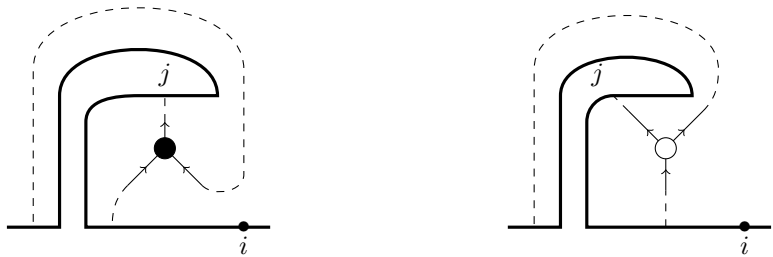
\begin{figure}[H]
    \centering
    \begin{tikzpicture}[scale=0.7]
    \draw [decoration={markings,mark=at position 0.5 with {\arrow{>}}},postaction={decorate}] (-135:1) -- (0,0);
    \draw [decoration={markings,mark=at position 0.5 with {\arrow{>}}},postaction={decorate}] (-45:1) -- (0,0);
    \draw [->] (0,0) -- (0,0.5);
    \draw [dashed] (0,0.5) -- (0,1) node [above] {$j$};
    \draw [fill=black] (0,0) circle [radius=0.2];
    \draw [dashed] (-45:1) to [out=-45, in=-90] (1.5,-0.5) -- (1.5,1.5) to [out=90,in=90] (-2.5,1.5) -- (-2.5,-1.5);
    \draw [dashed] (-135:1) to [out=-135,in=90] (-1,-1.5);
    \draw [very thick] (2,-1.5) -- (-1.5,-1.5) -- (-1.5,0.5) to [out=90,in=180] (-0.5,1) --  (1,1) to [out=90,in=90] (-2,1)-- (-2,-1.5) -- (-3,-1.5);
    \node at (1.5,-1.5) [] {$\bullet$};
    \node at (1.5,-1.5) [below] {$i$};
    \end{tikzpicture}\hspace{3cm}
    \begin{tikzpicture}[scale=0.7]
    \draw [decoration={markings,mark=at position 0.5 with {\arrow{>}}},postaction={decorate}] (0,-0.8) -- (0,0);
    \draw [dashed] (0,-1.5) -- (0,-0.8);
    \draw [decoration={markings,mark=at position 0.5 with {\arrow{>}}},postaction={decorate}] (0,0) -- (45:1) ;
    \draw [decoration={markings,mark=at position 0.5 with {\arrow{>}}},postaction={decorate}] (0,0) -- (135:1);
    \draw [fill=white] (0,0) circle [radius=0.2];
    \draw [dashed] (45:1) to [out=45,in=-90] (1,1.5) to [out=90,in=90] (-2.5,1.5) -- (-2.5,-1.5);
    \draw [very thick] (2,-1.5) -- (-1.5,-1.5) -- (-1.5,0.5) to [out=90,in=180] (-1,1) -- (0.5,1) to [out=90,in=90] (-2,1) -- (-2,-1.5) -- (-3,-1.5);
    \draw [dashed] (135:1) -- (-1,1) node [above left] {$j$}; 
    \node at (1.5,-1.5) [] {$\bullet$};
    \node at (1.5,-1.5) [below] {$i$};
    \end{tikzpicture}
    \caption{Allowable leftward pointing blocked sinks of perfect orientations.}
    \label{fig: allowed sinks}
\end{figure}

In order to make the positroid weave into a Demazure weave, we need to follow the boundary of the positroid weave in the counterclockwise direction and inductively clear these blocked sinks. Note that the counterclockwise direction along the weave boundary is the same as from left to right along the braid word $\beta_\cP$. The way we clear sinks is then analogous to how we move the crossings in between the $\tau_k$'s to the right of $\tau_m$ in  the proof of Proposition \ref{prop: containing w_0}.

Indeed, suppose we have cleared out all the blocked sinks in the clockwise direction of a blocked sink which is labelled $j$ (not necessarily a part of $\bGshift$ but possibly part of later T-shift). Then the external weave lines to the left of $j$ necessarily spell out the reduced word $(\tau_k\tau_{k+1}\cdots \tau_m)$, where $\tau_t=s_{m-t+1}s_{m-t+2}\cdots s_{m-1}$. Independently, the weave line of the blocked sink $j$ is $s_l$ for some $l>m-k+1$. Because $j$ is a leftward blocked sink, we can stretch the positroid weave vertically so that the sink $j$ is lower than the external weave lines in $(\tau_k\tau_{k+1}\cdots \tau_m)^{-1}$. Then the weave line of the blocked sink $j$ can now be extended leftward and upward through $(\tau_k\tau_{k+1}\cdots \tau_m)^{-1}$ using hexavalent and tetravalent weave vertices, analogous to the braid moves in the proof of Proposition \ref{prop: containing w_0}. Eventually this weave line comes out on the left, clockwise from the subword $w_0$, so that it can be extended upward while preserving its cyclic position. Since this extension only creates hexavalent and tetravalent weave vertices among external weave lines, it is a weave equivalence. In the end, after clearing out all the blocked sinks, we obtain a complete Demazure weave that is equivalent to the positroid weave $\ww(\bG)$, as required.
\end{proof}

Theorem~\ref{thm:Demazure weave} implies that positroid weaves are free, since all the Demazure weaves are free. In addition, Theorem~\ref{thm:Demazure weave} relates the mutable $Y$-trees of Theorem~\ref{thm:weave construction} to Lusztig cycles (Definition \ref{defn:Lusztig cycles}), which were used extensively in \cite{CGGLSS}. This corollary completes the proof of Theorem~\ref{thm:mainA}(ii).
\begin{cor}\label{cor:Lusztig cycles}  Let $\bG$ be a reduced plabic graph and $\ww(\bG)$ its positroid weave. Consider  the sequence of weave equivalences  in the proof of Theorem~\ref{thm:Demazure weave} that turn $\ww(\bG)$  into a complete Demazure weave. Then,  this sequence of weave equivalences turns the mutable $Y$-trees,  which correspond to internal faces of $\bG$, into mutable Lusztig cycles.
\end{cor}
\begin{proof} By Theorem \ref{thm:weave construction}, the mutable $Y$-trees are exactly the parts of $\bGshift$ contained inside non-boundary faces of $\bG$, and all vertices in this tree subgraph of $\bGshift$ are solid. Therefore, under the perfect orientation $O_i$, all edges of the $Y$-tree will be pointing upward, with a single sink at the top. This implies that the corresponding $Y$-tree is a Lusztig cycle, cf.~Definition \ref{defn:Lusztig cycles}.
\end{proof}

\begin{exmp} Let us once again use the positroid example in Examples \ref{exmp:positroid braid} and \ref{exmp: positroid weave}. In the first iteration of constructing the weave, the permutation associated with the green plabic graph is $\begin{pmatrix}2 & 3 & 4 & 6 & 7 & 1\\
3 & 2 & 1 & 7 & 4 & 6 \end{pmatrix}$. Based on the acyclic perfect orientation with source set $\overrightarrow{I_1}=(1,3,4)$, we rearrange the blue plabic graph into the configuration shown in Figure \ref{fig: demazure weave example}(left). Note that the sink $3$ is blocked because it would violate the cyclic ordering on $\{1,2,\dots, 6\}$ if we extend it all the way to the top.
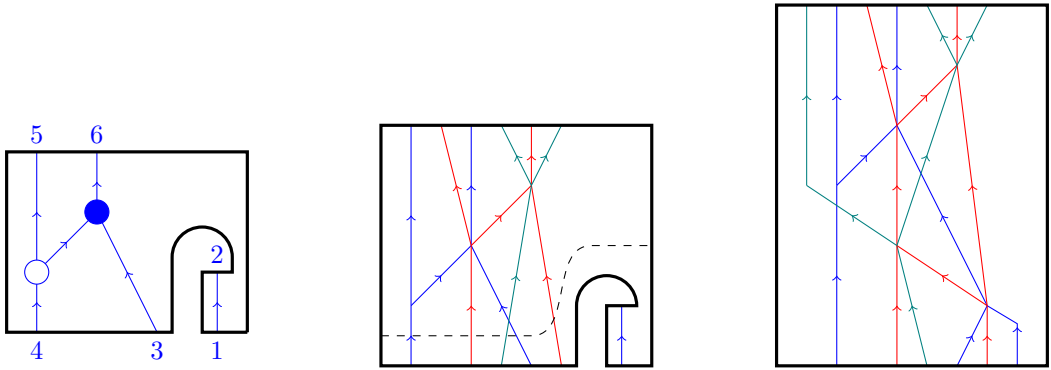
\begin{figure}[H]
    \centering
    \begin{tikzpicture}[scale=0.8]
        \draw [blue, decoration={markings,mark=at position 0.5 with {\arrow{>}}},postaction={decorate}] (0,0) node [below] {$4$} -- (0,1);
        \draw [blue, decoration={markings,mark=at position 0.5 with {\arrow{>}}},postaction={decorate}] (0,1) -- (0,3) node [above] {$5$};
        \draw [blue, decoration={markings,mark=at position 0.5 with {\arrow{>}}},postaction={decorate}] (0,1) -- (1,2);
        \draw [blue, decoration={markings,mark=at position 0.5 with {\arrow{>}}},postaction={decorate}] (2,0) node [below] {$3$} -- (1,2);
        \draw [blue, decoration={markings,mark=at position 0.5 with {\arrow{>}}},postaction={decorate}] (3,0) node [below] {$1$} -- (3,1) node [above] {$2$};
        \draw [blue, decoration={markings,mark=at position 0.5 with {\arrow{>}}},postaction={decorate}] (1,2) -- (1,3) node [above] {$6$};
        \draw [blue,fill=white] (0,1) circle [radius=0.2];
        \draw [blue, fill=blue] (1,2) circle [radius=0.2];
        \draw [very thick] (3.5,0) -- (2.75,0) -- (2.75,1) -- (3.25,1) -- (3.25,1.25) arc (0:180:0.5) -- (2.25,0) -- (-0.5,0) -- (-0.5,3) -- (3.5,3) -- (3.5,0);
    \end{tikzpicture} \hspace{1.5cm}
    \begin{tikzpicture}[scale=0.8]
        \draw [blue, decoration={markings,mark=at position 0.5 with {\arrow{>}}},postaction={decorate}] (0,0) -- (0,1);
        \draw [blue, decoration={markings,mark=at position 0.5 with {\arrow{>}}},postaction={decorate}] (0,1) -- (0,4);
        \draw [blue, decoration={markings,mark=at position 0.5 with {\arrow{>}}},postaction={decorate}] (0,1) -- (1,2);
        \draw [blue, decoration={markings,mark=at position 0.5 with {\arrow{>}}},postaction={decorate}] (2,0) -- (1,2);
        \draw [blue, decoration={markings,mark=at position 0.5 with {\arrow{>}}},postaction={decorate}] (3.5,0) -- (3.5,1);
        \draw [blue, decoration={markings,mark=at position 0.5 with {\arrow{>}}},postaction={decorate}] (1,2) -- (1,4);
        \draw [red, decoration={markings,mark=at position 0.5 with {\arrow{>}}},postaction={decorate}] (1,0) -- (1,2);
        \draw [red, decoration={markings,mark=at position 0.5 with {\arrow{>}}},postaction={decorate}] (1,2) -- (0.5,4);
        \draw [red, decoration={markings,mark=at position 0.5 with {\arrow{>}}},postaction={decorate}] (1,2) -- (2,3);
        \draw [red, decoration={markings,mark=at position 0.5 with {\arrow{>}}},postaction={decorate}] (2.5,0) -- (2,3);
        \draw [red, decoration={markings,mark=at position 0.5 with {\arrow{>}}},postaction={decorate}] (2,3) -- (2,4);
        \draw [teal, decoration={markings,mark=at position 0.5 with {\arrow{>}}},postaction={decorate}] (1.5,0) -- (2,3);
        \draw [teal, decoration={markings,mark=at position 0.5 with {\arrow{>}}},postaction={decorate}] (2,3) -- (1.5,4);
        \draw [teal, decoration={markings,mark=at position 0.5 with {\arrow{>}}},postaction={decorate}] (2,3) -- (2.5,4);
        \draw [very thick] (4,0) -- (3.25,0) -- (3.25,1)--(3.75,1) arc (0:180:0.5) -- (2.75,0) -- (-0.5,0) -- (-0.5,4) -- (4,4) -- cycle;
        \draw [dashed] (-0.5,0.5) -- (2,0.5) to [out=0,in=180] (3,2) -- (4,2);
    \end{tikzpicture}\hspace{1.5cm}
    \begin{tikzpicture}[scale=0.8]
        \draw [blue, decoration={markings,mark=at position 0.5 with {\arrow{>}}},postaction={decorate}] (0,-2) -- (0,1);
        \draw [blue, decoration={markings,mark=at position 0.5 with {\arrow{>}}},postaction={decorate}] (0,1) -- (0,4);
        \draw [blue, decoration={markings,mark=at position 0.5 with {\arrow{>}}},postaction={decorate}] (0,1) -- (1,2);
        \draw [blue, decoration={markings,mark=at position 0.5 with {\arrow{>}}},postaction={decorate}] (2.5,-1) -- (1,2);
        \draw [blue, decoration={markings,mark=at position 0.5 with {\arrow{>}}},postaction={decorate}] (1,2) -- (1,4);
        \draw [red, decoration={markings,mark=at position 0.5 with {\arrow{>}}},postaction={decorate}] (1,0) -- (1,2);
        \draw [red, decoration={markings,mark=at position 0.5 with {\arrow{>}}},postaction={decorate}] (1,2) -- (0.5,4);
        \draw [red, decoration={markings,mark=at position 0.5 with {\arrow{>}}},postaction={decorate}] (1,2) -- (2,3);
        \draw [red, decoration={markings,mark=at position 0.5 with {\arrow{>}}},postaction={decorate}] (2.5,-1) -- (2,3);
        \draw [red, decoration={markings,mark=at position 0.5 with {\arrow{>}}},postaction={decorate}] (2,3) -- (2,4);
        \draw [teal, decoration={markings,mark=at position 0.5 with {\arrow{>}}},postaction={decorate}] (1,0) -- (2,3);
        \draw [teal, decoration={markings,mark=at position 0.5 with {\arrow{>}}},postaction={decorate}] (2,3) -- (1.5,4);
        \draw [teal, decoration={markings,mark=at position 0.5 with {\arrow{>}}},postaction={decorate}] (2,3) -- (2.5,4);
        \draw [blue, decoration={markings,mark=at position 0.5 with {\arrow{>}}},postaction={decorate}] (3,-2) -- (3,-1.3) -- (2.5,-1);
        \draw [blue, decoration={markings,mark=at position 0.5 with {\arrow{>}}},postaction={decorate}] (2,-2) -- (2.5,-1);
        \draw [red, decoration={markings,mark=at position 0.5 with {\arrow{>}}},postaction={decorate}] (2.5,-2) -- (2.5,-1);
        \draw [red, decoration={markings,mark=at position 0.5 with {\arrow{>}}},postaction={decorate}] (2.5,-1) -- (1,0);
        \draw [teal, decoration={markings,mark=at position 0.5 with {\arrow{>}}},postaction={decorate}] (1.5,-2) -- (1,0);
        \draw [red, decoration={markings,mark=at position 0.5 with {\arrow{>}}},postaction={decorate}] (1,-2) -- (1,0);
        \draw [teal, decoration={markings,mark=at position 0.5 with {\arrow{>}}},postaction={decorate}] (1,0) -- (-0.5,1);
        \draw [teal, decoration={markings,mark=at position 0.5 with {\arrow{>}}},postaction={decorate}] (-0.5,1) -- (-0.5,4);
        \draw [very thick] (3.5,-2) -- (-1,-2) -- (-1,4) -- (3.5,4) -- cycle;
    \end{tikzpicture}
    \caption{A Demazure weave equivalent to the positroid weave in Example \ref{exmp: positroid weave}. The dashed line in the middle picture shows where to cut and stretch the lower part of the picture.}
    \label{fig: demazure weave example}
\end{figure}
\noindent By continuing the iterative construction with compatible perfect orientations, we  obtain the configuration in Figure \ref{fig: demazure weave example}(center). Lastly, we need to deal with the blocked sinks: in this example there is only one, the sink $3$ of the blue layer. By following the recipe, we need to stretch the weave vertically and then extend the outgoing external weave line towards the left. When crossing the incoming external weave lines, we add hexavalent weave vertices accordingly. The resulting final configuration is shown in Figure \ref{fig: demazure weave example}(right).\hfill$\Box$
\end{exmp}

\section{Weaves and Conjugate Surfaces}

This section proves Theorem \ref{thm:mainA}.(i). The core result is Theorem \ref{thm:weave_and_conjugate}, relating the Lagrangian projection of the Legendrian lift of the weave $\ww(\bG)$ to the conjugate surface of $\bG$. Both these surfaces are naturally embedded in the 4-dimensional space given by the cotangent bundle of the 2-disk and we show they are Hamiltonian isotopic and, in particular, smoothly isotopic. We also include Section \ref{ssec:weave_recap} as a succinct recollection of part of the topology underlying the theory of weaves. Section \ref{ssec:weave_recap} contains no new mathematical results but might be helpful to readers less familiar with weaves. Section \ref{ssec:Hamiltonian_comparison} contains the statement and proof of Theorem \ref{thm:weave_and_conjugate}, which is certainly new, both from the combinatorial and symplectic geometric viewpoints. This latter section makes particular use of part of our previous work, especially \cite[Section 3]{CW} and \cite[Section 3]{CL22}.

\subsection{Legendrian and Exact Lagrangian Surfaces from Weaves}\label{ssec:weave_recap}

In this short subsection, we follow \cite{CZ, CW, CGGLSS} and give a brief review on the constructions in Subsection \ref{subsec: prelim on weaves} from a geometric perspective.\\

A $k$-weave $\ww$ in a disk $\bD$ encodes a $k$-sheeted branched cover $\Sigma_\ww\lr\bD$ of the disk, cf.~\cite[Section 2]{CZ}. By construction, this surface $\Sigma_\ww$ is naturally included in $\bD^2_{x_1,x_2} \times \bR_z$, as a singular surface, and it is a wavefront, cf.~\cite[Section 2.2]{CZ}. Generically, the $k$ sheets are at distinct heights; the weave line of color $i$ is the locus where the $i$th and $(i+1)$th sheets cross. They cross transversely, creating an interval worth of immersed points along the corresponding weave line. By setting $y_i:=\partial_{x_i}z$, this singular surface in $\bD^2_{x_1,x_2} \times \bR_z$ lifts to an embedded Legendrian $\Lambda_\ww$ in the 5-dimensional $1$-jet space $J^1\bD\cong  T^*\bD^2 \times \bR_z$, where $(x_1,x_2,y_1,y_2)\in T^*\bD^2$ and $(y_1,y_2)$ are Cartesian coordinates for the cotangent fibers. If $\Lambda_\ww$ does not have any Reeb chords, its Lagrangian projection $L_\ww :=\pi_L(\Lambda_\ww)\subset T^*\bD^2$ is an embedded exact Lagrangian surface, where $\pi_L:T^*\bD^2 \times \bR\longrightarrow T^*\bD^2$ is the projection onto the first factor. If $\ww$ is a Demazure weave, then $\Lambda_\ww$ has no Reeb chords and thus the surface $L_\ww$ is embedded in the 4-dimensional space $T^*\bD^2$. By following \cite[Subsection 7.4]{CGGLSS} or \cite[Section 2.4]{CZ}, one can canonically interpret $Y$-cycles as topological $1$-cycles on $L_\ww$, cf.~Figure \ref{fig: perturbation}.

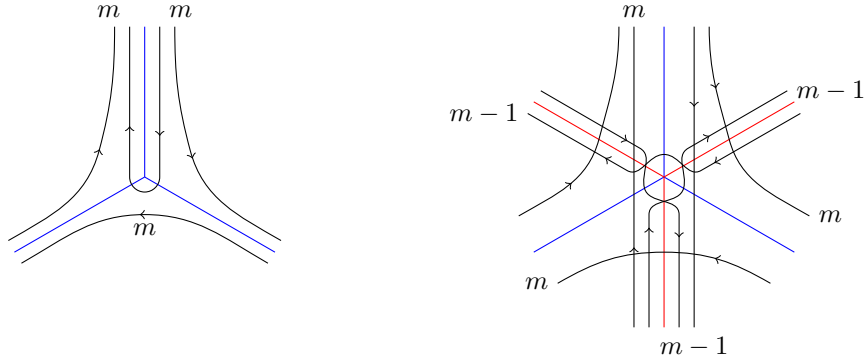
\begin{figure}[H]
    \centering
    \begin{tikzpicture}[baseline=0]
    \foreach \i in {0,1,2}
    {
    \draw [blue] (0,0) -- (90+\i*120:2);
    }
    \draw [decoration={markings, mark=at position 0.3 with {\arrow{<}}, mark=at position 0.7 with {\arrow{<}}}, postaction=decorate] (-0.2,2) node [above left] {$m$} -- (-0.2,0) arc (-180:0:0.2) -- (0.2,2) node [above right] {$m$};
    \draw [decoration={markings, mark=at position 0.5 with {\arrow{<}}}, postaction=decorate] (-0.4,2) to [out=-90,in=75] (150:0.7) to [out=-105,in=30] (-155:2);
    \draw [decoration={markings, mark=at position 0.5 with {\arrow{<}}}, postaction=decorate] (-145:2) to [out=30,in=-180] (0,-0.5)  node [below] {$m$} to [out=0,in=150] (-35:2);
    \draw [decoration={markings, mark=at position 0.5 with {\arrow{<}}}, postaction=decorate] (-25:2) to [out=150,in=-75] (30:0.7) to [out=105,in=-90] (0.4,2);
    \end{tikzpicture}\hspace{2cm}
    \begin{tikzpicture}[baseline=0]
    \foreach \i in {0,1,2}
    {
    \draw [blue] (0,0) -- (90+\i*120:2);
    \draw [red] (0,0) -- (30+\i*120:2);
    }
    \draw [decoration={markings, mark=at position 0.75 with {\arrow{<}}}, postaction=decorate] (-0.4,2) node [above] {$m$} -- (-0.4,-2);
    \draw [decoration={markings, mark=at position 0.75 with {\arrow{<}}}, postaction=decorate] (0.4,-2)  node [below] {$m-1$} -- (0.4,2);
    \draw [decoration={markings, mark=at position 0.75 with {\arrow{<}}}, postaction=decorate] (-0.6,2) to [out=-90,in=75] (150:1) to [out=-105,in=30] (-165:2);
    \draw [decoration={markings, mark=at position 0.75 with {\arrow{<}}}, postaction=decorate] (-135:2) node [left] {$m$} to [out=30,in=-180] (0,-1) to [out=0,in=150] (-45:2);
    \draw [decoration={markings, mark=at position 0.75 with {\arrow{<}}}, postaction=decorate] (-15:2) node [right] {$m$} to [out=150,in=-75] (30:1) to [out=105,in=-90] (0.6,2);
    \draw [decoration={markings, mark=at position 0.3 with {\arrow{<}}, mark=at position 0.7 with {\arrow{<}}}, postaction=decorate] (0.2,-2) to [out=90,in=-90] (0.2,-0.5) to [out=90,in=-75] (-150:0.3) to [out=105,in=-30] (130:0.5) -- (145:2);
    \draw [decoration={markings, mark=at position 0.3 with {\arrow{<}}, mark=at position 0.7 with {\arrow{<}}}, postaction=decorate] (155:2)  node [left] {$m-1$} -- (170:0.5) to [out=-30,in=180] (0,0.3) to [out=0,in=-150] (10:0.5) -- (25:2);
    \draw [decoration={markings, mark=at position 0.3 with {\arrow{<}}, mark=at position 0.7 with {\arrow{<}}}, postaction=decorate] (35:2) node [right] {$m-1$} -- (50:0.5) to [out=-150,in=75] (-30:0.3) to [out=-105,in=90] (-0.2,-0.5) -- (-0.2,-2);
    \end{tikzpicture}
    \caption{Local pictures of $Y$-cycle representatives as curves: near a trivalent weave vertex (left) and a hexavalent weave vertex (right).}
    \label{fig: perturbation}
\end{figure}

\noindent For context, the following lemma shows that the intersection pairing of $Y$-cycles in Definition \ref{defn:local intersection pairing} has a natural topological interpretation.
\begin{lem}[\protect{\cite[Section 7.4]{CGGLSS}}] The intersection pairing of $Y$-cycles defined in Definition \ref{defn:local intersection pairing} agrees with the homological intersection pairing of the corresponding topological $1$-cycles.
\end{lem}
\begin{proof} Note that according to the representatives chosen in Figure \ref{fig: perturbation}, intersections between these $1$-cycles can only occur near a weave vertex. Thus, the agreement of these two intersection pairings can be verified locally at each weave vertex, cf.~\cite[Lemma 7.17]{CGGLSS}.
\end{proof}

Let $\Lambda_0\sse(\bR^3,\xi_{st})$ be the unique max-tb Legendrian unknot and $L_0\sse(\bR^4,\lambda_{st})$ its unique embedded exact Lagrangian filling in the symplectization of $(\bR^3,\xi_{st})$. See \cite[Section 2.2]{CasalsNg} for more details on the following definition.

\begin{defn}[\cite{CasalsNg}]\label{defn: -1 closure} Let $\ww$ be a weave. Consider the satellite of $L_\ww$ along $L_0$, so that the boundary $\partial L_\ww$ is a Legendrian satellite along $\La_0$ of the Legendrian link in $(J^1\La_0,\xi_{st})$ associated to the cyclic positive braid word $\beta:=\partial \ww$. By definition, this Legendrian satellite link is denoted by $\Lambda_\beta\sse(\bR^3,\xi_{st})$ and said to be the \emph{$(-1)$-closure} of the positive braid $\beta$. A front projection of $\Lambda_\beta$ is drawn in Figure \ref{fig:-1 closure of a positive braid} (left).\hfill$\Box$
\end{defn}

\noindent Since $\partial L_\ww=\Lambda_\beta$, the embedded Lagrangian surface $L_\ww$ is an exact Lagrangian filling of $\Lambda_\beta$.

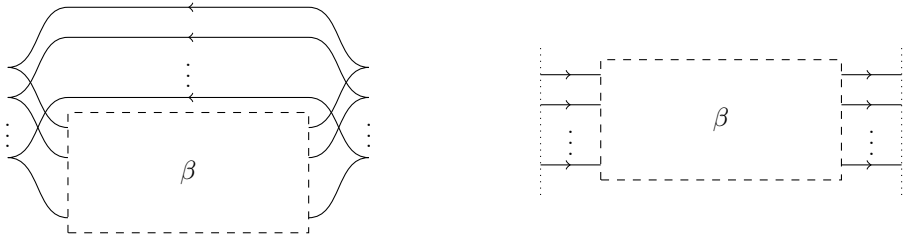
\begin{figure}[H]
    \centering
    \begin{tikzpicture}[baseline=0,scale=0.8]
    \draw [dashed] (0,0.25) rectangle node [] {$\beta$} (4,2.25);
    \foreach \i in {1,3,4}
    {
    \draw [decoration={markings,mark=at position 0.5 with {\arrow{>}}},postaction={decorate}] (4,0.5*\i) to [out=0,in=180] (5,1+0.5*\i) to [out=180,in=0] (4,2+0.5*\i) -- (0,2+0.5*\i) to [out=180,in=0] (-1,1+0.5*\i) to [out=0,in=180] (0,0.5*\i);
    }
    \node at (2,3) [] {\footnotesize{$\vdots$}};
    \node at (-1,2) [] {\footnotesize{$\vdots$}};
    \node at (5,2) [] {\footnotesize{$\vdots$}};
    \end{tikzpicture}\hspace{2cm}
    \begin{tikzpicture}[baseline=-20,scale=0.8]
        \draw [dashed] (0,0.25) rectangle node [] {$\beta$} (4,2.25);
        \foreach \i in {1,3,4}
    {
    \draw [decoration={markings,mark=at position 0.5 with {\arrow{>}}},postaction={decorate}](-1,0.5*\i) -- (0,0.5*\i);
    \draw [decoration={markings,mark=at position 0.5 with {\arrow{>}}},postaction={decorate}](4,0.5*\i) -- (5,0.5*\i);
    }
    \draw [dotted] (-1,0) -- (-1,2.5);
    \draw [dotted] (5,0) -- (5,2.5);
    \node at (-0.5,1) [] {\footnotesize{$\vdots$}};
        \node at (4.5,1) [] {\footnotesize{$\vdots$}};
    \end{tikzpicture}
    \caption{Left: $(-1)$-closure of a positive braid word $\beta$. Right: a drawing of $\hat{\beta}$, where the left dotted line is identified with the right dotted line so that the whole picture should be viewed as a subset of $S^1\times \mathbb{R}$.}
    \label{fig:-1 closure of a positive braid}
\end{figure}

\begin{rmk} It can be useful to understand the $(-1)$-closure $\Lambda_\beta$ as a Legendrian link in a tubular neighborhood $(J^1\La_0,\xi_{st})$ of $\La_0$. We denote by $\hat{\beta}$ the image of $\Lambda_\beta$ under the standard front projection $J^1\La_0\rightarrow \La_0\times \mathbb{R}$, which is a cyclic braid front with crossings as dictated by $\beta$. See \cite[Section 2]{CasalsNg} for more details.\hfill$\Box$
\end{rmk}

\subsection{Comparison of Hamiltonian isotopy classes}\label{ssec:Hamiltonian_comparison} Given a plabic graph $\bG$, consider its conjugate surface $\Sigma(\bG)$ as defined in \cite[Section 1.1.1]{GonKen}. See also \cite[Section 2]{Goncharov_IdealWebs}, \cite[Section 4]{STWZ} or \cite[Section 2.1]{CL22} for additional details on conjugate surfaces. For the purposes of this manuscript, this is a (ribbon) surface obtained from a plabic graph by using the four local models in Figure \ref{fig:ConjugateSurface_Models_Alt}. The boundary of the surface is given by the alternating strand diagram of $\bG$, see \cite{Pos,Goncharov_IdealWebs}. It retracts to $\bG$ and thus $H_1(\Sigma(\bG),\bZ)\cong H_1(\bG,\bZ)$, the latter being a free $\bZ$-module of rank equal to the number of faces of $\bG$.

\begin{center}
	\begin{figure}[H]
		\centering
		\includegraphics[scale=0.65]{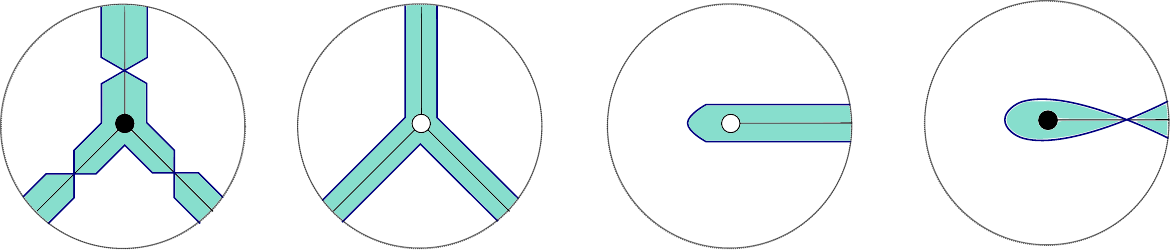}
		\caption{The four local models needed to draw a (projection of a) conjugate surface associated to a plabic graph $\bG$. The boundary of the surface is in dark blue and the surface itself in light blue.}\label{fig:ConjugateSurface_Models_Alt}
	\end{figure}
\end{center}

It is proven in \cite[Section 4.2]{STWZ} that $\Sigma(\bG)$ is the projection of an embedded exact Lagrangian surface $L(\bG)\sse(T^*\D^2,\la_{st})$ in the standard cotangent bundle of $\bD^2$. It is shown in \cite[Section 2.1]{CL22} that the Hamiltonian isotopy class of $L(\bG)$ is unique. In particular, $L(\bG)$ is a Lagrangian filling of the Legendrian lift of the alternating strand diagram. By the construction in Section \ref{sec:iterative_Tmap}, we can also associate a free weave $\ww(\bG)$ in $\D^2$ to $\bG$, whose Lagrangian projection defines another embedded exact Lagrangian surface in $(T^*\D^2,\la_{st})$. We now show that these two embedded exact Lagrangian surfaces are Hamiltonian isotopic.\\

\begin{thm}\label{thm:weave_and_conjugate}
Let $\bG\sse\D^2$ be a reduced plabic graph, $L(\bG)\sse(T^*\D^2,\la_{st})$ its conjugate Lagrangian surface, $\tau$ a sequence of $3$-regular trees for $\bG$, and $\ww(\bG,\tau)$ the associated positroid weave. Then the Lagrangian projection of the Legendrian lift of $\ww(\bG,\tau)$ is Hamiltonian isotopic to $L(\bG)$.
\end{thm}

\begin{figure}[H]
  \centering
  \includegraphics[scale=0.7]{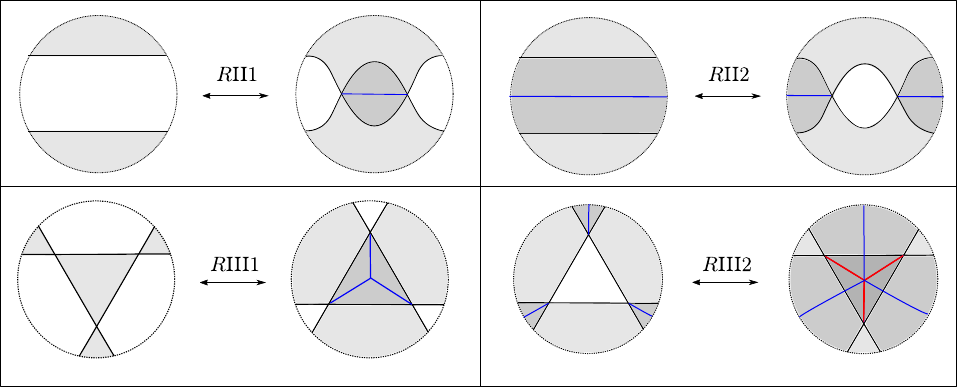}\\
  \caption{Reidemeister-type moves used in order to transition from a conjugate surface towards a weave. These are proven in \cite[Section 3]{CL22}.}
  \label{fig:TableMovesIntro}
\end{figure}

\begin{proof}
We use the framework of hybrid Lagrangians developed in \cite[Section 3]{CL22}. Specifically, we use the Reidemeister-type moves for hybrid surfaces shown in Figure \ref{fig:TableMovesIntro}. This allows us to translate the sympletic geometric aspects into diagrammatic combinatorics. We start with the conjugate Lagrangian surface $L(\bG)\sse(T^*\D^2,\la_{st})$, which we encode via its projection as a ribbon surface in $\D^2$, bounded by the alternating strand diagram of $\bG$. Figure \ref{fig:ConjugateSurface}.(i) draws in gray the conjugate surface for a reduced plabic graph of type $(3,6)$. We proceed algorithmically, as follows.

\begin{enumerate}
    \item First, perform an RIII--1 move at each of the (gray) triangles of the (projection of the) conjugate surface containing a black vertex of the plabic graph $\bG$. Figure \ref{fig:ConjugateSurface}.(ii) illustrates one instance of such an RIII-1 move being performed.
    \begin{center}
	\begin{figure}[H]
		\centering
		\includegraphics[scale=1.1]{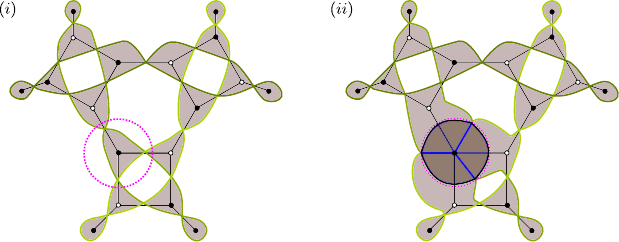}
		\caption{(Left) The conjugate surface associated to a reduced plabic graph of type $(3,6)$. (Right) The hybrid surface obtained by performing an RIII-1 move at the portion of the conjugate surface on the left surrounded by a pink circle.}
		\label{fig:ConjugateSurface}
	\end{figure}
\end{center}
Each face of $\bG$ has $2n$ sides for some $n\in\mathbb{N}$. For each face of $\bG$, these RIII-1 moves create $n$ trivalent vertices, of weave type and of top color, in the resulting hybrid surface. The projection of this hybrid surface to $\D^2$ is such that it bounds a $2n$-gon {\it inside} each $2n$-gonal face; $n$ of the vertices of this $2n$-gon have incoming weave edges with the top color. Figure \ref{fig:ConjugateSurface2}.(i) and (ii) illustrate what happens at a hexagonal and octogonal face after the RIII-1 moves above are performed.
\begin{center}
	\begin{figure}[H]
		\centering
		\includegraphics[scale=1.1]{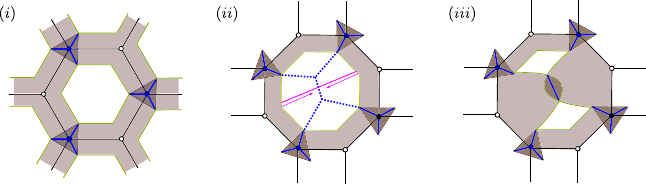}
		\caption{(Left) The hybrid surface near a hexagonal face of the reduced plabic graph after performing the sequence of RIII-1 moves in the first step of the proof of Theorem \ref{thm:weave_and_conjugate}. (Center) The hybrid surface near an octogonal face after the RIII-1 moves, a choice of $3$-regular tree connecting the weave lines pointing inward to the face (in dashed blue lines) and a choice of diagonal for the dual triangulation, in pink. (Right) The result of applying an RII-1 move to the hybrid surface in $(ii)$, performed along the pink diagonal. }
		\label{fig:ConjugateSurface2}
	\end{figure}
\end{center}
This is the only step in the process  using RIII-1 moves and these will be all trivalent vertices of the resulting weave: those created in the first step. Thus they all have the same (top) color.

    \item The construction of the weaves from the reduced plabic graphs in Section \ref{sec:iterative_Tmap} has the following choices: at each step in the iteration, we choose a 3-regular tree with $(n-2)$ vertices for each $2n$-gonal face of the plabic graph at that step. In this second step, we consider the hybrid surface near each $2n$-gonal face of the plabic graph, as depicted in Figure \ref{fig:ConjugateSurface2}.(i) and (ii) and Figure \ref{fig:ConjugateSurface3}. We choose a 3-regular tree inside of the face which connects the weave lines pointing inward into the face, as drawn in Figures \ref{fig:ConjugateSurface2}.(ii) and \ref{fig:ConjugateSurface3}. We then select a collection of $(n-3)$ disjoint embedded intervals in the face such that each interval is a (topological) diagonal for the triangulation of the face dual to the 3-regular tree. These intervals are depicted in pink in Figures \ref{fig:ConjugateSurface2}.(ii) and \ref{fig:ConjugateSurface3}, which respectively correspond to the cases $n=4$ and $5$.

    \begin{center}
	\begin{figure}[H]
		\centering
		\includegraphics[scale=1.3]{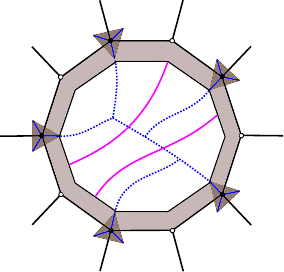}
		\caption{The hybrid surface near a decagonal face after the RIII-1 moves, with a choice of $3$-regular tree (in dashed blue lines) and diagonals for the dual triangulation (in pink).}
		\label{fig:ConjugateSurface3}
	\end{figure}
\end{center}

    Now, for each such interval in the collection, we perform an RII-1 move along that interval. That is, we pull together the two regions of the hybrid surface at the ends of each interval and create a weave line in the shape of an interval. This is illustrated in Figure \ref{fig:ConjugateSurface2}.(iii), which is the result of Figure \ref{fig:ConjugateSurface2}.(ii) after performing the (only) RII-1 along the (pink) interval; the arrows in Figure \ref{fig:ConjugateSurface2}.(ii) indicate the direction in which we pull. This creates a total of $(n-3)$ weave line intervals inside of the face of the reduced plabic graph.\\

    \item The sequence of RII-1 moves performed in the second step results in a hybrid surface such that, at each given $2n$-gonal face, the complement of (the projection of) the hybrid surface consists of $(n-2)$ topological triangles. In each such triangle, there are three incoming weave lines of the top color. The third step consists of performing a sequence of RIII-2 moves, one for each such triangle, and for each face. This introduces weave lines of a new color.\\

    \item Consider the plabic graph $\bGshift$ obtained by applying the construction in Section \ref{sec:iterative_Tmap} to the initial plabic graph $\bG$ and the hybrid surface at this stage. Then the weave lines of a new color of the hybrid surface, introduced by the RIII-2 moves of Step 3, combinatorially coincide with the weave part of the hybrid surface that we obtain by performing Step 1 of this algorithm to $\bGshift$. That is, if we ignore the top layer in the hybrid surface (and thus the top color weave lines disappear), the resulting hybrid surface is the hybrid surface that is obtained by performing Step 1 of this algorithm to $\bGshift$.\\
    
    Now, Steps (2) and (3) above can be performed to the weave lines of the new color in the hybrid surface at this stage without interfering with the top layer. Therefore, we now iteratively perform Steps (2) and (3) to the hybrid surface but we apply them to the weave lines of the new color, i.e. fixing the top layer and applying Steps (2) and (3) as if we had started with the plabic graph $\bGshift$.\\

    \item Finally, we iterate Step (4). That is, we repeatedly apply Steps (2) and (3), each time to the part of the hybrid surface which consists of the weave lines of the newest color introduced by the (previous) Step 3. After $(\mbox{rk}(\bG)-2)$ iterations of Steps (2) and (3), we obtain a hybrid surface which is a weave, consisting of weave lines of $(\mbox{rk}(\bG)-1)$ colors.\\
\end{enumerate}

The Steps (2) and (3) in the algorithm above reproduce exactly the weave construction of Section \ref{sec:iterative_Tmap}. Therefore, we conclude that the resulting weave is indeed of type $\ww(\bG)$. In fact, if we perform the same choices of $3$-regular trees as we would if we followed in the construction in Proposition~\ref{prop:weaves from plabic graphs}, the resulting weave would be identical.
\end{proof}

\noindent By Theorem \ref{thm:weave construction}, different choices of sequences $\tau$ and $\tau'$ lead to equivalent weaves $\ww(\bG,\tau)$ and $\ww(\bG,\tau')$, differing by the flop moves in Figure \ref{fig:weave equivalence}.(III). By \cite[Theorem 4.2]{CZ}, the Lagrangian projections of $\ww(\bG,\tau)$ and $\ww(\bG,\tau')$ are Hamiltonian isotopic. Thus Theorem \ref{thm:weave_and_conjugate} implies Theorem \ref{thm:mainA}.(i), where $\ww(\bG)$ in the latter denotes $\ww(\bG,\tau)$ for any choice of $\tau$.

\section{Flag Moduli Spaces}

In this section, we introduce flag moduli spaces associated to a positroid $\cP$  of type $(m,n)$. One of these spaces will be isomorphic to $\Pi^\circ_\cP$, cf.~ Proposition~\ref{thm: frame sheaf moduli = positroid} below. These flag moduli spaces were originally considered in the context of the microlocal theory of sheaves, cf.~\cite[Section 2.8]{CW} and references therein. That said, we follow \cite[Section 4.1]{CW} and describe them in terms of tuples of flags, cf.~Section~\ref{ssec: definition of sheaf moduli spaces} below, and also in terms of certain rank 1 local systems associated to $\Lambda_\cP$, cf.~ Section~\ref{subsec:local systems}. The first perspective is helpful when studying certain aspects of their associated cluster algebras, as in \cite{CW,CGGLSS,CG23}. The second perspective is particularly useful in proving that the Muller-Speyer twist map is DT, which we achieve in Section~\ref{sec:twist} below.


\subsection{Preliminaries on flags} Let $\Fl_m=\GL_m(\bC)/B$ denote the complete flag variety of $\bC^m$, where $m\in\mathbb{N}$ and $B\subset \GL_m(\bC)$ is the Borel subgroup of upper triangular matrices. We denote the $k$-dimensional subspace of a flag $\cF$ by $(\cF)_k$. The relative positions of two flags can be indexed by elements in the Weyl group $S_n$. Indeed, using the Bruhat decomposition $\GL_n=\bigsqcup_{w\in S_n}BwB$, we define  two flags $xB,yB$  to be in relative position $w \in S_n$, denoted $xB\xrightarrow{w}yB$, if $Bx^{-1}yB=BwB$.  This is independent of the representative matrix chosen for $w \in S_n$. For $\cF, \cF' \in \Fl_m$ and $i \in [m-1]$, we write $\cF \xrightarrow{i} \cF'$ if $\cF$ and $\cF'$ are in relative position $s_i \in S_n$, i.e.~ differ precisely at the $i$-dimensional subspace. Let $\beta= s_{i_1} \cdots s_{i_l}$ be a positive braid word. By definition, a \emph{$\beta$-chain of flags} is a tuple of flags $(\cF_1, \dots, \cF_l)$ satisfying
	\[\cF_1 \xrightarrow{i_1} \cF_2 \xrightarrow{i_2} \cdots \xrightarrow{i_l} \cF_l.\]
If $\beta$ is reduced word for $w \in S_{m}$, we write $\cF \xrightarrow{\beta} \cF'$ to indicate there exists a $\beta$-chain of flags which begins at $\cF$ and ends at $\cF'$. In that case, the pair of flags $\cF, \cF'$ is said to be in \emph{relative position} $w$ or $\beta$. If $\cF \xrightarrow{\beta} \cF'$, the $\beta$-chain $\cF \xrightarrow{i_1} \cF_2 \xrightarrow{i_2} \cdots \xrightarrow{i_l} \cF'$ can be uniquely recovered from the flags $\cF, \cF'$, cf.~ \cite[Lemma 2.5]{ShenWeng}. By definition, $\cF \xrightarrow{e} \cF'$ means $\cF=\cF'$.

\subsection{Definition of the Flag Moduli Spaces}\label{ssec: definition of sheaf moduli spaces} 



Let $\cP$ be a positroid and $\beta_\cP$  its associated positroid braid word, cf.~Definition \ref{defn:positroid-braid-word}. By Theorem \ref{thm:weave construction}, $\beta_\cP$ is the cyclic positive braid word at the boundary of a positroid weave $\ww$ associated with $\cP$.


\begin{defn}[Flag moduli]\label{defn:undecorated sheaf moduli space} Let $\cP$ be a positroid, $\beta_\cP=s_{i_1}s_{i_2} \cdots s_{i_l}= \beta_1 \beta_2 \cdots \beta_n$ its associated positive braid word. The \emph{flag moduli space} of $\beta_\cP$ is the configuration space
\[
\modsp(\beta_\cP):=\left.\left\{(\cF_0,\cF_1,\dots, \cF_l=\cF_0)\in (\Fl_m)^l \ \middle| \  \cF_{j-1}\xrightarrow{{i_j}}\cF_j \text{ for all $1\leq j\leq l$} \right\}\right/\GL_m.
\]
\noindent Since each piece $\beta_i$ is reduced, $\modsp(\beta_\cP)$ is isomorphic to 
\begin{equation}\label{eq:alternative def of sheaf moduli space}
	\modsp(\beta_\cP):=\left.\left\{(\cF_0,\cF_1,\dots, \cF_n=\cF_0)\in (\Fl_m)^n \ \middle| \  \cF_{i-1}\xrightarrow{\beta_i}\cF_i \text{ for all $1\leq i\leq n$} \right\}\right/\GL_m.
\end{equation}
\noindent We consider $\modsp(\beta_\cP)$ as an algebraic quotient stack, given by the data of the affine variety parameterizing  tuples of flags as above modulo the  algebraic $\GL_m$-action.
\hfill$\Box$
\end{defn}

\noindent A decorated version of $\modsp(\beta_\cP)$ is relevant in the study of positroids, it is defined as follows. 

\begin{defn}[Decorated flags] Let $\cF\in\GL_m(\bC)/B$ be a  complete flag in $\bC^m$. By definition, a \emph{decoration} of $\cF$ is a tuple $v_\cF:=((v)_1, \dots, (v)_m)$  of vectors, where $(v)_k$ is a non-zero element of the 1-dimensional quotient $(\cF)_k/(\cF)_{k-1}$. By definition, a \emph{decorated flag} $\tilde{\cF}:=(\cF,v_\cF)$ in $\bC^m$ is a pair consisting of a flag $\cF$ in $\bC^m$ together with a decoration $v_\cF$. Each decorated flag $\tilde{\cF}$ uniquely determines a volume form $\det(\tilde{\cF})$ on $\bC^m$, also denoted $\det(v)$, by the formula
\[
\det(\tilde{\cF})=\det(v):=(v)_1\wedge (v)_{2}\wedge \cdots \wedge (v)_m.
\]
Alternatively, if $U\sse B$ is the maximal unipotent subgroup, a decoration on a flag is a point in its fiber for the projection $p:\GL_m/U\rightarrow \GL_m/B$, and decorated flags are the elements of $\GL_m/U$.\hfill$\Box$
\end{defn}


Decorations are also referred to as trivializations, as they provide a basis for the 1-dimensional quotients $(\cF)_k/(\cF)_{k-1}$. The relative position of decorated flags is defined as follows. Let $\widetilde{\Fl}_m:=\GL_m/U$ be the space of decorated flags in $\bC^m$ and $p:\widetilde{\Fl}_m\rightarrow \Fl_m$ the forgetful map, forgetting the decoration.  Suppose $\tilde{\cF}, \tilde{\cF'}$ are decorated flags satisfying $p(\tilde{\cF})\xrightarrow{i}p(\tilde{\cF}')$. Then for any $1\leq k\leq m$ with $k\neq i,i+1$, we have
\begin{equation}\label{eq:equal quotients}
\frac{(\cF)_k}{(\cF)_{k-1}}=\frac{(\cF')_k}{(\cF')_{k-1}}.
\end{equation}
Since $(\cF)_i\neq (\cF')_i$, the following compositions are isomorphisms:\begin{equation}\label{eq:rho}
\rho:\frac{(\cF)_i}{(\cF)_{i-1}}\hookrightarrow\frac{(\cF)_{i+1}}{(\cF)_{i-1}}=\frac{(\cF')_{i+1}}{(\cF')_{i-1}}\twoheadrightarrow \frac{(\cF')_{i+1}}{(\cF')_i}
\end{equation}
\begin{equation}\label{eq:lambda}
\lambda:\frac{(\cF')_{i}}{(\cF')_{i-1}}\hookrightarrow \frac{(\cF')_{i+1}}{(\cF')_{i-1}}=\frac{(\cF)_{i+1}}{(\cF)_{i-1}}\twoheadrightarrow \frac{(\cF)_{i+1}}{(\cF)_i}.
\end{equation}

\begin{defn}\label{defn:compatible decoration} Two decorated flags $\tilde{\cF},\tilde{\cF}'\in\widetilde{\Fl}_m$ are said to be in $i$th relative position, denoted $\tilde{\cF}\doublearrow{i} \tilde{\cF}'$, if the following three conditions are satisfied:
\begin{itemize}
    \item[-] $p(\tilde{\cF})\xrightarrow{i}p(\tilde{\cF}')$,
    \item[-] $(v)_k=(v')_k$ for all $k\neq i,i+1$,
    \item[-] $\rho((v)_i)=(v')_{i+1}$ and $\lambda((v')_i)=(v)_{i+1}$.
\end{itemize}
Let $\beta=s_{i_1}s_{i_2}\cdots s_{i_l}$ be a positive braid word. A \emph{decorated $\beta$-chain} is a tuple of decorated flags satisfying $$\tcF\doublearrow{i_1} \tcF_1 \doublearrow{i_2} \cdots \doublearrow{i_l} \tcF'.$$
If $\beta$ is reduced, we write $\tilde{\cF}\doublearrow{\beta}\tilde{\cF}'$ if there exists a decorated $\beta$-chain between $\tilde{\cF}$ and $\tilde{\cF}'$.\hfill$\Box$
\end{defn}

\noindent Note that for two decorated flags in $i$th relative position, the decorations of one flag determine the decorations of the other, as $\rho$ and $\lambda$ are isomorphisms.
It is also possible to change the decoration of a flag without changing the underlying flag:  this allows us to introduce the appropriate decorated version of $\modsp(\beta_\cP)$, which will be isomorphic to a positroid stratum. First, we introduce \emph{base points} on $\beta_\cP$ in order to discuss exactly what type of decoration changes we allow, as follows.

\begin{defn}[Base points]\label{def:base-pts}
	Let $\cP$ be a positroid with bounded affine permutation $f$ and $\beta_\cP= \beta_1\beta_2\dots \beta_n$ the corresponding braid word. For each $i \in [n]$ which is not a fixed point of $f$, we add two base points $i^+$ and $i^-$ to the braid diagram  of $\beta_\cP$, inserted as follows:
	\begin{itemize}
		\item[(a)] If $l(\beta_i)>0$, then we add $i^+$, resp. $i^-$, to the segment at height $m$, resp. height $m-1$, just to the left of the first crossing of $\beta_i$. See Figure~\ref{fig:base-pts-flags}(right); we draw $i^+$ slightly to the left of $i_-$.\\
  
		\item[(b)] If $f(i-1)=i$ so $l(\beta_i)=0$, then we add $i^+$ and $i^{-}$ to the segment at height $m$ between $\beta_{i-1}$ and $\beta_{i+1}$, with $i^+$ on the left. See Figure~\ref{fig:base-pts-flags} (left).
	\end{itemize}
We denote this particular set of points in the braid diagram of $\beta_\cP$ by $\mathfrak{t}_\cP$, or by $\mathfrak{t}$ if $\cP$ is understood.\hfill$\Box$
\end{defn}

\begin{figure}[H]
    \centering
    \begin{tikzpicture}[scale=1.5]
    \draw (0,0) -- (1,0);
    \draw (0,0.5) -- (1,0.5);
    \draw (0,1) -- (1,1);
    \draw (0,1.5) -- (1,1.5);
    \node at (0.25,1.5) [] {$\bullet$};
    \node at (0.25,1.5) [below] {$i^+$};
    \node at (0.75,1.5) [] {$\bullet$};
    \node at (0.75,1.5) [below] {$i^-$};
    \end{tikzpicture}\hspace{2cm}
    \begin{tikzpicture}[scale=1.5]
    \draw (0,0) -- (1,0.5);
    \draw (0,0.5) -- (1,1);
    \draw (0,1) -- (1,1.5);
    \draw (0,1.5) -- (1,0);
    \node at (0.1,1.05) [] {$\bullet$};
    \node at (0.1,1.05) [below] {$i^-$};
    \node at (0.1,1.33) [] {$\bullet$};
    \node at (0.1,1.33) [above] {$i^+$};
    \end{tikzpicture}
    \caption{On the left, the placement of the base points $i^+,i^-$ when $f(i-1)=i$. On the right, the placement of basepoints for $l(\beta_i)>0$.}
    \label{fig:base-pts-flags}
\end{figure}
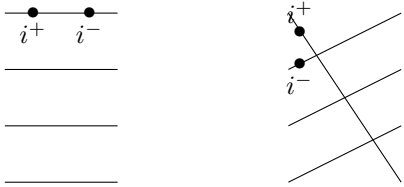

\begin{defn}
By definition, two decorated flags $\tilde{\cF}, \tilde{\cF'}$ with $p(\tilde{\cF})=p(\tilde{\cF}')$  are said to be $\tau_i$-scaled, $i\in[m]$, if all their decorations coincide except for $(v_{\tilde{\cF}})_i$ and $(v_{\tilde{\cF'}})_i$, which are different. That is, they are $\tau_i$-scaled if $(v_{\tilde{\cF}})_k=(v_{\tilde{\cF'}})_k$ for all $k\in[m]\setminus\{i\}$ and $(v_{\tilde{\cF}})_i\neq (v_{\tilde{\cF'}})_i$.\hfill$\Box$
\end{defn}
Intuitively, the decorated version of $\modsp(\beta_\cP)$ has changes in decoration when passing through base points in $\tt$ or crossings and the two changes in decoration passing through each pair $i^+$ and $i^-$ are themselves paired. The description reads as follows:

\begin{defn}[Decorated moduli space]\label{defn:framed-arrangement}
Let $\cP$ be a positroid of type $(m,n)$, $\beta_\cP$ the positroid braid and $\mathfrak{t}_\cP$ its set of base points. The decorated flag moduli space for $(\beta_\cP,\mathfrak{t}_\cP)$ is
$$\modsp(\beta_\cP;\mathfrak{t}_\cP):=\left.\left\{(F_0,F_1,\dots, F_n)\mbox{ where } F_i=(\cF_i^{(1)},\cF_i^{(2)},\cF_i^{(3)})\in (\widetilde{\Fl}_m)^3 \mbox{ s.t. } F_0=F_n\mbox{ and }(*)\right\}\right/\GL_m,$$
where $(*)$ is the following set of three conditions, which must hold for all $i\in[n]$:
\begin{enumerate}
    \item $p(\cF_i^{(1)})=p(\cF_i^{(2)})=p(\cF_i^{(3)})$ and $\cF^{(3)}_{i-1}\doublearrow{\beta_i}\cF^{(1)}_i$,\\

    \item If $l(\beta_i)>0$, then $$\displaystyle\frac{\left(v_{\cF^{(2)}_i}\right)_m}{\left(v_{\cF^{(1)}_i}\right)_m}=(-1)^{l(\beta_i)}\cdot\frac{\left(v_{\cF^{(2)}_i}\right)_{m-1}}{\left(v_{\cF^{(3)}_i}\right)_{m-1}},$$
    $\cF_i^{(1)}$ and $\cF_i^{(2)}$ are $\tau_m$-scaled, and $\cF_i^{(2)}$ and $\cF_i^{(3)}$ are $\tau_{m-1}$-scaled.\\

    \item If $l(\beta_i)=0$, then $\cF^{(1)}_i=\cF^{(3)}_i$ and $\cF_i^{(1)}$ and $\cF_i^{(2)}$ are $\tau_m$-scaled.\\
\end{enumerate}
As above, $\modsp(\beta_\cP;\mathfrak{t}_\cP)$ is considered as an algebraic stack.\footnote{In fact, it follows from Subsection \ref{ssec:iso_flagspositroid} below that it is smooth affine variety.}
\hfill$\Box$
\end{defn}

\noindent Theorem \ref{thm: frame sheaf moduli = positroid} below proves that $\modsp(\beta_\cP;\mathfrak{t}_\cP)$  is isomorphic to the positroid stratum for $\cP$, and it satisfies particularly good properties with regards to cluster structures. By construction, there is a forgetful map $p:\mathfrak{M}(\beta_\cP;\mathfrak{t})\longrightarrow\modsp(\beta_\cP)$. For context, this map coincides with the cluster-theoretic $p$-map $p:\cA\rightarrow  \cX^\uf$, once the appropriate cluster structures are introduced, cf.~\cite{CW}.\\

\begin{rmk}\label{rmk:visual_flags}
Finally, here is a visual description of the points in the moduli $\modsp(\beta_\cP)$,  following \cite[Section 4]{CW} that we use in some of the upcoming arguments. Consider the braid diagram for the positive braid word $\beta_\cP$,  drawn in the plane. Let $R_{j+1}\sse\bR^2$ be the open vertical strip between the $j$th and $(j+1)$st crossings, bounded to the right and to the left by a vertical line through the corresponding crossing. By definition, $R_0$ and $R_{\ell(\beta_\cP)+1}$ are the leftmost and rightmost regions of the braid diagram, before the first and after the last crossing, respectively. A point in $\modsp(\beta_\cP)$ can be described by assigning a complete flag $\cF_j$ in $\bC^m$ to each region $R_j$ such that:

\begin{itemize}
    \item[(i)] Two adjacent flags $\cF_j,\cF_{j+1}$ are in position $s_{k}$ if the $j$th crossing of $\beta_\cP$ is $s_{k}$, i.e.~ if the crossing between the regions $R_j$ and $R_{j+1}$ is $s_k$, then $\cF_j,\cF_{j+1}$ are in position $s_{k}$.\\

    \item[(ii)] The  flag in the leftmost region coincides with the flag in the rightmost region.
\end{itemize}

\noindent Points in the moduli $\modsp(\beta_\cP;\mathfrak{t}_\cP)$  can be similarly described, where now we consider vertical strip regions bounded by vertical lines which go through either a crossing or a base point. The $j$th triple of flags in Definition \ref{defn:framed-arrangement} encodes precisely the three flags in the three  following regions:  just to the left of the base point $i^+$,  between the base points $i^+$ and $i^-$,  and between $i^-$  and the first crossing of $\beta_{i_j}$.\hfill$\Box$
\end{rmk}

\begin{rmk}
Confer \cite[Section 4]{CW} for a geometric justification of condition (*) in Definition \ref{defn:framed-arrangement}. The first arXiv version of this manuscript also contains a description of these moduli spaces in terms of non-reduced wiring diagrams and chambers. See Remarks 6.1, 6.4 and Definitions 6.11 and 6.13 therein.\hfill$\Box$
\end{rmk}


\subsection{Invariance of flag moduli spaces}\label{subsec:invariance} Consider the positive braid word $\beta_\cP$. In \cite[Section 2.2]{CasalsNg} it is explained how to associate a Legendrian link to a positive braid word. This  Legendrian link is referred to as the $(-1)$-closure of the braid.

\begin{defn}\label{defn:positroid Legendrian link} Let $\cP$ be a positroid. By definition, the positroid Legendrian link $\Lambda_\cP\sse(\bR^3,\xi_{st})$ is the $(-1)$-closure of the positive braid word $\beta_\cP$. By definition, the positroid decoration of $\Lambda_\cP$ is the set of base points in $\Lambda_\cP$ given by  the decoration $\mathfrak{t}_\cP$ of $\beta_\cP$ from  Definition \ref{def:base-pts}.\hfill$\Box$
\end{defn}

\noindent Then the flag moduli space $\modsp(\beta_\cP)$ is isomorphic to the Toen-Vaquie moduli of pseudo-perfect objects of a certain decorated version of the dg-category of constructible sheaves on $\bR^2$ singularly supported on $\Lambda_\cP$. This category is of finite type and, in this case, this moduli derived stack is isomorphic to an affine variety modulo a $\GL_m$-action. These matters are discussed in \cite[Section 2.8]{CW}. In particular, since \cite[Section 3.3]{Sheaves1} shows that this category is invariant under Legendrian isotopies, the $D^-$-stack $\modsp(\beta_\cP)$ is also a Legendrian invariant. In fact, a Legendrian isotopy induces a specific functor that gives the required equivalence.

It thus follows that the algebraic isomorphism type of $\modsp(\beta_\cP)$ is invariant under braid moves in $\beta_\cP$, which can also be verified directly, cf.~\cite{CGSS20}. Similarly, $\modsp(\beta_\cP;\mathfrak{t}_\cP)$ is a contact isotropy invariant of the Legendrian link $\Lambda_\cP$ endowed with the set of marked points $\mathfrak{t}_\cP$.  Because of that, in a slight abuse of notation, we often denote $\modsp(\beta_\cP)$ and $\modsp(\beta_\cP;\mathfrak{t}_\cP)$  by $\modsp(\Lambda_\cP)$ and $\modsp(\Lambda_\cP;\mathfrak{t})$  respectively.


\subsection{\texorpdfstring{$\mathfrak{M}(\Lambda_\cP;\mathfrak{t})$ in terms of local systems}{}}\label{subsec:local systems}  Let us provide an alternative description of $\mathfrak{M}(\Lambda_\cP;\mathfrak{t})$  which is useful to us in Sections \ref{subsec:phi map} and \ref{sec:twist}. This viewpoint is closer to the treatment in \cite[Section 2.8]{CW}. It serves as a midpoint between the microlocal sheaf-theoretical perspective, from \cite{CW,CZ}, and the Lie-theoretic flag descriptions, from \cite{CG22,CGSS20,CGGLSS}, presented above.\\

First, we claim that each point $\cF_\bullet$ in these moduli defines a $\GL_1(\bC)$-local system $\mathcal{L}(\cF_\bullet)$ on the Legendrian link $\Lambda_\cP$. In the decorated case that local system comes with certain trivializations. The statement is:

\begin{constr}\label{lem:locsys}
Let $\cP$ be a positroid. Then:
\begin{itemize}
    \item[(i)] Every tuple of flag $\cF_\bullet\in\modsp(\beta_\cP)$ defines a $\GL_1(\bC)$-local system $\mathcal{L}(\cF_\bullet)$ on the Legendrian link $\Lambda_\cP$.\\

    \item[(ii)] Every tuple of decorated flags $\cF_\bullet\in\modsp(\beta_\cP;\mathfrak{t}_\cP)$ defines a $\GL_1(\bC)$-local system $\mathcal{L}(\cF_\bullet)$ on the Legendrian link $\Lambda_\cP$  endowed with a trivialization of $\mathcal{L}(\cF_\bullet)$  along each of the segments that constitute $\Lambda_\cP\setminus\mathfrak{t}_\cP$.
\end{itemize}

\end{constr}

\begin{proof}[Procedure] 
Let us start with Part $(i)$. This construction is done through the braid diagram of $\beta_\cP$ and  we use the description of points $\cF_\bullet\in\modsp(\beta_\cP)$ in Remark \ref{rmk:visual_flags}. Consider a point in a strand of the braid diagram for $\beta_\cP$ which is not a crossing, and so lies in a vertical strip labeled by a flag $\cF$.
Denote by $(\cF)_k$ and $(\cF)_{k-1}$ the vector spaces in the flag $\cF$ right above and below this point, respectively. These are $k$-dimensional  and $(k-1)$-dimensional respectively for some $k\in[m]$.

We define the fiber of the local system $\mathcal{L}(\cF_\bullet)$ to be the group linear automorphisms of the 1-dimensional quotient space $(\cF)_k/(\cF)_{k-1}$. By the transversality conditions in Definition \ref{defn:undecorated sheaf moduli space}, the isomorphisms $\rho$ and $\lambda$ in Equations \eqref{eq:rho} and \eqref{eq:lambda} can then be used to glue these fibers, uniquely up to homotopy. The $(-1)$-closure  simply identifies the rightmost and leftmost pieces of the braid diagram for $\beta_\cP$ via the identity. The result is a local system $\mathcal{L}(\cF_\bullet)$ on $\Lambda_\cP$: this follows from \cite[Section 2.8.1]{CW} and \cite[Section B.2]{CL22}, as the  construction  coincides with the microlocal monodromy functor, cf.~\cite[Example 5.3]{CL22}.\\

For Part $(ii)$ it suffices to note  that a decoration on a flag $\cF$ is precisely the data of a set of vector generators, one for each of its  associated quotient spaces $\cF_k/\cF_{k-1}$. Along each segment in $\Lambda_\cP\setminus\mathfrak{t}_\cP$, and using the  decorations of the flags in $\cF_\bullet$, this is equivalent to the data of a trivialization  of $\mathcal{L}(\cF_\bullet)$.\end{proof}

\noindent From this viewpoint, to enhance a point in $\mathfrak{M}(\Lambda_\cP)$ to a point in $\mathfrak{M}(\Lambda_\cP;\mathfrak{t})$, one chooses a particular trivialization of the local system $\mathcal{L}(\cF_\bullet)$. One may choose a trivialization for each segment of the braid diagram $\beta_\cP$ with the crossings removed. 
Condition (*) in Definition \ref{defn:framed-arrangement} then requires that these trivializations are compatible over crossings and base points.\footnote{The construction of $\rho$ and $\lambda$ in \cite[Section 4]{CW} was as parallel transport across crossings, leading to Condition (*.1).} For base points, ``compatibility"  means the following.\\ 

Recall that the base points in $\mathfrak{t}=\mathfrak{t}_\cP$ come in pairs of the form $i^\pm$. Let $L^-$ and $R^-$, resp.~$L^+$ and $R^+$, be the segments of the braid diagram for $\beta_\cP$ to the left and right of $i^-$, resp.~$i^+$. Let $\psi_{i^-}$, resp.~ $\psi_{i^+}$, be the parallel transport of one such local system $\mathcal{L}(\cF_\bullet)$ along a path from a point in $L^-$, resp.~$R^+$, to a point in $R^-$, resp.~$L^+$, through the base point $i^-$, resp.~$i^+$. Given  a decorated flag $\widetilde\cF$, we have trivializations $v_{L^-}$ and $v_{R^-}$, resp.~$v_{L^+}$ and $v_{R^+}$, along the corresponding segments. We  then define the following two functions:
\begin{equation}\label{eq: A_i}
A_{i^-}(\widetilde\cF):=\frac{\psi_-(v_{L^-})}{v_{R^-}} \quad \quad \text{and} \quad \quad A_{i^+}(\widetilde\cF):=\frac{\psi_+(v_{R^+})}{v_{L^+}}.
\end{equation}
Here we used the shorthand $\frac{v}{w}$ to denote the constant of proportionality between parallel vectors $v,w$. Note that $A_{i^-}$ and $A_{i^+}$ are $\bC^\times$-valued regular functions on $\mathfrak{M}(\Lambda_\cP;\mathfrak{t})$. Conditions (*.2) and (*.3) in Definition \ref{defn:framed-arrangement} are equivalent to setting $A_{i^-}=A_{i^+}$. Thus, the moduli space $\mathfrak{M}(\Lambda;\mathfrak{t})$ can be alternatively defined as: 

\begin{defn}\label{defn: framed sheaf moduli space in terms of local systems} The decorated flag moduli space $\mathfrak{M}(\Lambda_\cP;\mathfrak{t})$ is the moduli space of pairs consisting of a point $\cF_\bullet\in \mathfrak{M}(\Lambda_\cP)$ and a trivialization of the local system $\mathcal{L}(\cF_\bullet)$ along $\Lambda_\cP\setminus \mathfrak{t}$ so that $A_{i^-}=A_{i^+}$ for all $i$.\hfill$\Box$
\end{defn}

\noindent Due to the last condition in Definition \ref{defn: framed sheaf moduli space in terms of local systems}, we often omit the $\pm$ sign and denote both $A_{i^-}$ and $A_{i^+}$ by $A_i$.\\

Second, this description  in terms of local systems  leads to one more version of a flag moduli space, denoted by $\dM(\Lambda_\cP;\mathfrak{t})$. It has a less significant role than $\mathfrak{M}(\Lambda_\cP;\tt)$  in our manuscript, but it can nevertheless helpful to consider it. Its definition is as follows:

\begin{defn}\label{def:X-with-frozens} The \emph{clipped flag moduli space} $\dM(\Lambda_\cP;\mathfrak{t})$ is defined to be the moduli space  of pairs consisting of a point $\cF_\bullet\in\modsp(\Lambda_\cP)$ and an isomorphism $\phi_i:\mathcal{L}_{i^+}(\cF_\bullet)\rightarrow \mathcal{L}_{i^-}(\cF_\bullet)$ between the fibers of the local system $\mathcal{L}(\cF_\bullet)$ for each pair of base points $i^\pm$.\hfill$\Box$
\end{defn}

\noindent For readers that are familiar with cluster ensembles, $\dM(\Lambda_\cP;\mathfrak{t})$ corresponds to the full cluster Poisson $\mathcal{X}$-scheme, with frozens.\footnote{The adjectives {\it decorated}, {\it framed} and {\it pinned} have already been used for different variations of these flag moduli spaces. Thus our choice of {\it clipped} to refer to this particular variation.} In particular, $\dM(\Lambda_\cP;\mathfrak{t})$ is dual to the cluster $\mathcal{A}$-scheme $\mathfrak{M}(\Lambda_\cP;\mathfrak{t})$. The cluster Poisson $\mathcal{X}$-scheme without frozens is the (undecorated) flag moduli $\mathfrak{M}(\Lambda_\cP)$. See \cite[Section 2.8]{CW} or \cite[Section 2]{FockGoncharov_ensemble} for more details on such spaces. The cluster $\mathcal{A}$-scheme without frozens is not discussed in this article and we assign no particular terminology to it.


\subsection{An isomorphism \texorpdfstring{$\Pio_\cP \to \modsp(\Lambda_\cP;\mathfrak{t})$}{}}\label{ssec:iso_flagspositroid} This section presents an explicit isomorphism between the decorated flag moduli space $\frmodsp(\Lambda_\cP;\mathfrak{t})$ and the positroid stratum $\Pi^\circ_\cP$. We represent a point in $\Pio_\cP$ by an $m \times n$ matrix with columns $v_1, \dots, v_n$ and write $[i_1 i_2 \dots i_k]$ for the span of $v_{i_1}, v_{i_2}, \dots, v_{i_k}$. If the $i$th entry of the Grassmann necklace is $\tI{i}=\{i_1 <_i i_2 <_i \dots <_i i_m\}$, we denote $D_i:= v_{i_1} \wedge v_{i_2} \wedge \cdots \wedge v_{i_m}.$ Note that $D_i= \pm \Delta_{\tI{i}}$ is a Pl\"ucker function and it is non-vanishing on  the positroid stratum $\Pio_\cP$. 

\begin{rmk}
    We assume that none of $v_1, \dots, v_n$ are the zero vector, that is, the bounded affine permutation $f$ of $\cP$ has no fixed points. This is without loss of generality. Indeed, deleting a zero column gives a natural isomorphism $\Pio_{\cP}\xrightarrow{\sim}\Pio_{\cP'}$ and $(\Lambda_\cP; \tt) = (\Lambda_{\cP'}; \tt)$, so the flag moduli spaces are also equal. Furthermore, the isomorphism $\Pio_{\cP}\xrightarrow{\sim}\Pio_{\cP'}$ commutes with the twist and its pullback takes target seeds $\Sigma_T(\bG')$ to target seeds $\Sigma_T(\bG)$. Thus, we can and do assume $f$ has no fixed points in the definitions and proofs for the remainder of the paper, but all results hold for arbitrary $f$ and positroid $\cP$.\hfill$\Box$
\end{rmk}

Consider a point $V \in \Pio_\cP$ and the braid diagram for $\beta_\cP$, where we implicitly identify the leftmost and rightmost  ends of the strands via the identity. Let us now construct a tuple of flags $\cF_\bullet(V)\in \modsp(\beta_\cP)$ and  the necessary set of trivializations $\tau(V)$ for the local system $\mathcal{L}(\cF_\bullet(V))$:

\begin{enumerate}
    \item First, label the arc between $i^+$ and $\pi_f(i-1)^-$ in the braid diagram $\beta_\cP \setminus \tt$ with the vector $v_{\pi_f(i-1)}$ from $V \in \Pio_\cP$. Similarly, we label the arc between $i^-$ and $(i+1)^+$ with $D_i^{-1}v_{i}$. See Figure~\ref{fig:trivialization}, where such labels are depicted. This procedure labels each arc of $\beta_\cP$ with a vector that depends on the point $V \in \Pio_\cP$.\\
    
    \noindent These labels define a point in $\cF_\bullet(V)\in \modsp(\beta_\cP)$ as follows. For each region between two consecutive horizontal arcs, at levels $k$ and $(k+1)$ for some $k\in[m]$, we declare  $\cF(V)^k$ to be the subspace of $\bC^m$ spanned by the vectors labeling  all the strands passing exactly below the region. Here $\cF(V)^k$ labels the restriction of  the tuple $\cF_\bullet(V)$ to this region, which is itself a  subspace for one of its flags. Lemma \ref{lem: boundary frozen A coord}  will momentarily prove that this is well-defined, i.e.~that this gives a chain of flags satisfying the necessary transversality conditions.\\

    \item Second, following the description of $\modsp(\beta_\cP)$ in Subsection \ref{subsec:local systems}, the required trivialization of the local system $\mathcal{L}(\cF_\bullet(V))$ is given by projecting the vector labeling each arc, which lives in $\cF_k\sse\bC^m$ for some $k$, into the quotient line $\cF_k/\cF_{k-1}$ associated to the arc.
\end{enumerate}

\begin{defn}\label{def:phi} Let $\cP$ be a positroid and $V \in \Pio_\cP$. We define the map
    $$\Phi: \Pio_\cP \longrightarrow \modsp(\Lambda_\cP;\tt),\quad V\mapsto\Phi(V):=(\cF_\bullet(V),\tau(V)),$$
    where $\cF_\bullet(V)$ and $\tau(V)$ are constructed as above.\hfill$\Box$
\end{defn}

\begin{figure}[H]
    \centering
    \begin{tikzpicture}
    \draw [lime,line width=2ex] (-1,3) -- (0,3);
    \draw [lime, line width=2ex] (0,2.5) -- (2,3);
    \draw [yellow, line width=2ex] (-1,2.5) -- (0,2.5);
    \draw [yellow, line width=2ex] (0,3) -- (2,1);
    \draw [yellow, line width=2ex] (-1,2) -- (0,2) -- (2,2.5);
    \draw [yellow, line width=2ex] (-1,1) -- (0,1) -- (2,1.5);
    \draw [yellow, line width=2ex] (-1,0.5) -- (2,0.5);
    \draw [yellow, line width=2ex] (-1,-0.5) -- (2,-0.5);
    \draw (-1,-0.5) -- (2,-0.5);
    \draw (-1,0.5) -- (2,0.5);
    \draw (-1,3) -- (0,3) -- (2,1);
    \draw (-1,2.5) -- (0,2.5) -- (2,3);
    \draw (-1,2) -- (0,2) -- (2,2.5);
    \draw (-1,1) -- (0,1) -- (2,1.5);
    \node at (0,1.5) [] {$\vdots$};
    \node at (0,0) [] {$\vdots$};
    \node at (0,3) [] {$\bullet$};
    \node at (0,2.5) [] {$\bullet$};
    \node at (0,3) [above] {\footnotesize{$i^+$}};
    \node at (0,2.5) [below] {\footnotesize{$i^-$}};
    \node [teal] at (-1,3) [above left] {$D_{i-1}^{-1}v_{i-1}$};
    \node [teal] at (2,3) [above right] {$D_{i}^{-1}v_{i}$};
    \node [orange] at (-1,2.5) [left] {$v_i$};
    \node [orange] at (2,1) [right] {$v_{\pi_f(i-1)}$};
    \node [orange] at (-1,2) [left] {$v_{i_3}$};
    \node [orange] at (-1,-0.5) [left] {$v_{i_m}$};
    \node [orange] at (-1,1.75) [left] {$\vdots$};
    \node [orange] at (-1,0.25) [left] {$\vdots$};
    \end{tikzpicture}\hspace{2cm}
    \begin{tikzpicture}
    \foreach \i in {0,1,2,5,6}
    {
    \draw [yellow, line width=2ex] (-1,-0.5+\i*0.5) -- (2,-0.5+\i*0.5);
    \draw (-1,-0.5+\i*0.5) -- (2,-0.5+\i*0.5);
    }
    \draw [yellow, line width=2ex] (0,3) -- (1,3);
    \draw [lime, line width=2ex] (-1,3) -- (0,3);
    \draw [lime, line width=2ex] (1,3) -- (2,3);
    \draw (-1,3) -- (2,3);
    \node at (0,3) [] {$\bullet$};
    \node at (1,3) [] {$\bullet$};
    \node at (0,3) [above] {$i^+$};
    \node at (1,3) [above] {$i^-$};
    \node [teal] at (-1,3)[left] {$D_{i-1}^{-1}v_{i-1}$};
    \node [teal] at (2,3) [right] {$D_i^{-1}v_i$};
    \node [orange] at (0.5,3) [above] {$v_i$};
    \foreach \i in {2,3}
    {
    \node [orange] at (-1,3.5-\i*0.5) [left] {$v_{i_\i}$};
    }
    \foreach \i in {1,2}
    {
    \node [orange] at (-1,\i*0.5-0.5) [left] {$v_{i_{m-\i}}$};
    }
    \node [orange] at (-1,-0.5) [left] {$v_{i_m}$};
    \node [orange] at (-1,1.25) [left] {$\vdots$};
    \node at (0.5,1.25) [] {$\vdots$};
    \end{tikzpicture}
    \caption{Trivializations defined by a point $[v_1,v_2,\dots, v_n]$ in the positroid cell. 
    }
    \label{fig:trivialization}
\end{figure}
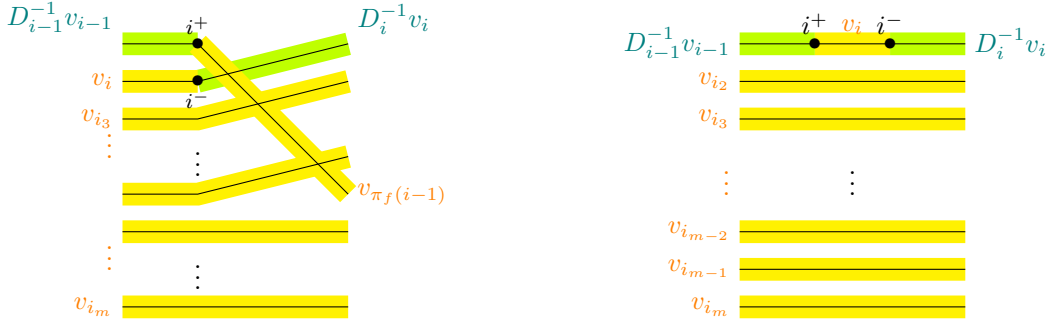

\begin{lem} \label{lem: boundary frozen A coord} 
    Let $\cP$ be a positroid and $V \in \Pio_\cP$. The map $\Phi(V)=(\cF_\bullet(V),\tau(V))$ is well-defined. In particular, we have the following two properties:
    
    \begin{itemize}
        \item[(i)] The  vertical strip in the  braid diagram for $\beta_{\cP}$ between $\beta_{i-1}$ and $\beta_i$ corresponds to the flag     
    \[
\cF[\tI{i-1}]:=0\subset [{i_m}] \subset [{i_{m-1}}{i_m}]\subset \cdots \subset [{i_2}{i_3}\cdots {i_m}]\subset [{i_1}{i_2}\cdots {i_m}]= \mathbb{C}^m,
\]
where $\tI{i-1}=\{i_1 <_{i-1} i_2 <_{i-1} \cdots <_{i-1} i_m\}$.\\

\item[(ii)] The trivializations satisfy $\Phi^*(A_{i^-})=\Phi^*(A_{i^+})=D_i$ for all $i$.
    \end{itemize} 
\end{lem}

\begin{proof}

First, since $\Delta_{\tI{i-1}}$ is non-vanishing on $\Pio_\cP$, the vectors $v_{i_1}, \dots, v_{i_m}$ are linearly independent.  Therefore, the chain of subspaces that these vectors span,  which by construction constitute the flags, have the expected dimension. Thus, the corresponding flag $\cF[\tI{i-1}]$ is a complete flag.\\

\noindent Second, we must now verify that the undecorated flag  remains invariant to the left and to the right of a base point. Note that $v_{i_1}=D_{i-1}^{-1}v_{i-1}$  and that passing through the basepoint $i^+$ changes $v_{i_1}$ to $v_{\pi_f(i-1)}$; sliding over $i^-$ rescales a single vector. Thus, all  the subspaces of the flag are preserved except for possibly the top one. Once we slide past $i^+$, Definition~\ref{def:phi} assigns the top-dimensional space spanned by $[{\pi_f(i-1)}{i_2}\cdots{i_m}]$. Since $\tI{i}= \tI{i-1} \setminus \{i-1\} \cup \{\pi_f(i-1)\}$ and $\Delta_{\tI{i}}$ is non-vanishing on $\Pio_\cP$, we have $[{\pi_f(i-1)}{i_2}\cdots {i_m}]=[{i_1}{i_2},\cdots ,{i_m}]=\bC^m$ as desired. This also implies that  flags separated by a crossing are labeled by different subspaces at that level,  with the right transversality condition, and implies Part $(i)$.\\

\noindent For Part $(ii)$, \cite[Section 4]{CW}  implies that the pull-back of the function $A_{i^-}$ can be computed as:
\[
\Phi^*(A_{i^-})=\frac{v_i\wedge v_{i_3} \wedge \cdots \wedge v_{i_m}}{D_i^{-1}v_i\wedge v_{i_3} \wedge \cdots \wedge v_{i_m}}=D_i.
\]
Similarly, the pull-back of the function $A_{i^+}$ is:
\[
\Phi^*(A_{i^+})=\frac{(-1)^{l(\beta_i)}v_{\pi_f(i-1)}\wedge v_i\wedge v_{i_3}\wedge \cdots \wedge v_{i_m}}{D_{i-1}^{-1}v_{i-1}\wedge v_i\wedge v_{i_3}\wedge \cdots \wedge v_{i_m}}=D_i.
\]
This concludes Part $(ii)$ and proves that $\Phi(V)$ is indeed well-defined.
\end{proof}

In Section \ref{subsec:phi map}, we define a seed $\Sigma(\ww(\bG);\mathfrak{t})$ for each positroid weave $\ww(\bG)$, in line with \cite{CGGLSS,CW}. The following theorem, which we prove at the end of Subsection \ref{subsec: toggles and RIII moves}, shows that the map $\Phi$ in Definition \ref{def:phi} above is an isomorphism and gives the relationship between the two seeds $\Sigma(\ww(\bG);\mathfrak{t})$ and $\Sigma_{T}(\bG)$.

\begin{thm}\label{thm: frame sheaf moduli = positroid}
    The map $\Phi:\Pi_\cP^\circ\longrightarrow \modsp(\Lambda_\cP;\tt)$
is an isomorphism and the pullback $\Phi^*$ sends the seed $\Sigma(\ww(\bG);\mathfrak{t})$ to $\Sigma_{T}(\bG)$ for any reduced plabic graph $\bG$.
\end{thm}

\noindent The fact that $\Phi$ is an isomorphism is interesting, but the main contribution of Theorem \ref{thm: frame sheaf moduli = positroid} is that $\Phi^*$ pulls-back the seed $\Sigma(\ww(\bG);\mathfrak{t})$ to $\Sigma_{T}(\bG)$. Theorem \ref{thm: frame sheaf moduli = positroid} implies that $\modsp(\Lambda_\cP;\tt)$ is a cluster $\mathcal{A}$-scheme with an initial seed given by any positroid weave. Finally, Theorem \ref{thm: frame sheaf moduli = positroid} proves Theorem~\ref{thm:mainA}.(iii) and (iv). Thus, once proven in Subsection \ref{subsec: toggles and RIII moves}, it concludes the proof of Theorem \ref{thm:mainA}.

\begin{rmk} In \cite[Section 3.2]{STWZ}, an isomorphism is given between the postroid stratum $\Pi_\cP^\circ$ and a variant of $\modsp(\Lambda_\cP; \tt)$, in a similar spirit as Definition \ref{def:phi}. Nevertheless, there is no construction of cluster varieties from Legendrian knots in \cite{STWZ}. Rather, for some specific Legendrian links, certain moduli spaces are shown to be isomorphic (just as affine varieties) to positroids, cf.~\cite[Theorem 3.9]{STWZ}. Therefore, the comparison of seeds via the pullback cannot occur in their context, as there is no additional cluster structure and seeds being constructed to which one can compare. Note that \cite{STWZ} predates the symplectic construction of cluster $\mathcal{A}$-structures on $\Pio_\cP$, independent of \cite{GL}, and the weave description of cluster seeds on $\modsp(\Lambda_\cP; \tt)$, which is only established after \cite{CGGLSS,CW}.\hfill$\Box$
\end{rmk}

\section{Cluster seeds from positroid weaves and toggles}\label{subsec:phi map}

The central goal of this section is the proof of Theorem \ref{thm: frame sheaf moduli = positroid}, which concludes the proof of Theorem \ref{thm:mainA}. The method of proof, including the use of toggles in Subsection \ref{subsec: toggles and RIII moves}, might be of independent interest. Indeed, the argument also serves as a partial transition between the geometric techniques developed in \cite{CW}, based on the microlocal theory of sheaves, and the combinatorial results on Grassmann necklaces presented in \cite{FSB}. It also provides an explicit description of the cluster seeds $\Sigma(\ww;\tt)$ associated to the positroid weaves $\ww=\ww(\bG)$ built in Section \ref{sec:iterative_Tmap} above.


\subsection{Construction of the Seed \texorpdfstring{$\Sigma(\ww;\mathfrak{t})$}{}}\label{subsec:cluster cycles} 

In this subsection, we construct the cluster seed $\Sigma(\ww;\tt)$ for the coordinate ring of functions of the affine variety $\modsp(\Lambda_\cP;\tt)$ associated with a positroid weave $\ww:=\ww(\bG)$, following the strategy in \cite[Section 4]{CW}. This is achieved in the following three steps:\\

\begin{enumerate}
    \item Obtaining a description of the particular open {\it toric chart} $(\bC^\times)^{d}\sse\modsp(\Lambda_\cP;\tt)$ underlying the seed $\Sigma(\ww;\tt)$, where $d=\dim_\bC\modsp(\Lambda_\cP;\tt)$. This is achieved in Subsection \ref{sssec:step1}, without describing the quiver or cluster variables at this stage.\\

    \item Constructing the {\it quiver} for the seed $\Sigma(\ww;\tt)$, in Subsection \ref{sssec:step2}, without discussing cluster variables.\\

    \item Constructing the {\it cluster variables} for the seed $\Sigma(\ww;\tt)$, which are irreducible regular functions in $\bC[\modsp(\Lambda_\cP;\tt)]$. This is achieved in Subsection \ref{sssec:step3}.\\
\end{enumerate}
Once these three steps are established, we then prove Proposition~\ref{prop:merodromy-is-plucker}, which shows that the map $\Phi:\Pi_\cP^\circ\to \modsp(\Lambda_\cP;\tt)$ maps the seed $\Sigma_T(\bG)$ to the seed $\Sigma(\ww;\tt)$. This proves part of Theorem~\ref{thm: frame sheaf moduli = positroid}.


\subsubsection{Toric chart for $\Sigma(\ww;\tt)$}\label{sssec:step1}  Consider a positroid weave $\ww$ in the disk $\D^2$. Its intersection with the boundary $\partial\D^2$ cyclically spells a positive braid word for $\beta_\cP$.  The data of the boundary alone, namely $\beta_\cP$, determines the space $\modsp(\beta_\cP;\tt)$. Let $\mathcal{O}p(\partial\D^2)$ be an arbitrarily small but fixed  open neighborhood of the boundary $\partial\D^2$.\\

\noindent Points in the space $\modsp(\beta_\cP;\tt)$ can be seen as the data of  one flag assigned to each region in the complement $\mathcal{O}p(\partial\D^2)\setminus (\mathcal{O}p(\partial\D^2)\cap\ww)$ of the weave in that neighborhood of the boundary, such that two flags in adjacent regions separated by an $i$-colored weave line are $s_i$-transverse  and with the appropriate decorations.  Given the weave $\ww$  in interior of $\D^2$, we can consider the following extension  locus.

\begin{defn}[Toric chart]\label{def:toric_chart}
By definition, $T_\ww\sse\modsp(\beta_\cP;\tt)$ is set of points whose associated collection of flags in $\mathcal{O}p(\partial\D^2)\setminus (\mathcal{O}p(\partial\D^2)\cap\ww)$  extends to a collection of flags in the interior  complement $\D^2\setminus\ww$ while obeying the  following  transversality rule: two flags in adjacent regions of $\D^2\setminus\ww$ separated by an $i$-colored weave line of $\ww$  must be $s_i$-transverse.\hfill$\Box$
\end{defn}

\noindent Importantly, \cite[Section 5]{CZ}  implies the following two facts:
\begin{enumerate}
    \item Such an extension, when it exists, is unique.
    
    \item The locus $T_\ww\sse\modsp(\beta_\cP;\tt)$ is an open toric chart.
\end{enumerate}

\noindent  In particular,  there is an algebraic isomorphism $T_\ww\cong(\bC^*)^d$ where $d:=\mbox{dim}_\bC\modsp(\beta_\cP;\tt)$. This can also be seen by direct computation, cf.~\cite[Section 5]{CGSS20}. Both these facts above use that positroid weaves are Demazure, cf.~Theorem \ref{thm:Demazure weave} above. Furthermore, the restriction map, given by restricting the flag data in $\D^2\setminus\ww$ to the boundary $\mathcal{O}p(\partial\D^2)\setminus (\mathcal{O}p(\partial\D^2)\cap\ww)$, defines an open embedding $T_\ww\longrightarrow \modsp(\beta_\cP;\tt)$.\\

\noindent By definition, the image $T_\ww\sse \modsp(\beta_\cP;\mathfrak{t})$ of this open embedding in $\modsp(\beta_\cP;\mathfrak{t})$ is the toric chart we associate with the cluster seed $\Sigma(\ww;\tt)$. For each such toric locus we will now associate a quiver and cluster variables.

\begin{rmk}
From the viewpoint of symplectic geometry, $T_\ww$ is naturally identified with a certain relative cohomology group of $\La_\ww$  with $\bC^*$-coefficients, where $\La_\ww$  is the surface discussed in Section \ref{ssec:weave_recap}. In line with  Subsection \ref{subsec:invariance}, $T_\ww$ can be intrinsically named is a certain moduli space of sheaves in $\bR^3$ with singular support on a front for $\La_\ww$, cf.~\cite[Section 5.3]{CZ}. This implies  that $T_\ww$  and the open embedding $T_\ww\longrightarrow \modsp(\beta_\cP;\tt)$  are invariant under weave equivalences.\hfill$\Box$
\end{rmk}

\begin{rmk}
For the undecorated version there is an analogous result: by forgetting decorations, the extension locus as in Definition \ref{def:toric_chart} gives an open toric  chart in $\modsp(\beta_\cP)$.  In this case, $T_\ww$ is identified with $H^1(\La_\ww;\bC^*)$.  These facts and the necessary invariance properties also follow from \cite[Section 5]{CZ}.
\hfill$\Box$
\end{rmk}


\subsubsection{Quiver for $\Sigma(\ww;\tt)$}\label{sssec:step2} The quiver $Q(\ww;\tt)$  for the cluster seed associated to $\Sigma(\ww;\tt)$ is built as follows. Given a positroid weave $\ww$, consider the smooth surface $L:=\La_\ww$ introduced in Section \ref{ssec:weave_recap}. Following \cite{FockGoncharov_ensemble}, we construct the quiver associated with $\Sigma(\ww;\tt)$ by choosing  the following data:
\begin{itemize}
    \item[-] A lattice $N$ endowed with a skew-symmetric form $\{\cdot, \cdot\}$.

    \item[-] A basis $B=\{\gamma_i\}$ of the lattice $N$.
\end{itemize}
By definition, the index set of the basis gives the quiver vertices and the exchange matrix of the quiver is given by $\epsilon_{ij}:=\{\gamma_i,\gamma_j\}$. In the case of $Q(\ww;\tt)$, the lattice $N$ will be a sub-lattice $N_\ww$ of the relative homology group $H_1(L,\tt)$. The skew-symmetric form will be the restriction of the intersection form on $H_1(L,\tt)$ to $N$. In order to describe the  lattice and the basis, we introduce the following definition.

\begin{defn}\label{def:cluster-cycles} Let $\bG$ be a plabic graph and $\ww=\ww(\bG)$ its positroid weave. By definition, the cluster cycles on $L=\La_\ww$ are the  following two collections of cycles:
\begin{itemize}
    \item[(i)] The $Y$-trees in Theorem \ref{thm:weave construction}.(1), which correspond to non-exceptional faces of $\bG$.

    \item[(ii)] The relative cycles in $H_1(L,\tt)$ depicted in Figure \ref{fig: exceptional base points}, one for each exceptional boundary face.
\end{itemize}
\noindent We denote the cluster cycle of a face $F$ by $\gamma_F$ and the set of all cluster cycles in $L$ by $B=B_\ww$.\hfill$\Box$\\
\end{defn}

\begin{figure}[H]
    \centering
    \begin{tikzpicture}
        \draw (-1,-0.5) -- (1,-0.5);
        \draw (-1,1) -- (1,1);
        \node at (-0.5,1) [] {$\bullet$};
        \node at (-0.5,1) [above] {$i^+$};
        \node at (0.5,1) [] {$\bullet$};
        \node at (0.5,1) [above] {$i^-$};
    \end{tikzpicture}\hspace{2cm}
    \begin{tikzpicture}
        \draw (-3,0) -- (-1,1) -- (1,0.5) -- (-1.25,-0.75);
        \draw [dashed] (-3,-1) -- (-1,0) -- (0,-0.25);
        \draw (0,-0.25) -- (1,-0.5) -- (0,-1);
        \node at (-0.5,0.875) [] {$\bullet$};
        \node at (-0.5,0.875) [above] {$i^+$};
        \node at (0.5,0.625) [] {$\bullet$};
        \node at (0.5,0.625) [above] {$i^-$};
        \draw [purple, decoration={markings,mark=at position 0.5 with {\arrow{<}}},postaction={decorate}] (-0.5,0.875) to [bend right] (0.5,0.625); 
    \end{tikzpicture}
    \caption{(Left) Base points for an exceptional boundary face. (Right) The corresponding frozen cluster cycle:  the arc depicted in purple is a relative cycle in $H_1(L,\tt)$.}
    \label{fig: exceptional base points}
\end{figure}
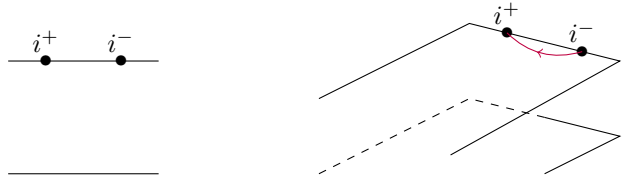

\noindent The endpoints of the $Y$-trees for a non-exceptional boundary face are depicted in Figure \ref{fig: base points}.

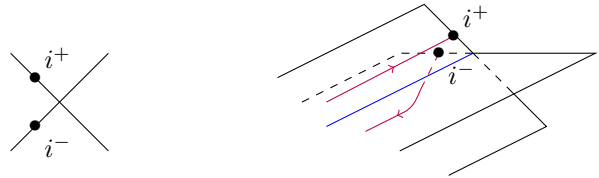
\begin{figure}[H]
    \centering
    \begin{tikzpicture}[baseline=0,scale=0.65]
    \draw (-1,-1) -- (1,1);
    \draw (-1,1) -- (1,-1);
    \node at (-0.5,-0.5) [] {$\bullet$};
    \node at (-0.5,-0.5) [below right] {$i^-$};
    \node at (-0.5,0.5) [above right] {$i^+$};
    \node at (-0.5,0.5) [] {$\bullet$};
    \end{tikzpicture} \hspace{2cm}
    \begin{tikzpicture}[baseline=0,scale=0.65]
    \draw [dashed] (-2.5,0) -- (-0.5,1) -- (1,1);
    \draw (1,1) -- (3.5,1) -- (-0.5,-1);
    \draw (-3,0.5) -- (0,2) -- (1,1);
    \draw [dashed] (1,1) -- (1.8,0.2);
    \draw (1.8,0.2)--(2.5,-0.5) -- (0.5,-1.5);
    \draw [blue] (1,1) -- (-2,-0.5);
    \draw [decoration={markings,mark=at position 0.5 with {\arrow{<}}},postaction={decorate},purple] (0.6,1.35) -- (-2,0);
    \draw [decoration={markings,mark=at position 0.5 with {\arrow{<}}},postaction={decorate},purple] (-1.2,-0.6) -- (-0.4,-0.2) to [out=30,in=-120] (0,0.4);
    \draw [dashed, purple] (0.3,1) -- (0,0.4);
    \node at (0.6, 1.35) [] {$\bullet$};
    \node at (0.6,1.35) [above right] {$i^+$};
    \node at (0.3,1) [] {$\bullet$};
    \node at (0.3,1) [below right] {$i^-$};
    \end{tikzpicture}
    \caption{(Left) Base points next to each crossing in the top row of $\beta_\cP$. (Right) Endpoints of the corresponding frozen cluster cycle,  depicted in purple, corresponding to a non-exceptional boundary face; these are not weave lines. The only weave line is colored blue.}
    \label{fig: base points}
\end{figure}

\begin{defn}[Quiver]\label{defn:quiver}
Let $\bG$ be a plabic graph and $\ww=\ww(\bG)$ its positroid weave. By definition, the lattice $N:=N_\ww$ is the $\mathbb{Z}$-span of the cluster cycles $B$ inside  the lattice $H_1(L,\tt)$. Cluster cycles in $B$ corresponding to non-boundary faces are said to be \emph{mutable}, and those corresponding to boundary faces are said to be \emph{frozen}. By definition, the resulting quiver $Q(\ww;\tt)$ is the quiver for the cluster seed $\Sigma(\ww;\tt)$.\hfill$\Box$
\end{defn}

\noindent Note that the cluster cycles corresponding to interior faces belong to the sub-lattice $H_1(L)\sse H_1(L,\tt)$. Those corresponding to boundary faces belong to $H_1(L,\tt)$ but are not absolute cycles, i.e.~they are not elements in $H_1(L)$. We still need to verify that the cluster cycles are linearly independent,  so that Definition \ref{defn:quiver}  gives the well-defined quiver. This is the content of the following result:

\begin{lem}\label{lem:B-is-basis} Let $\bG$ be a plabic graph and $\ww=\ww(\bG)$ its positroid weave. The mutable cluster cycles in $\ww(\bG)$ associated to the interior faces of $\bG$ form a basis for $H_1(L)$ and $B$ is a linearly independent subset of $H_1(L,\mathfrak{t})$. In particular, $B$ is a basis of $N$.
\end{lem}
\begin{proof} For the first claim, Theorem \ref{thm:weave_and_conjugate} shows that the exact Lagrangian filling $L$ is Hamiltonian isotopic (and thus diffeomorphic) to the conjugate surface of $\bG$, which deformation retracts to $\bG$. Therefore we have $H_1(L)\cong H_1(\bG)$, the latter of which admits a basis given by the boundary of each internal face of $\bG$. Under the Hamiltonian isotopy in the proof of Theorem \ref{thm:weave_and_conjugate}, these face boundaries are homotoped to the corresponding mutable $Y$-trees. Thus, the mutable cluster cycles form a basis of $H_1(L)$.\\

\noindent For the second claim, since $H_1(\mathfrak{t})=0$, the long exact sequence $$\cdots \rightarrow H_1(\mathfrak{t})\rightarrow H_1(L)\rightarrow H_1(L,\mathfrak{t})\rightarrow \cdots$$ implies that $H_1(L)$ injects into $H_1(L,\mathfrak{t})$. This shows that all mutable cluster cycles are linearly independent. Since there is only one frozen cluster cycle attached to any point in $\mathfrak{t}$, it follows that even after including the frozen cluster cycles, the set $B$ consists of a linearly independent collection of cycles.
\end{proof}


\subsubsection{Cluster variables  for $\Sigma(\ww;\tt)$}\label{sssec:step3} In Subsection \ref{sssec:step2},  we described the quiver $Q(\ww;\tt)$ for the cluster seed $\Sigma(\ww;\tt)$ in terms of a basis of a lattice $(N,\{\cdot, \cdot\})$ endowed with a skew-symmetric form. In line with \cite{FockGoncharov_ensemble, GHKK}, we now describe the cluster variables as functions associated with elements in the dual basis $B^\vee$ of the dual lattice $M:=N^*$. For the seed $\Sigma(\ww;\tt)$, we follow  the same strategy as in \cite[Section 3]{CW} and the dual basis is described as follows.\\

Let $\bG$ be a plabic graph, $\ww=\ww(\bG)$ its positroid weave, $L=L_\ww$ the associated surface, and $\Lambda:=\partial L$ the positroid Legendrian link. The lattice $N\sse H_1(L,\tt)$ is the $\mathbb{Z}$-span of the cluster cycles $B$. By Poincar\'{e} duality,  we have an isomorphism
$$H_1(L,\mathfrak{t})^*\cong H_1(L\setminus \mathfrak{t},\Lambda\setminus \mathfrak{t}).$$
This allows us to describe the dual lattice $M=N^*$ as  a certain quotient of $H_1(L\setminus \mathfrak{t},\Lambda\setminus \mathfrak{t})$, and Poincar\'{e} duality descends to a duality between $N$ and $M$.

\begin{defn}[Dual lattice]\label{def:dual lattices} By definition, the lattice $M$ is the quotient of $H_1(L\setminus \mathfrak{t}, \Lambda\setminus \mathfrak{t})$ by $N^\perp$. We denote by $B^\vee$ the basis of $M$ Poincar\'e dual to the basis $B$ of $N$. Elements of $B^\vee$ are referred to as dual cluster cycles. The dual of $\gamma_F\in B$ is denoted by $\eta_F$; we use the same notation for representatives of $\eta_F$.\hfill$\Box$
\end{defn}

By \cite[Section 4]{CW}, the expectation is that cluster variables in $\Sigma(\ww;\tt)$ should be given by microlocal merodromies along dual cluster cycles. In order to define microlocal merodromies, we observe that by equipping $L$ with a compatible collection of sign curves, we obtain a $\GL_1(\bC)$-local system $\mathcal{L}(p)$ on $L$ for each point in the toric chart $p\in T_\ww$, cf.~\cite[Section 4.3]{CW}. (The construction is analogous to Subsection \ref{subsec:local systems}  but applied to surface fronts instead of braid diagrams, see also \cite[Section 7.2.1]{CZ}.) Using these local systems, we define microlocal merodromies as follows.

\begin{defn}[Microlocal merodromies]\label{defn: merodromy} Let $\eta$ be an oriented path from a point of $\Lambda\setminus \mathfrak{t}$ to another point of $\Lambda\setminus \mathfrak{t}$. By definition, the microlocal merodromy $A_\eta:T_\ww\rightarrow \bC^\times$  along $\eta$ is the function
\[
A_\eta(p):=\frac{\psi_\eta(v)}{w},
\]
where $v$ is the trivialization at the start of $\eta$, $w$ is the trivialization at the end of $\eta$, and $\psi_\eta:=\psi_\eta(p)$ denotes the parallel transport in the local system $\mathcal{L}(p)$ along the path $\eta$. For a relative $1$-cycle $\eta\in H_1(L\setminus \mathfrak{t}, \Lambda\setminus \mathfrak{t})$, we represent $\eta$ as a sum of oriented paths and define its microlocal merodromy $A_\eta$ to be the product of all the microlocal merodromies along these oriented paths.\hfill$\Box$
\end{defn}

\noindent Note that $A_\eta=A_{\eta'}$ as functions on $T_\ww$ if and only if $[\eta]=[\eta']$ in the quotient lattice $M$, making $A_\eta$ well-defined on homology cycles.

\begin{defn}[Cluster variables]\label{def:cluster-var-for-weave}
Let $\bG$ be a plabic graph and $\ww(\bG)$ its positroid weave. By definition, the cluster variables of the seed $\Sigma(\ww(\bG);\tt)$ are the microlocal merodromies
    \[A(\ww(\bG);\tt):=\{A_{\eta_F}: F \text{ a face of }\bG\}\]
along the dual cluster cycles $\eta_F\in B^\vee$ associated to the faces $F$ of $\bG$.\hfill$\Box$
\end{defn}

\noindent At this stage, we have constructed a cluster seed in $\modsp(\Lambda_\cP;\mathfrak{t})$ for each positroid weave $\ww(\bG)$:

\begin{defn}[Initial seed]
Let $\bG$ be a plabic graph and $\ww(\bG)$ its positroid weave. By definition, the initial seed $\Sigma(\ww(\bG);\tt)$ in $\bC[\modsp(\Lambda_\cP;\mathfrak{t})]$ is given  by the quiver $Q(\ww(\bG);\tt)$  with cluster variables $A(\ww(\bG);\tt)$.\hfill$\Box$
\end{defn}

\noindent The underlying toric chart for this initial seed is $T_\ww$, i.e.~ the common non-vanishing locus in $\modsp(\Lambda_\cP;\mathfrak{t})$ of all the functions in $A(\ww(\bG);\tt)$ is $T_\ww$. This follows  similarly to the results in \cite[Section 4]{CW}.

\begin{rmk} By Theorem \ref{thm:Demazure weave}, a positroid weave $\ww(\bG)$ is equivalent to a complete Demazure weave. Therefore, one may employ the techniques we developed in \cite{CW,CGGLSS} to show that $\mathcal{O}(\modsp(\Lambda_\cP;\mathfrak{t}))$ is a cluster algebra with initial seed $\Sigma(\ww(\bG);\mathfrak{t})$. The construction above and Proposition~\ref{thm: frame sheaf moduli = positroid} give an alternative explicit proof via the isomorphism $\Phi$, better suited for our proof of Theorem \ref{thm:mainB}. 
\hfill$\Box$
\end{rmk}


\subsection{Dual cluster cycles and their merodromies}\label{ssec:dualcycles} The goal of this subsection is to prove Proposition \ref{prop:merodromy-is-plucker}, showing that $\Phi$ pulls-back microlocal merodromies in $A(\ww(\bG);\tt)$ to the appropriate Pl\"ucker functions.


\subsubsection{Representatives of dual cluster cycles}\label{sssec:representatives_dualcycles}

Let us describe explicit representatives of dual cluster cycles in $B^\vee$. We use these representatives in Proposition~\ref{prop:merodromy-is-plucker} to show that the map $\Phi^*$ sends cluster variables of $\Sigma(\ww(\bG);\tt)$ to cluster variables of $\Sigma_T(\bG)$.\\

\begin{enumerate}
    \item \textbf{Representative for $\eta_F$ for $F$ boundary face:} If $F$ is a boundary face, choose one of the base points at which $\gamma_F$ ends, and draw a little arc near the boundary of $L$ that jumps across the base point. We orient this arc so that it intersects $\gamma_F$ with index $+1$. This oriented arc defines a relative cycle $\eta_F$ and $\eta_F$ is dual to $\gamma_F$.\\

\item \textbf{Representative for $\eta_F$ for $F$ internal face:} If $\gamma_F\in B$ is a mutable cluster cycle, i.e., when $F$ is a non-boundary face, the construction is as follows.\footnote{This construction of $\eta_F$ was proposed to us by L. Shen in a private conversation.}\\

To start, fix a generic point $p$ in the interior of $F$ so that it is not on any weave line of the corresponding positroid weave $\ww$. Since the surface $L$ is an $m:1$ branched cover $\pi:L\rightarrow \bD$ of the disk $\bD$ and $p$ is not in the ramification locus, there are $m$ distinct lifts $p_1, p_2,\dots, p_m$ of $p$ to $L$. By the proof of Theorem \ref{thm:weave_and_conjugate}, there is an isotopy $H_t$ of $(T^*\mathbb{D}^2,\la_{st})$ from the conjugate Lagrangian surface associated to $\bG$, at time $t=0$, to the surface $L$, at time $t=1$. For each point $p_i\in L$, there is an initial time $0<t_i<1$ when $p_i$ appears in the image of $H_{t_i}$ for the first time. At this time $t=t_i$, $p_i$ is necessarily a point at the boundary of the exact Lagrangian surface $L$. Let $\eta_{p_i}$ be the trajectory of this boundary point under the isotopy $H_t$ from $t=t_i$ to $t=1$, understood as a subset of $L$.\\

\begin{figure}[H]
    \centering
    \begin{tikzpicture}[baseline=30]
    \foreach \i in {0,1,2}
    {
    \draw (0,\i) -- (2,\i) -- (3,\i+0.5) -- (1,\i+0.5) -- cycle;
    \node at (1.5,\i+0.25) [] {$\bullet$};
    }
    \draw [decoration={markings,mark=at position 0.5 with {\arrow{>}}},postaction={decorate},purple] (1.5,0.25) -- (0,0);
    \draw [decoration={markings,mark=at position 0.5 with {\arrow{>}}},postaction={decorate},purple] (1.5,1.25) -- (2,1);
    \draw [decoration={markings,mark=at position 0.5 with {\arrow{>}}},postaction={decorate},purple] (1.5,2.25) -- (3,2.5);
    \end{tikzpicture} \hspace{1cm} $\rightsquigarrow$ \hspace{1cm}
    \begin{tikzpicture}[baseline=30]
    \foreach \i in {0,1,2}
    {
    \draw (0,\i) -- (2,\i) -- (3,\i+0.5) -- (1,\i+0.5) -- cycle;
    \node at (1.5,\i+0.25) [] {$\bullet$};
    }
    \draw [decoration={markings,mark=at position 0.5 with {\arrow{<}}},postaction={decorate},purple] (0,1) -- (1.5,1.25);
    \draw [decoration={markings,mark=at position 0.5 with {\arrow{<}}},postaction={decorate},purple] (0,2) -- (1.5,2.25);
    \draw [decoration={markings,mark=at position 0.5 with {\arrow{<}}},postaction={decorate},purple] (1.5,1.25) -- (2,1);
    \draw [decoration={markings,mark=at position 0.5 with {\arrow{<}}},postaction={decorate},purple] (1.5,2.25) -- (3,2.5);
    \end{tikzpicture}
    \caption{Construction of the dual relative cycle $\eta_F$: in this example, we reverse the top two gradient flows and take the complementary lifts of the bottom gradient flow.}
    \label{fig: dual relative cycle}
\end{figure}
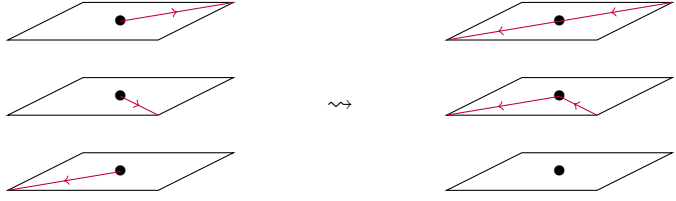

By definition, for each $i\in[m]$, $\eta_{p_i}$ is said to be the gradient flow of $p_i$.
This nomenclature is due to the fact that in the front projection the height function is decreasing along each $\eta_{p_i}$. See Figure \ref{fig: dual relative cycle} (left) for a depiction of some such $\eta_{p_i}$. In order to construct a relative 1-cycle from these gradiant flows we proceed as follows. First, reverse all but one of the $m$ gradient flows $\eta_{p_1}, \eta_{p_2},\dots, \eta_{p_m}$; for the remaining one, say $\eta_{p_j}$, consider the lifts $\pi^{-1}(\pi(\eta_{p_j}))\setminus \eta_{p_j}$. Second, concatenate these lifts with the $m-1$ reversed gradient flows to form a relative $1$-cycle $\eta_F$ on $L$. This is illustrated in Figure \ref{fig: dual relative cycle} (right). Note that such a relative $1$-cycle can be constructed so that it misses the set of base points $\mathfrak{t}$, thus defining a relative $1$-cycle in $H_1(L\setminus \mathfrak{t}, \Lambda\setminus \mathfrak{t})$.\\
\end{enumerate}


\noindent These representatives are used to compute microlocal merodromies in the next subsection. The following  result shows that the relative 1-cycle $\eta_F$ constructed in (2) above is a representative of $\eta_F \in B^\vee$ dual to $\gamma_F$. 

\begin{prop}\label{prop: construction of dual cluster cycles} In the notation above, let $p_j$ be a point where gradient flows are not reversed nor where the generic point $p$ is inside the non-boundary face $F$. Then the equivalence class $[\eta_F]\in M$ is the basis element dual to the mutable cluster cycle $\gamma_F\in N$.
\end{prop}
\begin{proof} It suffices to prove that $\inprod{\eta_F}{\gamma_F}=1$ and $\inprod{\eta_F}{\gamma_{F'}}=0$ for any other basis element $\gamma_{F'}$ in the cluster cycle basis $B$. First, for any generic $1$-chain $\xi$ on $\bD$, its lift $\pi^{-1}(\xi)$ to $L$ has zero intersection with any element in $B$. Indeed, if $F$ is an exceptional boundary face, then we can shrink $\gamma_F$ to an arc that misses $\pi^{-1}(\xi)$; if $F$ is not an exceptional boundary face, then generically all intersections between $\pi^{-1}(\xi)$ with $\gamma_F$ must occur at pre-images of weave lines, and such intersections must come in canceling pairs. See Figure \ref{fig: canceling pairs}. Therefore, the intersection number between $\eta_F$ and any cluster cycle $\gamma_{F'}$ is the same as the intersection number between $-\sum_{i=1}^m \eta_{p_i}$ and $\gamma_{F'}$.

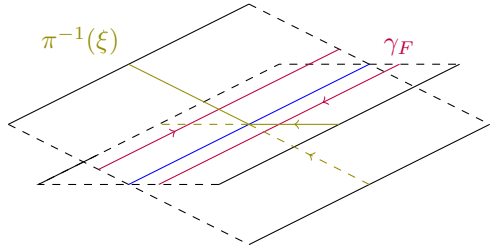
\begin{figure}[H]
    \centering
    \begin{tikzpicture}[baseline=0,scale=0.8]
    \draw [dashed] (3,0) -- (0,0) -- (4,2) -- (7,2);
    \draw (7,2) -- (3,0);
    \draw [dashed] (-0.5,1) -- (3.5,-1);
    \draw [dashed] (3.5,3) -- (7.5,1);
    \draw (-0.5,1) -- (3.5,3);
    \draw (3.5,-1) -- (7.5,1);
    \draw (0,0) -- (1,0.5);
    \draw [blue] (1.5,0) -- (5.5,2);
    \draw [purple, decoration={markings,mark=at position 0.7 with {\arrow{<}}},postaction={decorate}] (2,0) -- (6,2)  node [above] {$\gamma_F$};
    \draw [purple, decoration={markings,mark=at position 0.7 with {\arrow{<}}},postaction={decorate}] (5,2.25) -- (1,0.25);
    \draw [olive, dashed, decoration={markings,mark=at position 0.5 with {\arrow{>}}},postaction={decorate}] (5.5,0) -- (3.5,1);
    \draw [olive] (3.5,1) -- (1.5,2) node [above left] {$\pi^{-1}(\xi)$};
    \draw [olive, decoration={markings,mark=at position 0.5 with {\arrow{>}}},postaction={decorate}] (5,1) -- (3.5,1);
    \draw [olive,dashed] (3.5,1) -- (2,1);
    \end{tikzpicture}
    \caption{Intersections between $\pi^{-1}(\xi)$, in olive, and $\gamma_F$, in purple, always come in canceling pairs. The only weave line is colored in blue.}
    \label{fig: canceling pairs}
\end{figure}

\noindent Second, Theorem \ref{thm:weave construction}  implies that all cluster cycles are located on the top weave layer, and the height function decreases along every gradient flow $\eta_{p_i}$. Thus, the only gradient flow that can possibly cross a top layer weave line (which is the singular locus between the top two sheets in the front projection) is the one that starts at the top sheet, and this happens at most once. Denote this gradient flow by $\eta_{p_1}$.

It remains to argue that this crossing always happens with the cluster cycle $\gamma_F$. For this, we note that in the proof of Theorem \ref{thm:weave_and_conjugate}, the top lift of any generic point in the interior of the non-boundary face $F$ is swept up by one of the zig-zags adjacent to $F$, and this zig-zag then continues and eventually hits the top weave layer inside $F$. See Figure \ref{fig:ConjugateSurface2}. This shows that $-\eta_{p_1}$ only intersects $\gamma_F$ and does so exactly once.
\end{proof}

\begin{rmk} After fixing representatives $\eta_F$'s for the dual cluster cycles, one can follow \cite[Section 4.9]{CW} and perform cluster $\mathcal{A}$-mutations topologically. This is particularly useful, for example, when mutating at a non-square face in a reduced plabic graph (Figure \ref{fig:ConjugateSurface10}): the recipe in ibid.~ gives a new relative cycle whose microlocal merodromy is the mutated cluster $\mathcal{A}$-variable.\hfill$\Box$
\end{rmk}


\subsubsection{Computing microlocal merodromies} Let us use the representatives $\eta_F$ for the dual cluster cycles constructed in Subsection \ref{sssec:representatives_dualcycles} above to compute the pull-backs of the microlocal merodromies $A_{\eta_F}$ under the map $\Phi: \Pio_\cP \longrightarrow \modsp(\Lambda_\cP;\tt)$ constructed in Subection \ref{ssec:iso_flagspositroid}. The result reads as follows:

\begin{prop}\label{prop:merodromy-is-plucker} Let $\cP$  be a positroid, $\bG$ the associated plabic graph, and $\ww$ the corresponding positroid weave.  Then the map $\Phi: \Pio_\cP \longrightarrow \modsp(\Lambda_\cP;\tt)$ satisfies $\Phi^*(A_{\eta_F})=\Delta_{\overrightarrow{I_F}}$.
\end{prop}
\begin{proof} 
We prove the equality up to sign, that is $\Phi^*(A_{\eta_F})= \pm \Delta_{\overrightarrow{I_F}}$, for any choice of sign curves, cf.~\cite[Section 4.5]{CW} for details on sign curves. We then choose the canonical set of sign curves and conclude that the sign is $+1$, cf.~Remark \ref{rmk:sign_curves} below. Now, if $F$ is a boundary face containing the basepoints $i^\pm$, then $\eta_F$ is homologous to the paths we used to compute the functions $A_{i^-}=A_{i^+}=A_i$ \eqref{eq: A_i}. This implies that $\Phi^*(A_{\eta_F})= \Phi^*(\pm A_i)= \pm D_i = \pm \Delta_{\tI{i}}$.\\

If $F$ is not a boundary face, we proceed as follows. Let $p$ be a generic point inside the face $F$, $p_1,p_2,\dots, p_m$ be its preimages in $L_\ww$ and $\eta_{p_1},\dots, \eta_{p_m}$ the corresponding gradient flows. Let $q_j$ be the endpoint in $\eta_{p_j}$ at the boundary of the weave. By parallel transporting the boundary trivializations at each $q_j$ along each of these gradient flows $\eta_{p_j}$, we obtain a collection of vectors $w_j$: one for each fiber of local system on $L_\ww$ above $p_j$. By construction, we can lift each of these $w_j$ to a vector $v_j\in\bC^m$,  where $\bC^m$ is the top dimensional space in the flag in $\ww$ over $p$. Each of these vectors $v_j$ is then a column vector in the matrix representative $M$ of a point in $\Pi_\cP^\circ$.
Based on the construction of the trivialization data, the index of each such column vector $v_j$ is exactly the target of the zig-zag that sweeps up $p_i$ for each $i$. These zig-zags are precisely those for which the non-boundary face $F$ is on the left; see also Proposition \ref{prop:face label after a single T-shift}. Since the vectors $v_j$ form a flag, their wedge is non-zero. By following the proof of Proposition \ref{prop: construction of dual cluster cycles}, the microlocal merodromy $A_{\eta_F}$ is, up to sign, the ratio between the wedge of these $v_j$'s with the unique top determinant form. This ratio is precisely $\pm \Delta_{\overrightarrow{I_F}}$, since $I_F$, records the indices of all these $v_j$'s. Thus we conclude $\Phi^*(A_{\eta_F})= \pm \Delta_{\overrightarrow{I_F}}$ for non-boundary faces.
\end{proof}

\begin{rmk}\label{rmk:sign_curves}
As a technical aspect, we note that the signs of merodromies $A_{\eta_F}$ depend on a choice of sign curves, cf.~\cite[Section 4.5]{CW} and Appendix A.2 in the first arXiv version of this manuscript. In the above proposition, we implicitly consider the choice of sign curves in Corollary A.26 of Appendix A.2 in that version, which is referred to as the canonical set of sign curves.\hfill$\Box$
\end{rmk}

\subsection{Toggles and Reidemeister III moves}\label{subsec: toggles and RIII moves}
In this section, we define \emph{toggle} isomorphisms
$$\tog: \modsp(\beta_\cP; \tt) \to \modsp(\beta'; \tt).$$
These are particular isomorphisms induced by a sequence of Reidemeister III moves applied to $\beta_\cP$,  resulting in an equivalent positive braid word $\beta'$. These isomorphisms $\tog$ are crucial in the proof of Theorem \ref{thm:mainB}, showing that the twist map is DT. They are also used in Proof of Theorem~\ref{thm: frame sheaf moduli = positroid} in the next subsection. In course, we explain the connection between these isomorphisms and the toggles on Grassmann necklaces introduced in \cite{FSB}. We explicitly compute $\tog \circ \Phi$ in Proposition~\ref{prop:tog-circ-phi} and use it to prove that $\Phi:\Pio_\cP\rightarrow\modsp(\Lambda_\cP;\tt)$ is an isomorphism. An added corollary of these computations is that $\modsp(\Lambda_\cP)$ is an $\cX$-scheme with no frozens, cf.~Corollary~\ref{cor:flag-unfrozen-X} below.\\

Let us begin with the construction of these toggle isomorphisms. Recall the notion of periodic grid patterns and their chords from Definition~\ref{def:grid-pattern}.

\begin{defn}[Toggles]\label{def:toggle}
Given the periodic grid pattern for $\beta_\cP$, the leftward and rightward toggles at $i \in [n]$ are the sequences of Reidemeister III moves shown in Figure \ref{fig:toggles}, applied periodically. In Figure \ref{fig:toggles}, the vertical slice at $x=i$ crosses horizontal segments at heights $\tI{i}$. 

    \begin{figure}[H]
   \begin{center}  
    \begin{tikzpicture}[scale=0.5, baseline=15]
        \draw (0,6.2) node [left] {$\vdots$} -- (3,6.2);
        \draw (0,5.8) -- (3,5.8);
        \draw (0,5) node [left] {$a$} -- (1,5) -- (1,-1) -- (3,-1); 
        \draw (0,4.2) node [left] {$\vdots$} -- (3,4.2);
        \draw (0,3.8) -- (3,3.8);
        \draw (0,3) node [left] {$c$} -- (2,3) -- (2,1) -- (3,1);
        \draw (0,2.2) node [left] {$\vdots$}-- (3,2.2);
        \draw (0,1.8) -- (3,1.8);
        \node at (0,1) [left] {$d$};
        \draw (0,0.2) node [left] {$\vdots$} -- (3,0.2);
        \draw (0,-0.2) -- (3,-0.2);
        \node at (0,-1) [left] {$b$};
        \draw (0,-1.8) node [left] {$\vdots$} -- (3,-1.8);
        \draw (0,-2.2) -- (3,-2.2);
        \node at (1.5,-3) [] {\footnotesize $ (i-1) ~ i ~ (i+1)$};
    \end{tikzpicture} \quad \quad $\begin{array}{c}\overset{\text{rightward toggle}}{\longleftarrow} \\ \underset{\text{leftward toggle}}{\longrightarrow}\end{array}$ \quad \quad
    \begin{tikzpicture}[scale=0.5, baseline=15]
    \draw (0,6.2) node [left] {$\vdots$} -- (3,6.2);
        \draw (0,5.8) -- (3,5.8);
        \draw (0,5) node [left] {$a$} -- (2,5) -- (2,-1) -- (3,-1); 
        \draw (0,4.2) node [left] {$\vdots$} -- (3,4.2);
        \draw (0,3.8) -- (3,3.8);
        \draw (0,3) node [left] {$c$} -- (1,3) -- (1,1) -- (3,1);
        \draw (0,2.2) node [left] {$\vdots$}-- (3,2.2);
        \draw (0,1.8) -- (3,1.8);
        \node at (0,1) [left] {$d$};
        \draw (0,0.2) node [left] {$\vdots$} -- (3,0.2);
        \draw (0,-0.2) -- (3,-0.2);
        \node at (0,-1) [left] {$b$};
        \draw (0,-1.8) node [left] {$\vdots$} -- (3,-1.8);
        \draw (0,-2.2) -- (3,-2.2);
        \node at (1.5,-3) [] {\footnotesize $ (i-1) ~ i ~ (i+1)$};
    \end{tikzpicture}
\end{center}
    \caption{Leftward and rightward toggles.}\label{fig:toggles}
\end{figure}
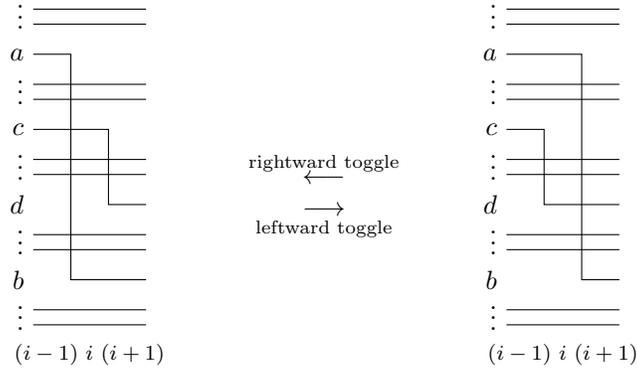

\noindent More formally, a leftward toggle at $i$ moves the $(c,d)$-chord at $x=i+0.5$ below and to the left of the $(a,b)$-chord at $x=i-0.5$, where $a<c<d<b$. A rightward toggle does the reverse. By definition, the result of applying a sequence of toggles to the grid pattern for $\beta_\cP$ is called a periodic grid pattern for $\cP$. If two periodic grid patterns are related by a toggle at $i$, we say the corresponding braid words $\beta, \beta'$ are related by a toggle at $i$, which at the level of braid words is a sequence of commutation moves $s_j s_k \to s_k s_j$ and braid moves $s_k s_j s_k \to s_j s_k s_j$. If $\beta_\cP, \beta'$ are related by a sequence of toggles, then they represent the same positive braid and thus, by the invariance in Subsection \ref{subsec:invariance}, there is a canonical isomorphism
\[\tog: \modsp(\beta_\cP) \to \modsp(\beta'),\]
which we refer to as a \emph{toggle} isomorphism.
\hfill$\Box$
\end{defn}

\noindent 

\begin{defn}[Toggles with basepoints]\label{def:toggle-with-basepoints}
The periodic grid pattern of $(\beta_{\cP};\tt)$ is the periodic grid pattern of $\beta_\cP$ with basepoints $\tt$ added periodically. If $l(\beta_i)>0$, we place $i^+$ in the closest $\begin{tikzpicture}[baseline=7]\draw (0,0.5) -- (0.5,0.5) -- (0.5,0);\end{tikzpicture}$ corner, which is at height $i-1$. The basepoint $i^-$ is free to move anywhere between the adjacent $+$ basepoints $(f^{-1}(i)+1)^+$ and $(i+1)^+$. We extend Definition~\ref{def:toggle} to grid patterns and braid words with basepoints, with the convention that $i^+$ remains in the same corner after a toggle and $i^-$ remains free to move. If $(\beta_\cP;\tt), (\beta';\tt)$ are related by a sequence of toggles, then the invariance in Subsection \ref{subsec:invariance} yields a canonical isomorphism
\[\tog: \modsp(\beta_\cP;\tt) \to \modsp(\beta';\tt),\]
which we refer to as a \emph{toggle} isomorphism. If trivializations for $\cF_\bullet \in \modsp(\beta_\cP;\tt)$ are given by projecting a vector $v \in \bC^m$ into the appropriate 1-dimensional vector space, so are trivializations for $\tog(\cF_\bullet)$.\hfill$\Box$
\end{defn}


\subsubsection{Toggles on necklaces}
The map $\Phi: \Pio_{\cP} \rightarrow \modsp(\beta_{\cP}; \tt)$ is defined using the target Grassmann necklace. In order to compute $\tog \circ \Phi$ explicitly, we need to use more general necklaces associated to the positroid $\cP$ \cite[Definition 3.5]{FSB} and define toggles on necklaces. For notational convenience, we write the target Grassmann necklace $\tNec_\cP = (\tI{1}, \dots, \tI{n})$ as 
\[\tI{n} \stacknumber{f(0)}{0} \tI{1} \stacknumber{f(1)}{1} \tI{2} \stacknumber{f(2)}{2} \dots \stacknumber{f({n-1})}{{n-1}} \tI{n}\stacknumber{f({n})}{{n}}. \]
Reducing modulo $n$, the ``stack" $\stacknumber{f(j)}{j}$ records which elements are swapped going from one set to the other. Let us briefly recall toggle moves and use them to recursively define necklaces.\footnote{In \cite{FSB}, these are called \emph{aligned toggles} and necklaces as defined here are called ``Grassmannlike necklaces with $\iota \leq_{\circ} \pi$''.}

\begin{defn} Necklaces with permutation $f$ are recursively defined as follows. The target necklace $\tNec_{\cP}$ is a necklace with permutation $f$, and so are all lists obtained from it by the  following leftward toggle operation. Let $\cI= (I_1, \dots, I_n)$ be a necklace with permutation $f$ and let $a\in [n]$ be an index such that
    \[\cI=\cdots I_{a-1}\stacknumber{f(x)}{x}I_a\stacknumber{f(y)}{y}I_{a+1}\cdots \]
    with $x < y$ and $f(x)>f(y)$. By definition, a \emph{leftward toggle} at $a$ produces
    \[\cI':=\cdots I_{a-1}\stacknumber{f(y)}{y} I_a' \stacknumber{f(x)}{x}I_{a+1}\cdots,\]
    another necklace with permutation $f$.
    More explicitly, all subsets in $\cI$ and $\cI'$ agree except the one indexed by $a$, which in  $\cI$ is $I_a := I_{a-1} \setminus \{x\} \cup \{f(x)\}$ and in $\cI'$ is $I_a' := I_{a-1} \setminus \{y\} \cup \{f(y)\}$ (modulo $n$). The inverse of a leftward toggle at $a$ is a  said to be \emph{rightward toggle} at $a$, which takes 
      \[\cI=\cdots I_{a-1}\stacknumber{f(u)}{u}I_a\stacknumber{f(w)}{w}I_{a+1}\cdots \quad \text{to} \quad \cI':=\cdots I_{a-1}\stacknumber{f(w)}{w} I_a' \stacknumber{f(u)}{u}I_{a+1}\cdots\]
    if $u>w$ and $f(u)<f(w)$.\hfill$\Box$
\end{defn}

\noindent See Example~\ref{exmp:toggling-seq} below for an instance of a toggling sequence going through several necklaces. Note that the permutation of a necklace can be computed by reading the ``stacks'', or equivalently computing which elements are swapped when going from $I_a$ to $I_{a+1}$. 

\begin{rmk}\label{rmk:necklace-poset}
    There is a natural poset structure on necklaces with permutation $f$, where $\cI \geq \cI'$ if $\cI'$ can be obtained from $\cI$ by a sequence of leftward toggles. Alternately, this poset structure can be phrased in terms of the weak order on bounded affine permutations, as follows. A necklace $\cI$ with permutation $f$ encodes another bounded affine permutation $\iota$ (the ``insertion permutation"), defined by
    \[\cI = I_n \stacknumber{\iota(0)}{j_0} I_1 \stacknumber{\iota(1)}{j_1} I_2 \stacknumber{\iota(2)}{j_2} \dots \stacknumber{\iota(n-1)}{j_{n-1}}I_n.\]
    By definition, performing a leftward toggle on $\cI$ right-multiplies $\iota$ by a length-decreasing simple transposition. By \cite[Lemma 4.13]{FSB}, in fact $\cI \geq \cI'$ if and only if $\iota \geq \iota'$ in the right weak order on $\bdmn$. The unique maximal element in this poset is the target necklace $\tNec$, which has insertion permutation $f$; the unique minimal element is a rotation of the source necklace $\sNec$, which has insertion permutation $i \mapsto i+k$. See also \cite[Figure 4]{FSB}.\hfill$\Box$
\end{rmk}

\begin{exmp} \label{exmp:toggling-seq}  Here is a leftward toggling sequence that turns the target Grassmann necklace associated with the positroid in Example \ref{exmp:positroid braid} into the corresponding source Grassmann necklace.
{\small 
  \setlength{\abovedisplayskip}{6pt}
  \setlength{\belowdisplayskip}{\abovedisplayskip}
  \setlength{\abovedisplayshortskip}{0pt}
  \setlength{\belowdisplayshortskip}{3pt}
\begin{align*}
        &7124\stacknumber{6}{0}1246\stacknumber{3}{1}2346\stacknumber{9}{2}3462\stacknumber{8}{3}4612\stacknumber{7}{4}6712\stacknumber{5}{5}6712\stacknumber{11}{6}7124 \\ 
        \sim \ & 
         7124\stacknumber{3}{1}{\color{red}7234}\stacknumber{6}{0}2346\stacknumber{9}{2}3462\stacknumber{8}{3}4612\stacknumber{7}{4}6712\stacknumber{5}{5}6712\stacknumber{11}{6}7124\\
        \sim \ &  
        {\color{red}6723}\stacknumber{4}{-1}7234\stacknumber{6}{0}2346\stacknumber{9}{2}3462\stacknumber{8}{3}4612\stacknumber{7}{4}6712\stacknumber{5}{5}6712\stacknumber{10}{8}{\color{red}6723}\\
        \sim \ &
            {6723}\stacknumber{4}{-1}7234\stacknumber{6}{0}2346\stacknumber{8}{3}{\color{red}2461}\stacknumber{9}{2}4612\stacknumber{7}{4}6712\stacknumber{5}{5}6712\stacknumber{10}{8}{6723}\\
        \sim \ & 
        {6723}\stacknumber{4}{-1}7234\stacknumber{6}{0}2346\stacknumber{8}{3}{2461}\stacknumber{7}{4}{\color{red}2671}\stacknumber{9}{2}6712\stacknumber{5}{5}6712\stacknumber{10}{8}{6723}\\
        \sim \ &
        {6723}\stacknumber{4}{-1}7234\stacknumber{6}{0}2346\stacknumber{7}{4}{\color{red}2367}\stacknumber{8}{3}{2671}\stacknumber{9}{2}6712\stacknumber{5}{5}6712\stacknumber{10}{8}{6723}\\
        \sim \ & 
        {6723}\stacknumber{4}{-1}7234\stacknumber{6}{0}2346\stacknumber{7}{4}{2367}\stacknumber{8}{3}{2671}\stacknumber{5}{5}{\color{red}2671}\stacknumber{9}{2}6712\stacknumber{10}{8}{6723}\\
        \sim  \ &
        {6723}\stacknumber{4}{-1}7234\stacknumber{6}{0}2346\stacknumber{7}{4}{2367}\stacknumber{5}{5} {\color{red}2367} \stacknumber{8}{3}{2671}\stacknumber{9}{2}6712\stacknumber{10}{8}{6723}\\
        \sim \ & 
        {6723}\stacknumber{4}{-1}7234\stacknumber{6}{0}2346\stacknumber{5}{5}{\color{red} 2346}\stacknumber{7}{4} {2367} \stacknumber{8}{3}{2671}\stacknumber{9}{2}6712\stacknumber{10}{8}{6723}\\       
        \sim \ &
        {6723}\stacknumber{4}{-1}7234\stacknumber{5}{5}{\color{red}7234}\stacknumber{13}{7}{2346}\stacknumber{7}{4} {2367} \stacknumber{8}{3}{2671}\stacknumber{9}{2}6712\stacknumber{10}{8}{6723}
    \end{align*}

}
\noindent In the sequence above, we have colored each new necklace entry in red.\hfill$\Box$
\end{exmp}


\subsubsection{Toggles and the map $\Phi$} In this subsection we explicitly compute the composition $\tog \circ  \Phi$, which we use when showing that $\Phi$ is an isomorphism.

\begin{prop}\label{prop:tog-circ-phi}
    Consider the braid word $\beta$ obtained from $\beta_\cP$ by a sequence of leftward toggles, and $\cI$ the necklace obtained from $\tNec_\cP$ by the same sequence of toggles. Then the image of $V \in \Pio_\cP$ under the map
    \[ \tog \circ \Phi: \Pio_\cP \stackrel{\Phi}{\rightarrow} \modsp(\beta_\cP;\tt) \stackrel{\tog}{\rightarrow} \modsp(\beta;\tt)\]
    is given by the construction of Definition~\ref{def:phi} using $\beta$ rather than $\beta_\cP$.\\ 
    
    \noindent In particular, the slice of the grid pattern of $\beta$ at $x=i$ corresponds to the flag 
        \[\cF[I_{i}]=0\subset [{i_m}] \subset [{i_{m-1}}{i_m}]\subset \cdots \subset [{i_2}{i_3}\cdots {i_m}]\subset [{i_1}{i_2}\cdots {i_m}]=\mathbb{C}^m\]
        where $I_{i}=\{i_1 <_{i} \dots <_{i} i_m\}$ is the $i$th entry of $\cI$.
\end{prop}

\noindent Before proving Proposition~\ref{prop:tog-circ-phi}, we need the following preparatory lemma.

\begin{lem}\label{lem:left span =  right span} Consider the following local configuration in a periodic grid pattern for $\cP$,
\begin{center}
    \begin{tikzpicture}[scale=0.7]
        \draw (0,4) node [left] {$i$} -- (1,4) -- (1,0) -- (2,0);
        \draw (0,3.2) node [left] {$j_1$} -- (2,3.2);
        \draw (0,2.4) node [left] {$j_2$} -- (2,2.4);
        \draw (0,0.8)  node [left] {$j_s$} -- (2,0.8);
        \node at (0,1.7) [left] {$\vdots$};
        \node at (0,0) [left] {$k$};
    \end{tikzpicture}
\end{center} 
where segments are labeled by their heights.
Then $[i{j_1}{j_2}\cdots {j_s}]=[{j_1}{j_2}\cdots {j_s} k]$ and it is $(s+1)$-dimensional. 
\end{lem}
\begin{proof} 
For convenience we set $v_{b+an}= v_b$. First, we suppose that the periodic grid pattern is for $\beta_\cP$. By Lemma~\ref{lem: boundary frozen A coord}, $\dim [ij_1j_2\cdots j_s]=\dim [j_1j_2\cdots j_sk]=s+1$. Furthermore, by Definition~\ref{def:perm of positroid}, $v_i \in [j_1j_2\cdots j_sk]$ but $v_k \notin [j_1j_2\cdots j_s]$. Therefore, when expanding $v_i$ as a linear combination of $v_{j_1},v_{j_2},\dots, v_k$, the coefficient of $v_k$ is non-zero. Thus $v_k$ is a linear combination of $v_i,v_{j_1},v_{j_2},\dots, v_{j_s}$ as well, or in other words $v_k\in [ij_1j_2\cdots j_s]$. This proves $[ij_1j_2\cdots j_s]=[j_1j_2\cdots j_sk]$ and shows the lemma holds for the periodic grid pattern for $\beta_{\cP}$.\\

Now we induct on $k-i$. Since there are $s$ horizontal lines present between levels $i$ and $k$, we must have $k-i \geq s$, so the base case is $k-i=s$. For this base case, the $(i,k)$-chord must cross the horizontal lines at heights $j_1,j_2,\dots, j_s$ in all periodic grid patterns for $\cP$, and does not cross any other horizontal lines. This is because toggles do not change the number of horizontal segments a chord crosses. Thus this local configuration appears in the grid pattern for $\beta_\cP$ and we can conclude by the argument above.\\

\noindent Inductively, let us assume the statement is correct for all $j$ with $f(j)-j < k-i$. Note that a toggle in which the $(i,k)$-chord is the short chord does not change the horizontal segements the $(i,k)$-chord intersects. A rightward toggle in which the $(i,k)$-chord is the long chord will replace some $j_{\ell}=f(a)$ with $a$ and we will have $i<a<j_{\ell}<k$. Since $j_{\ell}-a<k-i$, by the inductive hypothesis we have 
\[
[ij_1\cdots j_s]=[ij_1\cdots a \cdots j_s] \quad \quad \text{and} \quad \quad [j_1 \cdots j_sk]=[j_1\cdots a \cdots j_sk].
\]
\[
\begin{tikzpicture}[scale=0.7, baseline=20]
    \draw (0,3) node [left] {$i$} -- (2,3) -- (2,0) -- (3,0);
    \draw (0,2) node [left] {$a$} -- (1,2) -- (1,1) -- (3,1);
    \node at (0,1) [left] {$j_{\ell}$};
    \node at (0,0) [left] {$k$};
\end{tikzpicture}
\quad \quad \longleftrightarrow \quad \quad 
\begin{tikzpicture}[scale=0.7, baseline=20]
    \draw (0,3) node [left] {$i$} -- (1,3) -- (1,0) -- (3,0);
    \draw (0,2) node [left] {$a$} -- (2,2) -- (2,1) -- (3,1);
    \node at (0,1) [left] {$j_{\ell}$};
    \node at (0,0) [left] {$k$};
\end{tikzpicture}
\]
Apply a sequence of rightward toggles to the periodic grid pattern to get the periodic grid pattern for $\beta_\cP$. The horizontal segments crossing the $(i,k)$-chord are $h_1, \dots, h_s$ and by repeatedly applying the argument above, 
\[
[ij_1\cdots j_s]=[ih_1\cdots h_s] \quad \quad \text{and} \quad \quad [j_1 \cdots j_sk]=[h_1 \cdots h_sk].
\]
Applying the lemma for the periodic grid pattern for $\beta_\cP$ gives the desired conclusion.
\end{proof}

\begin{proof}[Proof of Proposition~\ref{prop:tog-circ-phi}]
Fix $V\in \Pio_\cP$. Since the trivializations for $(\tog \circ \Phi)(V)$ and for Definition~\ref{def:phi} using $\beta$ are given by projecting the same vectors $v_i$ into the lines associated to arcs, it suffices to show the proposition at the level of flags rather than decorated flags. That is, we will show that the construction in Subsection \ref{ssec:iso_flagspositroid} applied to $\beta$ and $(\tog \circ \Phi)(V)$ give the same point in the flag moduli.\\

First, we note that by \cite[Theorem 4.14]{FSB}, the Pl\"ucker coordinate $\Delta_{i_1, \dots, i_m}$ is nonzero on $\Pio_\cP$ and so the vectors $v_{i_1},v_{i_2},\cdots ,v_{i_m}$ are linearly independent. Thus $\cF[I_{i}]$ defined above is a complete flag in $\bC^m$. Regions between horizontal strands, a.k.a.~chambers, in the vertical slices $i-0.5 <x< i+0.5$ of the periodic grid pattern are labeled by the subspaces in this flag; since there are no $j^+$ basepoints in this range, sliding a vertical slice past a basepoint will just rescale a vector. Lemma~\ref{lem:left span =  right span} implies that for any such chamber region for $\beta$ above a $j^+$ basepoint, the labeling subspace can be defined using either $v_j$ or $v_{f(j)}$. Therefore, the construction in Subsection \ref{ssec:iso_flagspositroid} does define a point in the required moduli and vertical slices correspond to $\cF[I_{i}]$. That is, applying Definition~\ref{def:phi} using $\beta$ gives the point $(\cF[I_{n}], \cF[I_{1}], \dots, \cF[I_{n}]) \in \modsp(\beta)$.\\

Second, we now show that the image of $V$ under the map 
 \[ \tog \circ \Phi: \Pio_\cP \to \modsp(\beta_\cP) \to \modsp(\beta)\] is this point. We argue one leftward toggle at a time. Suppose the statement is true for a sequence of leftward toggles which gives the braid word $\beta'$ and necklace $\cI'$. Suppose  that we apply another leftward toggle, moving chords at $x=a-0.5$ and $x=a+0.5$, gives $\beta$. The toggle induces an isomorphism $\modsp(\beta') \to \modsp(\beta)$ in which a $\beta'$-chain of flags is sent to the unique $\beta$-chain of flags where the subspace labeling any chamber not strictly contained between the chord at $x=a-0.5$ and the chord at $x=a+0.5$ is unchanged. Since $I_i' = I_i$ for $i \neq a$, we have that $\cF[I_i'] = \cF[I_i]$ for $i \neq a$, and thus the point $(\cF[I_n], \cF[I_1], \cF[I_2], \dots, \cF[I_{n-1}], \cF[I_n])$ corresponds to this unique $\beta$-chain of flags.
\end{proof}


\subsection{Proof of Theorem~\ref{thm: frame sheaf moduli = positroid}} We are ready to prove that $\Phi$ is an isomorphism and maps the initial cluster seed $\Sigma_T(\bG)$ to $\Sigma(\ww;\tt)$, where $\ww$ is the positroid weave of $\bG$. As a result, this proves that $\modsp(\Lambda_\cP;\tt)$ is a cluster $\mathcal{A}$-variety. This completes the proof of Theorem~\ref{thm:mainA}.(iii) and (iv). Let us prove Theorem~\ref{thm: frame sheaf moduli = positroid}.\\

In order to show that $\Phi$ is an isomorphism, we construct an inverse map $\Psi$, as follows. Let $\tilde{\cF}_\bullet \in \modsp(\Lambda_\cP;\tt)$, we construct an $m \times n$ matrix $\Psi(\tilde{\cF}_\bullet)$ as follows. Choose a sequence of leftward toggles which takes $\tNec_\cP$ to $\sNec_\cP$. This is possible by Remark~\ref{rmk:necklace-poset}. Applying the same sequence of toggles to the periodic grid pattern of $\beta_\cP$ results in the periodic grid pattern of $\delta_\cP$, which corresponds to the source necklace; this changes $\tilde{\cF}_\bullet$ to $\tog(\tilde{\cF}_\bullet)$. An important feature of the periodic grid pattern of $\delta_\cP$ is that all $\begin{tikzpicture}[baseline=0]\draw (0, 0.25) -- (0,0) -- (0.25,0);\end{tikzpicture}$ corners appear along the lower left boundary.
 Now, we move the $i^-$ basepoint slightly to the right of the $\begin{tikzpicture}[baseline=0]\draw (0, 0.25) -- (0,0) -- (0.25,0);\end{tikzpicture}$ corner at height $y=i$ on the lower left boundary. We set $x_i$ to be the vector labeling the segment $S_i$ containing this corner, which is the segment to the left of $i^-$. Since this segment is at the bottom level of the wiring diagram, $x_i$ is an element of a 1-dimensional subspace of $\bC^m$, rather than a 1-dimensional quotient $(\cF_a)_{j+1}/(\cF_a)_j$. Therefore $x_i \in \bC^m$. By construction,  we declare the matrix $\Psi(\tilde{\cF}_\bullet)$ to have columns $x_1, \dots, x_n$.\\

\noindent The flag corresponding to the vertical slice passing through $S_i$ can be written as 
$$[x_i] \subset [x_i, x_{i_2}] \subset \cdots \subset [x_i, x_{i_2}, \dots, x_{i_{m-1}}] \subset [x_i, x_{i_2}, \dots, x_{i_{m-1}}, x_{i_m}]= \bC^m$$
where $\{i, i_2, \dots, i_m\}=\sI{i}$, the $i$th entry of the source necklace. In particular, the Plucker coordinate $\Delta_{\sI{i}}$ is non-vanishing on $\Psi(\tilde{\cF}_\bullet)$ for all $i$. Moreover, from the source grid pattern $\delta_\cP$, it follows that if $\sI{i} = \{i_m, i_{m-1}, \dots, i_1=i\}$
and $\pi^{-1}(i)=j$ with $i_m <_{i+1} \cdots <_{i+1} i_{k+1} <_{i+1} j <_{i+1} i_k <_{i+1} \cdots <_{i+1} i_1$, then $[x_{i_1} \dots x_{i_k}] = [x_{i_2} \dots x_{i_k} x_j]$. This is enough information to reconstruct the cyclic rank matrix and thus $\Psi(\tilde{\cF}_\bullet)$ is an element of $\Pio_\cP$.\footnote{See also \cite[Theorem 3.9]{STWZ} and its proof where this argument is elaborated.} Thus $\Psi$ is a well-defined map from $\modsp(\beta_{\cP}; \tt)$ to $\Pio_\cP$.\\

\noindent The trivialization of $\Phi(V)$ for the arc between $(\pi^{-1}(i+1))^+$ and $i^-$ is given by projecting $v_i$ into the appropriate 1-dimensional quotient, cf.~Definition \ref{def:phi}. Proposition~\ref{prop:tog-circ-phi} says that the same is true for $(\tog \circ \Phi) (V) \in \modsp(\delta_\cP;\tt)$. Thus the construction of the vector $x_i$ above exactly recovers $v_i$ and $(\Psi\circ \Phi)(V)=V$. If we have a configuration of flags $\tilde{\cF}_\bullet$, $\Psi(\tilde{\cF}_\bullet)$ gives us a sequence of vectors that can be used to recover the flags $\tilde{\cF}_\bullet$ using the map $\Phi$, and thus $(\Phi \circ \Psi)(\tilde{\cF}_\bullet)= \tilde{\cF}_\bullet$ as well.\\

Finally, Proposition~\ref{prop:merodromy-is-plucker} implies that $\Phi^*$ sends the cluster variables of $\Sigma(\ww(\bG); \tt)$ to the cluster variables of $\Sigma_T(\bG)$. Theorem~\ref{thm:weave construction}, together with the fact that each quiver has an isolated frozen vertex corresponding to each exceptional boundary face, implies that this identification of cluster variables induces a quiver isomorphism.\hfill$\Box$ 

\begin{exmp}\label{exmp:phi-inv} We illustrate the proof of Theorem~\ref{thm: frame sheaf moduli = positroid} using our running example from Example~\ref{exmp:positroid braid}. In the Figure~\ref{fig:phi-inv}, we show how to perform toggles to move between the periodic grid patterns for the target and source necklaces. We have also shown a possible configuration for the base points, keeping the $i^+$ basepoints near the $\begin{tikzpicture}[baseline=7]\draw (0,0.5) -- (0.5,0.5) -- (0.5,0);\end{tikzpicture}$ corners, and the $i^-$ basepoints near the $\begin{tikzpicture}[baseline=7]\draw (0,0.5) -- (0,0) -- (0.5,0);\end{tikzpicture}$ corners. Note that the trivializations by the vectors $v_i$ occur at the $\begin{tikzpicture}[baseline=7]\draw (0,0.5) -- (0,0) -- (0.5,0);\end{tikzpicture}$ corners, in particular, the corners between the basepoints $(\pi^{-1}(i)+1)^+$ and $i^-$. Note that we can read off the vectors $v_i$ from the $\begin{tikzpicture}[baseline=7]\draw (0,0.5) -- (0,0) -- (0.5,0);\end{tikzpicture}$ corners in the rightmost diagram.\hfill$\Box$

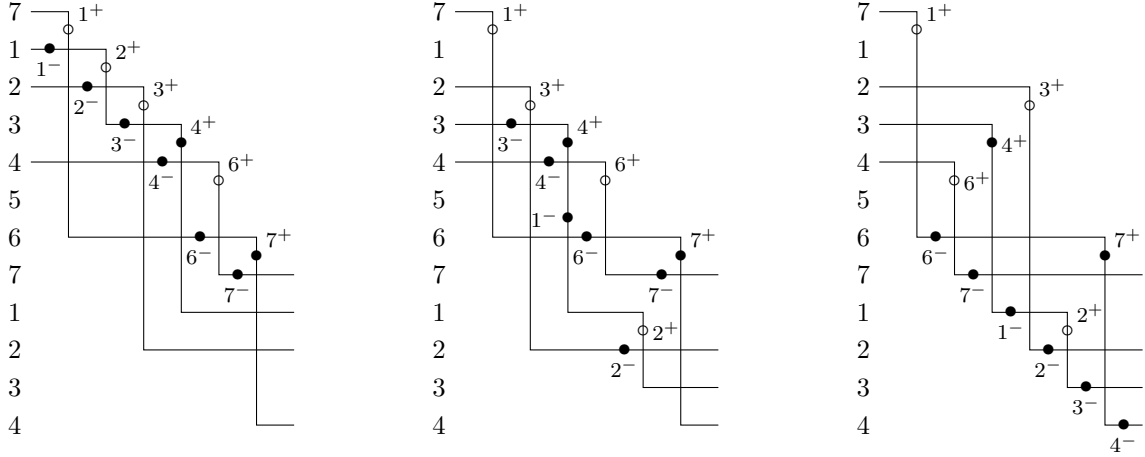
\begin{figure}[H]
    \centering
    \begin{tikzpicture}[scale=0.5,baseline=0]
        \draw (0,11) node [left] {$7$} --(1,11)--(1,5) -- (6,5) -- (6,0)  -- (7,0); 
        \draw (0,10) node [left] {$1$} -- (2,10) -- (2,8) -- (4,8) -- (4,3) -- (7,3);
        \draw (0,9) node [left] {$2$} -- (3,9) -- (3,2) -- (7,2);
        \draw (0,7) node [left] {$4$} -- (5,7) -- (5,4) -- (7,4);
        \foreach \i in {3,5,6,7}
        {
            \node at (0,11-\i) [left] {$\i$};
        }
        \foreach \i in {1,2,3,4}
        {
            \node at (0,4-\i) [left] {$\i$};
        }
        \foreach \i in {1,2,3}
        {
            \node at (\i, 11.5-\i) [] {$\circ$};
            \node at (\i, 11.5-\i) [above right] {\footnotesize{$\i^+$}};
            \node at (\i-0.5,11-\i) [] {$\bullet$};
            \node at (\i-0.5,11-\i) [below] {\footnotesize{$\i^-$}};
        }
        \node at (4,7.5) [] {$\bullet$};
        \node at (4,7.5) [above right] {\footnotesize{$4^+$}};
        \node at (3.5,7) [] {$\bullet$};
        \node at (3.5,7) [below] {\footnotesize{$4^-$}};
        \node at (5,6.5) [] {$\circ$};
        \node at (5,6.5) [above right] {\footnotesize{$6^+$}};
        \node at (4.5,5) [] {$\bullet$};
        \node at (4.5,5) [below] {\footnotesize{$6^-$}};
        \node at (6,4.5) [] {$\bullet$};
        \node at (6,4.5) [above right] {\footnotesize{$7^+$}};
        \node at (5.5,4) [] {$\bullet$};
        \node at (5.5,4) [below] {\footnotesize{$7^-$}};
    \end{tikzpicture}\hspace{1.5cm}
    \begin{tikzpicture}[scale=0.5,baseline=0]
        \draw (0,11) node [left] {$7$} --(1,11)--(1,5) -- (6,5) -- (6,0)  -- (7,0); 
        \draw (0,8) node [left] {$3$} -- (3,8) -- (3,3) -- (5,3) -- (5,1) -- (7,1);
        \draw (0,9) node [left] {$2$} -- (2,9) -- (2,2) -- (7,2);
        \draw (0,7) node [left] {$4$} -- (4,7) -- (4,4) -- (7,4);
        \foreach \i in {1,5,6,7}
        {
            \node at (0,11-\i) [left] {$\i$};
        }
        \foreach \i in {1,2,3,4}
        {
            \node at (0,4-\i) [left] {$\i$};
        }
        \node at (1, 10.5) [] {$\circ$};
        \node at (1, 10.5) [above right] {\footnotesize{$1^+$}};
        \node at (3,5.5) [] {$\bullet$};
        \node at (3,5.5) [left] {\footnotesize{$1^-$}};
        \node at (5, 2.5) [] {$\circ$};
        \node at (5, 2.5) [right] {\footnotesize{$2^+$}};
        \node at (4.5,2) [] {$\bullet$};
        \node at (4.5,2) [below] {\footnotesize{$2^-$}};
        \node at (2, 8.5) [] {$\circ$};
        \node at (2, 8.5) [above right] {\footnotesize{$3^+$}};
        \node at (1.5,8) [] {$\bullet$};
        \node at (1.5,8) [below] {\footnotesize{$3^-$}};        
        \node at (3,7.5) [] {$\bullet$};
        \node at (3,7.5) [above right] {\footnotesize{$4^+$}};
        \node at (2.5,7) [] {$\bullet$};
        \node at (2.5,7) [below] {\footnotesize{$4^-$}};
        \node at (4,6.5) [] {$\circ$};
        \node at (4,6.5) [above right] {\footnotesize{$6^+$}};
        \node at (3.5,5) [] {$\bullet$};
        \node at (3.5,5) [below] {\footnotesize{$6^-$}};
        \node at (6,4.5) [] {$\bullet$};
        \node at (6,4.5) [above right] {\footnotesize{$7^+$}};
        \node at (5.5,4) [] {$\bullet$};
        \node at (5.5,4) [below] {\footnotesize{$7^-$}};
    \end{tikzpicture}\hspace{1.5cm}
     \begin{tikzpicture}[scale=0.5,baseline=0]
        \draw (0,11) node [left] {$7$} --(1,11)--(1,5) -- (6,5) -- (6,0)  -- (7,0); 
        \draw (0,9) node [left] {$2$} -- (4,9) -- (4,2) -- (7,2);
        \draw (0,8) node [left] {$3$} -- (3,8) -- (3,3) -- (5,3) -- (5,1) -- (7,1);
        \draw (0,7) node [left] {$4$} -- (2,7) -- (2,4) -- (7,4);
        \foreach \i in {1,5,6,7}
        {
            \node at (0,11-\i) [left] {$\i$};
        }
        \foreach \i in {1,2,3,4}
        {
            \node at (0,4-\i) [left] {$\i$};
        }
        \node at (1, 10.5) [] {$\circ$};
        \node at (1, 10.5) [above right] {\footnotesize{$1^+$}};
        \node at (3.5,3) [] {$\bullet$};
        \node at (3.5,3) [below] {\footnotesize{$1^-$}};
        \node at (5, 2.5) [] {$\circ$};
        \node at (5, 2.5) [above right] {\footnotesize{$2^+$}};
        \node at (4.5,2) [] {$\bullet$};
        \node at (4.5,2) [below] {\footnotesize{$2^-$}};
        \node at (4, 8.5) [] {$\circ$};
        \node at (4, 8.5) [above right] {\footnotesize{$3^+$}};
        \node at (5.5,1) [] {$\bullet$};
        \node at (5.5,1) [below] {\footnotesize{$3^-$}};        
        \node at (3,7.5) [] {$\bullet$};
        \node at (3,7.5) [right] {\footnotesize{$4^+$}};
        \node at (6.5,0) [] {$\bullet$};
        \node at (6.5,0) [below] {\footnotesize{$4^-$}};
        \node at (2,6.5) [] {$\circ$};
        \node at (2,6.5) [right] {\footnotesize{$6^+$}};
        \node at (1.5,5) [] {$\bullet$};
        \node at (1.5,5) [below] {\footnotesize{$6^-$}};
        \node at (6,4.5) [] {$\bullet$};
        \node at (6,4.5) [above right] {\footnotesize{$7^+$}};
        \node at (2.5,4) [] {$\bullet$};
        \node at (2.5,4) [below] {\footnotesize{$7^-$}};
    \end{tikzpicture}
    \caption{On the left is periodic grid pattern for the target necklace. Performing leftward toggles at $1$ and $7$ gives the middle diagram. Performing further leftward toggles at $3$, $4$, and $3$ gives the periodic grid pattern for the source necklace on the right.}
    \label{fig:phi-inv}
\end{figure}
\end{exmp}

\noindent We also obtain the following consequence from the results in this subsection, included here for completeness.

\begin{cor}\label{cor:flag-unfrozen-X} Let $\cP$ be a positroid. Then the flag moduli space $\modsp(\Lambda_\cP)$ is a cluster $\mathcal{X}$-scheme. Each positroid weave $\ww$ for $\cP$ gives a seed with quiver equal to $Q(\ww;\mathfrak{t})$ with frozen vertices deleted.
\end{cor}
\begin{proof} Following \cite{CW}, the forgetful map $\modsp(\Lambda_\cP;\tt)\longrightarrow \modsp(\Lambda_\cP)$ commutes with the cluster map $p:\mathcal{A}\longrightarrow \mathcal{X}$. It follows from \cite[Corollary 4.53]{CW} that $\modsp(\Lambda_\cP)$ is a cluster $\mathcal{X}$-scheme, and each seed for $\modsp(\Lambda_\cP;\tt)$ gives a seed for $\modsp(\Lambda_\cP)$ whose quiver is obtained by deleting the frozen vertices of the seed for $\modsp(\Lambda_\cP;\tt)$.
\end{proof}

\begin{cor}\label{thm:cluster-ensemble} $\dM(\Lambda_\cP;\tt)$ is a cluster $\mathcal{X}$-scheme with the same initial quiver as $\modsp(\Lambda_\cP;\tt)$. In particular, $\dM(\Lambda_\cP;\tt)$ is the cluster dual of the cluster $\mathcal{A}$-scheme $\modsp(\Lambda_\cP;\tt)$.
\end{cor}
\begin{proof} Suppose $\cP$ is a positroid of type $(m,n)$, $\bG$ is a reduced plabic graph associated with $\cP$ and $\ww=\ww(\bG)$ is a positroid weave obtained from $\bG$. In order for $\dM(\Lambda_\cP;\tt)$ to have the same quiver as $\modsp(\Lambda_\cP;\tt)$, we need to specify how the frozen $\mathcal{X}$-coordinates are defined. By the construction above, for each $1\leq i\leq n$, there is a frozen cluster cycle $\gamma_i$ associated with the boundary face sandwiched between $i-1$ and $i$ of $\bG$. This frozen relative 1-cycle $\gamma_i$ starts at the base point $i^-$ and ends at the base point $i^+$, see Figures \ref{fig: exceptional base points} and \ref{fig: base points}. The local system $\cL$ on the exact Lagrangian surface $L_\ww$ provides a well-defined parallel transport along $\gamma_i$. Since the base point $i^-$ is located on a strand with stalks $l_{i-1}$ and the base point $i^+$ is located on a strand with stalks $l_i$, we can compose the parallel transport along $\gamma_i$ with the linear isomorphism $\phi_{i-1}:l_{i-1}\rightarrow l_i$. This composition defines a monodromy $X_i$, or more accurately a merodromy (see \cite[Section 4]{CW}), along $\gamma_i$. This monodromy $X_i$ is the frozen cluster $\mathcal{X}$-coordinate we associated with the frozen cluster cycle $\gamma_i$. A direct computation shows that under mutation these monodromies $X_i$ transform according to the cluster $\mathcal{X}$-transformation formula. Since these frozen cluster $\mathcal{X}$-coordinates occur along relative 1-cycles $\gamma_i$ that are contained in a neighborhood of the boundary, none of the mutable cluster $\mathcal{X}$-coordinates (including those that are not in the initial cluster seed) have them as factors. Therefore, Corollary \ref{cor:flag-unfrozen-X}, which shows that $\modsp(\Lambda_\cP)$ is a cluster $\mathcal{X}$-scheme, implies that $\dM(\Lambda_\cP;\tt)$ is also a cluster $\mathcal{X}$-scheme.
\end{proof}

\section{Donaldson-Thomas Transformation and the Muller-Speyer Twist Map}\label{sec:twist}
The goal of this section is proving Theorem \ref{thm:mainB} and Corollary \ref{cor:quasiequivalence}.


\subsection{Summary of ingredients} Theorem \ref{thm:mainB} compares two automorphisms of any positroid cell $\Pi^\circ_\cP$: the twist map and the Donaldson-Thomas (DT) transformation.

\begin{itemize}
    \item[(i)] First, let us recall that \cite{MullerSpeyertwist} defined the \emph{twist map} $\tw:\Pi^\circ_\cP\lr\Pi^\circ_\cP$, an automorphism\footnote{The map $\tw$ is called the \emph{right twist map}  in \cite{MullerSpeyertwist}. They also define a \emph{left twist map} using the source Grassmann necklace, and show that the right and the left twist maps are inverses of each other.} on any positroid variety $\Pi^\circ_\cP$. By definition, it acts on matrix representatives by
\begin{align*}
    \tw:\Pi^\circ_\cP&\longrightarrow \Pi^\circ_\cP\\
    [v_1,v_2,\dots, v_n]&\longmapsto \left[\Delta_1^{-1}\alpha_1,\Delta_2^{-1}\alpha_2,\dots \Delta_n^{-1}\alpha_n\right],
\end{align*}
where $\alpha_i:=\overrightarrow{\bigwedge}_{j\in\overrightarrow{I_i}\setminus \{i\}}v_j$ is a wedge of the $m-1$ column vectors indexed by $\overrightarrow{I_i}\setminus \{i\}$ according to the ordering $<_i$, viewed as a vector in $(\bC^m)^*\cong \bC^m$.\\

\item[(ii)] Independently, \cite{GL} shows that plabic graph seeds are part of a cluster structure on positroid varieties. By Theorem~\ref{thm:weave construction}, plabic graph quivers coincide with the quiver from a positroid weave. By Theorem \ref{thm:Demazure weave}, each positroid weave is equivalent to a complete Demazure weave. In consequence, the results from \cite[Section 8]{CGGLSS} apply and provide Donaldson-Thomas transformations for $\modsp(\Lambda_\cP)$, $\modsp(\Lambda_\cP;\tt)$, and thus $\Pio_\cP$. All these DT transformations are quasi-cluster automorphisms.\\
\end{itemize}

\noindent We take the following two steps to prove Theorem \ref{thm:mainB}:\\
\begin{enumerate}
    \item Describe the Donaldson-Thomas automorphism $\DT:\modsp(\Lambda_\cP;\tt)\lr \modsp(\Lambda_\cP;\tt)$ in a contact geometric manner, using a front diagram for the Legendrian link $\Lambda_\cP$ and a contact isotopy that induces DT in the flag moduli. This is the content of Subsection \ref{subsec: review of DT}. This description in the context of Legendrian links and the cluster theory of their associated microlocal moduli is in line with our previous work \cite[Section 5]{CW}, see also \cite[Section 8]{CGGLSS}.\\

    \item Show that the following diagram commutes and prove Theorem~\ref{thm:mainB}.
    \begin{equation}\label{eq:big-sq}
    \xymatrix{\Pi^\circ_\cP\ar[r]^(0.4)\Phi_(0.4)\cong \ar[d]_\tw & \modsp(\Lambda_\cP;\tt)\ar[r]^p \ar@[blue][d]^\DT & \modsp(\Lambda_\cP) \ar[d]^\DT \\
    	\Pi^\circ_\cP \ar[r]^(0.4)\Phi_(0.4)\cong & \modsp(\Lambda_\cP;\tt)\ar[r]^p & \modsp(\Lambda_\cP)}
    \end{equation}
The middle blue arrow $\DT$ in the above diagram is the Donaldson-Thomas transformation for $\modsp(\Lambda_\cP;\tt)$. We show the commutativity of this diagram by using the geometric description of $\DT$ in item (1). This is the content of Subsection \ref{ssec:unfrozen DT}, which proves that the outer rectangular diagram commutes, and Subsection \ref{ssec:proof_TheoremB}, that establishes the remaining commutativity. In particular, this implies that the twist map $\tw:\Pio_\cP\lr\Pio_\cP$ is a DT-transformation, and thus a quasi-cluster automorphism.
\end{enumerate}


\color{black}

\subsection{Donaldson-Thomas transformations on Flag Moduli Spaces}\label{subsec: review of DT}
In this subsection we introduce the building blocks of the DT transformation on flag moduli spaces. In particular, we explain how the DT automorphisms of $\modsp(\Lambda_{\cP})$, $\modsp(\Lambda_{\cP};\tt)$ and $\dM(\Lambda_{\cP};\tt)$ are induced by a Legendrian isotopy of $\Lambda_{\cP}$ composed with a contactomorphism (together with moving the base points in the latter two cases): specifically, by a half K\'{a}lm\'{a}n loop rotation and a reflection. In this subsection and hereafter, we endow the ring $\bC[\modsp(\Lambda_{\cP};\tt)]$ with the cluster algebra structure constructed in Section~\ref{subsec:phi map}. We also endow $\modsp(\Lambda_\cP)$ and $\dM(\Lambda_{\cP};\tt)$ with the $\cX$-cluster structures of Corollaries~\ref{cor:flag-unfrozen-X} and~\ref{thm:cluster-ensemble}, respectively.

\subsubsection{DT transformation on $\modsp(\Lambda_{\cP})$} The results from \cite[Section 8.2]{CGGLSS} allow us to  describe the DT transformation for $\modsp(\Lambda_{\cP})$ in this contact geometric manner. See also the cases  previously discussed in \cite[Section 5]{CW}. The first result is the following:

\begin{prop}\label{prop:braid-var-to-flag-moduli} Let $\cP$  be a positroid and $\Lambda_{\cP}$  its associated Legendrian link.
Then the cluster $\cX$-variety $\modsp(\Lambda_{\cP})$ admits a DT transformation. In addition, there exists a contact isotopy in $(\bR^3,\xi_{st})$, consisting of a sequence of rotations on $\Lambda_{\cP}$, and a contactomorphism of $(\bR^3,\xi_{st})$, consisting of a reflection, such that their composition induces the DT transformation on $\modsp(\Lambda_{\cP})$.
\end{prop}
\begin{proof} 
By Corollary~\ref{cor:flag-unfrozen-X}, $\modsp(\Lambda_{\cP})$ is a cluster $\cX$-variety with seeds given by positroid weaves. By Theorem \ref{thm:Demazure weave}, any positroid weave is equivalent to a Demazure weave, which defines a cluster seed for a braid variety \cite{CGGLSS}. In our case, let $\beta_\cP$ be the positive braid word associated with a positroid $\cP$. By Proposition \ref{prop: containing w_0},  $\beta_\cP$  can be assumed to be of the form $\gamma w_0$ for some positive braid word $\gamma$, after applying braid moves and cyclic rotations. Then, any positroid weave $\ww$ associated with $\cP$ is equivalent to a Demazure weave $\wv$ for the braid variety $X(\gamma)$. By \cite[Theorem 8.6]{CGGLSS}, the braid variety is a cluster $\mathcal{X}$-variety. By construction there is a quotient map $q:X(\gamma)\longrightarrow \modsp(\Lambda_\cP)$; it can be described by comparing the seeds for $X(\gamma)$ and $\modsp(\Lambda_\cP)$ that come from $\wv$ and $\ww$ respectively. This quotient map coincides with the cluster-theoretic map defined by forgetting the frozen $\mathcal{X}$-coordinates. By Definition~\ref{defn: DT}, the DT transformation on $X(\gamma)$ descends to the DT transformation on $\modsp(\Lambda_\cP)$ via the quotient map $q$. By \cite[Section 8.2]{CGGLSS}, the DT on the braid variety $X(\gamma)$ is given by a contact isotopy, induced by a Legendrian isotopy  given by rotations, followed by a reflection contactomorphism. Therefore, DT descends to the DT transformation on $\modsp(\Lambda_\cP)$ defined by the same contact isotopy followed by the same contactomorphism.
\end{proof}

In Subsections~\ref{ssec:rotation} and \ref{ssec:reflection}, we describe the contact isotopy referred to as a \emph{rotation} and the contactomorphism referred to as a \emph{reflection} in Proposition~\ref{prop:braid-var-to-flag-moduli}. Simultaneously, we generalize the results from \cite{CGGLSS} in the presence of base points. This generalization allows us to construct $\DT$ on the decorated moduli $\modsp(\Lambda_\cP; \tt)$.\\



\subsubsection{General flag moduli $\modsp(\Lambda; \tt)$} Individually, rotations and reflections do not induce automorphisms of a flag moduli. Rather, they  induce isomorphisms between two algebraic varieties, each of which can be described in a manner similar to our previous flag moduli. It is only the composition of several such isomorphisms that leads to an automorphism.  The next task is to define these intermediate algebraic varieties, which are a generalization of the flag moduli spaces introduced thus far.

\begin{defn}\label{def:general-modsp} Let $S:= \{s_1, \dots, s_{m-1}\}$ and $T:= \{\tau_1, \dots, \tau_m\}$.
Fix $\alpha=a_1 \dots a_l$ a word in $S \cup T$ with an even number of letters $a_{i_1}, \dots, a_{i_{2r}}\in T$ and label the letters in $\alpha$ from $T$ by fixing a bijection $p:\{1^+, 1^-, \dots, 2r^+, 2r^-\} \to \{i_1, \dots i_{2r}\}$. We also fix a sign $\sigma_i \in \{\pm1\}$ for each $i \in [r]$.\\

\noindent By definition, the flag moduli space  associated to such a triple $(\alpha;p, \sigma)$ is
\[\modsp(\alpha;p, \sigma):= \{(\cF_0, \cF_1, \dots, \cF_l) \in (\widetilde{\Fl}_m)^{l+1} : \cF_0=\cF_l \text{ and } (**)  \}/\GL_m\]
where $(**)$ is the following set of three conditions, which must hold for all $j \in [l]:$
\begin{enumerate}
    \item If $a_j=s_i$, then $\cF_{j-1} \doublearrow{i} \cF_{j}$.
    \item If $a_j= \tau_i$, then $\cF_{j-1}$ and $\cF_j$ are $\tau_i$-scaled.
    \item If $a_{p(i^+)}=a_j=\tau_c$ and $a_{p(i^-)}=a_k=\tau_d$, then 
    \[\displaystyle\frac{\left(v_{\cF_{j}}\right)_c}{\left(v_{\cF_{j-1}}\right)_c}= \sigma_i \frac{\left(v_{\cF_{k-1}}\right)_d}{\left(v_{\cF_{k}}\right)_d}. 
    \]
\end{enumerate}
\hfill$\Box$
\end{defn}

\begin{defn}\label{def:general-modsp-leg}  Let $(\alpha;p, \sigma)$ be a triple as in Definition \ref{def:general-modsp}. From $\alpha$, we consider the following braid diagram with base points: each $s_i$ gives a crossing and each $\tau_j= a_{p(i^\pm)}$ gives the base point $i^{\pm}$ on strand $j$. Let $\Lambda_\alpha\sse(\bR^3,\xi_{st})$ be the $(-1)$-closure of this braid diagram and $\tt_\alpha$ the  associated set of base points. Via Construction~\ref{lem:locsys}, each point of $\modsp(\alpha; p, \sigma)$ gives a local system on $\Lambda_\alpha$ with trivializations on $\Lambda_\alpha\setminus \tt_\alpha$. By definition, the moduli space $\modsp(\Lambda_\alpha;\tt_\alpha)$ is defined to be $\modsp(\alpha; p, \sigma)$, where $\sigma$ is chosen so that the microlocal merodromies $A_{i^+},A_{i^-}$ are equal for pairs of basepoints $(i^+, i^-)$ paired according to $p$.
\hfill$\Box$
\end{defn}

Definition~\ref{defn:framed-arrangement} of $\modsp(\Lambda_\cP; \tt)$ is a special case of Definition~\ref{def:general-modsp-leg} for the choice $$\alpha:=a_{p(1^+)} a_{p(1^-)} \beta_1 a_{p(2^+)} a_{p(2^-)} \beta_2 \cdots a_{p(n^+)} a_{p(n^-)} \beta_n,$$ where $a_{p(i^+)}= \tau_m$, $a_{p(i^-)}$ is $\tau_m$ if $l(\beta_i)=0$ and is $\tau_{m-1}$ otherwise, and the signs $\sigma_i$ are chosen so that the two microlocal merodromies $A_{i^+}=A_{i^-}$ coincide for each pair of base points. Finally, Definition~\ref{def:general-modsp-leg} can be modified in the same vein as Definition \ref{def:X-with-frozens}  so as to define the clipped versions $\dM(\alpha; p, \sigma)$ and $\dM(\Lambda_\beta;\tt)$.

\subsubsection{The rotation isotopy.}\label{ssec:rotation}
In this subsection, we describe the rotation isotopy for $(-1)$-closures of positive braid words with $w_0$ as a consecutive subword. By Proposition~\ref{prop: containing w_0}, any positroid braid word $\beta_\cP$ is equivalent to a braid of this form. We also recall the notation $s_{i^*}:=w_0 s_i w_0= s_{m-i}$ if $s_i\in S_m$.



\begin{defn}[Rotation]
Let $\Lambda_\beta \subset (\bR^3,\xi_{st})$ be the $(-1)$-closure of a positive braid word $\beta$. Suppose that $\beta$ is of the form $\beta=\beta_1 s_iw_0\beta_2$, where $\beta_1,\beta_2$ are positive braid words. By definition, a \emph{rotation} at  the crossing $s_i$ is the Legendrian isotopy given by any composition of Reidemeister III moves which takes $\beta_1 s_iw_0\beta_2$ to $\beta_1w_0 s_{i^*}\beta_2$. By definition, a Reidemeister III move acts on the set of base points according to Figure~\ref{fig: movement of base points in RIII},  leaving all the base points outside of that region  invariant. By definition, a rotation acts on the set of base points basepoints $\tt$  of $\Lambda_\beta$  by composing the action of its Reidemeister III moves.\hfill$\Box$
\end{defn}
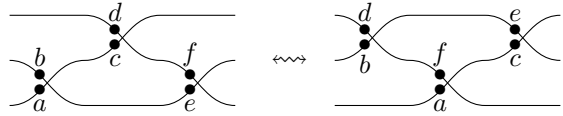
\begin{figure}[H]
\centering
\begin{tikzpicture}[baseline=15]
\draw (0,0) to [out=0,in=180] (1,0.6) to [out=0,in=180] (2,1.2) -- (3,1.2);
\draw (0,0.6) to [out=0,in=180] (1,0)  -- (2,0) to [out=0,in=180] (3,0.6);
\draw (0,1.2) -- (1,1.2) to [out=0,in=180] (2,0.6) to [out=0,in=180] (3,0);
\node (a) at (0.4,0.2) [] {$\bullet$};
\node (b) at (0.4,0.4) [] {$\bullet$};
\node (c) at (1.4,0.8) [] {$\bullet$};
\node (d) at (1.4,1) [] {$\bullet$};
\node (e) at (2.4,0.2) [] {$\bullet$};
\node (f) at (2.4,0.4) [] {$\bullet$};
\foreach \i in {a,c,e}
{
\node at (\i) [below] {$\i$};
}
\foreach \i in {b,d,f}
{
\node at (\i) [above] {$\i$};
}
\end{tikzpicture}
\quad $\leftrightsquigarrow$\quad 
\begin{tikzpicture}[baseline=15]
\draw (0,0) -- (1,0) to [out=0,in=180] (2,0.6) to [out=0,in=180] (3,1.2);
\draw (0,0.6) to [out=0,in=180] (1,1.2)  -- (2,1.2) to [out=0,in=180] (3,0.6);
\draw (0,1.2) to [out=0,in=180] (1,0.6) to [out=0,in=180] (2,0) -- (3,0);
\node (b) at (0.4,0.8) [] {$\bullet$};
\node (d) at (0.4,1) [] {$\bullet$};
\node (a) at (1.4,0.2) [] {$\bullet$};
\node (f) at (1.4,0.4) [] {$\bullet$};
\node (c) at (2.4,0.8) [] {$\bullet$};
\node (e) at (2.4,1) [] {$\bullet$};
\foreach \i in {b,a,c}
{
\node at (\i) [below] {$\i$};
}
\foreach \i in {d,f,e}
{
\node at (\i) [above] {$\i$};
}
\end{tikzpicture}
\caption{Movement of base points under Reidemeister III moves. Note that the cyclic order of basepoints along a strand is preserved.}
    \label{fig: movement of base points in RIII}
\end{figure}

\begin{lem}\label{lem:rotation_action} Let $\Lambda_\beta \subset (\bR^3,\xi_{st})$ be the $(-1)$-closure of a positive braid word $\beta$ and $\tt$ a set of base points. Suppose that $\beta$ is of the form $\beta=\beta_1 s_iw_0\beta_2$, where $\beta_1,\beta_2$ are positive braid words. Let $(\Lambda',\tt')$  be the result of applying a rotation at the crossing $s_i$ of $(\Lambda_\beta,\tt)$, and thus $\La'=\La_{\beta'}$ with $\beta'=\beta_1w_0 s_{i^*}\beta_2$. Then:

\begin{enumerate}
    \item Rotation at the crossing $s_i$ induces an isomorphism $r_i:\modsp(\Lambda_\beta;\tt)\lr \modsp(\Lambda';\tt')$.\\

    \item At the level of flags, $r_i$ acts according to
    \[(\cF_0\to \cdots \cF_{j-1} \xrightarrow{s_i} \cF_{j} \xrightarrow{w_0} \cF_{j+1} \to \cdots \to \cF_0) \mapsto (\cF_0\to \cdots \cF_{j-1} \xrightarrow{w_0} \cF_{j}' \xrightarrow{s_{i^*}} \cF_{j+1} \to \cdots \to \cF_0)\]
where $\cF_j'$ is the unique flag such that $\cF_{j} \xrightarrow{s_i w_0} \cF_j' \xrightarrow{s_{i^*}} \cF_{j+1}$.\\

\item Let $\cF_\bullet \in \modsp(\Lambda_\beta;\tt)$ and define $\cF_\bullet':= r_i(\cF_\bullet) \in \modsp(\Lambda_\beta';\tt')$.  Let $\cL(\cF_\bullet), \cL(\cF_\bullet')$ denote the induced $\GL_1(\bC)$-local systems on $\La$ and $\La'$ with trivializations away from the base points, described in Construction~\ref{lem:locsys}. Then $\cL(\cF_\bullet)$ and $\cL(\cF_\bullet')$ are isomorphic.
\end{enumerate}
In addition, the statements above also hold for the clipped moduli space $\dM(\Lambda_\beta;\tt)$.
\end{lem}

\begin{proof}
Part (1)  follows from the fact that any Legendrian isotopy is invertible, as is the action on base points given by  Figure~\ref{fig: movement of base points in RIII}.  In this case, each Reidemeister III move induces an isomorphism, see e.g.~ \cite[Section 5.3]{CZ} or \cite[Section 4.1.2]{CasalsNg}, and thus their composition induces an  isomorphism. For Part (2), \cite[Section 5.3]{CZ}  provides an explicit formula for the action of a Reidemeister III on flags.  The formula in Part (2)  is directly obtained by composing the formulas for the corresponding sequence of Reidemeister III moves.

For Part (3), note that the  the action given by according to Figure~\ref{fig: movement of base points in RIII} on base points is such that the relative ordering of base points along each link component remains unchanged. Therefore, the trivializations away from base points before the rotation induce trivializations away from base points after the rotation. This defines an isomorphism of the local systems as required.  Finally, the arguments for Parts (1), (2) and (3)  are identical in the case of $\dM(\Lambda_\beta;\tt)$.
\end{proof}

By \cite[Section 8.1]{CGGLSS}, the flag moduli space $\modsp(\Lambda_\beta)$ carries a  Poisson structure. It is inherited from the Poisson structure on a braid variety via the quotient map appearing in the proof of Proposition \ref{prop:braid-var-to-flag-moduli}. At a toric $\mathcal{X}$-chart given by a filling $L$ of $\Lambda$, the Poisson bracket is precisely the intersection form in $H_1(L,\bZ)$ and the cluster $\mathcal{X}$-variables are microlocal monodromies, see \cite[Section 7.2.1]{CZ} and \cite[Section 4.4]{CW}. In such a seed, the Poisson bracket is log-canonical when expressed in terms of the microlocal monodromies $\{X_i\}$ around the corresponding absolute cycles in the seed:
\begin{equation}\label{eq:log canonical bracket}
\{X_\gamma,X_\delta\}=\inprod{\gamma}{\delta}X_\gamma X_\delta,
\end{equation}
where $\inprod{\gamma}{\delta}$ denotes the intersection pairing between $\gamma$ and $\delta$, considered as absolute 1-cycles in $H_1(L,\bZ)$.

\noindent The same holds for the clipped moduli spaces $\dM(\Lambda_\beta;\tt)$. Using the notation above, we have:

\begin{lem}\label{lem:poisson}
The isomorphisms $r_i:\modsp(\Lambda_\beta)\lr \modsp(\Lambda')$ and $r_i:\dM(\Lambda_\beta;\tt)\lr \dM(\Lambda';\tt')$ induced by a rotation preserve the Poisson brackets.
\end{lem}

\begin{proof} Let us argue that each Reidemeister III move induces an isomorphism preserving the Poisson brackets. The result will then follow, as a rotation is a composition of Reidemeister III moves.

In the case of $\modsp(\Lambda_\beta)$, or when studying the mutable part of $\dM(\Lambda_\beta;\tt)$, the effect of a Reidemeister III move only affects the boundary of the weave. Since the mutable vertices correspond to absolute 1-cycles, they are confined in a compact region of the weave which is left invariant under Reidemeister III moves. By \cite[Section 7.2.1]{CZ}, the associated variables are microlocal monodromies around those 1-cycles and can be computed using cross-ratios of flags around them. Therefore, each Reidemeister III move acts as the identity in that mutable sub-lattice and it induces an isomorphism preserving the Poisson brackets. In the case of the frozen relative 1-cycles for $\dM(\Lambda_\beta;\tt)$ invariance under Reidemeister III moves follows from a computation similar to \cite[Section 8.1]{CGGLSS}. 

\end{proof}

\subsubsection{The reflection contactomorphism.}\label{ssec:reflection} Let $(\bR^3,\xi_{st})$ be equipped with the contact form $\xi_{st}=\ker\{dz-ydx\}$.
\begin{defn}
     The \emph{reflection contactomorphism} is $t:(\bR^3,\xi_{st})\lr(\bR^3,\xi_{st})$, $t(x,y,z):=(-x,y,-z)$. \hfill $\Box$
\end{defn}
 
Let $\Lambda^\eta$ be the {\it downwards} $(-1)$-closure of a positive braid $\eta$, which is the front in Figure \ref{fig:downwards -1 closure of a positive braid}.

\begin{figure}[H]
    \centering
    \begin{tikzpicture}[baseline=0,scale=0.8]
    \draw [dashed] (0,-0.25) rectangle node [] {$\eta$} (4,-2.25);
    \foreach \i in {1,3,4}
    {
    \draw [decoration={markings,mark=at position 0.5 with {\arrow{<}}},postaction={decorate}] (4,-0.5*\i) to [out=0,in=180] (5,-1-0.5*\i) to [out=180,in=0] (4,-2-0.5*\i) -- (0,-2-0.5*\i) to [out=180,in=0] (-1,-1-0.5*\i) to [out=0,in=180] (0,-0.5*\i);
    }
    \node at (2,-3) [] {\footnotesize{$\vdots$}};
    \node at (-1,-2) [] {\footnotesize{$\vdots$}};
    \node at (5,-2) [] {\footnotesize{$\vdots$}};
    \end{tikzpicture}
    \caption{Downwards $(-1)$-closure of a positive braid word $\eta$.}
    \label{fig:downwards -1 closure of a positive braid}
\end{figure}
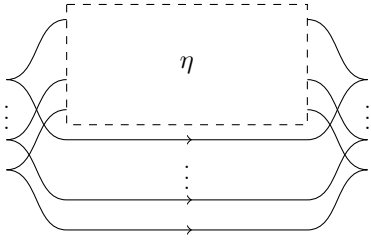

Consider a positive braid word $\beta= s_{i_1} \dots s_{i_\ell}$. Place the center of this front of $\Lambda_\beta$ at the origin of $\bR^2_{x,z}$, the front being the (standard) $(-1)$-closure of $\beta$. Then the contactomorphism $t$ sends $\Lambda_\beta$ to a Legendrian $\Lambda^\eta=t(\Lambda_\beta)$ whose front is the downwards $(-1)$-closure of a positive braid word $\eta$. In this case, we obtain $\eta=(\beta^*)^\circ= s_{i_\ell^*} \dots s_{i_1^*}$ (the $\circ$ stands for ``opposite"). In addition, if $\tt$ are base points for $\La_\beta$, the contactomorphism $t$ endows the image Legendrian $\La^\eta$ with a set of base points $t(\tt)$.\\

The reflection contactomorphism induces a homonymous isomorphism $t:\modsp(\La_\beta;\tt)\lr\modsp(\La^{(\beta^*)^\circ};t(\tt))$ between the corresponding flag moduli spaces. In order to explicitly describe this isomorphism, we must discuss the \emph{Hodge star} action on flags, as follows. For a $k$-dimensional subspace $V\subset \bC^m$, we define $\star V:=\ker((\bC^m)^*\rightarrow V^*)$, where the map $(\bC^m)^*\rightarrow V^*$ is dual to the inclusion; this is a $(m-k)$-dimensional subspace of $(\bC^m)^*$. Given a flag $\cF=(0\subset (\cF)_1\subset (\cF)_2\subset \cdots \subset (\cF)_{m-1}\subset (\cF)_m=\bC^m$, we define $\star\cF$ to be the flag
\[
\star\cF:=(0\subset \star (\cF)_{m-1}\subset \star(\cF)_{m-2}\subset \cdots \subset \star(\cF)_1\subset (\bC^m)^*.
\]
Note that if $\cF\xrightarrow{s_i}\cF'$, then $\star \cF \xrightarrow{s_{i^*}}\star \cF'$. The next lemma follows from \cite[Section 5]{CW}:

\begin{lem}\label{lem:reflection_flags}
Let $\Lambda_\beta \subset (\bR^3,\xi_{st})$ be the $(-1)$-closure of a positive braid word $\beta$ and $\tt$ a set of base points. In terms of flags, the isomorphism $t:\modsp(\La_\beta;\tt)\lr\modsp(\La^{(\beta^*)^\circ};t(\tt))$ induced by the eponymous reflection contactomorphism $t:(\bR^3,\xi_{st})\lr(\bR^3,\xi_{st})$ is
\[
t: \left[\cF_0 \xrightarrow{s_{i_1}} \cF_1 \xrightarrow{s_{i_2}}  \cdots \xrightarrow{s_{i_{l-1}}} \cF_{l-1}\xrightarrow{s_{i_l}} \cF_0\right] \longmapsto \left[\star \cF_0\xrightarrow{s_{i_l}^*}\star \cF_{l-1} \xrightarrow{s_{i_{l-1}}^*} \cdots \xrightarrow{s_{i_2}^*}\star \cF_1 \xrightarrow{s_{i_1}^*}\star \cF_0\right].
\]
\end{lem}

\noindent Let us now study the action of the reflection $t$ on the $\GL_1(\bC)$-local system $\cL$ from Construction~\ref{lem:locsys}, as follows. Consider a flag $\cF=(0\subset (\cF)_1\subset (\cF)_2\subset \cdots \subset (\cF)_{m-1}\subset (\cF)_m=\bC^m)$. Taking Hodge star and then taking quotients of consecutive subspaces, we obtain
\begin{equation}\label{eq:dual stalks}
\frac{\star (\cF)_k}{\star (\cF)_{k+1}}=\frac{\ker ((\bC^m)^*\rightarrow (\cF)_k^*)}{\ker ((\bC^m)^*\rightarrow (\cF)_{k+1}^*)}\cong \left(\frac{(\cF)_{k+1}}{(\cF)_k}\right)^*.
\end{equation}
Therefore, the stalks of the rank $\GL_1(\bC)$-local system $\cL'$ on $\Lambda':=t(\La)$, equipped with $t(\tt)$, are dual to the stalks of the $\GL_1(\bC)$-local system $\cL$ on $\Lambda$. Thus, away from the crossings in the front projection of $\Lambda$ and $\Lambda'$, we have a natural duality on stalks. See Figure \ref{fig:hodge star parallel transport}. We can extend this across the crossings as follows.\\

\begin{figure}[H]
\centering
\begin{tikzpicture}[baseline=0]
\draw (-2,0) node [blue,left] {$\left(\frac{(\cF_{j})_{i_j+1}}{(\cF_{j})_{i_j}}\right)^*$} -- (-1,0) to [out=0,in=180] (1,1) -- (2,1)node [blue,right] {$\left(\frac{(\cF_{j-1})_{i_j}}{(\cF_{j-1})_{i_j-1}}\right)^*$};
\draw (-2,1) node [blue,left] {$\left(\frac{(\cF_{j})_{i_j}}{(\cF_{j})_{i_j-1}}\right)^*$} -- (-1,1) to [out=0,in=180] (1,0) -- (2,0)node [blue,right] {$\left(\frac{(\cF_{j-1})_{i_j+1}}{(\cF_{j-1})_{i_j}}\right)^*$};
\node [red] at (-1.1,-0.5) [] {$\star(\cF_{j})_{i_j+1}$};
\node [red] at (-1.1,0.5) [] {$\star(\cF_{j})_{i_j}$};
\node [red] at (-1.1,1.5) [] {$\star(\cF_{j})_{i_j-1}$};
\node [red] at (1.1,-0.5) [] {$\star(\cF_{j-1})_{i_j+1}$};
\node [red] at (1.1,0.5) [] {$\star(\cF_{j-1})_{i_j}$};
\node [red] at (1.1,1.5) [] {$\star(\cF_{j-1})_{i_j-1}$};
\node [red] at (0,-0.5) [] {$=$};
\node [red] at (0,1.5) [] {$=$};
\draw (-2,-1) -- (2,-1);
\draw (-2,2) -- (2,2);
\node at (0,-1.5) [] {$\vdots$};
\node at (0,2.5) [] {$\vdots$};
\end{tikzpicture} 
\caption{Microlocal parallel transport under the Hodge star action.}
\label{fig:hodge star parallel transport}
\end{figure}

\noindent Consider the computation of the microlocal parallel transport of $\cL$ from Equations \eqref{eq:rho} and \eqref{eq:lambda}. Applying the Hodge star action yields the microlocal parallel transport for the local system $\cL'$. Since microlocal parallel transport in $\cL'$ is dual to the microlocal parallel transport of $\cL$, the microlocal monodromies of $\cL'$ are inverses of the corresponding ones in $\cL$. The duals of $\lambda,\rho$ have the following equations:\\
\[
\left(\frac{(\cF_{j})_{i_j+1}}{(\cF_{j})_{i_j}}\right)^*\hookrightarrow \left(\frac{(\cF_j)_{i_j+1}}{(\cF_j)_{i_j-1}}\right)^*=\left(\frac{(\cF_{j-1})_{i_j+1}}{(\cF_{j-1})_{i_j-1}}\right)^*\twoheadrightarrow \left(\frac{(\cF_{j-1})_{i_j}}{(\cF_{j-1})_{i_j-1}}\right)^*
\]
\[
\left(\frac{(\cF_{j})_{i_j}}{(\cF_{j})_{i_j-1}}\right)^* \twoheadleftarrow \left(\frac{(\cF_{j})_{i_j+1}}{(\cF_{j})_{i_j-1}}\right)^*=\left(\frac{(\cF_{j-1})_{i_j+1}}{(\cF_{j-1})_{i_j-1}}\right)^*\hookleftarrow \left(\frac{(\cF_{j-1})_{i_j+1}}{(\cF_{j-1})_{i_j}}\right)^*.
\]

This understanding of the change in local systems under $t$ allows us prove the following result:
\begin{lem}\label{lem:reflection-Poisson}
Let $\Lambda_\beta \subset (\bR^3,\xi_{st})$ be the $(-1)$-closure of a positive braid word $\beta$ and $\tt$ its base points. Then:\\
\begin{enumerate}
    \item For the undecorated moduli, $t:\modsp(\La_\beta)\lr\modsp(\La^{(\beta^*)^\circ})$ is a Poisson isomorphism.\\

    \item The isomorphism $t:\modsp(\La_\beta;\tt)\lr\modsp(\La^{(\beta^*)^\circ};t(\tt))$ induced by the eponymous reflection contactomorphism $t$  pulls back the microlocal merodromy $A'_s$ to $A_s^{-1}$, i.e.~ $t^*(A'_s)=A_s^{-1}$.
\end{enumerate}

\end{lem}

\begin{proof}
Let us ease notation with $\La:=\La_\beta$ and $\La':=\La^{(\beta^*)^\circ}$. For Part (1), we verify this in the log-canonical Darboux chart---cf. Equation~\eqref{eq:log canonical bracket}--- given by any Lagrangian filling $L_\ww$ of $\La$. Let $L_\ww$ is a filling of $\Lambda$ defined by a weave $\ww$. The Legendrian lift of $L_\ww$ is a Legendrian surface in $(\bR^5,\ker\{dz-ydx-pdq\})$. The front of the Legendrian lift of $L_\ww$ is a subset of $\bR^3_{x,z,q}$.  The contactomorphism $t$ extends to a contactomorphism of $(\bR^5,\ker\{dz-ydx-pdq\})$ by $(q,p)\mapsto (-q,p)$.  Therefore, the front in $\bR^3_{x,z,q}$ associated to the weave $\ww$ changes according to $(x,z,q)\mapsto (-x,-z,-q)$. From the diagrammatics of the weave, the weave resulting from applying $t$ to the lift of $L_\ww$ is obtained by applying the antipodal map to the plane and simultaneously interchanging the $i$th and $(m-i)$th colors for all $i\in[m-1]$, if there are $m-1$ colors in total.

The resulting weave $\ww'$ defines a filling $L_{\ww'}$ for $\Lambda'$. By construction, there is a canonical bijection between cycles on $L_\ww$ and $L_{\ww'}$. If $\gamma$ and $\gamma'$ are absolute cycles on $L_\ww$ and $L_{\ww'}$ corresponding to each other, then their microlocal monodromies satisfy $
X_\gamma=X_{\gamma'}^{-1}$. The intersection forms on $L_\ww$ and $L_{\ww'}$ are also opposite to each other: if $(\gamma',\delta')$ is a pair of 1-cycles on $L_{\ww'}$ corresponding to a pair of 1-cycles $(\gamma,\delta)$ on $L_\ww$, then $\inprod{\gamma}{\delta}_{L_\ww}=-\inprod{\gamma'}{\delta'}_{L_{\ww'}}$. These two properties and the log-canonical Poisson bracket formula \eqref{eq:log canonical bracket} imply that the induced isomorphism $t:\modsp(\Lambda)\rightarrow \modsp(\Lambda')$ is a Poisson isomorphism.\\

\begin{figure}[H]
    \centering
    \begin{tikzpicture}
        \draw (-2,0) node [above right] {$I_1$} -- (2,0) node [above left] {$I_2$};
        \node at (0,0) [] {$\bullet$};
        \node at (0,0) [below] {$s$};
        \draw [decoration={markings,mark=at position 0.5 with {\arrow{>}}},postaction={decorate}, blue] (-0.5,0) node [below] {$s_1$} arc (180:0:0.5) node [below] {$s_2$};
        \node [blue] at (0,0.5) [above] {$A_s$};
        \node at (0,-0.5) [below] {$\frmodsp(\Lambda)$};
    \end{tikzpicture}\hspace{2cm}
    \begin{tikzpicture}
        \draw (-2,0) node [above right] {$I_2$} -- (2,0) node [above left] {$I_1$};
        \node at (0,0) [] {$\bullet$};
        \node at (0,0) [below] {$s$};
        \draw [decoration={markings,mark=at position 0.5 with {\arrow{>}}},postaction={decorate}, blue] (-0.5,0) node [below] {$s_2$} arc (180:0:0.5) node [below] {$s_1$};
        \node [blue] at (0,0.5) [above] {$A'_s$};
        \node at (0,-0.5) [below] {$\frmodsp(\Lambda')$};
    \end{tikzpicture}
    \caption{Microlocal merodromy across base points under the isomorphism $t$.}
    \label{fig:transposed microlocal merodromy}
\end{figure}
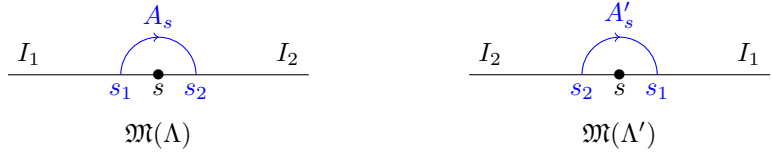

For Part (2), note that a trivialization of the local system $\cL$ over an open interval $I$ is the same as given an isomorphism between the global sections $\Gamma(I,\cL)$ and $\bC$. By Equation \eqref{eq:dual stalks}, $\Gamma(I,\cL')\cong \Gamma(I,\cL)^*$. Thus, any trivialization of $\cL$ over $I$ naturally gives rise to a trivialization of $\cL'$ over the corresponding open interval $I$. To compare microlocal merodromies, suppose that a base point $s$ is separating two open intervals $I_1$ and $I_2$. Each such interval is equipped with a trivialization and we can thus compare the two trivializations using parallel transport: this yields a non-zero microlocal merodromy $A_s$. In particular, $A_s$ is the composition $\bC\cong \cL_{s_1}\xrightarrow{\cong} \cL_{s_2}\cong \bC$, where $s_i$ is a point on $I_i$ that is sufficiently close to the base point $s$. After applying the contactomorphism $t$, the two intervals swap places and the composition above is transposed. See Figure \ref{fig:transposed microlocal merodromy}. Thus the microlocal merodromy $A'_s\in\bC[\frmodsp(\Lambda';t(\tt))]$ relates to $A_s\in\bC[\frmodsp(\Lambda;\tt)]$ under $t$ as $A'_s=A_s^{-1}$.
\end{proof}

\subsubsection{Donaldson-Thomas contactomorphism.} We now describe the DT transformation on $\modsp(\Lambda;\tt)$ in terms of the rotations in Subsection \ref{ssec:rotation} and the reflection in Subsection \ref{ssec:reflection}. Let $\Lambda_\beta \subset (\bR^3,\xi_{st})$ be the $(-1)$-closure of a positive braid word $\beta=w_0\theta$ and $\tt$ a set of base points.\\

Consider a Legendrian isotopy in $(\bR^3,\xi_{st})$  that moves the crossings of $\theta$  across the right cusps of the $(-1)$-closure of $\beta$ and moves the crossings of $w_0$  across the left cusps of this $(-1)$-closure. This leads to the front given by the downwards $(-1)$-closure of $w_0^\circ\theta^\circ$, where $\theta^\circ$ is the positive braid word obtained by reading $\theta$ right to left. Such a Legendrian isotopy can be obtained as a sequence of Reidemeister II moves.  Let us denote by $\Theta_t\in\mbox{Cont}(\bR^3,\xi_{st})$, $t\in[0,1]$, a contact isotopy realizing such Legendrian isotopy. In particular, $\Theta_1(\La_{w_0\theta})=\La^{w_0^\circ\theta^\circ}$. Let us still denote by $\Theta_1:\modsp(\La_{w_0\theta};\tt) \to \modsp(\La^{w_0^\circ\theta^\circ};\Theta_1(\tt))$ the associated isomorphism.

\begin{defn}\label{def:DT_link}
    Let $\beta=\eta\cdot w_0$ be a positive braid word with $\eta=s_{i_1} \dots s_{i_\ell}$, and ${r}_\beta:= r_{i_1} \circ r_{i_2} \circ \cdots \circ r_{i_\ell}$ the composition of rotations according to the crossings of $\eta$. By definition, the automorphism $\Psi_\beta:\modsp(\La_\beta;\tt)\lr \modsp(\La_\beta;\tt)$ is the composition of isomorphisms
    $$\Psi_\beta:\modsp(\La_\beta;\tt)\stackrel{r_\beta}{\lr}\modsp(\La_{w_0\eta^*};r_\beta(\tt))\stackrel{\Theta_1}{\lr}\modsp(\La^{(\eta^*w_0)^\circ};\Theta_1(r_\beta(\tt)))\stackrel{t}{\lr}\modsp(\La_\beta;\tt),$$
    where we have identified $\eta=(((\eta^*)^\circ)^*)^\circ$, $w_0=w_0^\circ$, and $\tt=t(\Theta_1(r_\beta(\tt)))$ in the target of the last map. Figure \ref{fig:half loop} depicts the underlying Legendrian isotopy and contactomorphism in the front inducing $\Psi_\beta$.\hfill$\Box$
\end{defn}

\noindent Let $\beta_1,\beta_2$ be two positive braid words related by a sequence of Reidemeister II and III moves. 
Such a sequence defines an isomorphism $R:\modsp(\La_{\beta_1};\tt)\lr\modsp(\La_{\beta_2};\tt)$, independent of the choice of such a sequence. If $\beta_1=s_{i_1} \dots s_{i_\ell}w_0$ ends in a word for $w_0$, then we define $\Psi_{\beta_2}:=R\circ\Psi_{\beta_1}\circ R^{-1}$.

\begin{figure}[H]
    \centering
    \begin{tikzpicture}[node1/.pic={
        \draw (0,0) rectangle (2,2);
        \node at (1,1) [] {$\eta$};
        \draw (3,0)  rectangle (5,2);
        \node at (4,1) [] {$w_0$};
        \draw (5,0.5) to [out=0,in=180] (7,2) to [out=180,in=0] (5,3.5) -- (0,3.5) to [out=180,in=0] (-2,2) to [out=0,in=180] (0,0.5);
        \draw (5,1.5) to [out=0,in=180] (7,3) to [out=180,in=0] (5,4.5) -- (0,4.5) to [out=180,in=0] (-2,3) to [out=0,in=180] (0,1.5);
        \node at (2,4.2) [] {$\vdots$};
        \node at (7,2.7) [] {$\vdots$};
        \node at (-2,2.7) [] {$\vdots$};
        \draw (2,0.5) -- (3,0.5);
        \draw (2,1.5) -- (3,1.5);
        \draw [red,dashed] (2.5,0) node [below] {$\cF$} -- (2.5,2);
        \draw [red,dashed] (2.5,3) -- (2.5,5) node [above] {$\cF'$};
    }, node2/.pic={
        \draw (0,0) rectangle (2,2);
        \node at (1,1) [] {$w_0$};
        \draw (3,0)  rectangle (5,2);
        \node at (4,1) [] {$\eta^*$};
        \draw (5,0.5) to [out=0,in=180] (7,2) to [out=180,in=0] (5,3.5) -- (0,3.5) to [out=180,in=0] (-2,2) to [out=0,in=180] (0,0.5);
        \draw (5,1.5) to [out=0,in=180] (7,3) to [out=180,in=0] (5,4.5) -- (0,4.5) to [out=180,in=0] (-2,3) to [out=0,in=180] (0,1.5);
        \node at (2,4.2) [] {$\vdots$};
        \node at (7,2.7) [] {$\vdots$};
        \node at (-2,2.7) [] {$\vdots$};
        \draw (2,0.5) -- (3,0.5);
        \draw (2,1.5) -- (3,1.5);
        \draw [red,dashed] (2.5,0) node [below] {$\cF''$} -- (2.5,2);
        \draw [red,dashed] (2.5,3) -- (2.5,5) node [above] {$\cF'$};
    }, node3/.pic={
        \draw (0,3) rectangle (2,5);
        \node at (1,4) [] {$w_0$};
        \draw (3,3)  rectangle (5,5);
        \node at (4,4) [] {$(\eta^*)^\circ$};
        \draw (5,3.5) to [out=0,in=180] (7,2) to [out=180,in=0] (5,0.5) -- (0,0.5) to [out=180,in=0] (-2,2) to [out=0,in=180] (0,3.5);
        \draw (5,4.5) to [out=0,in=180] (7,3) to [out=180,in=0] (5,1.5) -- (0,1.5) to [out=180,in=0] (-2,3) to [out=0,in=180] (0,4.5);
        \node at (2,1.2) [] {$\vdots$};
        \node at (7,2.7) [] {$\vdots$};
        \node at (-2,2.7) [] {$\vdots$};
        \draw (2,3.5) -- (3,3.5);
        \draw (2,4.5) -- (3,4.5);
        \draw [red,dashed] (2.5,0) node [below] {$\cF''$} -- (2.5,2);
        \draw [red,dashed] (2.5,3) -- (2.5,5) node [above] {$\cF'$};
    }, node4/.pic={
        \draw (0,0) rectangle (2,2);
        \node at (1,1) [] {$\eta$};
        \draw (3,0)  rectangle (5,2);
        \node at (4,1) [] {$w_0$};
        \draw (5,0.5) to [out=0,in=180] (7,2) to [out=180,in=0] (5,3.5) -- (0,3.5) to [out=180,in=0] (-2,2) to [out=0,in=180] (0,0.5);
        \draw (5,1.5) to [out=0,in=180] (7,3) to [out=180,in=0] (5,4.5) -- (0,4.5) to [out=180,in=0] (-2,3) to [out=0,in=180] (0,1.5);
        \node at (2,4.2) [] {$\vdots$};
        \node at (7,2.7) [] {$\vdots$};
        \node at (-2,2.7) [] {$\vdots$};
        \draw (2,0.5) -- (3,0.5);
        \draw (2,1.5) -- (3,1.5);
        \draw [red,dashed] (2.5,0) node [below] {$\star\cF'$} -- (2.5,2);
        \draw [red,dashed] (2.5,3) -- (2.5,5) node [above] {$\star\cF''$};
    }]
    \draw (0,0) pic [scale=0.5] {node4};
    \draw (8,0) pic [scale=0.5] {node3};
    \draw (8,4) pic [scale=0.5] {node2};
    \draw (0,4) pic [scale=0.5] {node1};
    \draw [->] (6,1) -- node [below] {contactomorphism $t$} (4,1);
    \draw [->] (4,5) -- node [above] {rotation $r_\beta$} (6,5);
    \draw[->] (10,3.5) -- node [right] {$\Theta_1$} (10,3);
    \end{tikzpicture}
    \caption{A sequence of moves inducing $\Psi_\beta$, which we momentarily prove coincides with the DT transformation. Here $\beta=\eta\cdot w_0$.
    }
    \label{fig:half loop}
\end{figure}
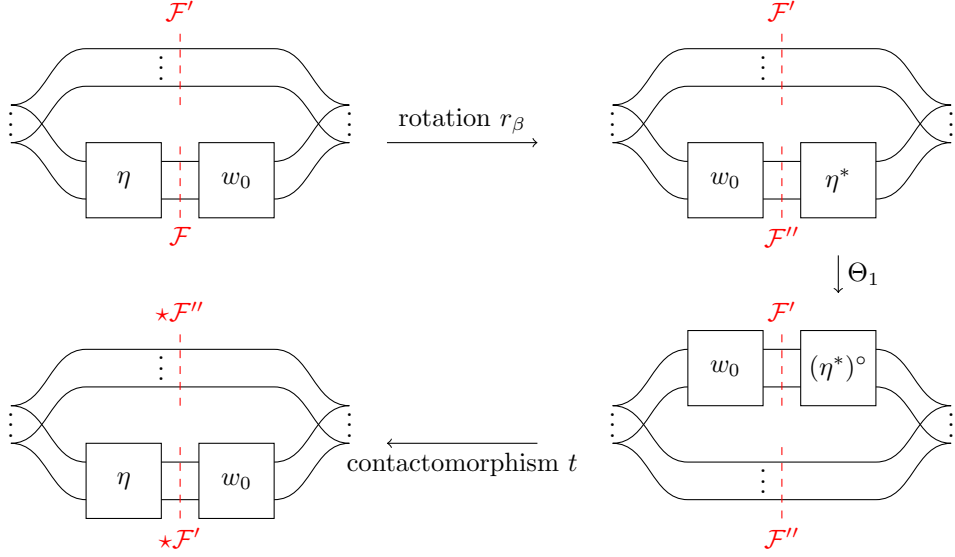

\begin{prop}\label{prop:Psi_DT}
Let $\cP$ be a positroid, $\beta_\cP$ its associated positroid braid and $\La_\cP$ its positroid Legendrian, endowed with the positroid base points $\tt$. Then $\Psi_{\beta_\cP}:\modsp(\La_\cP; \tt)\lr\modsp(\La_\cP; \tt)$ is a DT transformation.
\end{prop}
\begin{proof} We must verify conditions (i'),(ii') of Definition \ref{defn: DT}. Condition (i') is satisfied thanks to Proposition \ref{prop:braid-var-to-flag-moduli}. Lemma \ref{lem:reflection-Poisson}.(2) implies that Condition (ii') also holds, as $r_\beta$ and $\Theta_1$ do not affect the microlocal merodromies. Therefore, $\Psi_{\beta_\cP}$ is a quasi-cluster DT transformation.
\end{proof}

\noindent For reference, Figure \ref{fig:DT action on flags} illustrates how $\Psi_\beta$, and therefore the DT transformation, acts on the points of the flag moduli spaces. For a positroid $\cP$, we choose $\beta$ to be the braid word from Proposition~\ref{prop: containing w_0} which ends in $w_0$ and is cyclically equivalent to $\beta_\cP$.

\begin{figure}[H]
    \centering
    \begin{tikzpicture}[node1/.pic={
    \node (0)  at (0,0) [] {$\cF'$};
    \node (1) at (1,0) [] {$\cF_1$};
        \node (3) at (3,0) [] {$\cF_{l-1}$};
        \node (2) at (2,0) [] {$\cdots$};
        \node (4) at (4,0) [] {$\cF$};
        \node (a) at (2.5,2) [] {$\cF'$};
        \draw [->] (0) -- (1);
        \draw [->] (1) -- (2);
        \draw [->] (2) -- (3);
        \draw [->] (3) -- (4);
        \node at (0.5,0) [below] {\footnotesize{$s_{i_1}$}};
        \node at (1.5,0) [below] {\footnotesize{$s_{i_2}$}};
        \node at (2.5,0) [below] {\footnotesize{$s_{i_{l-1}}$}};
        \node at (3.5,0) [below] {\footnotesize{$s_{i_l}$}};
        \node (5) at (5,0) [] {$\cF'$};
        \draw [double] (0) -- (a);
        \draw [->] (4) -- (5);
        \draw [double] (a) -- (5);
        \node at (4.5,0) [below] {\footnotesize{$w_0$}};
    }, node2/.pic={
        \node (0) at (0,0) [] {$\cF'$};
        \node (1) at (1,0) [] {$\cF''$};
        \node (2) at (2,0) [] {$\cF'_1$};
        \node (3) at (3,0) [] {$\cdots$};
        \node (4) at (4,0) [] {$\cF'_{l-1}$};
        \node (5) at (5,0) [] {$\cF'$};
        \node (a) at (2.5,2) [] {$\cF'$};
        \draw [->] (0) -- (1);
        \draw [->] (1) -- (2);
        \draw [->] (2) -- (3);
        \draw [->] (3) -- (4);
        \draw [->] (4) -- (5);
        \node at (0.5,0) [below] {\footnotesize{$w_0$}};
        \node at (1.5,0) [below] {\footnotesize{$s^*_{i_1}$}};
        \node at (2.5,0) [below] {\footnotesize{$s^*_{i_2}$}};
        \node at (3.5,0) [below] {\footnotesize{$s^*_{i_{l-1}}$}};
        \node at (4.5,0) [below] {\footnotesize{$s^*_{i_l}$}};
        \draw [double] (0) -- (a);
        \draw [double] (5) -- (a);
    }, node3/.pic={
        \node (1) at (5,2) [] {$\cF''$};
        \node (2) at (4,2) [] {$\cF'_1$};
        \node (3) at (3,2) [] {$\cdots$};
        \node (4) at (2,2) [] {$\cF'_{l-1}$};
        \node (0) at (1,2) [] {$\cF'$};
        \node (-1) at (0,2) [] {$\cF''$};
        \node (a) at (2.5,0) [] {$\cF''$};
        \draw [->] (4) -- (0);
        \draw [->] (1) -- (2);
        \draw [->] (2) -- (3);
        \draw [->] (3) -- (4);
        \draw [->] (0) -- (-1);
        \node at (4.5,2) [above] {\footnotesize{$s_{i_1}^*$}};
        \node at (3.5,2) [above] {\footnotesize{$s_{i_2}^*$}};
        \node at (2.5,2) [above] {\footnotesize{$s_{i_{l-1}}^*$}};
        \node at (1.5,2) [above] {\footnotesize{$s_{i_l}^*$}};
        \node at (0.5,2) [above] {\footnotesize{$w_0$}};
        \draw [double] (-1) -- (a);
        \draw [double] (a) -- (1);
    }, node4/.pic={
        \node (0)  at (0,0) [] {$\star\cF''$};
    \node (1) at (1,0) [] {$\star\cF'_1$};
        \node (3) at (3.5,0) [] {$\star\cF'_{l-1}$};
        \node (2) at (2,0) [] {$\cdots$};
        \node (4) at (5,0) [] {$\star\cF'$};
        \node (a) at (3,2) [] {$\star\cF''$};
        \draw [->] (0) -- (1);
        \draw [->] (1) -- (2);
        \draw [->] (2) -- (3);
        \draw [->] (3) -- (4);
        \node at (0.5,0) [below] {\footnotesize{$s_{i_1}$}};
        \node at (1.5,0) [below] {\footnotesize{$s_{i_2}$}};
        \node at (2.75,0) [below] {\footnotesize{$s_{i_{l-1}}$}};
        \node at (4.25,0) [below] {\footnotesize{$s_{i_l}$}};
        \node (5) at (6,0) [] {$\star\cF''$};
        \draw [double] (0) -- (a);
        \draw [->] (4) -- (5);
        \draw [double] (a) -- (5);
        \node at (5.5,0) [below] {\footnotesize{$w_0$}};
    }
    ]
    \draw (0,4) pic {node1};
    \draw (8,4) pic {node2};
    \draw (8,0) pic {node3};
    \draw (-1,0) pic {node4};
    \draw [->] (7,1) -- node [below] {reflection $t$} (5,1);
    \draw [->] (5,5) -- node [above] {rotation $r_\beta$} (7,5);
    \draw[->] (10,3.5) -- node [right] {$\Theta_1$} (10,3);
    \end{tikzpicture}
    \caption{The action of DT on the flag moduli space.
    }
    \label{fig:DT action on flags}
\end{figure}
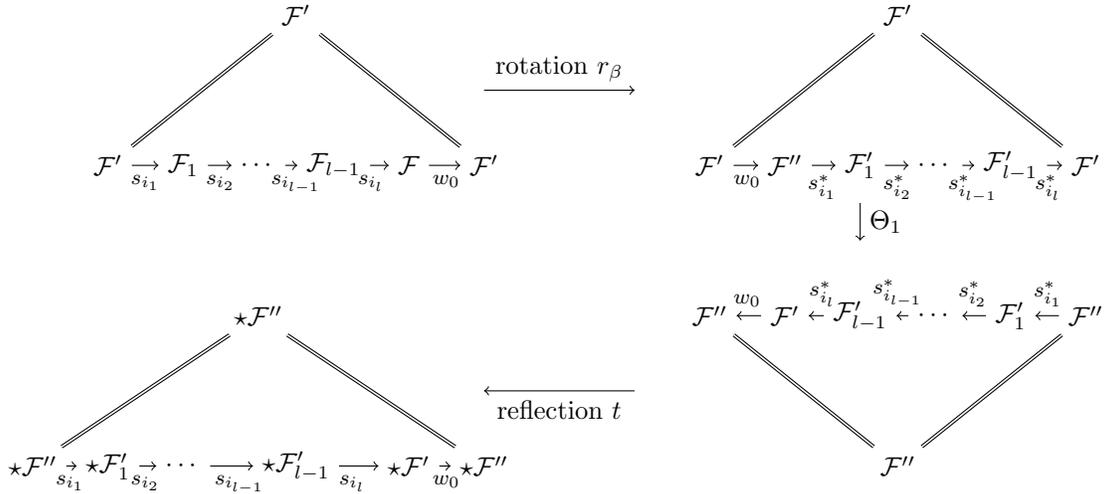

\begin{rmk} By construction, $\Psi_{\beta_\cP}:\modsp(\La_\cP; \tt)\lr\modsp(\La_\cP; \tt)$ descend to an automorphism $\modsp(\La_\cP)\lr\modsp(\La_\cP)$, still denoted by $\Psi_{\beta_\cP}$, of the undecorated moduli. Proposition \ref{prop:braid-var-to-flag-moduli} implies that this induced map on the undecorated moduli is the DT transformation.\hfill$\Box$
\end{rmk}


\subsection{The Twist Map and Donaldson-Thomas transformation on \texorpdfstring{$\modsp(\Lambda_\cP)$}{}}\label{ssec:unfrozen DT}
Proposition \ref{prop:Psi_DT} shows that the automorphism in Definition \ref{def:DT_link} is a Donaldson-Thomas transformation. For ease of notation, we also refer to the map $\Psi_{\beta_\cP}$, as an automorphism of either $\modsp(\La_{\beta};\tt)$ or $\modsp(\La_{\beta})$, by the name of $\DT$.\\

Let $\cP$ be a positroid of type $(m,n)$. In this section, we compare the twist map on $\Pio_\cP$ and DT on the undecorated moduli $\modsp(\Lambda_\cP)$. We show the commutativity of the diagram
\[
\xymatrix{\Pio_\cP\ar[r]^{p\circ\Phi} \ar[d]_\tw & \modsp(\Lambda_\cP) \ar[d]^\DT \\
\Pio_\cP\ar[r]_{p\circ\Phi} & \modsp(\Lambda_\cP)}
\]
which is the outer rectangle of the diagram \eqref{eq:big-sq} above. This shows that $\tw$ descends to $\DT$ on $\modsp(\Lambda_\cP)$. It is the first step towards showing that $\tw$ is a $\DT$-transformation on $\Pio_\cP$, as it proves Condition (i') in Definition~\ref{defn: DT}.\\

\noindent By construction, the quotient map $p:\modsp(\Lambda_\cP;\tt)\lr \modsp(\Lambda_\cP)$ forgets the data of trivializations. Thus, when comparing $\tw$ on $\Pi^\circ_\cP$ and $\DT$ on $\modsp(\Lambda_\cP)$ we only need to consider the column spans of the matrix representative $M\in\Pi^\circ_\cP$, instead of its actual column vectors. Let us then define $\Conf(\cP):=\Pi^\circ_\cP/(\bC^\times)^n$, where $(\bC^\times)^n$ acts on a point $M\in\Pi^\circ_\cP$ by scaling its column vectors. Then the Muller-Speyer twist map $\tw:\Pi^\circ_\cP\lr\Pi^\circ_\cP$ descends to a homonymous automorphism $\tw:\Conf(\cP)\lr \Conf(\cP)$. This map can be described as
\begin{align*} \tw: \Conf(\cP) & \longrightarrow \Conf(\cP)\\
[l_1,l_2,\dots, l_n] & \longmapsto [h_1,h_2,\dots, h_n]
\end{align*}
where $h_i:=\star\left[\overrightarrow{I_i}\setminus \{i\}\right]\sse\bC^m$ is the hyperplane spanned by $m-1$ lines, viewed as a line in the dual vector space $(\bC^m)^*\cong \bC^m$. Note that the quotient of the $(\bC^\times)^n$-scaling action on $\Pi^\circ_\cP$ is equivalent to the forgetful map $p:\modsp(\Lambda_\cP;\tt)\lr \modsp(\Lambda_\cP)$. Therefore, the isomorphism $\Phi:\Pi_\cP^\circ\lr\modsp(\Lambda_\cP;\tt)$ descends to the respective quotients as an isomorphism of algebraic stacks:

\begin{defn} By definition, $\Phi:\Conf(\cP)\lr \modsp(\Lambda_\cP)$ is the isomorphism obtained from the eponymous isomorphism $\Phi:\Pi_\cP^\circ\lr \modsp(\Lambda_\cP;\tt)$ in Definition \ref{def:phi} by taking the quotient of $\Pi_\cP^\circ$ by the scaling action and the quotient of $\modsp(\Lambda_\cP;\tt)$ that forgets the data of the decorations.\hfill$\Box$
\end{defn}

By construction, the upper and lower faces of the following cubic diagram commute:

\[
\xymatrix{ & \Conf(\cP) \ar[rr]^\Phi_\cong \ar[dd]_(0.35){\tw} & & \modsp(\Lambda_\cP) \ar[dd]^\DT \\
\Pi_\cP^\circ \ar[ur]^p \ar[rr]^(0.65){\Phi}_(0.65){\cong} \ar[dd]_\tw & & \modsp(\Lambda_\cP;\tt)  \ar[ur]_p \ar[dd]^(0.35){\DT} & \\
& \Conf(\cP) \ar[rr]^(0.35){\Phi}_(0.35){\cong} & & \modsp(\Lambda_\cP) \\
\Pi_\cP^\circ \ar[ur]^p \ar[rr]^\Phi_\cong & & \modsp(\Lambda_\cP;\tt) \ar[ur]_p & }
\]

\noindent We now show that the back face of the cube commutes. Section \ref{ssec:proof_TheoremB} shows commutativity of the front face.

\begin{thm}\label{thm:DT without frozens} Let $\cP$ be a  positroid. Then the following diagram commutes:
\[
\xymatrix{\Conf(\cP)\ar[r]^\Phi \ar[d]_\tw & \modsp(\Lambda_\cP) \ar[d]^\DT \\
\Conf(\cP)\ar[r]_\Phi & \modsp(\Lambda_\cP)}
\]
In particular, the quotient of the Muller-Speyer twist $\tw:\Pi^\circ_\cP/(\bC^\times)^n\lr\Pi^\circ_\cP/(\bC^\times)^n$ is the DT-transformation on the quotient $\Pi^\circ_\cP/(\bC^\times)^n$ of the positroid variety $\Pi^\circ_\cP$.
\end{thm}

\begin{proof} We directly compute the image $(\Phi^{-1}\circ \DT\circ \Phi)([l_1,l_2,\dots, l_n])$ and compare it with $[h_1,h_2,\dots, h_n]:=\tw([l_1,l_2,\dots, l_n])$, the result of applying the twist map. Geometrically, for each $i$ we choose a particular flag $\cF$ whose 1-dimensional subspace is $l_i$. Then we compute the image of $\cF$ under the DT-transformation, which involves another flag $\cF'$. The position of these two flags in the front, and how $\cF$ changes under the DT-transformation, is sketched in Figure \ref{fig:half loop}. Let us now provide the details. Throughout, we view points of $\modsp(\Lambda_\cP)$ as images of points in $\Conf(\cP)$ under $\Phi$; that is, we obtain flags using the construction of Definition~\ref{def:phi}.\\

First, for each $i \in [n]$, we choose a particular sequence of Reidemeister II and III moves to turn $\beta_\cP$ into a braid word ending in $w_0$. That is, we choose a particular isomorphism $R$ as mentioned right after Definition~\ref{def:DT_link}. We now use toggles from Section~\ref{subsec: toggles and RIII moves} to define this sequence, as follows. Fix $i\in[n]$. We start with the target Grassmann necklace. Suppose $\overrightarrow{I_i}=\{i<_i j_2<_i\cdots<_i j_m\}$ and $j_k<_{j_2}\pi_f(i)<_{j_2}<j_{k+1}$, cf.~ Figure \ref{fig: convenient representative} (left). To the left of the $(i, \pi_f(i))$ chord and between heights $y=i$ and $y=j_m$, there are $\begin{tikzpicture}\draw (0,0.25) -- (0,0) -- (0.25,0);\end{tikzpicture}$ corners exactly at $y=i$ and $y=j_2, \dots, j_m$. We claim that we can perform a sequence of leftward toggles to the chord $(\pi_f^{-1}(i),i)$ and the chords $(\pi_f^{-1}(j_l),j_l)$ for $1<l\leq m$ so that these $\begin{tikzpicture}\draw (0,0.25) -- (0,0) -- (0.25,0);\end{tikzpicture}$ corners appear with respect to the $<_i$ ordering.

Indeed, if any two such  $\begin{tikzpicture}\draw (0,0.25) -- (0,0) -- (0.25,0);\end{tikzpicture}$ corners are out of order, the right one $R$ appears northeast of the left one $L$. Since we are in the grid pattern for $\beta_\cP$, the chord ending at $R$ has its $\begin{tikzpicture}\draw (0,0.25) -- (0.25,0.25) -- (0.25,0);\end{tikzpicture}$ corner southeast of the $\begin{tikzpicture}\draw (0,0.25) -- (0.25,0.25) -- (0.25,0);\end{tikzpicture}$ corner of the chord ending at $L$. Therefore, these chords are aligned and can be toggled, putting the corners in the correct order as claimed. In the end, we obtain the initial position as depicted in Figure \ref{fig: convenient representative} (right). This coincides with the position of $\cF,\cF'$ in the front as in Figure \ref{fig:half loop} (upper left).

\begin{figure}[H]
\centering
\begin{tikzpicture}[scale=0.6,baseline=0]
    \draw [dashed] (-2,5) -- (0,5);
    \node [teal] at (-1,5) [] {$\bullet$};
    \node [teal] at (-1,5) [above] {$i^-$};
    \draw (0,5) node [left] {$i$} -- (1,5) -- (1,1.5) -- (2,1.5) node[right] {$\pi_f(i)$};
    \draw (0,4) node [left] {$j_2$}-- (2,4);
    \draw (0,2.5) node [left] {$j_k$} -- (2,2.5);
    \draw (0,0.5) node [left] {$j_{k+1}$}-- (2,0.5);
    \draw (0,-1) node [left] {$j_m$} -- (2, -1);
    \node at (0,3.25) [] {$\vdots$};
    \node at (2,3.25) [] {$\vdots$};
    \node at (1,-0.25) [] {$\vdots$};
    \node [teal] at (1,4.5) [] {$\bullet$};
    \node [teal] at (1,4.5) [right] {$(i+1)^+$};
\end{tikzpicture}\hspace{2cm}
\begin{tikzpicture}[baseline=-10]
    \draw (-0.5,3) node [above] {$i$} -- (-0.5,1.75) -- (3.5,1.75) -- (3.5,0) -- (4,0);
    \draw [blue,dashed] (-0.2,1.9) -- (3,1.9) -- (3,-1.5) -- cycle;
    \node [blue] at (1.7,1) [] {$w_0$};
    \foreach \i in {2,3,6}
    {
    \draw (0.5*\i-0.25,3) -- (0.5*\i-0.25,2.5-0.5*\i-0.25) -- (4,2.5-0.5*\i-0.25);
    }
    \node at (2,3) [] {$\cdots$};
    \node at (4,0.4) [] {$\vdots$};
    \node at (4,-0.4) [] {$\vdots$};
    \node at (0.75,3) [above] {$j_2$};
    \node at (1.25,3) [above] {$j_3$};
    \node at (2.75,3) [above] {$j_m$};
    \node at (4,1.25) [right] {$j_2$};
    \node at (4,0.75) [right] {$j_3$};
    \node at (4,-0.75) [right] {$j_m$};
    \node at (4,0) [right] {$\pi_f(i)$};
    \draw[red,dashed] (-1,2.5) node [left] {$\cF$} -- (3,2.5);
    \draw [red,dashed] (3.25,2) -- (3.25,-2) node [below] {$\cF'$};
    \node [teal] at (3.5,1.5) [] {$\bullet$};
    \node [teal] at (2,1.75) [] {$\bullet$};
    \node [teal] at (3.5,1.5) [above right] {$(i+1)^+$};
    \node [teal] at (2,1.75) [above] {$i^-$};
\end{tikzpicture}
\caption{The cyclic braid representative we use in the proof of Theorem \ref{thm:DT without frozens} and Theorem~\ref{thm:DT-with-froz}. Along the strands labeled by $j_l$, the trivialization is $v_{j_l}$. Along the strand labeled by $i$, up until the basepoint $i^-$, the trivialization is $v_{i}$.}
\label{fig: convenient representative}
\end{figure}


Second, we now compute $DT(\cF)$ with this particular choice of periodic braid pattern, using the description of the DT-transformation as in Definition \ref{def:DT_link} and $\cF$ is chosen to be the flag exactly to the left of the fixed $w_0$ subword. By Lemma \ref{lem:rotation_action}, the $r_\beta$ isomorphism leaves $\cF':=\cF_i$ invariant. The Legendrian isotopy that induces $\Theta_1$, which consists of Reidemeister II moves near the cusps, leaves the region near $\cF'$ fixed. Therefore the map induced by $\Theta_1$ also leaves $\cF'$ invariant. Therefore, the flag $t(\cF')$ is precisely the flag that replaces $\cF$ under DT, i.e.~ we must apply the reflection contactomorphism $t$ to $\cF'$ to obtain $\DT(\cF)=t(\cF')$. By Lemma \ref{lem:reflection_flags}, the line $l_i$ is replaced by the $1$-dimensional subspace in the flag $\star \cF'$, namely $\star[j_2j_3\cdots j_m]$. Since $\cF'=\cF_i$ by definition, they share the same $(m-1)$-dimensional subspace $[j_m] \subset [j_{m-1}, j_{m}] \subset \cdots \subset [j_2, \dots, j_m] \subset [i, j_2, \dots, j_m]$. Therefore, the image of $l_i$ under the DT transformation is $\star[j_2j_3\cdots j_m]=h_i$. This coincides with the descent of the Muller-Speyer twist automorphism $\tw:\Conf(\cP)\lr\Conf(\cP)$, i.e.~ $DT(\cF)=\tw(\cF)$. For any other flag $\mathcal{G}$ in the tuple of flags specifying a point in $\modsp(\Lambda_\cP)$, the argument above also implies $DT(\mathcal{G})=\tw(\mathcal{G})$. Indeed, consider a sequence of Reidemeister III moves that moves the $w_0$ piece (to the left) so that the given $\mathcal{G}$ is the flag exactly to the left of $w_0$. Conjugating by the isomorphism induced by these Reidemeister III moves, the argument above also establishes the equality $DT(\mathcal{G})=\tw(\mathcal{G})$. Therefore $DT=\tw$, up to conjugation by $\Phi$, thus establishing the required commutativity of the diagram.
\end{proof}

\subsection{The Twist Map and Donaldson-Thomas on \texorpdfstring{$\frmodsp(\Lambda_\cP;\tt)$}{}}\label{ssec:proof_TheoremB}
This subsection proves that the Muller-Speyer twist map $\tw:\Pi^\circ_\cP\lr\Pi^\circ_\cP$ is a DT-transformation. This proves Theorem~\ref{thm:mainB} in the introduction. We also show that the source and target cluster structures on $\Pio_\cP$ are related by a quasi-cluster transformation, proving Corollary~\ref{cor:quasiequivalence} in the introduction. First, we establish Theorem~\ref{thm:mainB} by proving the following result:

\begin{thm}\label{thm:DT-with-froz} Let $\cP$ be a  positroid. Then the following diagram commutes:
\[
\xymatrix{\Pi^\circ_\cP\ar[r]^(0.4)\Phi \ar[d]_\tw & \frmodsp(\Lambda_\cP;\tt) \ar[d]^\DT \\
\Pi^\circ_\cP\ar[r]_(0.4)\Phi & \frmodsp(\Lambda_\cP;\tt)}.
\]
\noindent In addition, the Muller-Speyer twist map $\tw:\Pio_\cP\lr\Pio_\cP$ is a DT-transformation, where $\Pio_\cP$ is considered as a cluster $\cA$-scheme with initial seed $\Sigma_T(\bG)$.
\end{thm}
\begin{proof} The main difference between this argument and the proof of Theorem \ref{thm:DT without frozens} is that we must also understand how the base points move (and affect the flags) under the DT-transformation. Let us use the same cyclic braid representative as in the proof of Theorem \ref{thm:DT without frozens}, in which the decoration vector $v_i$ can be conveniently read off, cf.~Figure \ref{fig: convenient representative}. We will now compare the action of the twist map on the $i$th column vector $v_i$ and the action of $\DT$ on the corresponding trivialization, for each $1\leq i\leq n$. Since we argue for each $i$ separately, we only keep track of the basepoints $i^-$ and $(i+1)^+$ at a given time. Note that immediately to the left of $i^-$, the trivialization is given by (a projection of) the vector $v_i$. Consequently, in Figure \ref{fig: convenient representative} we only indicate the two base points $i^-$ and $(i+1)^+$. In fact, only $i^-$ will be moved in the end, whereas there is no need to move $(i+1)^+$. 

We are working in the decorated moduli space $\frmodsp(\Lambda_\cP;\tt)$, instead of the undecorated moduli $\frmodsp(\Lambda_\cP)$. Therefore, given the cyclic braid representative as in the proof of Theorem \ref{thm:DT without frozens}, we also need to specify where the base points are located in this particular representative. By Definition~\ref{def:base-pts}, the $j^+$ base points are in the $\begin{tikzpicture}[baseline=7]\draw (0,0.5) -- (0.5,0.5) -- (0.5,0);\end{tikzpicture}$ corners and remain there throughout a sequence of toggles. From the proof of Proposition~\ref{thm: frame sheaf moduli = positroid}, in order to compute the column vector $v_i$ we need to move the {\it minus} base points so that they occur after the $\begin{tikzpicture}[baseline=7]\draw (0,0.5) -- (0,0) -- (0.5,0);\end{tikzpicture}$ corners. Thus, we choose to place the base points in a way such that the only base points that show up in the local picture of Figure \ref{fig: convenient representative} are $(i+1)^+$ and $i^-$, which  are colored in green. In particular, $i^-$ is the base point where the trivialization changes from $v_i$ to $\Delta_{\overrightarrow{I_i}}^{-1}v_i$ and $(i+1)^+$ is the base point that changes trivialization from $\Delta_{\overrightarrow{I_i}}^{-1}v_i$ to $v_{\pi(i)}$.\\

\noindent Let us fix these initial choices and start applying the DT transformation, following Definition \ref{def:DT_link} and Proposition \ref{prop:Psi_DT}. After applying $r_\beta$, the situation is as depicted in Figure \ref{fig: local picture after DT} (left). The contact isotopy $\Theta_1$ has no effect and therefore we apply the contactomorphism $t$ next. Since Figure \ref{fig: local picture after DT} (left) is part of a {\it cyclic} braid pattern, instead of a front projection, the contactomorphism $t$ only reflects across the northwest-southeast axis. The result of applying $t$ to Figure \ref{fig: local picture after DT} (left) is depicted in Figure \ref{fig: local picture after DT} (right).
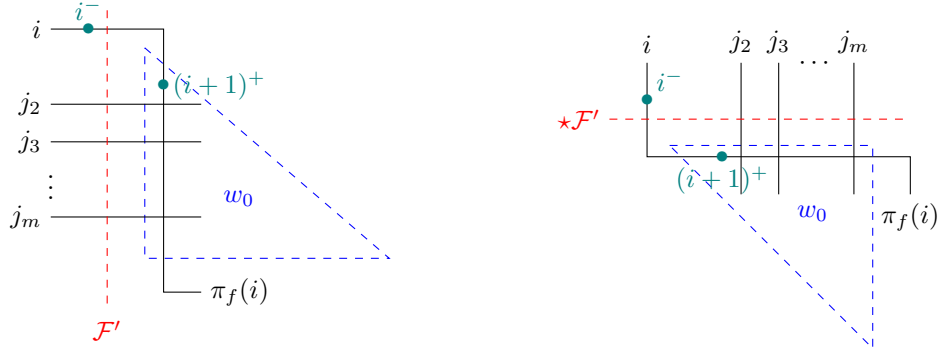
\begin{figure}[H]
    \centering
    \begin{tikzpicture}
    \draw (2,2.25) -- (3.5,2.25) -- (3.5,-1.25) -- (4,-1.25);
    \draw [blue,dashed] (3.25,2) -- (6.5,-0.8) -- (3.25,-0.8) -- cycle;
    \node [blue] at (4.5,0) [] {$w_0$};
    \foreach \i in {2,3,5}
    {
    \draw (2,2.5-0.5*\i-0.25) -- (4,2.5-0.5*\i-0.25);
    }
    \node at (2,0.25) [] {$\vdots$};
    \node at (2,1.25) [left] {$j_2$};
    \node at (2,0.75) [left] {$j_3$};
    \node at (2,-0.25) [left] {$j_m$};
    \node at (2,2.25) [left] {$i$};
    \node at (4,-1.25) [right] {$\pi_f(i)$};
    \draw [red,dashed] (2.75,2.5) -- (2.75,-1.5) node [below] {$\cF'$};
    \node [teal] at (3.5,1.5) [] {$\bullet$};
    \node [teal] at (2.5,2.25) [] {$\bullet$};
    \node [teal] at (3.5,1.5) [right] {$(i+1)^+$};
    \node [teal] at (2.5,2.25) [above] {$i^-$};
    \end{tikzpicture}\hspace{2cm}
    \begin{tikzpicture}
    \draw (-0.5,3) node [above] {$i$} -- (-0.5,1.75) -- (3,1.75) -- (3,1.25);
    \draw [blue,dashed] (-0.2,1.9) -- (2.5,1.9) -- (2.5,-0.8) -- cycle;
    \node [blue] at (1.7,1) [] {$w_0$};
    \foreach \i in {2,3,5}
    {
    \draw (0.5*\i-0.25,3) -- (0.5*\i-0.25,1.25);
    }
    \node at (1.75,3) [] {$\cdots$};
    \node at (0.75,3) [above] {$j_2$};
    \node at (1.25,3) [above] {$j_3$};
    \node at (2.25,3) [above] {$j_m$};
    \node at (3,1.25) [below] {$\pi_f(i)$};
    \draw[red,dashed] (-1,2.25) node [left] {$\star\cF'$} -- (3,2.25);
    \node [teal] at (0.5,1.75) [] {$\bullet$};
    \node [teal] at (-0.5,2.5) [] {$\bullet$};
    \node [teal] at (0.5,1.75) [below] {$(i+1)^+$};
    \node [teal] at (-0.5,2.5) [above right] {$i^-$};
\end{tikzpicture}
    \caption{(Left) Cyclic braid diagram near $w_0$ after applying $r_\beta$. (Right) Cyclic braid diagram near $w_0$ after applying $r_\beta,\Theta_1$ and the contactomorphism $t$. On the left, along the strands labeled by $j_l$, the trivialization is $v_{j_l}$. Along the strand labelled by $i$, up until the basepoint $i^-$, the trivialization is $v_{i}$.}
    \label{fig: local picture after DT}
\end{figure}

Let us compute the trivialization along the interval $(i^-,(i+1)^+)$ in Figure \ref{fig: local picture after DT} (right). As discussed in Subsection \ref{subsec: review of DT}, it follows from Lemma \ref{lem:reflection_flags} that the trivialization after the contactomorphism $t$ is dual to the trivialization in Figure \ref{fig: local picture after DT} (left), which depicts the (decorated) cyclic braid pattern before applying $t$. The trivialization along the interval $(i^-,(i+1)^+)$ in Figure \ref{fig: local picture after DT} (left) is
given by the image of the vector $\Delta_{\overrightarrow{I_i}}^{-1}v_i$ in the projection $\bC^m/(\cF')_{m-1}$ and $\Delta_{\overrightarrow{I_i}}=\det(v_i\wedge v_{j_2}\wedge \cdots \wedge v_{j_m})$. Therefore, the dual trivialization of $(\star\cF')_1\cong \star[j_2j_3\cdots j_m]$ along $(i^-,(i+1)^+)$ is given by $\star(v_{j_2}\wedge v_{j_3}\wedge \cdots \wedge v_{j_m})$. In order to complete the DT transformation, we also need to move the base points accordingly. Indeed, although the cyclic braid pattern (equivalently, the front of the Legendrian link) is the same as the initial one after applying the rotations and the reflection, the base points are not in the right place, cf.~Figure \ref{fig: convenient representative}. The final step is to move the base points back to where they initially were. By construction, it suffices to move the $i^-$ base point through the $\begin{tikzpicture}[baseline=7]\draw (0,0.5) -- (0,0) -- (0.5,0);\end{tikzpicture}$ corner in Figure \ref{fig: local picture after DT} (right). This action scales the trivialization at the corner by $A'_i=A_i^{-1}=\Delta_{\overrightarrow{I_i}}^{-1}$. Therefore, the DT-transformation is such that the trivialization along the $\begin{tikzpicture}[baseline=7]\draw (0,0.5) -- (0,0) -- (0.5,0);\end{tikzpicture}$ corner in the local picture changes from $v_i$ to $\Delta_{\overrightarrow{I_i}}^{-1}\star(v_{j_2}\wedge v_{j_3}\wedge \cdots \wedge v_{j_m})$, i.e.~DT sends $v_i$ to $\Delta_{\overrightarrow{I_i}}^{-1}\star(v_{j_2}\wedge v_{j_3}\wedge \cdots \wedge v_{j_m})$. This agrees with the Muller-Speyer twist and hence the diagram commutes as required.\\

\noindent We have now established that $\tw = \Phi^{-1} \circ \DT \circ \Phi$ as automorphisms of $\Pio_\cP$, where $\DT$ is the Donaldson-Thomas transformation on $\modsp(\Lambda_\cP;\tt)$ considered as an $\cA$-scheme with initial seed given by a positroid weave. This implies that $\tw:\Pio_\cP\lr\Pio_\cP$ is a quasi-cluster transformation, considered as an $\cA$-scheme with initial seed $\Sigma_T(\bG)$. Indeed, Theorem \ref{thm: frame sheaf moduli = positroid} implies that $\Phi^*$ sends $\Sigma(\ww(\bG);\tt)$ to $\Sigma_T(\bG)$ and commutes with mutations. Finally, by Theorem~\ref{thm:DT without frozens}, the mutation sequence in the quasi-cluster transformation $\tw$ is a reddening sequence. In addition, \cite[Equation 9]{MullerSpeyertwist} shows that $\tw$ inverts the target frozens. Therefore, $\tw$ is a $\DT$-transformation on $\Pio_\cP$ considered as an $\cA$-scheme with initial seed $\Sigma_T(\bG)$.
\end{proof}

\noindent We conclude Corollary~\ref{cor:quasiequivalence}, showing that the target and source seeds are quasi-cluster equivalent.

\begin{cor}\label{cor:twist-quasi-equivalence}
Let $\cP$ be a positroid and $\bG$ a plabic graph for $\cP$. Then $\Sigma_T(\bG)$ and $\Sigma_S(\bG)$ are related by a quasi-cluster transformation.
\end{cor}

\begin{proof}
First, \cite[Proposition 7.13]{MullerSpeyertwist} shows that  $\Sigma_S(\bG)$ differs from $(\tw^*)^{-2} \Sigma_T(\bG)$ by rescaling by target necklace frozens. These rescalings preserve exchange ratios (see e.g. the last paragraph of the proof of \cite[Theorem 5.17]{FSB}), and thus they yield a quasi-cluster transformation. Since $(\tw^*)^{-2}$ is also a quasi-cluster transformation by Theorem~\ref{thm:DT-with-froz}, $\Sigma_S(\bG)$ and $\Sigma_T(\bG)$ are related by a quasi-cluster transformation.
\end{proof}

\noindent In particular, this corollary implies that the seeds in the target and source cluster structures on $\Pio_\cP$ are the same up to Laurent monomials in frozens.

\renewcommand*{\thesection}{\Alph{section}}
\setcounter{section}{0}

\section{Appendix}

\subsection{Quasi-Cluster Donaldson-Thomas Transformation} Quasi-cluster transformations of cluster algebras were introduced  in \cite{Fra}. We use a slightly more restrictive version for the cluster ensemble setting: we give a brief review in this subsection, cf.~ \cite[Appendix]{CW} for more details. We assume basic familiarity with cluster algebras and cluster varieties, see e.g.~ \cite{FZI, FockGoncharov_ensemble}. 

Let us denote the mutable part of a quiver $Q$ by $Q^\uf$, and denote the exchange matrix of $Q$ by $\epsilon$. By following the cluster ensemble construction \cite[Section 1.2.9]{FockGoncharov_ensemble}, we obtain two pairs of cluster ensembles $(\mathcal{A}_Q,\mathcal{X}_Q)$ and $(\mathcal{A}_{Q^\uf}, \mathcal{X}_{Q^\uf})$ that fit into the following commutative diagram:
\[
\xymatrix{ \mathcal{A}_{Q^\uf} \ar@{^(->}[r]\ar[d] & \mathcal{A}_Q \ar[d] \\
\mathcal{X}_Q\ar@{->>}[r] & \mathcal{X}_{Q^\uf}}
\]
In particular, the two cluster schemes at the bottom row carry canonical Poisson structures defined in terms of the initial cluster $\mathcal{X}$-coordinates by $\{X_i,X_j\}=\epsilon_{ij}X_iX_j$. Let $I$, resp.~$I^\uf$ be the set of quiver vertices in $Q$, resp.~$Q^\uf$. Let $I^\fr:=I\setminus I^\uf$. Let $\mu$ be a mutation sequence and let $Q'=\mu(Q)$ be the quiver obtained by applying the mutation sequence $\mu$ to $Q$. Denote the cluster coordinates associated with $Q$, resp.~$Q'$, by $A_i$ and $X_i$, resp.~$A'_i$ and $X'_i$.

\begin{defn}[{\cite[Definitions A.31, A.32]{ShenWeng}}] A mutation sequence $\mu$ and an $I\times I$ matrix $N=(n_{ij})$ with $\mathbb{Z}$-entries are said to define a \emph{quasi-cluster transformation} if
\begin{enumerate}
    \item $\det N=\pm 1$;
    \item the submatrix $N|_{I^\uf\times I^\uf}$ is a permutation matrix;
    \item the submatrix $N|_{I^\uf \times I^\fr}=0$;
    \item $\epsilon_{ij}=\sum_{k,l} n_{ik}n_{jl}\epsilon'_{kl}$, where $\epsilon'$ is the exchange matrix of $Q'$.
\end{enumerate}
When the above conditions are satisfied, we define $M:=(N^t)^{-1}$. A quasi-cluster transformation $\sigma$ automatically acts biregularly on $\mathcal{A}_Q$ and $\mathcal{X}_Q$ by
\[
\sigma^*(A_i):=\prod_j(A'_j)^{m_{ij}} \quad \text{and} \quad \sigma^*(X_i):=\prod_j(X'_j)^{n_{ij}}. 
\]
In particular, the biregular morphism $\sigma:\mathcal{X}_Q\longrightarrow \mathcal{X}_Q$ is a Poisson morphism.\hfill$\Box$ 
\end{defn}

Note that if $Q=Q^\uf$, then a quasi-cluster transformation is the same as a cluster transformation, i.e.~a sequence of mutations. If the mutation sequence $\mu$ is empty, then the matrix $M$ encodes a permutation of the mutable $A$-variables and a rescaling of each $A$-variable by Laurent monomials in frozens. These rescalings have the property that the two monomials appearing on the right hand side of an exchange relation are scaled by the same Laurent monomial. Such scalings are said to \emph{preserve exchange ratios} in \cite{FSB}.

\begin{defn}\label{defn: DT} A quasi-cluster transformation $\sigma$ is said to be a \emph{quasi-cluster Donaldson-Thomas (DT) transformation} if 
\begin{itemize}
    \item[(i)] the mutation sequence $\mu$ of $\sigma$ is a reddening sequence, cf.~\cite[Section 3]{KellerDT};
    \item[(ii)] the matrix $M$ satisfies $M|_{I^\fr\times I^\fr}=-\mathrm{Id}$.
\end{itemize}
Each of these two conditions are equivalent to the following two conditions:
\begin{itemize}
    \item[(i')] the descent of $\sigma$ onto $\mathcal{X}_{Q^\uf}$ is the cluster DT transformation on $\mathcal{X}_{Q^\uf}$, cf.~\cite[Definition A.35]{ShenWeng};
    \item[(ii')] $\sigma^*(A_i)=A_i^{-1}$ for all frozen cluster $\mathcal{A}$-coordinates $A_i$.\hfill$\Box$
\end{itemize}
\end{defn}

Conditions (i') and (ii') above are used in Section \ref{sec:twist}.

\begin{rmk} If $Q=Q^\uf$, a quasi-cluster DT transformation is unique if it exists and it is central in the cluster modular group. If frozen vertices exist, quasi-cluster DT transformations need not be unique.\hfill$\Box$
\end{rmk}

\vspace{5cm}


\bibliographystyle{biblio}

\bibliography{main}

\newcommand{\etalchar}[1]{$^{#1}$}
\begin{thebibliography}{GHKK18}

\bibitem[CG22]{CG22}
Roger Casals and Honghao Gao.
\newblock Infinitely many {L}agrangian fillings.
\newblock {\em Ann. of Math. (2)}, 195(1):207--249, 2022.
\newblock \href {https://doi.org/10.4007/annals.2022.195.1.3}
  {\path{doi:10.4007/annals.2022.195.1.3}}.

\bibitem[CG23]{CG23}
Roger Casals and Honghao Gao.
\newblock A {L}agrangian filling for every cluster seed.
\newblock Preprint, 2023.
\newblock \href {http://arxiv.org/abs/arXiv:2308.00043}
  {\path{arXiv:arXiv:2308.00043}}.

\bibitem[CGG{\etalchar{+}}22]{CGGLSS}
Roger Casals, Eugene Gorsky, Mikhail Gorsky, Ian Le, Linhui Shen, and Jos\'{e}
  Simental.
\newblock Cluster structures on braid varieties.
\newblock Preprint, 2022.
\newblock \href {http://arxiv.org/abs/2207.11607} {\path{arXiv:2207.11607}}.

\bibitem[CGGS20]{CGSS20}
Roger Casals, Eugene Gorsky, Mikhail Gorsky, and Jos\'e Simental.
\newblock Algebraic weaves and braid varieties.
\newblock Preprint, 2020.
\newblock \href {http://arxiv.org/abs/arXiv:2012.06931}
  {\path{arXiv:arXiv:2012.06931}}.

\bibitem[CL22]{CL22}
Roger Casals and Wenyuan Li.
\newblock Conjugate fillings and {L}egendrian weaves.
\newblock Preprint, 2022.
\newblock \href {http://arxiv.org/abs/2210.02039} {\path{arXiv:2210.02039}}.

\bibitem[CN22]{CasalsNg}
Roger Casals and Lenhard Ng.
\newblock Braid loops with infinite monodromy on the {L}egendrian contact
  {DGA}.
\newblock {\em J. Topol.}, 15(4):1927--2016, 2022.

\bibitem[CW23]{CW}
Roger Casals and Daping Weng.
\newblock Microlocal theory of {L}egendrian links and cluster algebras.
\newblock {\em Geometry \& Topology}, pages 1--119, 2023.
\newblock \href {http://arxiv.org/abs/2204.13244} {\path{arXiv:2204.13244}}.

\bibitem[CZ22]{CZ}
Roger Casals and Eric Zaslow.
\newblock Legendrian weaves: N-graph calculus, flag moduli and applications.
\newblock {\em Geometry \& Topology}, pages 1--116, 2022.
\newblock \href {http://arxiv.org/abs/2007.04943} {\path{arXiv:2007.04943}}.

\bibitem[FG09]{FockGoncharov_ensemble}
Vladimir~V. Fock and Alexander~B. Goncharov.
\newblock Cluster ensembles, quantization and the dilogarithm.
\newblock {\em Ann. Sci. \'{E}c. Norm. Sup\'{e}r. (4)}, 42(6):865--930, 2009.
\newblock \href {https://doi.org/10.24033/asens.2112}
  {\path{doi:10.24033/asens.2112}}.

\bibitem[Fra16]{Fra}
Chris Fraser.
\newblock Quasi-homomorphisms of cluster algebras.
\newblock {\em Adv. in Appl. Math.}, 81:40--77, 2016.
\newblock \href {https://doi.org/10.1016/j.aam.2016.06.005}
  {\path{doi:10.1016/j.aam.2016.06.005}}.

\bibitem[FS18]{FordSerhiyenko18}
Nicolas Ford and Khrystyna Serhiyenko.
\newblock Green-to-red sequences for positroids.
\newblock {\em J. Combin. Theory Ser. A}, 159:164--182, 2018.
\newblock \href {https://doi.org/10.1016/j.jcta.2018.06.001}
  {\path{doi:10.1016/j.jcta.2018.06.001}}.

\bibitem[FSB22]{FSB}
Chris Fraser and Melissa Sherman-Bennett.
\newblock Positroid cluster structures from relabeled plabic graphs.
\newblock {\em Algebr. Comb.}, 5(3):469--513, 2022.
\newblock \href {http://arxiv.org/abs/2006.10247} {\path{arXiv:2006.10247}},
  \href {https://doi.org/10.5802/alco.220} {\path{doi:10.5802/alco.220}}.

\bibitem[FWZ21]{FWZ}
Sergey Fomin, Lauren Williams, and Andrei Zelevinsky.
\newblock {I}ntroduction to {C}luster {A}lgebras.
\newblock {\em arXiv preprint arXiv:2106.02160}, 2021.

\bibitem[FZ02]{FZI}
Sergey Fomin and Andrei Zelevinsky.
\newblock Cluster algebras. {I}. {F}oundations.
\newblock {\em J. Amer. Math. Soc.}, 15(2):497--529, 2002.
\newblock \href {https://doi.org/10.1090/S0894-0347-01-00385-X}
  {\path{doi:10.1090/S0894-0347-01-00385-X}}.

\bibitem[Gal21]{Galashin21_Critical}
Pavel Galashin.
\newblock Critical varieties in the {G}rassmannian.
\newblock {\em S\'{e}m. Lothar. Combin.}, 85B:Art. 34, 12, 2021.

\bibitem[GHKK18]{GHKK}
Mark Gross, Paul Hacking, Sean Keel, and Maxim Kontsevich.
\newblock Canonical bases for cluster algebras.
\newblock {\em J. Amer. Math. Soc.}, 31(2):497--608, 2018.
\newblock \href {http://arxiv.org/abs/1411.1394} {\path{arXiv:1411.1394}},
  \href {https://doi.org/10.1090/jams/890} {\path{doi:10.1090/jams/890}}.

\bibitem[GK13]{GonKen}
Alexander~B. Goncharov and Richard Kenyon.
\newblock Dimers and cluster integrable systems.
\newblock {\em Ann. Sci. \'{E}c. Norm. Sup\'{e}r. (4)}, 46(5):747--813, 2013.
\newblock \href {https://doi.org/10.24033/asens.2201}
  {\path{doi:10.24033/asens.2201}}.

\bibitem[GKS12]{Sheaves1}
St\'{e}phane Guillermou, Masaki Kashiwara, and Pierre Schapira.
\newblock Sheaf quantization of {H}amiltonian isotopies and applications to
  nondisplaceability problems.
\newblock {\em Duke Math. J.}, 161(2):201--245, 2012.
\newblock \href {https://doi.org/10.1215/00127094-1507367}
  {\path{doi:10.1215/00127094-1507367}}.

\bibitem[GL19]{GL}
Pavel Galashin and Thomas Lam.
\newblock Positroid varieties and cluster algebras.
\newblock Preprint, 2019.
\newblock \href {http://arxiv.org/abs/1906.03501} {\path{arXiv:1906.03501}}.

\bibitem[Gon17]{Goncharov_IdealWebs}
A.~B. Goncharov.
\newblock Ideal webs, moduli spaces of local systems, and 3d {C}alabi-{Y}au
  categories.
\newblock In {\em Algebra, geometry, and physics in the 21st century}, volume
  324 of {\em Progr. Math.}, pages 31--97. Birkh\"{a}user/Springer, Cham, 2017.

\bibitem[GT95]{garg1995upward}
Ashim Garg and Roberto Tamassia.
\newblock Upward planarity testing.
\newblock {\em Order}, 12(2):109--133, 1995.

\bibitem[Kel17]{KellerDT}
Bernhard Keller.
\newblock Quiver mutation and combinatorial dt-invariants.
\newblock Preprint, 2017.
\newblock \href {http://arxiv.org/abs/1709.03143} {\path{arXiv:1709.03143}}.

\bibitem[KLS13]{KLS}
Allen Knutson, Thomas Lam, and David~E. Speyer.
\newblock Positroid varieties: juggling and geometry.
\newblock {\em Compos. Math.}, 149(10):1710--1752, 2013.
\newblock \href {http://arxiv.org/abs/1111.3660} {\path{arXiv:1111.3660}},
  \href {https://doi.org/10.1112/S0010437X13007240}
  {\path{doi:10.1112/S0010437X13007240}}.

\bibitem[Lec16]{Lec}
B.~Leclerc.
\newblock Cluster structures on strata of flag varieties.
\newblock {\em Adv. Math.}, 300:190--228, 2016.
\newblock \href {http://arxiv.org/abs/1402.4435} {\path{arXiv:1402.4435}},
  \href {https://doi.org/10.1016/j.aim.2016.03.018}
  {\path{doi:10.1016/j.aim.2016.03.018}}.

\bibitem[Lus98]{Lusztig98}
G.~Lusztig.
\newblock Total positivity in partial flag manifolds.
\newblock {\em Represent. Theory}, 2:70--78, 1998.
\newblock \href {https://doi.org/10.1090/S1088-4165-98-00046-6}
  {\path{doi:10.1090/S1088-4165-98-00046-6}}.

\bibitem[MS16]{MarshScott16}
R.~J. Marsh and J.~S. Scott.
\newblock Twists of {P}l\"{u}cker coordinates as dimer partition functions.
\newblock {\em Comm. Math. Phys.}, 341(3):821--884, 2016.
\newblock \href {https://doi.org/10.1007/s00220-015-2493-7}
  {\path{doi:10.1007/s00220-015-2493-7}}.

\bibitem[MS17]{MullerSpeyertwist}
Greg Muller and David~E. Speyer.
\newblock The twist for positroid varieties.
\newblock {\em Proc. Lond. Math. Soc. (3)}, 115(5):1014--1071, 2017.
\newblock \href {http://arxiv.org/abs/1606.08383} {\path{arXiv:1606.08383}},
  \href {https://doi.org/10.1112/plms.12056} {\path{doi:10.1112/plms.12056}}.

\bibitem[Oh11]{Oh}
Suho Oh.
\newblock Positroids and {S}chubert matroids.
\newblock {\em J. Combin. Theory Ser. A}, 118(8):2426--2435, 2011.
\newblock \href {http://arxiv.org/abs/0803.1018} {\path{arXiv:0803.1018}},
  \href {https://doi.org/10.1016/j.jcta.2011.06.006}
  {\path{doi:10.1016/j.jcta.2011.06.006}}.

\bibitem[Pos06]{Pos}
Alexander Postnikov.
\newblock Total positivity, {Grassmannians}, and networks.
\newblock Preprint, 2006.
\newblock \href {http://arxiv.org/abs/math/0609764}
  {\path{arXiv:math/0609764}}.

\bibitem[Pre23]{Pressland23}
Matthew Pressland.
\newblock Quasi-coincidence of cluster structures on positroid varieties.
\newblock Preprint, 2023.
\newblock \href {http://arxiv.org/abs/arXiv:2307.13369}
  {\path{arXiv:arXiv:2307.13369}}.

\bibitem[PSBW21]{PSBW}
Matteo Parisi, Melissa Sherman-Bennett, and Lauren Williams.
\newblock The m=2 amplituhedron and the hypersimplex: signs, clusters,
  triangulations, {E}ulerian numbers.
\newblock Preprint, 2021.
\newblock \href {http://arxiv.org/abs/2104.08254} {\path{arXiv:2104.08254}}.

\bibitem[Rie06]{Rietsch06}
K.~Rietsch.
\newblock Closure relations for totally nonnegative cells in {$G/P$}.
\newblock {\em Math. Res. Lett.}, 13(5-6):775--786, 2006.
\newblock \href {https://doi.org/10.4310/MRL.2006.v13.n5.a8}
  {\path{doi:10.4310/MRL.2006.v13.n5.a8}}.

\bibitem[Sco06]{Scott06}
Joshua~S. Scott.
\newblock Grassmannians and cluster algebras.
\newblock {\em Proc. London Math. Soc. (3)}, 92(2):345--380, 2006.
\newblock \href {https://doi.org/10.1112/S0024611505015571}
  {\path{doi:10.1112/S0024611505015571}}.

\bibitem[SSBW19]{SSBW}
K.~Serhiyenko, M.~Sherman-Bennett, and L.~Williams.
\newblock Cluster structures in {S}chubert varieties in the {G}rassmannian.
\newblock {\em Proc. Lond. Math. Soc. (3)}, 119(6):1694--1744, 2019.
\newblock \href {http://arxiv.org/abs/1902.0080} {\path{arXiv:1902.0080}},
  \href {https://doi.org/10.1112/plms.12281} {\path{doi:10.1112/plms.12281}}.

\bibitem[STWZ19]{STWZ}
Vivek Shende, David Treumann, Harold Williams, and Eric Zaslow.
\newblock Cluster varieties from {L}egendrian knots.
\newblock {\em Duke Math. J.}, 168(15):2801--2871, 2019.
\newblock \href {http://arxiv.org/abs/1512.08942} {\path{arXiv:1512.08942}},
  \href {https://doi.org/10.1215/00127094-2019-0027}
  {\path{doi:10.1215/00127094-2019-0027}}.

\bibitem[SW21]{ShenWeng}
Linhui Shen and Daping Weng.
\newblock Cluster structures on double {B}ott-{S}amelson cells.
\newblock {\em Forum Math. Sigma}, 9:Paper No. e66, 89, 2021.
\newblock \href {https://doi.org/10.1017/fms.2021.59}
  {\path{doi:10.1017/fms.2021.59}}.

\end{thebibliography}

\end{document}